\documentclass[10pt, reqno]{amsart}
\usepackage{amssymb}
\usepackage{amsthm}
\usepackage{nccmath}
\usepackage[bookmarks=false]{hyperref}
\usepackage{etoolbox}
\usepackage{array}
\usepackage{xparse}
\usepackage{enumitem}
\usepackage{tikz}
\usepackage{mathtools}
\usepackage[font=scriptsize]{caption}

\calclayout

\newtheorem{thm}{Theorem}[section]

\newtheorem{lem}[thm]{Lemma}
\newtheorem{lemma}[thm]{Lemma}
\newtheorem{prop}[thm]{Proposition}
\newtheorem{cor}[thm]{Corollary}

\theoremstyle{definition}

\theoremstyle{remark}
\newtheorem{rem}[thm]{Remark}

\newcommand*{\thmref}[1]{Theorem~\ref{#1}}
\newcommand*{\corref}[1]{Corollary~\ref{#1}}
\newcommand*{\secref}[1]{\S\ref{#1}}
\newcommand*{\lemref}[1]{Lemma~\ref{#1}} 
\newcommand*{\propref}[1]{Proposition~\ref{#1}}
\newcommand*{\remref}[1]{Remark~\ref{#1}}
\newcommand*{\figref}[1]{Figure~\ref{#1}}

\newcommand*{\dis}{\displaystyle}
\newcommand*{\q}{\quad}
\newcommand*{\qq}{\qquad}
\newcommand*{\lpar}{ }

\newcommand*{\tx}[1]{\text{#1}}
\newcommand*{\Bigskip}{\bigskip \medskip}
\newcommand*{\lamb}{\lambda}
\newcommand{\del}{\delta}

\newcommand*{\ep}{\epsilon}
\newcommand*{\kap}{\kappa}

\newcommand*{\suchthat}{\, \middle| \,}

\newcommand*{\la}{_{\!\! a}}
\newcommand*{\lb}{_{\!\! b}}
\newcommand*{\n}{\!\!}

\newcommand*{\Pminus}{P_{-}}

\newcommand*{\myoverline}[3]{\mkern -#1mu\overline{\mkern#1mu#3\mkern#2mu}\mkern -#2mu}	

\newcommand*{\Zbar}{\myoverline{-3}{0}{\Z}}

\newcommand*{\fbar}{\bar{f}}
\newcommand*{\ubar}{\bar{u}}

\newcommand*{\Dapbar}{\myoverline{-3}{-1}{D}_\ap}
\newcommand*{\wbar}{\myoverline{0}{-0.5}{\omega}}

\newcommand*{\Mconst}{M}

\newcommand*{\half}{\frac{1}{2}}
\newcommand*{\threebytwo}{\frac{3}{2}}
\newcommand*{\fivebytwo}{\frac{5}{2}}
\newcommand*{\sevenbytwo}{\frac{7}{2}}

\newcommand*{\onebythree}{\frac{1}{3}}
\newcommand*{\twobythree}{\frac{2}{3}}
\newcommand*{\fourbythree}{\frac{4}{3}}

\newcommand*{\threebyfour}{\frac{3}{4}}

\newcommand*{\onebysix}{\frac{1}{6}}
\newcommand*{\fivebysix}{\frac{5}{6}}

\newcommand{\eightbyseven}{\frac{8}{7}}

\newcommand*{\Rsp}{\mathbb{R}}
\newcommand*{\Csp}{\mathbb{C}}

\newcommand*{\Zsp}{\mathbb{Z}}

\newcommand*{\Scalsp}{\mathcal{S}}
\newcommand*{\Dcalsp}{\mathcal{D}}

\newcommand*{\Lone}{L^1}
\newcommand*{\Ltwo}{L^2}
\newcommand*{\Lfourbythree}{L^\fourbythree}
\newcommand{\Leightbyseven}{L^\eightbyseven}

\newcommand*{\Linfty}{L^{\infty}}
\newcommand*{\Hhalf}{\dot{H}^\half}
\newcommand*{\Ccal}{\mathcal{C}}
\newcommand*{\Wcal}{\mathcal{W}}

\newcommand*{\Ltwolambhalf}{\Ltwo_{\sqrt{\lamb}}}
\newcommand*{\Linftylambhalf}{\Linfty_{\sqrt{\lamb}}}
\newcommand*{\Hhalflambhalf}{\Hhalf_{\sqrt{\lamb}}}
\newcommand*{\Ccallambhalf}{\Ccal_{\sqrt{\lamb}}}
\newcommand*{\Wcallambhalf}{\Wcal_{\sqrt{\lamb}}}

\newcommand*{\LoneDelta}{\Lone_{\Delta}}
\newcommand*{\LtwoDeltahalf}{\Ltwo_{\sqrt{\Delta}}}
\newcommand*{\LtwoDh}{\LtwoDeltahalf}
\newcommand*{\LinftyDeltahalf}{\Linfty_{\sqrt{\Delta}}}
\newcommand*{\LinftyDh}{\LinftyDeltahalf}
\newcommand*{\HhalfDeltahalf}{\Hhalf_{\sqrt{\Delta}}}
\newcommand*{\HhalfDh}{\HhalfDeltahalf}
\newcommand*{\CcalDeltahalf}{\Ccal_{\sqrt{\Delta}}}
\newcommand*{\CcalDh}{\CcalDeltahalf}
\newcommand*{\WcalDeltahalf}{\Wcal_{\sqrt{\Delta}}}
\newcommand*{\WcalDh}{\WcalDeltahalf}

\newcommand*{\LtwoDeltaonebythree}{\Ltwo_{\Delta^\onebythree}}
\newcommand*{\LinftyDeltaonebysix}{\Linfty_{\Delta^\onebysix}}
\newcommand*{\HhalfDeltaonebysix}{\Hhalf_{\Delta^\onebysix}}

\newcommand*{\approxLtwoDeltahalf}{\approx_{\LtwoDeltahalf}}

\newcommand*{\approxHhalfDeltahalf}{\approx_{\HhalfDeltahalf}}

\newcommand*{\al}{\alpha}
\newcommand*{\ap}{{\alpha'}}

\newcommand*{\bp}{{\beta'}}

\newcommand*{\xp}{{x'}}
\newcommand*{\yp}{{y'}}
\newcommand*{\zp}{{z'}}

\newcommand*{\diff}{\mathop{}\! d}
\newcommand*{\difff}{\mathop{}\!\! d}
\newcommand*{\compose}[1]{\circ{#1}}

\newcommand*{\conv}{*}
\newcommand*{\Hil}{\mathbb{H}}
\newcommand*{\Hcal}{\mathcal{H}}
\newcommand*{\Hcaltil}{\widetilde{\mathcal{H}}}
\newcommand*{\Htil}{\Hcaltil}
\newcommand*{\Pa}{\mathbb{P}_A}
\newcommand*{\Ph}{\mathbb{P}_H}

\newcommand*{\Id}{\mathbb{I}}
\newcommand*{\Imag}{\tx{Im}}
\newcommand*{\Real}{\tx{Re}}

\newcommand*{\grad}{\nabla}
\newcommand*{\Dt}{D_t}

\newcommand*{\pt}{\partial_t}
\newcommand*{\ps}{\partial_s}
\newcommand*{\px}{\partial_x}
\newcommand*{\py}{\partial_y}
\newcommand*{\pz}{\partial_z}
\newcommand*{\pxp}{\partial_\xp}
\newcommand*{\pyp}{\partial_\yp}
\newcommand*{\pzp}{\partial_\zp}
\newcommand*{\pap}{\partial_\ap}
\newcommand*{\papabs}{\abs{\pap}}
\newcommand*{\pbp}{\partial_\bp}

\newcommand*{\Dap}{D_{\ap}}

\newcommand*{\Dapabs}{\abs{D_{\ap}}}
\newcommand*{\Dapfrac}{\frac{1}{\Zap}\pap}
\newcommand*{\Dapbarfrac}{\frac{1}{\Zapbar}\pap}

\newcommand*{\Dapabsfrac}{\frac{1}{\Zapabs}\pap}

\newcommand*{\E}{E} 

\newcommand*{\Ecal}{\mathcal{E}}

\newcommand*{\Esigmazero}{E_{\sigma,0}}
\newcommand*{\Esigmaone}{E_{\sigma,1}}
\newcommand*{\Esigmatwo}{E_{\sigma,2}}
\newcommand*{\Esigmathree}{E_{\sigma,3}}
\newcommand*{\Esigmafour}{E_{\sigma,4}}
\newcommand*{\Esigma}{E_{\sigma}}
\newcommand*{\Ecalsigmaone}{\mathcal{E}_{\sigma,1}}
\newcommand*{\Ecalsigmatwo}{\mathcal{E}_{\sigma,2}}
\newcommand*{\Ecalsigma}{\mathcal{E}_{\sigma}}

\newcommand*{\Ehigh}{E_{high}}
\newcommand*{\Ecalhigh}{\mathcal{E}_{high}}

\newcommand*{\Eaux}{E_{aux}}

\newcommand*{\Ecalaux}{\mathcal{E}_{aux}}

\newcommand*{\EDelta}{E_{\Delta}}
\newcommand*{\EDeltazero}{E_{\Delta,0}}
\newcommand*{\EDeltaone}{E_{\Delta,1}}
\newcommand*{\EDeltatwo}{E_{\Delta,2}}
\newcommand*{\EDeltathree}{E_{\Delta,3}}
\newcommand*{\EDeltafour}{E_{\Delta,4}}

\newcommand*{\EcalDelta}{\mathcal{E}_{\Delta}}
\newcommand*{\EcalDeltazero}{\mathcal{E}_{\Delta,0}}
\newcommand*{\EcalDeltaone}{\mathcal{E}_{\Delta,1}}
\newcommand*{\EcalDeltatwo}{\mathcal{E}_{\Delta,2}}

\newcommand*{\Fcal}{\mathcal{F}}

\newcommand*{\Aone}{A_1}

\newcommand*{\bvar}{b}

\newcommand*{\bap}{\bvar_\ap}
\newcommand*{\bvarap}{\bap}

\newcommand*{\h}{h}

\newcommand*{\hvart}{\h_t}
\newcommand*{\hal}{\h_\al}

\newcommand*{\hinv}{\h^{-1}}
\newcommand*{\htil}{\widetilde{\h}}
\newcommand*{\htilap}{\htil_\ap}
\newcommand*{\htilbp}{\htil_\bp}

\newcommand*{\Util}{\widetilde{U}}

\newcommand*{\g}{g}

\newcommand*{\thvar}{\theta}
\newcommand*{\Th}{\Theta}

\newcommand*{\f}{f}

\newcommand*{\w}{\omega}

\newcommand*{\Psizp}{\Psi_{\zp}}

\newcommand*{\Jone}{J_1}
\newcommand*{\Jtwo}{J_2}

\newcommand*{\Rone}{R_1}
\newcommand*{\Rtwo}{R_2}
\newcommand*{\Rthree}{R_3}
\newcommand*{\Rfour}{R_4}

\newcommand{\U}{U}

\newcommand*{\z}{z}

\newcommand*{\zal}{\z_\al}

\newcommand*{\zt}{\z_t}

\newcommand*{\ztal}{\z_{t\al}}

\newcommand*{\ztt}{\z_{tt}}

\newcommand*{\zttal}{\z_{tt\al}}

\newcommand*{\Z}{Z}

\newcommand*{\Zap}{\Z_{,\ap}}

\newcommand*{\Zaphalf}{\Zap^{1/2}}
\newcommand*{\Zapbar}{\Zbar_{,\ap}}

\newcommand*{\Zbarap}{\Zapbar}
\newcommand*{\Zapabs}{\abs{\Zap}}

\newcommand*{\Zt}{\Z_t}

\newcommand*{\Ztbar}{\Zbar_t}

\newcommand*{\Ztap}{\Z_{t,\ap}}

\newcommand*{\Ztapbar}{\Zbar_{t,\ap}}

\newcommand*{\Ztbarap}{\Ztapbar}

\newcommand*{\Ztt}{\Z_{tt}}

\newcommand*{\Zttbar}{\Zbar_{tt}}

\newcommand*{\Zttap}{\Z_{tt,\ap}}
\newcommand*{\Zttapbar}{\Zbar_{tt,\ap}}
\newcommand*{\Zttbarap}{\Zttapbar}

\newcommand*{\Ztttbar}{\Zbar_{ttt}}

\newcommand*{\nobrac}[1]{ #1 }

\DeclarePairedDelimiter{\oldbrac}{\lparen}{\rparen}			
\NewDocumentCommand{\brac}{ s o m }{						
	\IfBooleanT{#1}{
  		\IfValueT{#2}{\oldbrac[#2]{#3}}
		\IfValueF{#2}{\oldbrac{#3}} 
	}
	\IfBooleanF{#1}{
  		\IfValueT{#2}{\PackageError{mypackage}{Incorrect use of brac. Insert star}{}}
		\IfValueF{#2}{\oldbrac*{#3}} 
	}		
}

\DeclarePairedDelimiter\oldcbrac{\lbrace}{\rbrace}				
\NewDocumentCommand{\cbrac}{ s o m }{					
	\IfBooleanT{#1}{
  		\IfValueT{#2}{\oldcbrac[#2]{#3}}
		\IfValueF{#2}{\oldcbrac{#3}} 
	}
	\IfBooleanF{#1}{
  		\IfValueT{#2}{\PackageError{mypackage}{Incorrect use of cbrac. Insert star}{}}
		\IfValueF{#2}{\oldcbrac*{#3}} 
	}		
}

\DeclarePairedDelimiter\oldsqbrac{\lbrack}{\rbrack}				
\NewDocumentCommand{\sqbrac}{ s o m }{					
	\IfBooleanT{#1}{
  		\IfValueT{#2}{\oldsqbrac[#2]{#3}}
		\IfValueF{#2}{\oldsqbrac{#3}} 
	}
	\IfBooleanF{#1}{
  		\IfValueT{#2}{\PackageError{mypackage}{Incorrect use of sqbrac. Insert star}{}}
		\IfValueF{#2}{\oldsqbrac*{#3}} 
	}		
}

\DeclarePairedDelimiter{\oldabs}{\lvert}{\rvert}
\NewDocumentCommand{\abs}{ s o m }{						
	\IfBooleanT{#1}{
  		\IfValueT{#2}{\oldabs[#2]{#3}}
		\IfValueF{#2}{\oldabs{#3}} 
	}
	\IfBooleanF{#1}{
  		\IfValueT{#2}{\PackageError{mypackage}{Incorrect use of abs. Insert star}{}}
		\IfValueF{#2}{\oldabs*{#3}} 
	}		
}

\DeclarePairedDelimiterX{\oldnorm}[1]{\lVert}{\rVert}{#1}
\NewDocumentCommand{\norm}{ s o o m }{					
	\IfValueT{#2} {
		\IfBooleanT{#1}{
  			\IfValueT{#3}{\oldnorm[#2]{#4}_{#3}}
			\IfValueF{#3}{\oldnorm{#4}_{#2}} 
		}
		\IfBooleanF{#1}{
  			\IfValueT{#3}{\PackageError{mypackage}{Incorrect use of norm. Insert star}{}}
			\IfValueF{#3}{\oldnorm*{#4}_{#2}} 
		}
	}
	\IfValueF{#2} {
		\IfBooleanT{#1}{\oldnorm{#4}}	
		\IfBooleanF{#1}{\oldnorm*{#4}}		
	}	
}

\makeatletter
\def\black@#1{%
    \noalign{%
        \ifdim#1>\displaywidth
            \dimen@\prevdepth
            \nointerlineskip
            \vskip-\ht\strutbox@
            \vskip-\dp\strutbox@
            \vbox{\noindent\hbox to \displaywidth{\hbox to#1{\strut@\hfill}}}%
            \prevdepth\dimen@
        \fi
    }%
}
\makeatother

\makeatletter
\renewcommand{\tocsection}[3]{%
  \indentlabel{\@ifnotempty{#2}{\bfseries\ignorespaces#1 #2\quad}}\bfseries#3}
\renewcommand{\tocsubsection}[3]{%
  \indentlabel{\@ifnotempty{#2}{\ignorespaces#1 #2\quad}}#3}

\newcommand\@dotsep{4.5}
\def\@tocline#1#2#3#4#5#6#7{\relax
  \ifnum #1>\c@tocdepth 
  \else
    \par \addpenalty\@secpenalty\addvspace{#2}%
    \begingroup \hyphenpenalty\@M
    \@ifempty{#4}{%
      \@tempdima\csname r@tocindent\number#1\endcsname\relax
    }{%
      \@tempdima#4\relax
    }%
    \parindent\z@ \leftskip#3\relax \advance\leftskip\@tempdima\relax
    \rightskip\@pnumwidth plus1em \parfillskip-\@pnumwidth
    #5\leavevmode\hskip-\@tempdima{#6}\nobreak
    \leaders\hbox{$\m@th\mkern \@dotsep mu\hbox{.}\mkern \@dotsep mu$}\hfill
    \nobreak
    \hbox to\@pnumwidth{\@tocpagenum{\ifnum#1=1\bfseries\fi#7}}\par
    \nobreak
    \endgroup
  \fi}
\AtBeginDocument{%
\expandafter\renewcommand\csname r@tocindent0\endcsname{0pt}
}
\def\l@subsection{\@tocline{2}{0pt}{2.5pc}{5pc}{}}
\makeatother

\makeatletter
 \def\@testdef #1#2#3{%
   \def\reserved@a{#3}\expandafter \ifx \csname #1@#2\endcsname
  \reserved@a  \else
 \typeout{^^Jlabel #2 changed:^^J%
 \meaning\reserved@a^^J%
 \expandafter\meaning\csname #1@#2\endcsname^^J}%
 \@tempswatrue \fi}
\makeatother

\makeatletter
\newcommand*{\rom}[1]{\expandafter\@slowromancap\romannumeral #1@}
\makeatother

\makeatletter
\patchcmd{\@sect}{\@addpunct.}{}{}{}
\patchcmd{\subsection}{-.5em}{1em}{}{}
\makeatother

\begin{document}

\title[water waves]{Angled crested like water waves with surface tension \rom{2}: Zero surface tension limit}
\author{Siddhant Agrawal}
\address{Department of Mathematics, University of Massachusetts Amherst, MA 01003}
\email{agrawal@math.umass.edu}

\begin{abstract}
This is the second paper in a series of papers analyzing angled crested like water waves with surface tension. We consider the 2D capillary gravity water wave equation and assume that the fluid is inviscid, incompressible, irrotational and the air density is zero. In the first paper \cite{Ag19} we constructed a weighted energy which generalizes the energy of Kinsey and Wu \cite{KiWu18} to the case of non-zero surface tension, and proved a local wellposedness result. In this paper we prove that under a suitable scaling regime, the zero surface tension limit of these solutions with surface tension are solutions to the gravity water wave equation which includes waves with angled crests. 
\end{abstract}


\maketitle

\tableofcontents

\section{Introduction}

This is the second paper in a series of papers analyzing angled crested like water waves with surface tension. As in the first paper we will identify 2D vectors with complex numbers. We consider a 2D fluid which we assume to be inviscid, incompressible and irrotational and the fluid is subject to a uniform gravitational field $-i$ acting in the downward direction. The fluid domain $\Omega(t)$ and the air region are separated by an interface $\Sigma(t)$ which we assume to be homeomorphic to $\Rsp$ and which tends to the real line at infinity. We do not assume that the interface is a graph. The air and the fluid are assumed to have constant densities of 0 and 1 respectively. The fluid is below the air region and its motion is governed by the Euler equation
\begin{equation}\label{eq:Euler}
\begin{aligned}
& \mathbf{v_t + (v.\nabla)v} = -i -\nabla P  \qq\text{ on } \Omega (t) 	\\
& \tx{div } \mathbf{ v} = 0, \quad \tx{curl } \mathbf{ v }=0 \qq\text{ on } \Omega(t) 
\end{aligned}
\end{equation}
along with the boundary conditions
\begin{equation}\label{eq:Eulerboundary}
\begin{aligned}
& P = -\sigma \partial_{s}\thvar \qq\text{ on } \Sigma (t) \\
& (1,\mathbf{v}) \tx{ is tangent to the free surface } (t, \Sigma(t)) \\
& \mathbf{v} \to 0, \grad P \to -i \qq\text{ as } |(x,y)| \to \infty
\end{aligned}
\end{equation}
Here $\thvar = $ angle the interface makes with the $x -axis$, $\ps = $ arc length derivative, $\sigma = $ coefficient of surface tension $\geq0$. 

The earliest results on local well-posedness for the Cauchy problem are for small data in 2D and were obtained by Nalimov \cite{Na74}, Yoshihara \cite{Yo82,Yo83} and Craig \cite{Cr85}. In the case of zero surface tension, Wu \cite{Wu97,Wu99} obtained the proof of local well-posedness for arbitrary data in Sobolev spaces. This result has been extended in various directions, see the works in \cite{ChLi00, Li05, La05, ZhZh08, CaCoFeGaGo13, AlBuZu14, HuIfTa16, AlBuZu18, HaIfTa17, Po19, Ai17, AiIfTa19, Ai20}. In the case of non-zero surface tension, the local well-posedness of the equation in Sobolev spaces was established by Beyer and Gunther in \cite{BeGu98}. See also related works in \cite{Ig01, Am03, Sc05, CoSh07, ShZe11, AlBuZu11, CaCoFeGaGo12, Ng17}. The zero surface tension limit of the water wave equation in Sobolev spaces was proved by Ambrose and Masmoudi \cite{AmMa05, AmMa09}. See also the works in \cite{OgTa02, ShZe08, MiZh09,CaCoFeGaGo12, ShSh19}.

All of the above results are for interfaces with regularity at least $C^{1,\alpha}$ for some $\alpha >0$. In an important work Kinsey and Wu \cite{KiWu18} proved an a priori estimate for angled crested water waves in the case of zero surface tension. Using this Wu \cite{Wu19} proved a local wellposedness result which allows for angled crested interfaces as initial data. In \cite{Ag20} the author proved that these singular solutions are rigid and in particular the angle of the corner does not change with time. 

In the first part of this series of papers \cite{Ag19}, we took the first step in extending this theory of angled crested water waves to the case of non-zero surface tension. We constructed an energy functional $\Ecalsigma(t)$ which generalizes the energy of \cite{KiWu18} to the case of $\sigma \geq 0$, and proved a local wellposedness result based on this energy (see \thmref{thm:existence}). For initial data in appropriate Sobolev spaces, this existence result gives us a uniform time of existence of solutions for $0\leq \sigma \leq \sigma_0$ for arbitrary $\sigma_0>0$,  thereby recovering the uniform time of existence result of Ambrose and Masmoudi \cite{AmMa05} in this case. In addition to this, the energy $\Ecalsigma(t)$ has several interesting properties: for example if $\sigma = 0$ then it reduces to a lower order version of the energy of  \cite{KiWu18} and allows angled crested interfaces, however for $\sigma>0$ the energy $\Ecalsigma(t)$ does not allow any singularties of the interface. On the other hand the energy does allow large curvature when $\sigma>0$ and in particular the $\Linfty$ norm of the initial curvature can be as large as $\sigma^{-\onebythree + \ep}$ for any $\ep>0$ (see \corref{cor:example}).  This growth rate of $\sigma^{-\onebythree}$ is explained by the fact that the quantity $\norm*[\infty]{\sigma^\onebythree \kap}$ where $\kap$ is the curvature, is a scaling invariant quantity for the problem of zero surface tension limit (see the introduction and Remark 3.4 of \cite{Ag19} for more details). 

In this paper we continue the study of angled crested like water waves and consider the zero surface tension limit. We construct a weighted energy $\EcalDelta(t)$ for the difference of two solutions of the water wave equation, one with zero surface tension and one with surface tension $\sigma$, and prove in our main result \thmref{thm:convergence} that $\EcalDelta(t) \to 0$ as the surface tension $\sigma \to 0$. This is a generalization of the convergence result of \cite{AmMa05} where convergence is proven in Sobolev spaces. The advantage of using this energy $\EcalDelta(t)$ for the difference of the solutions, instead of the usual Sobolev norms, is that the rate of growth of the energy $\EcalDelta(t)$ does not depend on the Sobolev norms of the initial data but only on weighted norms such as $\Ecalsigma(0)$. In particular the growth rate does not depend on the $C^{1,\alpha}$ norm of the initial interface (for any $0<\alpha \leq 1$), which is assumed in all previous results on zero surface tension limit. Hence this result allows us to control the difference of the solutions independent of how close the initial interface is to an angled crested interface. 

As an application of our main result \thmref{thm:convergence}, we show that under a suitable scaling regime, smooth solutions of the water wave equation with surface tension converge to the singular solution of the water wave equation with zero surface tension. More precisely we consider an initial data $(\Z,\Zt)(0)$ with angled crested interface and consider the singular solution $(\Z,\Zt)(t)$ to the water wave equation with zero surface tension as obtained from \cite{Wu19}. We then consider smooth solutions $(\Z^{\ep,\sigma}, \Zt^{\ep,\sigma})(t)$ to the water wave equation with surface tension $\sigma$ with initial data  $(\Z^{\epsilon,\sigma}, \Zt^{\epsilon,\sigma})(0) = (\Z\conv P_\epsilon, \Zt\conv P_\epsilon)(0)$ where $P_\epsilon$ is the Poisson kernel. We show in \corref{cor:examplenew} that if $\max\cbrac*[\big]{\sigma/\ep^\threebytwo, \ep} \to 0$, then the solutions $(\Z^{\ep,\sigma}, \Zt^{\ep,\sigma})(t) \to (\Z,\Zt)(t)$ in a suitable norm. The existence of these solutions  $(\Z^{\ep,\sigma}, \Zt^{\ep,\sigma})(t)$ on a uniform time interval was already shown in \cite{Ag19} (see \corref{cor:example}) and this factor of $\sigma/\ep^\threebytwo$ already appeared there (see Section 3.1 of \cite{Ag19} for more details). The novelty of \corref{cor:examplenew} is the convergence aspect. 

The proof of \thmref{thm:convergence} is based on energy estimates. The usual strategy of directly subtracting terms in the energy $\Ecalsigma(t)$ and using that as the energy for the difference of the solutions fails in our case because of the weighted nature of the energy $\Ecalsigma(t)$. To remedy this, we introduce a coupling term $\sigma(\Ecalaux)_b(t)$ which couples a carefully constructed weighted energy $(\Ecalaux)_b(t)$ of the zero surface tension solution, with the surface tension $\sigma$ coming from the solution with surface tension.  This coupling term $\sigma(\Ecalaux)_b(t)$ is a part of the energy $\EcalDelta(t)$ and is crucial to closing the energy estimate of $\EcalDelta(t)$ at the highest order. We believe that using such coupling terms could be useful in other convergence problems more generally, especially when dealing with weighted Sobolev norms. We discuss the necessity and usefulness of this coupling term in more detail in \secref{sec:discussion}. 

The paper is organized as follows: In \secref{sec:notation} we introduce the notation, the main system of equations and recall the results from \cite{Ag19}. In \secref{sec:resultsanddiscussion} we state our main results and discuss the main ideas behind the proof. Then in \secref{sec:part1} we collect all the identities and estimates from \cite{Ag19} that we frequently use in our calculations. In \secref{sec:aprioriEcalhigh} we prove an a priori estimate for the energy $\Ecalhigh(t)$ which is a weighted energy for the zero surface tension solutions, and this is used in the a priori estimate for $\Ecalaux(t)$. In \secref{sec:aprioriEcalaux} we prove an a priori estimate for the energy $\Ecalaux(t)$ which is also a weighted energy for the zero surface tension solution. As explained before this energy is crucial in proving \thmref{thm:convergence}. In \secref{sec:aprioriEcalDelta} we prove an a priori estimate for our main energy $\EcalDelta(t)$. Finally  in \secref{sec:proof} we complete the proofs of our results \thmref{thm:convergence} and \corref{cor:examplenew}. In appendix A \secref{sec:appendixA} we collect all the identities and estimates from the appendix of \cite{Ag19} which are needed for this paper and in appendix B \secref{sec:appendixB} we prove some new estimates that we use throughout the paper. 
\bigskip

\noindent \textbf{Acknowledgment}: This work was part of the author's Ph.D. thesis and he is very grateful to his advisor Prof. Sijue Wu for proposing the problem and for her guidance during this project. The author would also like to thank Prof. Jeffrey Rauch for many helpful discussions. The author was supported in part by NSF Grants DMS-1101434, DMS-1361791 through his advisor.


\section{Notation and previous work from part \rom{1}}\label{sec:notation}
\bigskip

\subsection{Notation}

In this section we recall the notation used in \cite{Ag19} and explain the main results proved there.  The Fourier transform is defined as
\[
\hat{f}(\xi) = \frac{1}{\sqrt{2\pi}}\int e^{-ix\xi}f(x) \diff x
\] 
We will denote by $\Scalsp(\Rsp)$ the Schwartz space of rapidly decreasing functions and $\Scalsp'(\Rsp)$ is the space of tempered distributions. A Fourier multiplier with symbol $a(\xi)$ is the operator $T_a$ defined formally by the relation $\dis \widehat{T_a{f}} = a(\xi)\hat{f}(\xi)$. The operators $\papabs^s $ for $s\in\Rsp$ are defined as the Fourier multipliers with symbol $\abs{\xi}^s$. 
The Sobolev space $H^s(\Rsp)$ for $s\geq 0$  is the space of functions with  $\norm[H^s]{f} = \norm*[\Ltwo(\diff \xi)]{(1+\abs{\xi}^2)^{\frac{s}{2}}\hat{f}(\xi)} < \infty$. The homogenous Sobolev space $\Hhalf(\Rsp)$ is the space of functions modulo constants with  $\norm[\Hhalf]{f} = \norm*[\Ltwo(\diff \xi)]{\abs{\xi}^\half \hat{f}(\xi)} < \infty$.  The Poisson kernel is given by
\begin{align}\label{eq:Poissonkernel}
K_\ep(x) = \frac{\ep}{\pi(\ep^2 + x^2)} \qquad \tx{ for } \ep>0
\end{align}

From now on compositions of functions will always be in the spatial variables. We write $f = f(\cdot,t), g = g(\cdot,t), f \compose g(\cdot,t) :=  f(g(\cdot,t),t)$. Define the operator $U_g$ as given by $U_g f = f\compose g$. Observe that $U_f U_g = U_{g\compose f}$.  Let $[A,B] := AB - BA$ be the commutator of the operators $A$ and $B$. If $A$ is an operator and $f$ is a function, then $(A + f)$ will represent the addition of the operators A and the multiplication operator $T_f$ where $T_f (g) = fg$.  We denote the convolution of $f$ and $g$ by $f \conv g$. We will denote the spacial coordinates in $\Omega(t) $ with $z = x+iy$, whereas $\zp = \xp + i\yp$ will denote the coordinates in the lower half plane $\Pminus = \cbrac{(x,y) \in \Rsp^2 \suchthat y<0}$. As we will frequently work with holomorphic functions, we will use the holomorphic derivatives $\pz = \half(\px-i\py)$ and $\pzp = \half(\pxp-i\pyp)$. In this paper all norms will be taken in the spacial coordinates unless otherwise specified. For example for a function $f:\Rsp\times [0,T] \to \Csp$ we write $\norm[2]{f} = \norm[2]{f(\cdot,t)} = \norm[\Ltwo(\Rsp,\diff \ap)]{f(\cdot,t)}$. Also for a function $f:\Pminus \to \Csp$ we write $\sup_{\yp<0}\norm[\Ltwo(\Rsp,\diff\xp)]{f} = \sup_{\yp<0}\norm[\Ltwo(\Rsp,\diff\xp)]{f(\cdot,\yp)}$. 

We write $a \lesssim b$ if there exists a universal constant $C>0$ so that $a\leq Cb$. This notation may be changed in different sections to simplify calculations in those sections. For functions $f_1,f_2,f_3:\Rsp \to \Csp$ we define the function $\sqbrac{f_1,f_2 ; f_3}:\Rsp \to \Csp$ as 
\begin{align}\label{eq:foneftwofthree}
\sqbrac{f_1, f_2;  f_3}(\ap) = \frac{1}{i\pi} \int \brac{\frac{f_1(\ap) - f_1(\bp)}{\ap - \bp}}\brac{\frac{f_2(\ap) - f_2(\bp)}{\ap-\bp}} f_3(\bp) \diff \bp
\end{align}
and if $g:\Rsp \to \Rsp$ is a homeomorphism then we define $\sqbrac{f_1,f_2 ; f_3}_{g}:\Rsp \to \Csp$ as 
\begin{align}\label{eq:foneftwofthreeg}
\sqbrac{f_1, f_2;  f_3}_{g}(\ap) = \frac{1}{i\pi} \int \brac{\frac{f_1(\ap) - f_1(\bp)}{g(\ap) - g(\bp)}}\brac{\frac{f_2(\ap) - f_2(\bp)}{g(\ap) - g(\bp)}} f_3(\bp) \diff \bp
\end{align}

Let the interface be parametrized in Lagrangian coordinates by $\z(\cdot,t): \Rsp \to \Sigma(t)$ satisfying $\z_{\al}(\al,t) \neq 0 $ for all  $\al \in \Rsp$. Hence $\zt(\al,t) = \mathbf{v}(\z(\al,t),t)$ is the velocity of the fluid on the interface and $\ztt(\al,t) = (\mathbf{v_t + (v.\nabla)v})(\z(\al,t),t)$ is the acceleration. 

Let $\Psi(\cdot,t): \Pminus \to  \Omega(t)$ be conformal maps satisfying $\lim_{\z\to \infty} \Psi_\z(\z,t) =1$  and $\lim_{\z\to \infty} \Psi_t(\z,t) =0$.  With this, the only ambiguity left in the definition of $\Psi$ is that of the choice of translation of the conformal map at $t=0$, which does not play any role in the analysis. Let $\Phi(\cdot,t):\Omega(t) \to \Pminus $ be the inverse of the map $\Psi(\cdot,t)$ and define $\h(\cdot,t):\Rsp \to \Rsp$ as
\begin{align*}
\h(\al,t) = \Phi(\z(\al,t),t)
\end{align*}
hence $\h(\cdot,t)$ is a homeomorphism. As we use both Lagrangian and conformal parameterizations, we will denote the Lagrangian parameter by $\al$ and the conformal parameter by $\ap$. Let $\hinv(\cdot,t)$ be its spacial inverse i.e.
\[
\h(\hinv(\ap,t),t) = \ap
\] 
From now on, we will fix our Lagrangian parametrization at $t=0$ by imposing
\begin{align*}
h(\al,0)= \al \quad  \tx { for all } \al \in \Rsp
\end{align*}
Hence the Lagrangian parametrization is the same as conformal parametrization at $t=0$. Define the variables
\[
\begin{array}{l l l}
 \Z(\ap,t) = \z\compose \hinv (\ap,t)  & \Zap(\ap,t) = \pap \Z(\ap,t) &  \quad  \tx{ Hence } \quad  \brac*[]{\dfrac{\zal}{\hal}} \compose \hinv = \Zap \\
  \Zt(\ap,t) = \zt\compose \hinv (\ap,t)  & \Ztap(\ap,t) = \pap \Zt(\ap,t) &  \quad   \tx{ Hence } \quad  \brac*[]{\dfrac{\ztal}{\hal}} \compose \hinv = \Ztap \\
 \Ztt(\ap,t) = \ztt\compose \hinv (\ap,t)  & \Zttap(\ap,t) = \pap \Ztt(\ap,t) &  \quad   \tx{ Hence } \quad  \brac*[]{\dfrac{\zttal}{\hal}} \compose \hinv = \Zttap \\
\end{array}
\]
Hence $\Z(\ap,t), \Zt(\ap,t)$ and $\Ztt(\ap,t)$ are the parameterizations of the boundary, the velocity and the acceleration in conformal coordinates and in particular $\Z(\cdot,t)$ is the boundary value of the conformal map $\Psi(\cdot,t)$. Note that as $\Z(\ap,t) = \z(\hinv(\ap,t),t)$ we see that  $\pt \Z \neq \Zt$. Similarly $\pt \Zt \neq \Ztt$.  The substitute for the time derivative is the material derivative. Define the operators
\begingroup
\allowdisplaybreaks
\begin{fleqn}
\begin{align}\label{eq:mainoperators}
\begin{split}
\qquad & \Dt =  \tx{material derivative} = \pt + \bvar\pap  \qquad \tx{ where } \bvar = \hvart \compose \hinv \\
 & \Dap = \Dapfrac \qquad \Dapbar = \Dapbarfrac \qquad \Dapabs = \Dapabsfrac \\
& \Hil  =   \text{Hilbert transform  = Fourier multiplier with symbol } -sgn(\xi) \\
& \qquad  \Hil f(\alpha ' ) =  \frac{1}{i\pi} p.v. \int \frac{1}{\alpha ' - \beta'}f(\beta') \diff\beta' \\
& \Ph = \text{Holomorphic projection} = \frac{\Id + \Hil}{2} \\
& \Pa = \text{Antiholomorphic projection} = \frac{\Id - \Hil}{2} \\ 
& \papabs   = \ i\Hil \partial_{\alpha'} = \sqrt{-\Delta} = \text{ Fourier multiplier with symbol } |\xi| \ \\
& \papabs^{1/2}  = \text{ Fourier multiplier with symbol } \abs{\xi}^{1/2} 
\end{split}
\end{align}
\end{fleqn}
\endgroup
Now we have $\Dt \Z = \Zt$ and $\Dt \Zt = \Ztt$ and more generally $\Dt (f(\cdot,t)\compose \hinv) = (\pt f(\cdot,t)) \compose \hinv$ or equivalently $\pt (F(\cdot,t)\compose \h) = (\Dt F(\cdot,t)) \compose \h$. This means that $\Dt = U_h^{-1}\pt U_h$ i.e. $\Dt$ is the material derivative in conformal coordinates. We also define a few more variables related to the interface:
\begin{fleqn}
\begin{align}\label{eq:mainvariables}
\begin{split}
& \g =\Imag(\log(\Zap)) \qquad  \tx{ Hence }\quad \Dapabs \g = -i\Dap \frac{\Zap}{\Zapabs}  \\
& \Th = (\Id+\Hil)\Dapabs \g = -i(\Id+\Hil) \Dap \frac{\Zap}{\Zapabs} \\
& \w = e^{i\g} = \frac{\Zap}{\Zapabs} \qquad  \tx{ Hence } \quad \Dapabs \w = i \w\Real \Th
\end{split}
\end{align} 
\end{fleqn} 
Observe that $g$ is the angle the interface makes with the x-axis in conformal coordinates, i.e. $g = \thvar\compose\hinv$ where $\frac{\zal}{\abs{\zal}} = e^{i\thvar}$. Hence $\Real \Th = \kap\compose\hinv$ where $\kap$ is the curvature of the interface.

\medskip
\subsection{The system}

To solve the system \eqref{eq:Euler}, \eqref{eq:Eulerboundary} in \cite{Ag19} we obtained a  system for the variables $(\Zap,\Zt)$ which we then solve. The system is as follows:
\begin{align}\label{eq:systemone}
\begin{split}
\bvar & = \Real(\Id - \Hil)\brac{\frac{\Zt}{\Zap}} \\
\Aone & = 1 - \Imag \sqbrac{\Zt,\Hil}\Ztapbar \\
(\pt + \bvar\pap)\Zap &= \Ztap - \bap\Zap  \\
(\pt + \bvar\pap)\Ztbar & = i -i\frac{\Aone}{\Zap} + \frac{\sigma}{\Zap}\pap(\Id + \Hil)\cbrac{\Imag\brac{\frac{1}{\Zap}\pap \frac{\Zap}{\Zapabs} }}
\end{split}
\end{align}
along with the condition that their harmonic extensions, namely $\Psizp(\cdot + iy) = K_{-y}\conv \Zap$ and $\U(\cdot + iy) = K_{-y}\conv \Ztbar$ for all $y<0$, \footnote{here $K_{-y}$ is the Poisson kernel \eqref{eq:Poissonkernel}} are holomorphic functions on $\Pminus$ and satisfy \footnote{We observe that for such a $\Psizp$  we can uniquely define $\log(\Psizp) : \Pminus \to \Csp$ such that $\log(\Psizp)$ is a continuous function with $\Psizp = \exp\cbrac{\log(\Psizp)}$ and $(\log(\Psizp))(\zp) \to 0$ as $\zp \to \infty$. }
\begin{align*}
\lim_{c \to \infty} \sup_{\abs{\zp}\geq c}\cbrac{\abs{\Psizp(\zp) - 1} + \abs{\U(\zp)}}  = 0 
\qquad \tx{ and } \quad \Psizp(\zp) \neq 0 \quad \tx{ for all } \zp \in \Pminus 
\end{align*}
After solving the above system one can obtain $\Z(\cdot,t)$ by the formula 
\begin{align*}
\Z(\ap,t) = \Z(\ap,0) +  \int_0^t \cbrac{\Zt(\ap,s) - \bvar(\ap,s)\Zap(\ap,s)} \diff s
\end{align*}
and hence $(\pt + \bvar\pap)\Z = \Dt\Z = \Zt$. Hence one can view the system being in variables $(\Z,\Zt)$ instead of the variables $(\Zap,\Zt)$,

We observe that the above system allows self intersecting interfaces. However if the interface is self-intersecting then it becomes nonphysical and so its relation to the Euler equation  \eqref{eq:Euler}, \eqref{eq:Eulerboundary} is lost. See \cite{Ag19} for more details. 

From the calculation in \cite{KiWu18} we have $\Aone \geq 1$. Now to get the function $\h(\al,t)$, we solve the ODE
\begin{align}\label{eq:h}
\begin{split}
\frac{\diff h}{\diff t} &= \bvar(\h, t) \\
h(\al,0) &= \al
\end{split}
\end{align}
Observe that as long as $\sup_{[0,T]} \norm[\infty]{\bvarap}(t) <\infty$ we can solve this ODE uniquely and for any $t\in [0,T]$ we have that $h(\cdot,t)$ is a homeomorphism. Hence it makes sense to talk about the functions $z = \Z\compose \h, \zt = \Zt \compose \h$ which are Lagrangian parametrizations of the interface and the velocity on the boundary. We also note that the last  equation in \eqref{eq:systemone} can be written as 
\begin{align}  \label{form:Zttbar}
\Zttbar -i = -i \frac{\Aone}{\Zap} +\sigma \Dap \Th
\end{align}

\subsection{Previous result}\label{sec:previousresult}

Let us now describe the main result of \cite{Ag19}. For $\sigma\geq 0$ define the energy
\begingroup
\allowdisplaybreaks
\begin{align*}
\Ecalsigmaone & =  \norm[2]{\pap\frac{1}{\Zap}}^2 + \norm*[\bigg][\Hhalf]{\frac{1}{\Zap}\pap\frac{1}{\Zap}}^2  +  \norm[\Hhalf]{\sigma\pap\Th}^2 + \norm*[\Bigg][2]{\sigma^\onebysix\Zap^\half\pap\frac{1}{\Zap}}^6 + \norm*[\Bigg][\infty]{\sigma^\half\Zap^\half\pap\frac{1}{\Zap}}^2  \\*
& \quad  +  \norm*[\Bigg][2]{\frac{\sigma^\half}{\Zap^\half}\pap^2\frac{1}{\Zap}}^2  + \norm*[\Bigg][\Hhalf]{\frac{\sigma^\half}{\Zap^\threebytwo}\pap^2\frac{1}{\Zap}}^2     + \norm[2]{\frac{\sigma}{\Zap}\pap^3\frac{1}{\Zap}}^2 + \norm[\Hhalf]{\frac{\sigma}{\Zap^2}\pap^3\frac{1}{\Zap}}^2 \\
\Ecalsigmatwo &= \norm[2]{\Ztapbar}^2 + \norm*[\Bigg][2]{\frac{1}{\Zap^2}\pap\Ztapbar}^2 + \norm*[\Bigg][2]{\frac{\sigma^\half}{\Zap^\half}\pap\Ztapbar}^2 + \norm*[\Bigg][2]{\frac{\sigma^\half}{\Zap^\fivebytwo}\pap^2\Ztapbar}^2 \\
\Ecalsigma & = \Ecalsigmaone + \Ecalsigmatwo
\end{align*}
\endgroup

\begin{thm}[\cite{Ag19}]\label{thm:existence}
Let $\sigma >0$ and assume the initial data $(\Z,\Zt)(0)$ satisfies $(\Zap-1,\frac{1}{\Zap} - 1, \Zt)(0) \in H^{3.5}(\Rsp)\times H^{3.5}(\Rsp)\times H^3(\Rsp)$. Then $ \Ecalsigma(0) < \infty$ and there exists $T,C_1 > 0$ depending only on $\Ecalsigma(0)$ such that the initial value problem to  \eqref{eq:systemone} has a unique solution $(\Z,\Zt)(t)$ in the time interval $[0,T]$ satisfying  
$(\Zap-1,\frac{1}{\Zap} - 1, \Zt) \in C^l([0,T], H^{3.5 - \threebytwo l}(\Rsp)\times H^{3.5 - \threebytwo l}(\Rsp)\times H^{3 - \threebytwo l}(\Rsp))$ 
for $l = 0,1$ and $\dis \sup_{t\in [0,T]} \Ecalsigma (t)\leq C_1 <\infty$
\end{thm}

We would like to emphasize that the above result does not allow angled crested interfaces when $\sigma >0$. The main feature of this existence result is that the time of existence depends only on $\Ecalsigma(0)$ and not on the $C^{1,\alpha}$ norm of the initial data (for any $0<\alpha \leq 1$). The energy $\Ecalsigma(t)$ can be viewed as a weighted Sobolev norms of $(\Z,\Zt)$ with the weight being powers of $\frac{1}{\Zap}$ along with appropriate powers of $\sigma$. The energy $\Ecalsigma(t)$ has several interesting properties such as:
\begin{enumerate}
\item For $\sigma=0$ the energy $\Ecalsigma(t)$ reduces to a lower order version of the energy of Kinsey and Wu \cite{KiWu18}. In particular it allows singular interfaces such as interfaces with angled crests and cusps (see \cite{Ag20}). 
\item For $\sigma>0$ the energy $\Ecalsigma(t)$ does not allow any singularities in the interface. In particular it does not allow angled crested interfaces. 
\item For $\sigma>0$ even though the energy $\Ecalsigma(t)$ does not allow singularities in the interface, it does allow interfaces with large curvature. It allows the $\Linfty$ norm of the curvature of the initial interface to be as large as $\sigma^{-\onebythree + \ep}$ for any $\ep>0$ (see \corref{cor:example} below). In particular for $\sigma$ small, the energy allows interfaces with very large curvature. 
\item Note that in the statement of the theorem, there are no assumptions on the Taylor sign condition. Also the energy $\Ecalsigma(t)$ is an increasing function of $\sigma$ and hence for initial data in appropriate Sobolev spaces, this result implies a uniform time of existence of solutions for $0\leq \sigma \leq \sigma_0$ for arbitrary $\sigma_0>0$, thereby recovering the uniform time of existence result of Ambrose and Masmoudi \cite{AmMa05} in this case. 
\end{enumerate}

These properties are discussed in more detail in \cite{Ag19}. One of the more interesting consequences of this theorem is that it can used to prove the existence of solutions to \eqref{eq:systemone} with the initial interface close to being angled crested. Let us now explain this result.  

As before let $(\Psi,\U)(\cdot,0):\Pminus \to \Csp$ be holomorphic maps with $\Psi_\zp \neq 0$ for all $z\in P_{-}$ and with their boundary values being the initial data namely $(\Z,\Ztbar)(\ap,0) = (\Psi,\U)(\ap,0)$ for all $\ap\in\Rsp$. 
Recall the notation namely that for $\zp \in \Pminus$ we write $\zp = \xp + i\yp$.   At $t=0$ define the quantity
\begin{align}\label{eq:M}
\begin{split}
\hspace*{-2mm}M & =   \sup_{\yp<0}\norm[\Leightbyseven(\Rsp, \diff \xp)]{\Psi_\zp^\threebyfour\partial_\zp \brac{\frac{1}{\Psi_\zp}}} + \sup_{\yp<0}\norm[\Lfourbythree(\Rsp, \diff \xp)]{\Psi_\zp^\half\partial_\zp \brac{\frac{1}{\Psi_\zp}}} + \sup_{\yp<0}\norm[\Ltwo(\Rsp, \diff \xp)]{\partial_\zp \brac{\frac{1}{\Psi_\zp}}}   \\
& \quad  + \sup_{\yp<0}\norm[\Linfty(\Rsp, \diff \xp)]{\frac{1}{\Psi_\zp}\partial_\zp \brac{\frac{1}{\Psi_\zp}}}  + \sup_{\yp<0}\norm[\Lone(\Rsp, \diff \xp)]{\frac{1}{\Psi_\zp}\partial_{\zp}^2 \brac{\frac{1}{\Psi_\zp}}} + \sup_{\yp<0}\norm[\Ltwo(\Rsp, \diff \xp)]{\frac{1}{\Psi_{\zp}^2}\partial_{\zp}^2 \brac{\frac{1}{\Psi_\zp}}}  \\
& \quad + \sup_{\yp<0}\norm[\Lone(\Rsp, \diff \xp)]{\frac{1}{\Psi_{\zp}^3}\partial_{\zp}^3 \brac{\frac{1}{\Psi_{\zp}}}} + \sup_{\yp<0}\norm[\Ltwo(\Rsp,\diff \xp)]{\frac{1}{\Psi_\zp} - 1} + \sup_{\yp<0}\norm[H^{3.5}(\Rsp,\diff \xp)]{\U}
\end{split}
\end{align}
It is easy to check that if $M<\infty$, then the initial data satisfies the hypothesis of Theorem 3.9 of \cite{Wu19} and we get a unique solution $(\Z,\Zt)(t)$ to \eqref{eq:systemone} for $\sigma=0$. Also by exactly the same argument as in section 5 of \cite{Ag20}, $M<\infty$ allows interfaces with angled crests of angles $\nu\pi$ with $0<\nu<\half$ and also allows certain cusps (see \cite{Ag20} for more details). With this we can now state the main corollary of this theorem.

\begin{cor}[\cite{Ag19}]\label{cor:example}
Consider an initial data $(\Z,\Zt)(0)$ with $M<\infty$. Let $(\Z,\Zt)(t)$ be the unique solution of equation \eqref{eq:systemone} for  $\sigma = 0$ with initial data $(\Z,\Zt)(0)$  as obtained in \cite{Wu19}. For $0<\epsilon \leq 1$ and $ \sigma \geq 0$ denote by $(\Z^{\epsilon,\sigma}, \Zt^{\epsilon,\sigma})(t)$ the unique solution to the equation \eqref{eq:systemone}  with surface tension $\sigma $ and with initial data  $(\Z^{\epsilon,\sigma}, \Zt^{\epsilon,\sigma})(0) = (\Z\conv P_\epsilon, \Zt\conv P_\epsilon)(0)$ where $P_\epsilon$ is the Poisson kernel. Then we have the following
\begin{enumerate}
\item For any $c>0$, there exists $T,C_1>0$ depending only on $c$ and $M$ such that for all $\sigma\geq 0$ and $0<\ep\leq 1$ satisfying $\dis \frac{\sigma}{\ep^\threebytwo} \leq c$, the solutions $(\Z^{\epsilon,\sigma}, \Zt^{\epsilon,\sigma})(t)$ exist in the time interval $[0,T]$ with $\sup_{t\in[0,T]} \Ecalsigma(\Z^{\ep,\sigma},\Zt^{\ep,\sigma})(t) \leq C_1 <\infty$.
\item If the initial interface $\Z(\cdot,0)$ has only one angled crest of angle $\nu\pi$ with $0<\nu<\half$, then the $\Linfty$ norm of the curvature $\kap^{\ep,\sigma}$ of the initial interface $\Z^{\ep,\sigma}(\cdot,0)$ satisfies $\norm[\infty]{\kap^{\ep,\sigma}} \sim \ep^{-\nu}$ as $\ep\to0$. In particular for any $0<\del<\onebythree$ arbitrarily small, choosing $\nu = \half - \threebytwo\del$ and $\sigma = \ep^\threebytwo$ we obtain $\norm[\infty]{\kap^{\ep,\sigma}} \sim \sigma^{-\onebythree + \del}$ as $\sigma \to 0$. Hence \thmref{thm:existence} allows initial interfaces with large curvature when $\sigma$ is small.  
\end{enumerate}
\end{cor}
See \figref{fig:main} for a comparison between the interfaces $\Z^{\epsilon,\sigma}(\cdot,t)$ and $\Z(\cdot,t)$. The interesting thing about this result is that under the assumption of $\sigma\lesssim \ep^\threebytwo$, it proves the existence of the solutions $(\Z^{\epsilon,\sigma}, \Zt^{\epsilon,\sigma})(t)$ on a uniform time interval $[0,T]$. Generally if one uses standard energy estimates (as compared to working with the energy $\Ecalsigma(t)$), one gets an upper bound on the time of existence as  $T \lesssim \norm[\infty]{\kap}^{-1}$ where $\kap$ is the initial curvature. As one can see, the initial interface of $\Z^{\ep,\sigma}(\cdot,0)$ has a very large curvature for $\ep$ small, and so one cannot obtain a result such as the above corollary by standard energy estimates and one has to use weighted energy estimates as done in \thmref{thm:existence}.  The scaling factor $\sigma/\ep^\threebytwo$ comes from the scaling of the equation and we refer to the introduction, Remark 3.4 along with its proceeding paragraph and Section 3.2  of \cite{Ag19} for more details.

\bigskip
\section{Main results and discussion}\label{sec:resultsanddiscussion}

\subsection{Results}\label{sec:results}

We now explain our results about convergence. Let $(\Z,\Zt)_a$ and $(\Z,\Zt)_b$ be two solutions of the water wave equation \eqref{eq:systemone} with surface tensions $\sigma_a$ and $\sigma_b$ respectively. We denote the two solutions as $A$ and $B$ respectively for simplicity. We will denote the terms and operators for each solution by their subscript $a$ or $b$. For example $(\Ztap)_a$ and $(\Ztap)_b$ denotes the spacial derivative of the velocity for the respective solutions. Similarly we also have the operators
\begin{align*}
(\Dapabs)_a = \frac{1}{\Zapabs_a}\pap \quad (\Dapabs)_b = \frac{1}{\Zapabs_b}\pap \quad \tx{ etc. }
\end{align*} 
Let $h_a, h_b$ be the homeomorphisms from \eqref{eq:h} for the respective solutions and let the material derivatives by given by $(\Dt)_a = U_{h_a}^{-1}\pt U_{h_a}$ and $(\Dt)_b = U_{h_b}^{-1}\pt U_{h_b}$. We define
\begin{align*}
\htil = h_b \compose h_a^{-1} \quad \tx{ and } \quad \Util = U_{\htil} = U_{h_a}^{-1}U_{h_b}
\end{align*}
While taking the difference of the two solutions, we will subtract in Lagrangian coordinates and then bring it to the conformal coordinate system of $A$. The reason we want to subtract in the Lagrangian coordinate system is because in our proof of the energy estimate we mainly use the material derivative, and in the Lagrangian coordinate system the material derivative for both the solutions is given by the same operator $\pt$ and subtracting in Lagrangian coordinates helps us avoid a loss of derivatives. The operator $\Util$ takes a function in the conformal coordinate system of $B$ to the conformal coordinate system of $A$. We define 
\begin{align*}
\Delta (f) = f_a - \Util(f_b)
\end{align*}
For example $\Delta (\Ztapbar) = (\Ztapbar)_a - \Util(\Ztapbar)_b$, where we have written $\Util(f)_b$ instead of $\Util(f_b)$ for easier readability for the term  $\Util(\Ztapbar)_b$. This notation allows us subtract the corresponding quantities of the two solutions in the correct manner, while still using conformal coordinates.

To describe our main result \thmref{thm:convergence} on the zero surface tension limit, we let $(\Z,\Zt)_a$ be the solution with surface tension $\sigma$ and let $(\Z,\Zt)_b$ denote the solution with zero surface tension.  We want to show that 
\begin{align*}
(\Z,\Zt)_a \to (\Z,\Zt)_b \qq \tx{ as } \quad \sigma \to 0
\end{align*}
in a suitable norm.  Note that with this notation equation \eqref{form:Zttbar} becomes
\begin{align*}
(\Zttbar)_a -i = -i \brac{\frac{\Aone}{\Zap}}\la +\sigma (\Dap \Th)_a \quad \tx{ where as } \quad (\Zttbar)_b -i = -i \brac{\frac{\Aone}{\Zap}}\lb
\end{align*}
and we have
\begin{align*}
\Delta (\Zttbar) = -i \cbrac{\brac{\frac{\Aone}{\Zap}}_{\n a} - \Util\brac{\frac{\Aone}{\Zap}}_{\n b} } +\sigma (\Dap \Th)_a = -i\Delta\brac{\frac{\Aone}{\Zap}} +\sigma (\Dap \Th)_a
\end{align*}

To state our convergence result, we first need to define a few more norms. We define the higher order energy for zero surface tension solutions as
\begin{align}\label{def:Ecalhigh}
\Ecalhigh  = \norm[2]{\pap\frac{1}{\Zap}}^2  + \norm[2]{\frac{1}{\Zap^2}\pap^2\frac{1}{\Zap}}^2 + \norm[2]{\Ztapbar}^2 + \norm*[\Bigg][2]{\frac{1}{\Zap^2}\pap\Ztapbar}^2 + \norm*[\Bigg][\Hhalf]{\frac{1}{\Zap^3}\pap\Ztapbar}^2
\end{align}
The energy $\Ecalhigh $ is used in Kinsey-Wu \cite{KiWu18} and Wu \cite{Wu19} to prove local wellposedness for angled crested water waves.
\footnote{The energy used in  \cite{KiWu18} and \cite{Wu19} also has the lower order term $\norm[\infty]{\frac{1}{\Zap}}^2(t)$ or $\abs{\frac{1}{\Zap}}(0,t)$ which we don't have.}.
Note that this energy is half spacial derivative higher order than the energy $\Ecalsigma\vert_{\sigma=0}$.  We also define an auxiliary energy for the zero surface tension solution
\begin{align}\label{def:Ecalaux}
\begin{split}
\Ecalaux & = \norm[\infty]{\Zap^\half\pap\frac{1}{\Zap}}^2 + \norm*[\Bigg][2]{\frac{1}{\Zap^\half}\pap^2\frac{1}{\Zap}}^2 + \norm*[\Bigg][2]{\frac{1}{\Zap^\fivebytwo}\pap^3\frac{1}{\Zap}}^2 \\
& \quad + \norm*[\Bigg][2]{\frac{1}{\Zap^\half}\pap\Ztapbar}^2 +  \norm*[\Bigg][2]{\frac{1}{\Zap^\fivebytwo}\pap^2\Ztapbar}^2 + \norm*[\Bigg][\Hhalf]{\frac{1}{\Zap^\sevenbytwo}\pap^2\Ztapbar}^2
\end{split}
\end{align}
This energy will be key in proving the convergence result. We can now define the norm in which we prove convergence \footnote{The energy $\EcalDelta$ described here is not the full energy that we use to prove the result. The full energy is given in \eqref{def:EcalDelta} and also has additional terms called $\EcalDeltazero$. The terms of $\EcalDeltazero$ are zero at time $t=0$ if the two solutions have the same initial data but $\EcalDeltazero$ may become non-zero at a later time. We do not include the full energy here as $\EcalDeltazero$ is zero at time $t=0$ and to simplify the presentation.}. Define the energy $\EcalDelta$ as
\begingroup
\allowdisplaybreaks
\begin{align*}
\EcalDeltaone & = \norm[2]{\Delta\brac{\pap\frac{1}{\Zap}}}^2 + \norm*[\bigg][\Hhalf]{\Delta\brac{\frac{1}{\Zap}\pap\frac{1}{\Zap}}}^2 + \norm*[\Bigg][2]{\brac{\sigma^\onebysix\Zap^\half\pap\frac{1}{\Zap}}\la}^6 \\*
& \quad + \norm*[\Bigg][\infty]{\brac{\sigma^\half\Zap^\half\pap\frac{1}{\Zap}}\la}^2  +  \norm*[\Bigg][2]{\brac*[\Bigg]{\frac{\sigma^\half}{\Zap^\half}\pap^2\frac{1}{\Zap}}\la}^2 + \norm*[\Bigg][\Hhalf]{\brac*[\Bigg]{\frac{\sigma^\half}{\Zap^\threebytwo}\pap^2\frac{1}{\Zap}}\la}^2\\*
& \quad    +  \norm[\Hhalf]{\brac{\sigma\pap\Th}_a}^2   + \norm[2]{\brac{\frac{\sigma}{\Zap}\pap^3\frac{1}{\Zap}}\la}^2 + \norm[\Hhalf]{\brac{\frac{\sigma}{\Zap^2}\pap^3\frac{1}{\Zap}}\la}^2\\
\EcalDeltatwo &= \norm[2]{\Delta\brac{\Ztapbar}}^2 + \norm*[\Bigg][2]{\Delta\brac{\frac{1}{\Zap^2}\pap\Ztapbar}}^2 + \norm*[\Bigg][2]{\brac*[\Bigg]{\frac{\sigma^\half}{\Zap^\half}\pap\Ztapbar}\la}^2 + \norm*[\Bigg][2]{\brac*[\Bigg]{\frac{\sigma^\half}{\Zap^\fivebytwo}\pap^2\Ztapbar}\la}^2 \\
\EcalDelta & =   \EcalDeltaone + \EcalDeltatwo + \sigma\brac{\Ecalaux}_b
\end{align*}
\endgroup

Let us now understand this norm. Observe that the terms of $\EcalDeltaone$ and $\EcalDeltatwo$ are obtained by taking the difference of the terms of the energy $\Ecalsigma$ defined in \secref{sec:previousresult} for the solutions $A$ and $B$. The term $\sigma\brac{\Ecalaux}_b$ is a coupling term as it couples the energy $\brac{\Ecalaux}_b$ of the zero surface tension solution $B$ with the surface tension coefficient $\sigma$ of the solution $A$. This term is crucial to close the energy estimate for $\EcalDelta$ at the highest order. The reason for the need to include this term and why it works is explained in more detail in \secref{sec:discussion}.

We now state our main result on convergence. We will state the theorem here only for the special case of the two solutions having the same initial data and a more general result is stated in \secref{sec:aprioriEcalDelta}. The existence part of the result follows from earlier results: For $\sigma=0$, one can use Theorem 3.9 of \cite{Wu19}, where a local wellposedness result is proved in terms of the energy $\Ecalhigh$ \footnote{In Theorem 3.9 of \cite{Wu19} the time of existence also depends on the lower order term $\frac{1}{\Zap}$. In our case this can be avoided as we show in \thmref{thm:aprioriEhigh} and the proof of \thmref{thm:convergence}.}. For $\sigma>0$ we can use \thmref{thm:existence} for an existence result in terms of the energy $\Ecalsigma(t)$. The main result of this paper is as follows:

\begin{thm}\label{thm:convergence}
Let $\sigma>0$ and let $(\Z,\Zt)(0)$ be an initial data such that $(\Zap-1,\frac{1}{\Zap} - 1, \Zt)(0) \in H^{3.5}(\Rsp)\times H^{3.5}(\Rsp)\times H^{3.5}(\Rsp)$. Let $L>0$ be such that
\begin{align*}
\Ecalhigh(0), \mathcal{E}_{\sigma}(0) \leq L
\end{align*}
Then there exists $T,C_0, C_1, C_2> 0$ depending only on $L$ so that we have a unique solution $(\Z,\Zt)(t)$  to \eqref{eq:systemone} with zero surface tension with initial data $(\Z,\Zt)(0)$ in time interval $[0,T]$ satisfying $(\Zap-1,\frac{1}{\Zap} - 1, \Zt) \in C^l([0,T], H^{3-\half l}(\Rsp)\times H^{3 - \half l}(\Rsp)\times H^{3.5 - \half l}(\Rsp))$ for $l=0,1$ along with the estimate $\sup_{t \in [0,T]} \Ecalhigh(\Z,\Zt)(t) \leq C_2$, and we have a unique solution $(\Z^\sigma,\Zt^\sigma)(t)$  to  \eqref{eq:systemone} with surface tension $\sigma$ with the same initial data $(\Z,\Zt)(0)$ in the time interval $[0,T]$ satisfying $(\Zap^\sigma-1,\frac{1}{\Zap^\sigma} - 1, \Zt^\sigma) \in C^l([0,T], H^{3.5 - \threebytwo l}(\Rsp)\times H^{3.5 - \threebytwo l}(\Rsp)\times H^{3 - \threebytwo l}(\Rsp))$ for $l=0,1$ and $\sup_{t\in [0,T]}\Ecalsigma(\Z^\sigma,\Zt^\sigma)(t) \leq C_2$, and we have the estimate
\begin{align*}
\sup_{t\in[0,T]}\EcalDelta(\Z^\sigma,\Z)(t) \leq C_1e^{C_0T}\EcalDelta(\Z^\sigma,\Z)(0)
\end{align*}
\end{thm}

In the above result $\EcalDelta(\Z^\sigma,\Z)(t) $ is the energy $\EcalDelta(t)$ where $(\Z^\sigma,\Zt^\sigma)(t)$ is the solution $A$ and $(\Z,\Zt)(t)$ is the solution $B$. From the energy $\EcalDelta$ defined above we observe that if the initial data for the two solutions is the same, then $\EcalDelta(\Z^\sigma,\Z)(0) \lesssim \sigma \to 0$ as $\sigma \to 0$, and hence $\sup_{t\in[0,T]}\EcalDelta(\Z^\sigma,\Z)(t) \to 0$ as $\sigma \to 0$.  This result should be contrasted with the result of Ambrose-Masmoudi \cite{AmMa05} where the convergence is proved in Sobolev spaces. Note that the theorem above implies convergence in Sobolev spaces so we recover the result of \cite{AmMa05}. The novelty of the above convergence result is that the rate of growth of the energy $\EcalDelta$ does not depend on the $C^{1,\alpha}$ norm of the interface (as is the case in \cite{AmMa05}), but only on the weighted energies $\Ecalhigh$ and $\Ecalsigma$, as the constant $C_0$ appearing in the estimate for $\EcalDelta(t)$, namely
\begin{align*}
\sup_{t\in[0,T]}\EcalDelta(\Z^\sigma,\Z)(t) \leq C_1e^{C_0T}\EcalDelta(\Z^\sigma,\Z)(0)
\end{align*}
depends only on $L$ which in turn depends only on $\Ecalhigh(0)$, $\Ecalsigma(0)$. If the initial interface is close to being an angled crest, then $\Ecalhigh$ and $\Ecalsigma$ remain bounded \footnote{$\Ecalsigma(t)$ remains bounded provided the surface tension is small enough depending on how close it is an angled crest interface. See \corref{cor:examplenew}.} where as the $C^{1,\alpha}$ norm (for any $0<\alpha\leq1$) of the interface $\Z$ blows up as the interface gets closer to being angled crested. Hence this result allows us to control the difference of the solutions independent of how close the initial interface is to an angled crest interface. It is also worthwhile to note that in the proof we show that the energy $\EcalDelta$ is quite strong, and in particular we have the angle of the interface $\thvar^\sigma \to \thvar$ in $\Linfty$ as $\sigma \to 0$ \footnote{See the paragraph following \thmref{thm:aprioriEDelta} for more details.}. Hence the approximation between the solutions with non-zero surface tension and zero surface tension is quite strong. 

We prove this theorem in \secref{sec:proof}. To prove this theorem, we first prove an a priori energy estimate for the energy $\Ecalaux$, by proving that as long as $\Ecalhigh$ is controlled, then $\Ecalaux$ is also controlled. This is proved in \secref{sec:aprioriEcalaux}. Then using this energy estimate we prove an a priori energy estimate for $\EcalDelta$ in \secref{sec:aprioriEcalDelta}. We then use both of these energy estimates to finish the proof in \secref{sec:proof}. 

Let us now apply this result explicitly to angled crested interfaces.  For two solutions $A$ and $B$ of water waves, we define the following norm
\begin{align}\label{eq:Fcal}
\begin{split}
\Fcal_\Delta & = \norm[\Hhalf]{(\zt)_a - (\zt)_b} + \norm[\Hhalf]{(\ztt)_a - (\ztt)_b} + \norm[\Hhalf]{\brac{\frac{\hal}{\zal}}_{\n a} - \brac{\frac{\hal}{\zal}}_{\n b}} + \norm[2]{(\hal)_a - (\hal)_b} \\
& \quad + \norm[2]{\brac{\frac{\ztal}{\zal}}_{\n a} - \brac{\frac{\ztal}{\zal}}_{\n b}} + \norm[2]{\brac{\Aone\compose \h}_a - \brac{\Aone\compose \h}_b} + \norm[2]{\brac{\frac{h_{t\alpha}}{\hal}}_{\n a} - \brac{\frac{\h_{t\alpha}}{\hal}}_{\n b}}
\end{split}
\end{align}
This norm was introduced by Wu in Theorem 3.7 of \cite{Wu19} to establish uniqueness of angled crested water waves without surface tension \footnote{See \eqref{eq:Ftilcal} for an equivalent form of $\Fcal_\Delta$ written in the conformal coordinate system of solution $A$, as was written in \cite{Wu19}. Also instead of the term $\norm[2]{(\hal)_a - (\hal)_b}$ above,  in \cite{Wu19} the term which shows up is $\norm[2]{\frac{(\hal)_b}{(\hal)_a}  - 1} $. Both of these are equivalent as $\hal$ has a lower and upper bound as long as $\Ecalsigma(t)$ remains bounded. See the paragraph following \eqref{eq:h}}. Let us now consider  \corref{cor:example} and let solution $A$ be $(\Z^{\epsilon,\sigma}, \Zt^{\epsilon,\sigma})(t)$ and solution $B$ be  $(\Z, \Zt)(t)$.  We will use the above norm $\Fcal_\Delta$ to establish convergence for $(\Z^{\epsilon,\sigma}, \Zt^{\epsilon,\sigma})(t) \to (\Z, \Zt)(t)$ as $\ep,\sigma \to 0$ under suitable conditions. 

\begin{figure}[ht]
\centering
 \begin{tikzpicture}[scale=0.61]
\draw [thick, xshift=-25pt] (-4,-0.5) to [out=0,in=-89] (-2,4) to [out=-87+180,in=180] (-0,0) to [out=0,in=180+45] (6,3.5) to [out=-65,in=180] (10,0) to [out=0,in=150] (14,4) to [out=180+38,in=180] (17,-0.5);
\draw [densely dashed, xshift=-25pt, yshift=-20pt] (-4,0.2-0.4) to [out=0,in=-100] (-1.9,2-0.0) to [out=-100+180,in=180] (-0.26,0.3) to [out=0,in=180+45] (5.8,3.5) to (6,3.5) to [out=-65,in=180] (10,0.2) to [out= 0,in=90](13.2,3.8) to [out=-90,in=180] (17,0.2-0.4);
\node at (17.1,-0.3) {\ \ $Z(\cdot,t)$};
\node at (17.7,-0.8-0.3) {\!\!$Z^{\ep,\sigma}(\cdot,t)$};
\draw [thin,xshift=-25pt,  <->] (-4.3,0-0.45) to (-4.3,-0.52-0.4) ;
\node at (-5.5,-0.3-0.4) {${\ep}$};
\node at (1.8,2.6) {$\rho_{air} =0$};
\node at (4.2,0) {$\rho_{water} =1$};
\end{tikzpicture}
\caption{Waves with and without surface tension}\label{fig:main}
\end{figure}
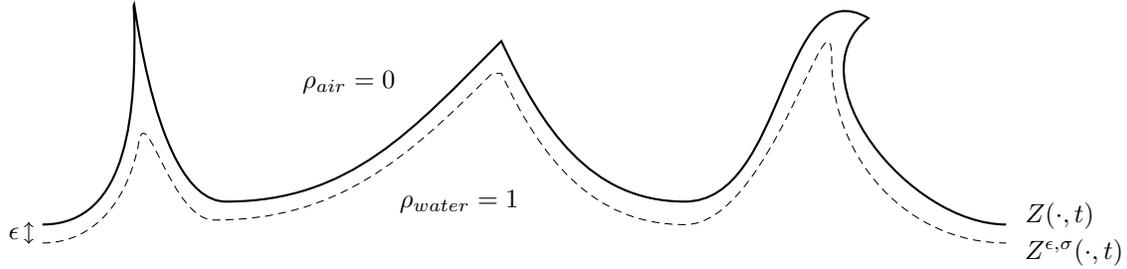

\begin{cor}\label{cor:examplenew}
Consider an initial data $(\Z,\Zt)(0)$ with $M<\infty$ where $M$ is defined in \eqref{eq:M} and fix $c>0$. Let $(\Z,\Zt)(t)$ be the unique solution of equation \eqref{eq:systemone} with zero surface tension with initial data $(\Z,\Zt)(0)$ in the time interval $[0,T_1]$, where $T_1>0$ and depends only on $M$, as obtained in Theorem 3.9 of \cite{Wu19}. For $\sigma\geq 0$, $0<\ep\leq 1$ let $(\Z^{\epsilon,\sigma}, \Zt^{\epsilon,\sigma})(t)$ be the unique solution to the equation \eqref{eq:systemone} with surface tension $\sigma $ and with initial data  $(\Z^{\epsilon,\sigma}, \Zt^{\epsilon,\sigma})(0) = (\Z\conv P_\epsilon, \Zt\conv P_\epsilon)(0)$ where $P_\epsilon$ is the Poisson kernel.  If $\dis \frac{\sigma}{\ep^\threebytwo} \leq c$, then there exists $T_2, C_0, C_1, C_2>0$ depending only on $M$ and $c$ so that the solutions $(\Z^{\epsilon,\sigma}, \Zt^{\epsilon,\sigma})(t)$ exist in the time interval $[0,T_2]$ and $\sup_{t\in[0,T_2]} \Ecalsigma(\Z^{\ep,\sigma},\Zt^{\ep,\sigma})(t) \leq C_2$ along with 
\begin{enumerate}
\item $\sup_{t\in [0,T_2]} \EcalDelta(\Z^{\epsilon,\sigma}, \Z^\epsilon)(t) \leq \frac{\sigma}{\ep^\threebytwo}(C_1e^{C_0 T_2})$
\item  $\sup_{t\in[0,T]} \Fcal_{\Delta}(\Z^{\epsilon,\sigma}, \Z)(t) \to 0$ as $\max\cbrac{\frac{\sigma}{\ep^\threebytwo}, \ep} \to 0$, where $T = \min\cbrac{T_1,T_2}>0$.
\end{enumerate}
\end{cor}
In the above result $\EcalDelta(\Z^{\epsilon,\sigma},\Z^\ep)(t) $ is the energy $\EcalDelta(t)$  for the two solutions $(\Z^{\epsilon,\sigma}, \Zt^{\epsilon,\sigma})(t)$ and $(\Z^\ep,\Zt^\ep)(t)$  and similarly $\Fcal_\Delta(\Z^{\epsilon,\sigma},\Z)(t) $ is the energy $\Fcal_\Delta(t)$ for the two solutions $(\Z^{\epsilon,\sigma}, \Zt^{\epsilon,\sigma})(t)$ and $(\Z,\Zt)(t)$.  As explained in \secref{sec:previousresult}, $M<\infty$ allows interfaces with angled crests of angles $\nu\pi$ with $0<\nu<\half$ and also allows certain cusps, and so $(\Z,\Zt)(0)$ is allowed to be singular. As explained previously the existence part of this result has already been proved in \corref{cor:example} and the novelty here is the convergence aspect. The first part of the convergence result says that as long as $\sigma \lesssim \ep^\threebytwo$, then the difference of the solution with surface tension $(\Z^{\epsilon,\sigma}, \Zt^{\epsilon,\sigma})(t)$ to the solution without surface tension $(\Z^{\epsilon}, \Zt^{\epsilon})(t)$ does not depend on how close the initial interface is to being angled crested (i.e. independent of $\ep$). The second part shows that the smooth solutions $(\Z^{\epsilon,\sigma}, \Zt^{\epsilon,\sigma})(t)$ to \eqref{eq:systemone} with surface tension $\sigma$ converge to the singular solution $(\Z, \Zt)(t)$ of \eqref{eq:systemone} with zero surface tension, and the convergence happens in the norm $\Fcal_{\Delta}$. 

This corollary is proved in \secref{sec:proof}. The first part of the convergence result follows directly from \thmref{thm:convergence}. The second part follows from the observation that the norm $\Fcal_\Delta$ is weaker than the norm $\EcalDelta$ and hence we essentially have 
\begin{align*}
 \Fcal_{\Delta}(\Z^{\epsilon,\sigma}, \Z)(t) & \leq  \Fcal_{\Delta}(\Z^{\epsilon,\sigma}, \Z^\ep)(t) +  \Fcal_{\Delta}(\Z^{\epsilon}, \Z)(t) \\
 & \lesssim \cbrac{\Ecal_{\Delta}(\Z^{\epsilon,\sigma}, \Z^\ep)(t)}^\alpha +  \Fcal_{\Delta}(\Z^{\epsilon}, \Z)(t)
\end{align*}
for some $\alpha >0$. Now the proof follows from using the fact that $\Ecal_{\Delta}(\Z^{\epsilon,\sigma}, \Z^\ep)(t) \to 0$ as $\max\cbrac{\frac{\sigma}{\ep^\threebytwo}, \ep} \to 0$ from part (1) of the corollary, and $\Fcal_{\Delta}(\Z^{\epsilon}, \Z)(t) \to 0$ by Theorem of 3.7 of \cite{Wu19}. See the proof in \secref{sec:proof} for more details. 

The norm $\Fcal_\Delta$ being weaker than the norm $\EcalDelta$ has one consequence being that it can only show $\thvar^{\ep,\sigma} \to \thvar$ in $\Ltwo$ instead of the convergence in $\Linfty$, which is what you would obtain if one could use the stronger norm $\EcalDelta$ \footnote{To see that we indeed get convergence in $\Ltwo$, observe from \eqref{form:Dtg} that $\abs{\pt (\thvar_a - \thvar_b)} \leq  \abs{\brac{\frac{\ztal}{\zal}}_{ a} - \brac{\frac{\ztal}{\zal}}_{ b}}$ and hence we obtain convergence in $\Ltwo$ by using the fact that $\Fcal_\Delta$ controls $\norm[2]{\brac{\frac{\ztal}{\zal}}_{ a} - \brac{\frac{\ztal}{\zal}}_{ b}} $.}. However observe that one in fact cannot expect convergence of $\thvar^{\ep,\sigma} \to \thvar$ in $\Linfty$, as if the initial data is symmetric with $\Z(\cdot,0)$ having an angled crest at $\ap = 0$, then by \cite{Ag20} the function $ \thvar (\cdot,t)$ will have a discontinuity at $\al = 0$ for all $t\in [0,T]$, whereas $\thvar^{\ep,\sigma} (0,t) = 0$ whenever $\ep>0$, for all $t\in [0,T]$. In particular this implies that one cannot have $\Ecal_{\Delta}(\Z^{\epsilon,\sigma}, \Z)(t) \to 0$ as $\max\cbrac{\frac{\sigma}{\ep^\threebytwo}, \ep} \to 0$, in this situation. We also note that the factor $\sigma/\ep^\threebytwo$ comes from the scaling of the equation and is optimal for angled crested interfaces with angles close to $\pi/2$ (see the introduction and Remark 3.4 of \cite{Ag19} for more details about the scaling). This shows that the convergence result \corref{cor:examplenew} that we prove is as strong as one can expect in this scenario.


\subsection{Discussion}\label{sec:discussion} 

One of the fundamental difficulties in proving a zero surface tension limit result for angled crested like water waves, is that one should maintain the optimal scaling rate $\sigma/\ep^\threebytwo$ that one has in \corref{cor:example} in the convergence result \corref{cor:examplenew} as well. This makes the proof a lot more nontrivial. To do this we have to subtract the energy at the level of $\Ecalsigma$ as this was the energy used to prove \corref{cor:example}, but then the difference of the energy does not close because of the weighted nature of the energy (this issue does not arise when one is working with non weighted energies in Sobolev spaces as explained below). To remedy this we introduce the coupling term $\sigma(\Ecalaux)_b(t)$ to close the energy at the highest order. Let us now explain the main difficulty and exactly where the energy  $\sigma(\Ecalaux)_b(t)$ is used. 

Let us first discuss how an energy estimate for convergence in Sobolev spaces generally works. Let $u_a$ and $u_b$ be solutions to the equations
\begin{align*}
(\pt^2 + \papabs + \sigma \papabs^3)u_a = N_a \qq (\pt^2 + \papabs)u_b = N_b
\end{align*}
Here $N_a, N_b$ are nonlinear lower order terms. These are highly simplified versions of the water wave equation with surface tension $\sigma$ and with zero surface tension respectively. Now the corresponding energies for the solutions $u_a$ and $u_b$ are of the form
\begin{align*}
(E_s)_a(t) = \norm[\Hhalf]{\pt\pap^s u_a}^2 + \norm[2]{\pap^{s+1} u_a}^2 + \norm[2]{\sigma^\half \pap^{s+2}u_a}^2 \qq (E_s)_b(t) = \norm[\Hhalf]{\pt\pap^s u_b}^2 + \norm[2]{\pap^{s+1} u_b}^2
\end{align*}
Hence to prove convergence at the level of $E_s$, we first subtract the equations to get 
\begin{align*}
(\pt^2 + \papabs)(u_a - u_b) + \sigma\papabs^3 u_a = N_a - N_b
\end{align*}
From which we get the energy
\begin{align*}
E_{\Delta,s}(t) = \norm[\Hhalf]{\pt\pap^s (u_a - u_b)}^2 + \norm[2]{\pap^{s+1} (u_a - u_b)}^2 + \norm[2]{\sigma^\half \pap^{s+2}u_a}^2
\end{align*}
Differentiating this we see that
\begin{align*}
\frac{\diff }{\diff t} E_{\Delta,s}(t) & = 2\Real\cbrac{ \int (\papabs \pt \pap^s(\ubar_a - \ubar_b))(\pap^s \pt^2(u_a - u_b))  } \\
& \quad + 2\Real\cbrac{ \int (\pt\pap^{s+1}(\ubar_a - \ubar_b))(\pap^{s+1}(u_a - u_b)) } \\
& \quad + 2\Real\cbrac{ \int (\pt\pap^{s+2}\ubar_a)(\sigma\pap^{s+2}u_a) }
\end{align*}
Now in the last integral we rewrite $\pt\pap^{s+2}\ubar_a = \pt\pap^{s+2}(\ubar_a - \ubar_b) + \pt\pap^{s+2}\ubar_b$ and use the identities $\papabs = i\Hil\pap$ and $\Hil^2 = \Id$ to get
\begin{align*}
\frac{\diff }{\diff t} E_{\Delta,s}(t) & = 2\Real \cbrac{\int \brac{\papabs\pt\pap^s(\ubar_a - \ubar_b)}\pap^s\brac{\pt^2(u_a - u_b) + \papabs(u_a - u_b) + \sigma\papabs^3 u_a}} \\
& \quad + 2\Real\cbrac{ \int (\sigma^\half \pt \pap^{s+2}\ubar_b )(\sigma^\half\pap^{s+2}u_a) } 
\end{align*}
Now the first integral can be controlled by using the equation for the difference. To control the second integral, one can assume control of the energy $(E_{s+ 3/2})_b(t)$ to get control of $\norm[2]{\pt\pap^{s+2}u_b}$, and so the second integral can be estimated via
\begin{align}\label{eq:inthigher}
\abs{ 2\Real\cbrac{ \int (\sigma^\half \pt \pap^{s+2}\ubar_b )(\sigma^\half\pap^{s+2}u_a) }  } \lesssim \sigma^\half\norm[2]{\pt \pap^{s+2}u_b} \norm*[\big][2]{\sigma^\half\pap^{s+2}u_a} \to 0 \qq \tx{ as }\sigma \to 0
\end{align} 
It is crucial to note that in the above estimate we have convergence to zero fundamentally because we are using the fact that $\sigma^\half \to 0$ as $\sigma \to 0$, and the terms $\norm[2]{\pt \pap^{s+2}u_b}$ and $ \norm*[\big][2]{\sigma^\half\pap^{s+2}u_a}$ remain bounded. 

Note that this was a very basic strategy to prove convergence and one can hope to do better if one is working with smooth enough data. However as we are interested in angled crested interfaces, the quasilinear equation becomes degenerate and hence one cannot use a lot of the machinery (such as dispersive estimates) one has when one is working in a smooth enough regime (see Section 3.2 of \cite{Ag19} for more details). Hence we follow essentially the same strategy as above except that at the last step, namely the estimate \eqref{eq:inthigher}, this strategy fails. To see this, first note that the integral we are left to estimate by trying to control $\EDeltafour$ in \thmref{thm:aprioriEDelta} is an integral of the form
\begin{align*}
2  \Real  \int \Util\cbrac{\frac{\sigma^\half}{\Zapabs^\half}\pap\Dapabs(\Dt\Dap\Zt) }_{\n b} \cbrac{\frac{\sigma^\half}{\Zapabs^\half}\pap\Dapabs\Dapbar\Ztbar }\la 
\end{align*}
See \secref{sec:controlEDeltafour} for the details. Now $\brac{\frac{\sigma^\half}{\Zapabs^\half}\pap\Dapabs\Dapbar\Ztbar}_{\n a} \in \Ltwo$ as it is part of the energy $\EDeltafour$ (This is analogous to the control of $\norm*[\big][2]{\sigma^\half\pap^{s+2}u_a}$ one has from $\E_{\Delta,s}$ in \eqref{eq:inthigher}). However we cannot assume control of $\cbrac{\frac{1}{\Zapabs^\half}\pap\Dapabs(\Dt\Dap\Zt)}\lb \in \Ltwo$ in an analogous manner as was done in \eqref{eq:inthigher} (by assuming control of $\norm[2]{\pt\pap^{s+2}u_b}$), as heuristically if the zero surface tension solution $B$ has an angle crest of angle $\nu \pi$ then $\Z(\ap) \sim (\ap)^{\nu}$ near $\ap = 0$, and hence using \eqref{form:Zttbar} for $\sigma = 0$ we heuristically have near $\ap = 0$
\begin{align*}
\frac{1}{\Zapabs^\half}\pap\Dapabs(\Dt\Dap\Zt) \sim \frac{1}{\Zap^\fivebytwo}\pap^3\frac{1}{\Zap} \sim (\ap)^{-\sevenbytwo \nu + \half}
\end{align*}
which does not belong to $\Ltwo$ for $\frac{2}{7}\leq \nu <\half$. Hence assuming control of this term will force us to severely restrict the initial data allowed in our results by imposing the restriction $0<\nu < \frac{2}{7}$. Another way to try to resolve it would be to use higher order norms which allow all angles $\nu\pi$ with $0<\nu<\half$ such as an energy even higher order than $\Ecalhigh$, which corresponds to taking a weighted derivative $\frac{1}{\Zapabs^2}\pap$ on the energy $\Ecalhigh$ (see Section 3.2 in \cite{Ag19} for more details on these weighted norms). If one does this, then the term we would get control over would be $\cbrac{\frac{1}{\Zapabs^2}\pap\Dapabs(\Dt\Dap\Zt)}\lb \in \Ltwo$ which does allow all angles less than $\pi/2$. However now the scaling becomes a problem as we observe that
\begin{align*}
\norm[2]{\brac{\frac{\sigma^\half}{\Zapabs^\half}\pap\Dapabs(\Dt\Dap\Zt)}\lb} \leq \norm[\infty]{\sigma^\onebythree(\Zap)_b}^\threebytwo \norm[2]{\brac{\frac{1}{\Zapabs^2}\pap\Dapabs(\Dt\Dap\Zt)}\lb}
\end{align*}
If $\Z(\cdot,0)$ is an interface with an angle crest at $\ap = 0$ of angle $\nu\pi$, then we have $\Zap(\ap) \sim (\ap)^{\nu -1}$ near $\ap =0$, and hence for $\nu$ close to $0$ the best one can assume is that $\Zap(\cdot,0) \in \Lone$ in general. Hence for $\Z^\ep(\cdot,0) = \Z(\cdot,0) \conv P_\ep$, where $P_\ep$ is the Poisson kernel, we get the condition
\begin{align*}
\norm[\infty]{\sigma^\onebythree(\Zap)_b(\cdot,0)} = \norm[\infty]{\sigma^\onebythree \Zap^\ep(\cdot,0)} \lesssim \frac{\sigma^\onebythree}{\ep}\norm[1]{\Zap(\cdot,0)} \to 0
\end{align*} 
which forces us to assume that $\sigma/\ep^3 \to 0$, which is very far from the optimal rate of $\sigma/\ep^\threebytwo \to 0$. 

To overcome this difficulty, we abandon the approach used to prove the analogous estimate \eqref{eq:inthigher}. Instead of adding a norm to refine the zero surface tension solution, we instead add  exactly the required terms to close the energy $\EcalDelta$ to the energy $\EcalDelta$ itself, and control these additional terms seperately. We observe that to close the energy estimate for $\EcalDelta$, what we really need to control is $\cbrac{\frac{\sigma^\half}{\Zapabs^\half}\pap\Dapabs(\Dt\Dap\Zt)}\lb \in \Ltwo$ and ensure that this goes to zero as $\sigma \to 0$. So we control this term directly and the term $\sigma(\Ecalaux)_b(t)$ in the energy $\EcalDelta$ controls precisely terms of this type and ensures that it goes to zero as $\sigma \to 0$. The energy $\sigma(\Ecalaux)_b(t)$ not only allows us to prove the a priori estimate for $\EcalDelta$ in \thmref{thm:convergence} but moreover ensures that the scaling $\sigma/\ep^\threebytwo$ obtained originally in \corref{cor:example} as part of \cite{Ag19} remains the same in \corref{cor:examplenew}. This shows that the energy $\sigma(\Ecalaux)_b(t)$ scales in the correct manner and does not weaken the statement of our results in \thmref{thm:convergence} or  \corref{cor:examplenew}.

\section{Identities and equations from previous work}\label{sec:part1}

In this section we collect the identities and estimates proved in \cite{Ag19} which we use in this paper. We are collecting them essentially all in this section as this is easier than to constantly refer to the paper \cite{Ag19} when we prove our energy estimates in later sections. 

\smallskip
\subsection{Main identities}\label{sec:equations}

Here we first collect the main identities commonly used in this paper. See Section 4 of \cite{Ag19} for a proof of these identities. 

\begin{enumerate}[leftmargin =*, align=left, label=\alph*)]
\item We have
\[
\frac{\Zap}{\Zapabs}\pap\frac{1}{\Zap}   = \pap \frac{1}{\Zapabs} + \w\Dapabs \wbar
\]
Observe that $\dis \pap \frac{1}{\Zapabs}$ is real valued and $\dis \w\Dapabs \wbar$ is purely imaginary. From this we obtain 
\begin{align} \label{form:RealImagTh}
\Real\brac{\frac{\Zap}{\Zapabs}\pap\frac{1}{\Zap} } = \pap \frac{1}{\Zapabs} \quad \qquad \Imag\brac{\frac{\Zap}{\Zapabs}\pap\frac{1}{\Zap}} = i\brac{ \wbar\Dapabs \w} = -\Real \Th
\end{align}

\item We have
\begin{align}  \label{form:Dtg}
\Dt \g = - \Imag (\Dapbar \Ztbar) 
\end{align}

\item For any complex valued function $f$, we have $\Hil (\Real f) =  i\Imag(\Hil f)$ and  $\Hil(i \Imag f) = \Real (\Hil f)$. Hence we get the following identities
\begin{align}
(\Id + \Hil)(\Real f) = f -i\Imag (\Id - \Hil) f  \label{form:Real}\\
(\Id + \Hil)(i \Imag f) = f - \Real (\Id - \Hil)f  \label{form:Imag}
\end{align}

\item We have
\begin{align}\label{form:Aonenew}
\Aone = 1 + i\Zt\Ztapbar -i(\Id + \Hil)\cbrac{\Real(\Zt\Ztapbar)}
\end{align}

\item We have 
\begin{align*}
\bvar = \frac{\Zt}{\Zap} -i(\Id + \Hil)\cbrac{\Imag\brac{\frac{\Zt}{\Zap}}}
\end{align*}
and hence 
\begin{align}\label{form:bvarapnew}
\bvarap = \Dap\Zt + \Zt\pap\frac{1}{\Zap}  -i\pap(\Id + \Hil)\cbrac{\Imag\brac{\frac{\Zt}{\Zap}}}
\end{align}

\item We now record some frequently used commutator identities. 
\begin{align}\label{eq:commutator}
 [\pap, \Dt ] &= \bap \pap   \qquad &[\Dapabs, \Dt] &= \Real{(\Dap \Zt)} \Dapabs = \Real{(\Dapbar \Ztbar) \Dapabs} \\   \relax 
 [\Dap, \Dt]  &=  \brac{\Dap \Zt} \Dap   \qquad &  [\Dapbar, \Dt] &= \brac{\Dapbar \Ztbar } \Dapbar 
\end{align}
Using these we also obtain the following formulae
\begin{align}
\Dt \Zapabs & = \Dt e^{\Real\log \Zap} =  \Zapabs \cbrac{\Real(\Dap\Zt) - \bvarap} \label{form:DtZapabs} \\ 
 \Dt \frac{1}{\Zap} & = \frac{-1}{\Zap}(\Dap\Zt - \bvarap) = \frac{1}{\Zap}\cbrac{(\bvarap - \Dap\Zt - \Dapbar \Ztbar) + \Dapbar \Ztbar} \label{form:DtoneoverZap}
\end{align}
Observe that $(\bvarap - \Dap\Zt - \Dapbar \Ztbar)$ is real valued and this fact will be useful later on.
\medskip

\item We have the formula
\begin{align} \label{form:Th}
\Th = i \frac{\Zap}{\Zapabs} \pap \frac{1}{\Zap} - i \Real (\Id - \Hil) \brac{\frac{\Zap}{\Zapabs} \pap \frac{1}{\Zap}} 
\end{align}

\item We have the formula
\begin{align}  \label{form:RealDtTh}
\Real (\Dt \Th)   = - \Imag\cbrac{(\Dapabs + i \Real \Th) \Dapbar \Ztbar}
\end{align}
We also have the formula
\begin{align} \label{form:DtTh}
\Dt \Th = i\brac{\Dapabs + i \Real \Th} \Dapbar \Ztbar - i \Real (\Id - \Hil)\cbrac{(\Dapabs + i \Real \Th) \Dapbar \Ztbar}  + i \Imag \nobrac{(\Id - \Hil) \Dt \Th} 
\end{align}
along with 
\begin{align}\label{eq:IdHilTh}
(\Id - \Hil)\Dt\Th = \sqbrac{\Dt,\Hil}\Th = \sqbrac{\bvar,\Hil}\pap\Th
\end{align}

\end{enumerate}

\smallskip
\subsection{Quasilinear equations} 

We now write down the quasilinear equations obtained in \cite{Ag19} that we will need to prove our energy estimates. See Section 4 of \cite{Ag19} for a derivation of these equations. 

\begin{enumerate}
\item Define the real valued variable $\Jone$ as
\begin{align} \label{form:Jone}
\begin{split}
\Jone &=  \Dt\Aone +  \Aone\brac{\bvarap - \Dap\Zt - \Dapbar\Ztbar} + \sigma \pap \Real \brac{\Id - \Hil} \cbrac{(\Dapabs + i\Real \Th) \Dapbar \Ztbar} \\
& \quad - \sigma\pap \Imag\brac{\Id - \Hil} \Dt\Th
\end{split}
\end{align}
Using this we get
\begin{align}\label{eq:Ztbar}
\Ztttbar + i\frac{\Aone}{\Zap}\Dapbar\Ztbar -i\sigma\Dap(\Dapabs + i\Real\Th)\Dapbar \Ztbar = -\sigma(\Dap\Zt) \Dap\Th -i\frac{\Jone}{\Zap}
\end{align}
We also have
\begin{align} \label{eq:ZtZap}
\begin{split}
& \Ztttbar\Zap + i\Aone\Dapbar\Ztbar -i\sigma\pap\brac{\frac{1}{\Zapbar}\Dapabs\Ztapbar  } \\
& = i\sigma\pap\cbrac{\brac{\Dapabs \frac{1}{\Zapbar}}\Ztapbar } -\sigma(\Dap\Zt)\pap\Th -\sigma\pap\cbrac{(\Real\Th)\Dapbar\Ztbar} -i\Jone
\end{split}
\end{align}
This equation gives rise to the energy $\Esigmaone$ in the energy estimate \thmref{thm:aprioriEsigma}.

\item \begin{equation} \label{eq:DapbarZtbar}
 \brac{\Dt^2 +i\frac{\Aone}{\Zapabs^2}\pap  -i\sigma\Dapabs^3} \Dapbar\Ztbar  =  \Rone -i\brac{\Dapbar\frac{1}{\Zap}} \Jone -i\frac{1}{\Zapabs^2}\pap \Jone 
\end{equation}
where
\begin{align} \label{form:Rone}
\begin{split}  
\Rone & =  -2(\Dapbar\Ztbar)(\Dt\Dapbar\Ztbar) -2\sigma\Real(\Dap\Zt)\Dapbar\Dap\Th -\sigma(\Dapbar\Dap\Zt)\Dap\Th \\ 
& \quad  +i\sigma\brac{2i\Real \brac{\Dapabs \Th} + \brac{\Real \Th}^2} \Dapabs \Dapbar\Ztbar -\sigma \Real \brac{\Dapabs^2 \Th}\Dapbar\Ztbar  \\
& \quad   + i\sigma \brac{\Real \Th} \brac{ \Real \brac{\Dapabs \Th}}\Dapbar\Ztbar 
\end{split}
\end{align}
and $\Jone$ was defined in \eqref{form:Jone}. This equation gives rise to the energy $\Esigmafour$ in the energy estimate \thmref{thm:aprioriEsigma}.

\item Multiply the equation for $\Dapbar\Ztbar$ in \eqref{eq:DapbarZtbar} by $\Zbarap$ to get the equation
\begin{align} \label{eq:Ztbarap}
\begin{split}
& \brac{\Dt^2 +i\frac{\Aone}{\Zapabs^2}\pap  -i\sigma\Dapabs^3} \Ztbarap \\
 & =  \Rone\Zapbar -i\brac{\pap\frac{1}{\Zap}} \Jone -i\Dap \Jone - \Zapbar\sqbrac{\Dt^2 +i\frac{\Aone}{\Zapabs^2}\pap -i\sigma\Dapabs^3, \frac{1}{\Zapbar}} \Ztapbar
\end{split}
\end{align}
This equation gives rise to the energy $\Esigmatwo$ in the energy estimate \thmref{thm:aprioriEsigma}. 

\item We have

\begin{align} \label{eq:Th}
 \brac{\Dt^2 +i\frac{\Aone}{\Zapabs^2}\pap  -i\sigma\Dapabs^3} \Th = \Rtwo +i\Jtwo
\end{align}
where
\begin{align}\label{form:Rtwo}
\begin{split}
\Rtwo &=  -2i\brac{\Dapbar \Ztbar}\brac{\Dapabs \Dapbar \Ztbar} + \brac{\Real \Th }\cbrac{ \brac{\Dapbar \Ztbar}^2 +i\Dapbar\brac{\frac{\Aone}{\Zap}} +i\sigma\brac{\Real\Th} \Dapabs\Th} \\
& \quad +\sigma \Real\brac{\Dapabs\Th} \Dapabs\Th + \brac{\Dapabs\frac{\Aone}{\Zapabs}}\brac{\frac{\Zap}{\Zapabs}\pap \frac{1}{\Zap}} + \Dapabs\brac{\frac{1}{\Zapabs^2}\pap \Aone}  \\
& \quad +  (\Id + \Hil) \Imag\cbrac{\Real(\Dapbar\Ztbar)\Dapabs\Dapbar\Ztbar  -i\Real (\Dt \Th)\Dapbar\Ztbar}\\
& \quad + \frac{\Aone}{\Zapabs^2}\pap\Real (\Id - \Hil) \brac{\frac{\Zap}{\Zapabs} \pap \frac{1}{\Zap}} 
\end{split} \\
 \Jtwo&= \Imag (\Id - \Hil) \brac*[\big]{\Dt^2 \Th} - \Real (\Id - \Hil)\cbrac{(\Dapabs + i \Real \Th) \Dt\Dapbar \Ztbar} 
\end{align}
Note that the variable $\Jtwo$ is real valued. This equation gives rise to the energy $\Esigmathree$ in the energy estimate \thmref{thm:aprioriEsigma}.
\end{enumerate}

\medskip

\subsection{Previous a priori estimate}\label{sec:aprioriEcalsigma} 

We now describe the main a priori estimate proved in Section 5 of \cite{Ag19}. We will need a modification of this energy when we prove our main energy estimate \thmref{thm:aprioriEDelta}. Define
\begingroup
\allowdisplaybreaks
\begin{align*}
& \Esigmazero =  \norm[\infty]{\sigma^{\half} \Zapabs^\half \pap\frac{1}{\Zap}}^2  +   \norm[2]{\sigma^\onebysix\Zapabs^\half\pap\frac{1}{\Zap}}^6  +  \norm[2]{\pap\frac{1}{\Zap}}^2  + \norm[2]{\frac{\sigma^\half}{\Zapabs^\half}\pap^2\frac{1}{\Zap}}^2 \\
& \Esigmaone = \norm[\Hhalf]{(\Zttbar -i)\Zap}^2 + \norm[2]{\sqrt{\Aone}\Ztapbar}^2 + \norm[2]{\frac{\sigma^\half}{\Zapabs^\half}\pap\Ztapbar}^2 \\
& \Esigmatwo = \norm[2]{\Dt\Ztapbar}^2 + \norm[\Hhalf]{\sqrt{\Aone}\frac{\Ztapbar}{\Zapabs}}^2 + \norm[\Hhalf]{\frac{\sigma^\half}{\Zapabs^\threebytwo}\pap\Ztapbar}^2\\
& \Esigmathree = \norm[2]{\Dt\Th}^2 + \norm[\Hhalf]{\sqrt{\Aone}\frac{\Th}{\Zapabs}}^2 + \norm[\Hhalf]{\frac{\sigma^\half}{\Zapabs^\threebytwo}\pap\Th}^2\\
& \Esigmafour = \norm[\Hhalf]{\Dt\Dapbar\Ztbar}^2 + \norm[2]{\sqrt{\Aone}\Dapabs\Dapbar\Ztbar}^2 + \norm[2]{\frac{\sigma^\half}{\Zapabs^\half}\pap\Dapabs\Dapbar\Ztbar}^2 \\
& \Esigma = \Esigmazero + \Esigmaone + \Esigmatwo + \Esigmathree + \Esigmafour
\end{align*}
\endgroup

\begin{thm}[\cite{Ag19}]\label{thm:aprioriEsigma}
Let $\sigma\geq 0$ and let $(\Z, \Zt)(t)$ be a smooth solution to \eqref{eq:systemone} in $[0,T]$ for $T>0$, such that for all $s\geq 3$ we have $(\Zap-1,\frac{1}{\Zap} - 1, \Zt) \in \Linfty([0,T], H^{s+\half}(\Rsp)\times H^{s+\half}(\Rsp)\times H^{s}(\Rsp))$. Then $\sup_{t\in[0,T]} \Esigma(t) < \infty$ and there exists a polynomial P with universal non-negative coefficients such that for all $t \in [0,T)$ we have
\begin{align*}
\frac{d\Esigma(t)}{dt} \leq P(\Esigma(t))
\end{align*}
\end{thm}

\begin{rem} \label{rem:aprioriEsigma}
Note that the energy $\Esigmazero$ contains a term which is the $L^\infty$ norm of a function and hence may not in general be differentiable in time even for smooth solutions. Hence for this term  the time derivative is replaced by the upper Dini derivative $  \limsup_{s \to 0^{+}} \frac{\norm[\infty]{f(t+s)}-\norm[\infty]{f(t)}}{s}$
\end{rem}

In this section when we write $\dis f \in \Ltwo$, what we mean is that there exists a universal polynomial $P$ with non-negative coefficients such that $\dis \norm*[2]{f} \leq P(\Esigma)$. Similar definitions for $\dis f\in \Hhalf$, $\dis  f \in \Linfty$ etc. This notation was used in \cite{Ag19} and we use it here as well as it simplifies the presentation. In order to prove this theorem, two new weighted spaces were also introduced to simplify the calculations. Define the spaces $\Ccal$ and $\Wcal$:
\begin{enumerate}
\item  If $\dis w \in \Linfty$ and $\dis \Dapabs w \in \Ltwo$, then we say $w\in \Wcal$. Define 
\begin{align}\label{def:Wcal}
 \norm[\Wcal]{w} = \norm[\infty]{w}+ \norm[2]{\Dapabs w}
\end{align}

\item If $\dis f \in \Hhalf$ and $\dis f\Zapabs \in \Ltwo$, then we say $f\in \mathcal{C}$. Define
\begin{align}\label{def:Ccal}
\norm[\Ccal]{f} = \norm[\Hhalf]{f} + \brac*[\bigg]{1+ \norm*[\bigg][2]{\pap\frac{1}{\Zapabs}}}\norm[2]{f\Zapabs}
\end{align}
\end{enumerate}

Also define the norm $\norm[\Wcal\cap\Ccal]{f} = \norm[\Wcal]{f} + \norm[\Ccal]{f}$. The main lemma governing the behavior of these spaces is given as follows: 

\medskip
\begin{lem} \label{lem:CW} 
The following properties hold for the spaces $\Wcal$ and $\Ccal$
\begin{enumerate}[leftmargin =*, align=left]
\item If $w_1,w_2 \in \Wcal$, then $w_1w_2 \in \Wcal$. Moreover $\norm[\Wcal]{w_1w_2} \leq \norm[\Wcal]{w_1}\norm[\Wcal]{w_2}$
\item If $f\in \Ccal$ and $w\in \Wcal$, then $fw \in \Ccal$.  Moreover $\norm[\Ccal]{fw} \lesssim \norm[\Ccal]{f}\norm[\Wcal]{w}$
\item If $f,g \in \Ccal$, then $fg\Zapabs \in \Ltwo$. Moreover $\norm[2]{fg\Zapabs} \lesssim \norm[\Ccal]{f}\norm[\Ccal]{g}$
\end{enumerate}
\end{lem}
This lemma was presented in \cite{Ag19} and its proof follows directly from \propref{prop:Hhalfweight}. We will continue to use these spaces as they are quite useful in our energy estimates. 

In the proof of \thmref{thm:aprioriEsigma}, several quantities were shown to be controlled by the energy $\Esigma$. We now list some of the quantities controlled by $\Esigma$ which we will use frequently in our energy estimates in \secref{sec:aprioriEcalhigh} and \secref{sec:aprioriEcalaux}. By Proposition 6.1 of \cite{Ag19}, the energy $\Esigma$ is equivalent to the energy $\Ecalsigma$ and so the following terms are also controlled by $\Ecalsigma$.  It is important to note that as the estimates are true for all $\sigma\geq0$, they are in particular true for $\sigma=0$. We will now list all the terms controlled in Section 5.1 in \cite{Ag19} which remain non-zero when we specialize to $\sigma=0$ case.

\medskip
\begin{enumerate}[widest = 99, leftmargin =*, align=left, label=\arabic*)]

\item $\dis \Ztapbar \in \Ltwo, \Dapabs\Dapbar\Ztbar \in \Ltwo$ 

\item $\dis \Aone \in \Linfty \cap \Hhalf$ 

\item $\dis \pap\frac{1}{\Zap} \in \Ltwo, \pap\frac{1}{\Zapabs} \in \Ltwo, \Dapabs \w \in \Ltwo$ and hence $\w \in \Wcal$

\item $\dis \Dapbar\Ztbar \in \Linfty, \Dapabs\Ztbar \in \Linfty, \Dap\Ztbar \in \Linfty$  
 
\item $\dis\Dapbar^2\Ztbar \in \Ltwo, \Dapabs^2\Ztbar \in \Ltwo, \Dap^2\Ztbar \in \Ltwo$ 

\item $\dis  \Dapbar\Ztbar \in \Wcal\cap\Ccal, \Dapabs\Ztbar \in \Wcal\cap\Ccal, \Dap\Ztbar \in \Wcal\cap\Ccal$

\item $\dis \pap \Pa\brac{\frac{\Zt}{\Zap}} \in \Linfty$

\item $\dis \Dapabs \Aone \in \Ltwo$ and hence $\dis \Aone \in \Wcal,  \sqrt{\Aone} \in \Wcal, \frac{1}{\Aone} \in \Wcal, \frac{1}{\sqrt{\Aone}} \in \Wcal$

\item $\dis \Th \in \Ltwo$, $\dis \Dt\Th \in \Ltwo$

\item $\dis \frac{\Th}{\Zapabs} \in \Ccal$

\item $\dis \Dap\frac{1}{\Zap}\in \Ccal$, $ \dis \Dapabs\frac{1}{\Zap}\in \Ccal$, $ \dis \Dapabs\frac{1}{\Zapabs}\in \Ccal$, $\dis \frac{1}{\Zapabs^2}\pap \w \in \Ccal$

\item $\dis \frac{1}{\Zapabs^2}\pap \Aone \in \Linfty\cap\Hhalf  $ and hence $\dis \frac{1}{\Zapabs^2}\pap \Aone \in \Ccal$

\item $\dis \frac{1}{\Zapabs^3}\pap^2\Aone \in \Ltwo $, $\dis \Dapabs\brac*[\bigg]{\frac{1}{\Zapabs^2}\pap\Aone} \in \Ltwo $ and hence $\dis \frac{1}{\Zapabs^2}\pap \Aone \in \Wcal$

\item $\dis \bvarap \in \Linfty\cap\Hhalf $ and $\Hil(\bvarap) \in \Linfty\cap\Hhalf $

\item $\dis \Dapabs \bvarap \in \Ltwo$ and hence $\bvarap \in \Wcal$ 

\item $\dis \pap\Dt\frac{1}{\Zap} \in \Ltwo$, $\dis \Dt\pap\frac{1}{\Zap} \in \Ltwo$

\item $\dis \Zttapbar \in \Ltwo $

\item $\dis \Dapbar \Zttbar \in \Ccal, \Dapabs \Zttbar \in \Ccal, \Dt\Dapbar\Ztbar \in \Ccal \tx{ and }\Dt\Dapabs\Ztbar \in \Ccal$

\item $\dis \Dt\Aone \in \Linfty\cap\Hhalf$

\item $\Dt(\bvarap - \Dap\Zt - \Dapbar\Ztbar) \in \Linfty\cap\Hhalf$ and hence $\dis \Dt\bvarap \in \Hhalf, \pap\Dt \bvar \in \Hhalf $

\item $\dis (\Id - \Hil)\Dt^2\Th \in \Ltwo$, $\dis (\Id - \Hil)\Dt^2\Ztapbar \in \Ltwo$, $\dis (\Id - \Hil)\Dt^2 \Dap\Ztbar \in \Hhalf$

\item $\dis \sqbrac{\Dt^2, \frac{1}{\Zap}}\Ztapbar \in \Ccal$, $\dis \sqbrac{\Dt^2, \frac{1}{\Zapbar}}\Ztapbar \in \Ccal$

\item $\dis \sqbrac{i\frac{\Aone}{\Zapabs^2}\pap, \frac{1}{\Zap}}\Ztapbar \in \Ccal$, $\dis \sqbrac{i\frac{\Aone}{\Zapabs^2}\pap, \frac{1}{\Zapbar}}\Ztapbar \in \Ccal$

\item $\dis \Rone \in \Ccal$

\item $\dis \Jone \in \Linfty\cap\Hhalf$

\item $\dis \Dapabs\Jone \in \Ltwo$ and hence $\dis \Jone \in \Wcal$

\item $\dis \Rtwo \in \Ltwo$

\item $\dis \Jtwo \in \Ltwo$

\item $\dis (\Id - \Hil)\Dt^2 \Dapbar\Ztbar \in \Hhalf$

\item $\dis \frac{1}{\Zapabs^2}\pap\Jone \in \Hhalf$ and hence $\dis \frac{1}{\Zapabs^2}\pap\Jone \in \Ccal$

\end{enumerate} 
\medskip

In addition to controlling these terms, some other estimates were also proved to prove the energy estimate \thmref{thm:aprioriEsigma}. We now give a useful lemma proved in Section 5.2 in \cite{Ag19}.

\begin{lem} \label{lem:timederiv}
Let $T>0$ and let $f,\bvar \in C^2([0,T), H^2(\Rsp))$ with $\bvar$ being real valued. Let $\Dt = \pt + \bvar\pap$. Then 
\begin{enumerate}[leftmargin =*, align=left]
\item $\dis \frac{\diff }{\diff t} \int f \diff \ap = \int \Dt f \diff\ap + \int \bvarap f \diff\ap$
\item $\dis \abs*[\Big]{ \frac{d}{dt}\int \abs{f}^2 \diff \ap - 2\Real \int \bar{f} (\Dt f) \diff \ap} \lesssim \norm[2]{f}^2 \norm[\infty]{\bvarap}$

\item $\dis \abs{ \frac{d}{dt}\int (\papabs\bar{f})f \diff\ap- 2\Real \cbrac{\int (\papabs\bar{f})\Dt f \diff \ap}} \lesssim   \norm[\Hhalf]{f}^2 \norm[\infty]{\bvarap}$
\end{enumerate}
\end{lem}

In proving  \thmref{thm:aprioriEsigma}, several estimates were proven in Section 5.2.3 of \cite{Ag19} in order to close the energy estimate. We now list some of them below. In the following $P$ represents a universal polynomial with non-negative coefficients and $(\Z,\Zt)(t)$ is a solution to \eqref{eq:systemone} with surface tension $\sigma$. 
\begin{enumerate}[leftmargin =*, align=left]
\item We have the estimate
\begin{align}\label{est:DtsqrtAone}
\norm[\Hhalf]{\Dt\brac{\frac{\sqrt{\Aone}}{\Zapabs} \fbar} - \frac{\sqrt{\Aone}}{\Zapabs}\Dt \fbar} \lesssim P(\Esigma)\norm[\Ccal]{\frac{f}{\Zapabs}}
\end{align}

\item If $\Ph f = f$ then we have the estimate
\begin{align}\label{est:DapabssqrtAonetwo}
\norm[2]{\frac{\sqrt{\Aone}}{\Zapabs}\papabs\brac{\frac{\sqrt{\Aone}}{\Zapabs}f } - i\frac{\Aone}{\Zapabs^2}\pap f } \lesssim P(\Esigma)\cbrac{\norm[\Hhalf]{\frac{\sqrt{\Aone}}{\Zapabs}f} + \norm[\Ccal]{\frac{f}{\Zapabs}}}
\end{align}

\end{enumerate}
\medskip

Finally in addition to the above estimates it was shown in Section 6 of \cite{Ag19} that the energies $\Esigma$ and $\Ecalsigma$ are equivalent.
\begin{prop}\label{prop:equivEsigmaEcalsigma} 
There exists universal polynomials $P_1, P_2$ with non-negative coefficients so that if $(\Z,\Zt)(t)$ is a smooth solution to the water wave equation \eqref{eq:systemone} for $\sigma \geq 0$ in the time interval $[0,T]$ satisfying $(\Zap-1,\frac{1}{\Zap} - 1, \Zt) \in \Linfty([0,T], H^{s+\half}(\Rsp)\times H^{s+\half }(\Rsp)\times H^{s }(\Rsp))$ for all $s\geq 3$,  then for all $t \in [0,T]$ we have
\begin{align*}
\Esigma(t) \leq P_1(\Ecalsigma(t)) \quad \tx{ and }\quad  \Ecalsigma(t) \leq P_2(\Esigma(t))
\end{align*}
\end{prop}

\medskip
\section{Higher order energy  $\Ecalhigh$ }\label{sec:aprioriEcalhigh}

In this section we prove an a priori energy estimate for solutions of \eqref{eq:systemone} for $\sigma=0$ for the energy $\Ecalhigh$ defined in \secref{sec:previousresult}.  We will first construct an energy $\Ehigh$, prove an a priori estimate for it in \thmref{thm:aprioriEhigh} and then show in \propref{prop:equivEhighEcalhigh} that $\Ehigh$ and $\Ecalhigh$ are equivalent energies. Note that we already have an energy estimate for $\sigma=0$ by simply taking the special case of $\sigma=0$ in the energy $\Esigma$ in \thmref{thm:aprioriEsigma}. The energy $\Ehigh$ is higher order than $\Esigma\vert_{\sigma=0}$ by half spacial derivative. We need an estimate for a higher order norm as we will need it to control $\Ecalaux$ in \secref{sec:aprioriEcalaux} which is in turn needed to prove \thmref{thm:convergence}. 

This higher order energy $\Ehigh$ is essentially equivalent to the energy used in \cite{KiWu18, Wu19},  (more precisely we do not have the dependence on the lower order term $\norm[\infty]{\frac{1}{\Zap}}(t)$ or $\abs{\frac{1}{\Zap}}(0,t)$ in the energy which they have) and in those papers an a priori estimate is already established. The main reason we are proving an a priori estimate here is because in the proof of the energy estimate of  $\Ehigh$, we prove several additional things such as control of several more terms (for example $\Dapabs\frac{1}{\Zapabs}\in\Wcal\cap\Ccal, \frac{1}{\Zapabs^2}\pap\w \in \Wcal\cap\Ccal,  \frac{1}{\Zapabs^2}\pap\Dapbar\Ztbar \in \Ccal$ etc.) and estimates regarding the time derivative of the energy, which are not available in \cite{KiWu18, Wu19}. We need these additional estimates in \secref{sec:aprioriEcalaux} to prove \thmref{thm:aprioriEaux}. An additional benefit of proving the a priori estimate is that it also allows us to remove the dependence on lower order terms such as  $\norm[\infty]{\frac{1}{\Zap}}(t)$ or $\abs{\frac{1}{\Zap}}(0,t)$. 

Let us now define the energy $\Ehigh$ and state the a priori estimate. Define
\begin{align}\label{def:Ehigh}
\Ehigh = \Esigma\vert_{\sigma=0} + \norm[2]{\Dt\Dap^2\Ztbar}^2 + \norm[\Hhalf]{\frac{\sqrt{\Aone}}{\Zapabs}\Dap^2\Ztbar}^2
\end{align}

\begin{thm}\label{thm:aprioriEhigh}
Let $T>0$ and let $(\Z, \Zt)(t)$ be a smooth solution to \eqref{eq:systemone} with $\sigma=0$ in the time interval $[0,T]$, such that for all $s\geq 2$ we have $(\Zap-1,\frac{1}{\Zap} - 1, \Zt) \in \Linfty([0,T], H^{s}(\Rsp)\times H^{s}(\Rsp)\times H^{s+\half}(\Rsp))$. Then $\sup_{t\in[0,T]} \Ehigh(t) < \infty$ and there exists a universal polynomial P with non-negative coefficients such that for all $t \in [0,T)$ we have
\begin{align*}
\frac{d\Ehigh(t)}{dt} \leq P(\Ehigh(t))
\end{align*}
\end{thm}
\medskip

To prove this theorem, we first obtain the quasilinear equation in \secref{sec:quasiEhigh}, then control the quantities controlled by $\Ehigh$ in \secref{sec:quantEhigh} and then finish the proof in \secref{sec:closeEhigh}. Finally we show in \secref{sec:equivEhighEcalhigh} that $\Ehigh$ and $\Ecalhigh$ are equivalent energies.

\subsection{Quasilinear equation}\label{sec:quasiEhigh}

Let us now derive the quasilinear equation relevant for the energy $\Ehigh$. Plugging in $\sigma = 0$ in the equation for $\Ztbar$ from \eqref{eq:Ztbar} we obtain
\begin{align*}
\brac{\Dt^2 +i\frac{\Aone}{\Zapabs^2}\pap } \Ztbar = -i\frac{\Jone}{\Zap}
\end{align*}
with $\Jone$ defined by \eqref{form:Jone} with $\sigma=0$, i.e. $\Jone =  \Dt\Aone +  \Aone\brac{\bvarap - \Dap\Zt - \Dapbar\Ztbar}$.  Applying $\Dap^2$ to the above equation we obtain
\begin{align*}
 \brac{\Dt^2 +i\frac{\Aone}{\Zapabs^2}\pap } \Dap^2\Ztbar  = -i\Dap^2\brac{\frac{\Jone}{\Zap}} + \sqbrac{\Dt^2 +i\frac{\Aone}{\Zapabs^2}\pap , \Dap^2} \Ztbar 
\end{align*}
Let us try to simplify the terms above. We will heavily use the commutator estimates from  \eqref{eq:commutator} and also use the equation \eqref{form:Zttbar} for $\sigma=0$. 

\begin{enumerate}[leftmargin =*, align=left, label=\alph*)]

\item We see that
\begin{align*}
& \sqbrac{\Dt^2 +i\frac{\Aone}{\Zapabs^2}\pap , \Dap} \\
&= \Dt\sqbrac{\Dt,\Dap} + \sqbrac{\Dt,\Dap}\Dt + \sqbrac{(\Ztt + i)\Dap,\Dap} \\
& = -\Dt\cbrac{(\Dap\Zt)\Dap} - (\Dap\Zt)\Dap\Dt - (\Dap\Ztt)\Dap \\
& = - \cbrac{\Dt(\Dap\Zt)}\Dap - (\Dap\Zt)\Dt\Dap - (\Dap\Zt)\Dap\Dt - (\Dap\Ztt)\Dap  \\
& = \cbrac{-2(\Dap\Ztt) + 2(\Dap\Zt)^2}\Dap -2(\Dap\Zt)\Dap\Dt
\end{align*}

\item We have the relation
\begin{align*}
& \sqbrac{\Dt^2 +i\frac{\Aone}{\Zapabs^2}\pap , \Dap}\Dap \\
\quad &= \cbrac{-2(\Dap\Ztt) + 2(\Dap\Zt)^2}\Dap^2 -2(\Dap\Zt)\Dap\Dt\Dap \\
&= \cbrac{-2(\Dap\Ztt) + 2(\Dap\Zt)^2}\Dap^2 -2(\Dap\Zt)\Dap\cbrac{-(\Dap\Zt)\Dap + \Dap\Dt} \\
& = 2(\Dap\Zt)(\Dap^2\Zt)\Dap + \cbrac{-2(\Dap\Ztt) + 4(\Dap\Zt)^2}\Dap^2 - 2(\Dap\Zt)\Dap^2\Dt
\end{align*}

\item We similarly have
\begin{align*}
& \Dap\sqbrac{\Dt^2 +i\frac{\Aone}{\Zapabs^2}\pap , \Dap} \\
\quad &=  \cbrac{-2(\Dap^2\Ztt) + 4(\Dap\Zt)(\Dap^2\Zt)}\Dap  + \cbrac{-2(\Dap\Ztt) + 2(\Dap\Zt)^2}\Dap^2 \\
& \quad   -2(\Dap^2\Zt)\Dap\Dt - 2(\Dap\Zt)\Dap^2\Dt
\end{align*}

\item Hence we have 
\begin{align*}
& \sqbrac{\Dt^2 +i\frac{\Aone}{\Zapabs^2}\pap , \Dap^2}  \\
\quad &= \sqbrac{\Dt^2 +i\frac{\Aone}{\Zapabs^2}\pap , \Dap}\Dap + \Dap\sqbrac{\Dt^2 +i\frac{\Aone}{\Zapabs^2}\pap , \Dap} \\
&=  \cbrac{-2(\Dap^2\Ztt) + 6(\Dap\Zt)(\Dap^2\Zt)}\Dap  + \cbrac{-4(\Dap\Ztt) + 6(\Dap\Zt)^2}\Dap^2 \\
& \quad   -2(\Dap^2\Zt)\Dap\Dt - 4(\Dap\Zt)\Dap^2\Dt
\end{align*}

\item We see that
\begin{align*}
-i\Dap^2\brac{\frac{\Jone}{\Zap}} &= -i\Dap\cbrac{ \frac{\wbar^2}{\Zapabs^2}\pap\Jone + \Jone\brac{\Dap\frac{1}{\Zap}}}  \\
& = -i\wbar^3\Dapabs\brac{\frac{1}{\Zapabs^2}\pap\Jone} -2i\wbar(\Dap\wbar)\brac{\frac{1}{\Zapabs^2}\pap\Jone}  \\
& \quad - i(\Dap\Jone)\brac{\Dap\frac{1}{\Zap}} -i\Jone\brac{\Dap^2\frac{1}{\Zap}}
\end{align*}

\end{enumerate}
\bigskip
Combining the above identities we get the equation for $\Dap^2\Ztbar$
\begin{align}\label{eq:Dap^2Ztbar}
\brac{\Dt^2 +i\frac{\Aone}{\Zapabs^2}\pap}\Dap^2\Ztbar = -i\wbar^3\Dapabs\brac{\frac{1}{\Zapabs^2}\pap\Jone} + \Rthree
\end{align}
where
\begin{align}\label{form:Rthree}
\begin{split}
\Rthree & =  \cbrac{-2(\Dap^2\Ztt) + 6(\Dap\Zt)(\Dap^2\Zt)}(\Dap\Ztbar)  + \cbrac{-4(\Dap\Ztt) + 6(\Dap\Zt)^2}(\Dap^2\Ztbar) \\
& \quad   -2(\Dap^2\Zt)(\Dap\Zttbar) - 4(\Dap\Zt)(\Dap^2\Zttbar)  -2i\wbar(\Dap\wbar)\brac{\frac{1}{\Zapabs^2}\pap\Jone}  \\
& \quad - i(\Dap\Jone)\brac{\Dap\frac{1}{\Zap}} -i\Jone\brac{\Dap^2\frac{1}{\Zap}}
\end{split}
\end{align}
\bigskip

\subsection{Quantities controlled by the energy  $\Ehigh$}\label{sec:quantEhigh}

In this section whenever we write $\dis f \in \Ltwo$, what we mean is that there exists a universal polynomial $P$ with nonnegative coefficients such that $\dis \norm*[2]{f} \leq P(\Ehigh)$. Similar definitions for $\dis f\in \Hhalf$, $\dis  f \in \Linfty$, $\dis f \in \Ccal$ or $\dis f \in \Wcal$ where the definitions for the spaces $\Ccal$ and $\Wcal$ are as in \eqref{def:Wcal},\eqref{def:Ccal}. Note that $\Ehigh$ controls the energy $\Esigma\vert_{\sigma=0}$ and hence we already have control of a lot of quantities as listed in \secref{sec:aprioriEcalsigma}. We will freely use the quantities controlled by $\Esigma\vert_{\sigma=0}$ to prove the above theorem. In particular we will also be making use of \lemref{lem:CW}. Let us now establish the quantities controlled by $\Ehigh$ which are not controlled by $\Esigma\vert_{\sigma=0}$. 

\medskip
\begin{enumerate}[widest = 99, leftmargin =*, align=left, label=\arabic*)]

\item $\dis \Dt\Dap^2\Ztbar \in \Ltwo$, $\dis \Dap^2\Zttbar \in \Ltwo$, $\dis \Dapabs^2\Zttbar \in \Ltwo$ and $\dis \Dap^2\Ztt \in \Ltwo$, $\dis \Dapabs\Dt\Dap\Zt \in \Ltwo$
\begin{flalign*}
\lpar \tx{Proof:  We see that \qquad} \Dt\Dap^2\Ztbar &= \sqbrac{\Dt,\Dap^2}\Ztbar + \Dap^2\Zttbar &\\
& = \Dap\cbrac{-(\Dap\Zt)\Dap\Ztbar} -(\Dap\Zt)\Dap^2\Ztbar + \Dap^2\Zttbar \\
& = -(\Dap^2\Zt)\Dap\Ztbar -2(\Dap\Zt)\Dap^2\Ztbar + \Dap^2\Zttbar
\end{flalign*}
Now $\Dt\Dap^2\Ztbar \in \Ltwo$ as it part of the energy and hence we have
\[
\norm[2]{\Dap^2\Zttbar} \lesssim \norm[2]{\Dt\Dap^2\Ztbar} + \norm[2]{\Dap^2\Zt}\norm[\infty]{\Dap\Ztbar} + \norm[2]{\Dap^2\Ztbar}\norm[\infty]{\Dap\Zt}
\]
Now we observe that
\[
\Dap^2\Zttbar = \Dap\brac{\wbar\Dapabs\Zttbar} = \wbar(\Dapabs\wbar)\Dapabs\Zttbar + \wbar^2\Dapabs^2\Zttbar
\]
Hence from \lemref{lem:CW} we have
\[
\norm*[\big][2]{\Dapabs^2\Zttbar} \lesssim \norm[2]{\Dap^2\Zttbar} + \norm*[\bigg][\Ccal]{\frac{1}{\Zapabs^2}\pap\wbar}\norm[\Ccal]{\Dapabs\Zttbar}
\]
The terms $\Dap^2\Ztt \in \Ltwo$, $\dis \Dapabs\Dt\Dap\Zt \in \Ltwo$ are proven similarly. 
\Bigskip

\item $\dis \Dap\Zttbar \in \Wcal\cap\Ccal$, $\dis \Dapabs\Zttbar \in \Wcal\cap\Ccal$ and $\dis \Dapbar\Zttbar \in \Wcal\cap\Ccal$
\medskip\\
Proof: We already know that $\dis \Dap\Zttbar \in \Ccal$ and $\Dap^2\Zttbar \in \Ltwo$. Hence using $f = \Dap\Zttbar$ and $w = \frac{1}{\Zap}$ in  \propref{prop:LinftyHhalf} we obtain
\[
\norm[\infty]{\Dap\Zttbar}^2 \lesssim \norm[2]{\Zttbarap}\norm[2]{\Dap^2\Zttbar}
\]
Hence $\Dap\Zttbar \in \Wcal$. We also have from \lemref{lem:CW} that $\dis \norm[\Wcal\cap\Ccal]{\Dapabs\Zttbar} \lesssim \norm[\Wcal]{\w}\norm[\Wcal\cap\Ccal]{\Dap\Zttbar}$. The proof of  $\dis \Dapbar\Zttbar \in \Wcal\cap\Ccal$ is done similarly.

\Bigskip

\item $\dis \Dapabs\frac{1}{\Zap} \in \Linfty$, $\dis \Dapabs\frac{1}{\Zapabs} \in \Linfty$ and $\dis \frac{1}{\Zapabs^2}\pap\w \in \Linfty$
\medskip\\
Proof: We observe from \eqref{form:Zttbar} that $\dis \Zttbar -i = -i\frac{\Aone}{\Zap}$ and hence we have
\[
\Dapabs\Zttbar = -i\frac{\wbar}{\Zapabs^2}\pap\Aone -i\Aone\Dapabs\frac{1}{\Zap}
\]
As $\Aone \geq 1$ we have
\[
\norm*[\bigg][\infty]{\Dapabs\frac{1}{\Zap}} \lesssim \norm*[\bigg][\infty]{\frac{1}{\Zapabs^2}\pap\Aone} + \norm[\infty]{\Dapabs\Zttbar}
\]
Hence $\dis \Dapabs \frac{1}{\Zap} \in \Linfty$. Now using \eqref{form:RealImagTh} we see that
\[
\Real\brac{\Dapbar\frac{1}{\Zap} } = \Dapabs \frac{1}{\Zapabs} \quad \qquad \Imag\brac{\Dapbar\frac{1}{\Zap}} =  i\brac{ \frac{\wbar}{\Zapabs^2}\pap \w} 
\]
Hence we obtain $\dis \Dapabs\frac{1}{\Zapabs} \in \Linfty$ and $\dis \frac{1}{\Zapabs^2}\pap\w \in \Linfty$
\Bigskip

\item  $\dis \Dapabs^2\frac{1}{\Zap} \in \Ltwo$, $\dis \Dap^2\frac{1}{\Zap} \in \Ltwo$ and similarly   $\dis \Dapabs^2\frac{1}{\Zapabs} \in \Ltwo$, $\dis \frac{1}{\Zapabs^3}\pap^3\w \in \Ltwo$. We also have $\dis \frac{1}{\Zapabs^2}\pap^2\frac{1}{\Zap} \in \Ltwo$, $\dis \frac{1}{\Zapabs^2}\pap\Th \in \Ltwo$
\medskip\\
Proof: From \eqref{form:Zttbar} we have $\dis \Zttbar -i = -i\frac{\Aone}{\Zap}$ and hence
\begin{align*}
 & \Dapabs^2\Zttbar \\
 &= \Dapabs\brac*[\bigg]{-i\frac{\wbar}{\Zapabs^2}\pap\Aone -i\Aone\Dapabs\frac{1}{\Zap}} \\
&= -i\wbar\Dapabs\brac*[\bigg]{\frac{1}{\Zapabs^2}\pap\Aone} -i\brac{\Dapabs\wbar}\brac*[\bigg]{\frac{1}{\Zapabs^2}\pap\Aone} -i(\Dapabs\Aone)\brac{\Dapabs\frac{1}{\Zap}} \\
& \quad   -i\Aone\Dapabs^2\frac{1}{\Zap}
\end{align*}
As $\Aone \geq 1$ we see that
\begin{align*}
\lpar \norm[2]{\Dapabs^2\frac{1}{\Zap}} & \lesssim \norm*[\big][2]{\Dapabs^2\Zttbar} + \norm*[\bigg][2]{\Dapabs\brac*[\bigg]{\frac{1}{\Zapabs^2}\pap\Aone}}  + \norm[2]{\Dapabs\w}\norm*[\bigg][\infty]{\frac{1}{\Zapabs^2}\pap\Aone} \\
& \quad + \norm[2]{\Dapabs\Aone}\norm[\infty]{\Dapabs\frac{1}{\Zap}}
\end{align*}
Now we see that
\[
\Dap^2\frac{1}{\Zap} = \Dap\brac{\wbar\Dapabs\frac{1}{\Zap}} = \wbar(\Dapabs\wbar)\Dapabs\frac{1}{\Zap} + \wbar^2\Dapabs^2\frac{1}{\Zap}
\]
Hence we have
\[
\norm[2]{\Dap^2\frac{1}{\Zap}} \lesssim \norm[2]{\Dap\wbar}\norm[\infty]{\Dapabs\frac{1}{\Zap}} + \norm[2]{\Dapabs^2\frac{1}{\Zap}}
\]
Now using the formula \eqref{form:RealImagTh} we have
\[
\Real\brac{\Dapbar\frac{1}{\Zap} } = \Dapabs \frac{1}{\Zapabs} \quad \qquad \Imag\brac{\Dapbar\frac{1}{\Zap}} =  i\brac{ \frac{\wbar}{\Zapabs^2}\pap \w} 
\]
Hence we have
\begin{align*}
\norm[2]{\Dapabs^2\frac{1}{\Zapabs}} & \lesssim \norm[2]{\Dapabs\w}\norm[\infty]{\Dapabs\frac{1}{\Zap}} + \norm[2]{\Dapabs^2\frac{1}{\Zap}} \\
\norm*[\bigg][2]{\frac{1}{\Zapabs^3}\pap^2\w} & \lesssim \norm*[\big][2]{\Dapabs\w}\norm[\infty]{\Dapabs\frac{1}{\Zap}} + \norm[2]{\Dapabs^2\frac{1}{\Zap}} + \norm[2]{\Dapabs\wbar}\norm*[\bigg][\infty]{\frac{1}{\Zapabs^2}\pap\w}  \\
& \quad + \norm[\infty]{\Dapabs\frac{1}{\Zapabs}}\norm*[\big][2]{\Dapabs\w}
\end{align*}
We also see that 
\begin{align*}
\Dapabs^2\frac{1}{\Zap} = \Dapabs\brac{\frac{1}{\Zapabs}\pap \frac{1}{\Zap}} = \brac{\Dapabs \frac{1}{\Zapabs}}\pap\frac{1}{\Zap} + \frac{1}{\Zapabs^2}\pap^2\frac{1}{\Zap}
\end{align*}
Hence we have
\begin{align*}
\norm[2]{ \frac{1}{\Zapabs^2}\pap^2\frac{1}{\Zap}} \lesssim \norm[\infty]{\Dapabs\frac{1}{\Zapabs}}\norm[2]{\pap\frac{1}{\Zap}} + \norm[2]{\Dapabs^2\frac{1}{\Zap} }
\end{align*}
Now recall the formula of $\Th$ from \eqref{form:Th}
\begin{align*}
\Th = i \frac{\Zap}{\Zapabs} \pap \frac{1}{\Zap} - i \Real (\Id - \Hil) \brac{\frac{\Zap}{\Zapabs} \pap \frac{1}{\Zap}} 
\end{align*}
We have
\begin{align*}
\norm[2]{\frac{1}{\Zapabs^2}\pap\brac{\frac{\Zap}{\Zapabs}\pap\frac{1}{\Zap}}} \lesssim \norm[\infty]{\frac{1}{\Zapabs^2}\pap\w}\norm[2]{\pap\frac{1}{\Zap}} + \norm[2]{ \frac{1}{\Zapabs^2}\pap^2\frac{1}{\Zap} }
\end{align*}
and hence from \propref{prop:commutator} we obtain 
\begin{align*}
\norm[2]{\frac{1}{\Zapabs^2}\pap\Th} \lesssim \norm[\infty]{\Dapabs\frac{1}{\Zapabs}}\norm[2]{\pap\frac{1}{\Zap}} + \norm[2]{\frac{1}{\Zapabs^2}\pap\brac{\frac{\Zap}{\Zapabs}\pap\frac{1}{\Zap}}}
\end{align*}
\bigskip

\item $\dis \Dapabs\frac{1}{\Zap} \in \Wcal\cap\Ccal$, $\dis \Dapabs\frac{1}{\Zapabs} \in \Wcal\cap\Ccal$ and $\dis \frac{1}{\Zapabs^2}\pap\w \in \Wcal\cap\Ccal$
\medskip\\
Proof: The inclusion in $\Ccal$ is known as it is part of energy estimate for $\Esigma$ for $\sigma = 0$. Now we have already shown that all the quantities are in $\Linfty$ and using the above estimates like $\dis \Dapabs^2\frac{1}{\Zap} \in \Ltwo$ we are done. 
\bigskip

\item $\dis \Dt\bvarap \in \Linfty$, $\dis \pap\Dt\bvar \in \Linfty$
\medskip\\
Proof: We already know that $\Esigma\vert_{\sigma=0}$ controls $\Dt(\bvarap -\Dap\Zt -\Dapbar\Ztbar) \in \Linfty$. Now 
\[
\Dt\Dap\Zt = -(\Dap\Zt)^2 + \Dap\Ztt
\]
and hence as $\Dap\Zt \in \Linfty$, $\Dap\Ztt \in \Linfty$ we have $\Dt\Dap\Zt \in \Linfty$. Hence we have $\Dt\bvarap \in \Linfty$. We now have $\pap\Dt\bvar = \bvarap^2 + \Dt\bvarap$ and as $\bvarap \in \Linfty$ we see that $\pap\Dt\bvar \in \Linfty$. 
\Bigskip

\item $\dis \frac{1}{\Zapabs}\Dap^2\Ztbar \in \Ccal$,$\dis \frac{1}{\Zapabs^2}\pap\Dap\Ztbar \in \Ccal$,  $\dis \frac{1}{\Zapabs^3}\pap\Ztbarap \in \Ccal$ and  $\dis \frac{1}{\Zapabs^2}\pap\Dapbar\Ztbar \in \Ccal$ and similarly $\dis \Dt^2\Dapbar\Ztbar \in \Ccal$
\medskip\\
Proof: From the energy we know that  $\dis \frac{\sqrt{\Aone}}{\Zapabs}\Dap^2\Ztbar \in \Hhalf$. Hence as $\dis \sqrt{\Aone}\Dap^2\Ztbar \in \Ltwo$ we see that $\dis \frac{\sqrt{\Aone}}{\Zapabs}\Dap^2\Ztbar \in \Ccal$. Hence from \lemref{lem:CW} we see that
\[
\norm[\Ccal]{\frac{1}{\Zapabs}\Dap^2\Ztbar} \lesssim \norm[\Ccal]{\frac{\sqrt{\Aone}}{\Zapabs}\Dap^2\Ztbar}\norm[\Wcal]{\frac{1}{\sqrt{\Aone}}}
\]
As $\w \in \Wcal$ we again get from \lemref{lem:CW} that $\dis \frac{1}{\Zapabs^2}\pap\Dap\Ztbar \in \Ccal$. Now 
\[
\frac{1}{\Zapabs}\Dap^2\Ztbar = \brac{\Dap\frac{1}{\Zap}}\Dapabs\Ztbar + \frac{\wbar^2}{\Zapabs^3}\pap\Ztbarap
\]
Hence from  \lemref{lem:CW} we have
\[
\norm*[\bigg][\Ccal]{\frac{1}{\Zapabs^3}\pap\Ztbarap} \lesssim \norm[\Wcal]{\w}^2\norm[\Ccal]{\frac{1}{\Zapabs}\Dap^2\Ztbar} + \norm[\Wcal]{\w}^2\norm[\Ccal]{\Dap\frac{1}{\Zap}}\norm[\Wcal]{\Dapabs\Ztbar}
\]
Now we see that
\begin{align*}
\frac{1}{\Zapabs^2}\pap\Dapbar\Ztbar = \brac{\Dapabs\frac{1}{\Zapbar}}\Dapabs\Ztbar + \frac{\w}{\Zapabs^3}\pap\Ztapbar
\end{align*}
Hence by using  \lemref{lem:CW} we obtain
\begin{align*}
\norm[\Ccal]{\frac{1}{\Zapabs^2}\pap\Dapbar\Ztbar} \lesssim \norm[\Ccal]{\Dapabs\frac{1}{\Zapbar}}\norm[\Wcal]{\Dapabs\Ztbar} + \norm[\Wcal]{w}\norm[\Ccal]{\frac{1}{\Zapabs^3}\pap\Ztapbar}
\end{align*}
Now we recall the equation for $\Dapbar\Ztbar$ for $\sigma = 0$ from \eqref{eq:DapbarZtbar} and \eqref{form:Rone}
\begin{align*}
\brac{\Dt^2 +i\frac{\Aone}{\Zapabs^2}\pap } \Dapbar\Ztbar  =  -2(\Dapbar\Ztbar)(\Dt\Dapbar\Ztbar) -i\brac{\Dapbar\frac{1}{\Zap}} \Jone -i\frac{1}{\Zapabs^2}\pap \Jone 
\end{align*}
Hence from  \lemref{lem:CW} we have
\begin{align*}
\norm[\Ccal]{\Dt^2\Dapbar\Ztbar} & \lesssim \norm[\Wcal]{\Aone}\norm[\Ccal]{\frac{1}{\Zapabs^2}\pap\Dapbar\Ztbar} + \norm[\Wcal]{\Dapbar\Ztbar}\norm[\Ccal]{\Dt\Dapbar\Ztbar}  \\
& \quad  + \norm[\Ccal]{\Dapbar\frac{1}{\Zap}}\norm[\Wcal]{\Jone} + \norm[\Ccal]{\frac{1}{\Zapabs^2}\pap\Jone}
\end{align*}

\medskip

\item $\dis (\Id - \Hil)\Dt^2\Dap^2\Ztbar \in \Ltwo$
\medskip\\
Proof: For a function $f$ satisfying $\Pa f = 0$ we have from \propref{prop:tripleidentity}
\begin{align*}
(\Id - \Hil)\Dt^2f & = \sqbrac{\Dt,\Hil}\Dt f + \Dt\sqbrac{\Dt,\Hil}f \\
& = \sqbrac{\bvar,\Hil}\pap\Dt f + \Dt\sqbrac{\bvar,\Hil}\pap f \\
& = 2\sqbrac{\bvar,\Hil}\pap\Dt f + \sqbrac{\Dt\bvar,\Hil}\pap f - \sqbrac{\bvar, \bvar ; \pap f}
\end{align*}
As $\Pa \Dap^2\Ztbar = 0$ we obtain from \propref{prop:commutator} and \propref{prop:triple}
\[
\norm[2]{(\Id - \Hil)\Dt^2\Dap^2\Ztbar} \lesssim \norm[\infty]{\bvarap}\norm[2]{\Dt\Dap^2\Ztbar} + \norm[\Hhalf]{\pap\Dt\bvar}\norm[2]{\Dap^2\Ztbar} + \norm[\infty]{\bvarap}^2\norm[2]{\Dap^2\Ztbar}
\]
\medskip

\item $\dis (\Id - \Hil)\brac{i\frac{\Aone}{\Zapabs^2}\pap\Dap^2\Ztbar} \in \Ltwo$
\medskip\\
Proof: We see that
\[
(\Id - \Hil)\brac{i\frac{\Aone}{\Zapabs^2}\pap\Dap^2\Ztbar} = i\sqbrac{\frac{\Aone}{\Zapabs^2},\Hil}\pap\Dap^2\Ztbar
\]
and hence we have from \propref{prop:commutator}
\begin{align*}
\!\! \norm*[\Bigg][2]{(\Id - \Hil)\brac{i\frac{\Aone}{\Zapabs^2}\pap\Dap^2\Ztbar}} \lesssim \norm[2]{\Dap^2\Ztbar}\cbrac{\norm*[\bigg][\infty]{\frac{1}{\Zapabs^2}\pap\Aone} + \norm[\infty]{\Aone}\norm[\infty]{\Dapabs\frac{1}{\Zapabs}} }
\end{align*}
\Bigskip

\item $\Rthree \in \Ltwo$
\medskip\\
Proof: We recall from \eqref{form:Rthree} the formula of $\Rthree$
\begin{align*}
\Rthree & =  \cbrac{-2(\Dap^2\Ztt) + 6(\Dap\Zt)(\Dap^2\Zt)}(\Dap\Ztbar)  + \cbrac{-4(\Dap\Ztt) + 6(\Dap\Zt)^2}(\Dap^2\Ztbar) \\
& \quad   -2(\Dap^2\Zt)(\Dap\Zttbar) - 4(\Dap\Zt)(\Dap^2\Zttbar)  -2i\wbar(\Dap\wbar)\brac{\frac{1}{\Zapabs^2}\pap\Jone}  \\
& \quad - i(\Dap\Jone)\brac{\Dap\frac{1}{\Zap}} -i\Jone\brac{\Dap^2\frac{1}{\Zap}}
\end{align*}
Hence using \lemref{lem:CW} we easily have the estimate
\begin{align*}
& \norm[2]{\Rthree} \\
& \lesssim \cbrac*[\Big]{\norm[2]{\Dap^2\Ztt} + \norm[\infty]{\Dap\Zt}\norm[2]{\Dap^2\Zt}}\norm[\infty]{\Dap\Ztbar} + \cbrac{\norm[\infty]{\Dap\Ztt} + \norm[\infty]{\Dap\Zt}^2 }\norm[2]{\Dap^2\Ztbar}  \\
& \quad + \norm[2]{\Dap^2\Zt}\norm[\infty]{\Dap\Zttbar} + \norm[\infty]{\Dap\Zt}\norm[2]{\Dap^2\Zttbar} + \norm[\Ccal]{\frac{1}{\Zapabs^2}\pap\wbar}\norm*[\bigg][\Ccal]{\frac{1}{\Zapabs^2}\pap\Jone} \\
& \quad + \norm[2]{\Dap\Jone}\norm[\infty]{\Dap\frac{1}{\Zap}} + \norm[\infty]{\Jone}\norm[2]{\Dap^2\frac{1}{\Zap}}
\end{align*}

\Bigskip

\item $\dis \Dapabs\brac{\frac{1}{\Zapabs^2}\pap\Jone} \in \Ltwo$ and hence $\dis \frac{1}{\Zapabs^2}\pap\Jone \in \Linfty$
\medskip\\
Proof: As $\Jone$ is real valued we have
\begin{align*}
 & \Dapabs\brac{\frac{1}{\Zapabs^2}\pap\Jone} \\
 & = \Real(\Id - \Hil)\cbrac{\frac{1}{\Zapabs}\pap\brac{\frac{1}{\Zapabs^2}\pap\Jone}} \\
& = \Real\cbrac{\sqbrac{\frac{1}{\Zapabs},\Hil}\pap\brac{\frac{1}{\Zapabs^2}\pap\Jone} - \w^3\sqbrac{\frac{\wbar^3}{\Zapabs},\Hil}\pap\brac{\frac{1}{\Zapabs^2}\pap\Jone}} \\
& \quad + \Real\cbrac{\w^3(\Id - \Hil)\cbrac{\frac{\wbar^3}{\Zapabs}\pap\brac{\frac{1}{\Zapabs^2}\pap\Jone}} }
\end{align*}
Now applying $(\Id - \Hil)$ on the equation for $\Dap^2\Ztbar$ from \eqref{eq:Dap^2Ztbar} we obtain 
\begin{align*}
\norm[2]{(\Id - \Hil)\cbrac{\wbar^3\Dapabs\brac{\frac{1}{\Zapabs^2}\pap\Jone}}} & \lesssim \norm[2]{(\Id - \Hil)\Dt^2\Dap^2\Ztbar}  + \norm[2]{\Rthree} \\
& \quad + \norm[2]{(\Id - \Hil)\brac{i\frac{\Aone}{\Zapabs^2}\pap\Dap^2\Ztbar}} 
\end{align*}
Hence we have from \propref{prop:commutator}
\begin{align*}
\norm[2]{\Dapabs\brac{\frac{1}{\Zapabs^2}\pap\Jone}} & \lesssim \cbrac{\norm[2]{\pap\frac{1}{\Zapabs}} + \norm[2]{\Dapabs\wbar}}\norm[\Hhalf]{\frac{1}{\Zapabs^2}\pap\Jone} \\
& \quad + \norm[2]{(\Id - \Hil)\cbrac{\wbar^3\Dapabs\brac{\frac{1}{\Zapabs^2}\pap\Jone}}}
\end{align*}
Now we just use  \propref{prop:LinftyHhalf} with the functions $\dis f = \frac{1}{\Zapabs^2}\pap\Jone$ and $\dis w = \frac{1}{\Zapabs}$ and we easily get that $\dis \frac{1}{\Zapabs^2}\pap\Jone \in \Linfty$
\Bigskip

\end{enumerate}

\subsection{Closing the energy estimate for  $\Ehigh$}\label{sec:closeEhigh}

We now complete the proof of \thmref{thm:aprioriEhigh}. To simplify the calculations, we will continue to use the notation used in \secref{sec:quantEhigh} and introduce another notation: If $a(t), b(t) $ are functions of time we write $a \approx b$ if there exists a universal  polynomial $P$ with non-negative coefficients such that $\abs{a(t)-b(t)} \leq P(\Ehigh(t))$. Observe that $\approx$ is an equivalence relation. With this notation, proving \thmref{thm:aprioriEhigh} is equivalent to showing $ \frac{d\Ehigh(t)}{dt}  \approx 0$. 

Now we know from \thmref{thm:aprioriEsigma} that 
\[
\frac{d\Esigma(t)}{dt} \leq P(\Esigma(t))
\]
and hence this is true for $\sigma = 0$ with the same polynomial $P$. Hence we have
\[
\frac{d(\Esigma\vert_{\sigma=0})(t)}{dt} \leq P((\Esigma\vert_{\sigma=0})(t)) \leq P(\Ehigh(t))
\]
Hence we only need to control the time derivative of $\Ehigh - \Esigma\vert_{\sigma=0}$. Hence
\begin{align*}
\frac{d\Ehigh(t)}{dt} \approx \frac{d}{dt}\cbrac{ \int \abs{\Dt\Dap^2\Ztbar}^2 \diff\ap + \int \abs{\papabs^\half \brac{\frac{\sqrt{\Aone}}{\Zapabs}\Dap^2\Ztbar}}^2 \diff\ap}
\end{align*}
The right hand side is the time derivative of 
\begin{align*}
 \int \abs{\Dt f}^2\difff\ap +  \int \abs*[\bigg]{ \papabs^\half \brac*[\bigg]{\frac{\sqrt{\Aone}}{\Zapabs} f}}^2\difff\ap
\end{align*}
where $f = \Dap^2\Ztbar$ and we have $\Ph f = f$. Now by using \lemref{lem:timederiv} we see that
\begin{align*}
\frac{d\Ehigh(t)}{dt} \approx 2\Real\cbrac{ \int (\Dt \fbar) (\Dt^2 f)\diff \ap + \int  \cbrac{\papabs\brac*[\bigg]{\frac{\sqrt{\Aone}}{\Zapabs} f}} \Dt\brac*[\bigg]{\frac{\sqrt{\Aone}}{\Zapabs} \bar{f}} \difff\ap }
\end{align*}
Now using \eqref{est:DtsqrtAone} we obtain
\begin{align*}
\frac{d\Ehigh(t)}{dt} \approx 2\Real\cbrac{ \int (\Dt \fbar) (\Dt^2 f)\diff \ap + \int  \cbrac{\frac{\sqrt{\Aone}}{\Zapabs}\papabs\brac*[\bigg]{\frac{\sqrt{\Aone}}{\Zapabs} f}} (\Dt\bar{f}) \diff\ap  }
\end{align*}
Now using \eqref{est:DapabssqrtAonetwo} and combing terms we get
\begin{align*}
\frac{d\Ehigh(t)}{dt} \approx 2 \Real \int \brac{ \Dt^2 f +i\frac{\Aone}{\Zapabs^2}\pap f }(\Dt \bar{f}) \diff\ap
\end{align*}
As $\Dt f = \Dt\Dap^2\Ztbar \in \Ltwo$ we only need to show that the other term in in $\Ltwo$. Now the equation for $\Dap^2\Ztbar$ from \eqref{eq:Dap^2Ztbar} implies
\begin{align*}
\brac{\Dt^2 +i\frac{\Aone}{\Zapabs^2}\pap}\Dap^2\Ztbar = -i\wbar^3\Dapabs\brac{\frac{1}{\Zapabs^2}\pap\Jone} + \Rthree
\end{align*}
As we have already shown that $\dis \Dapabs\brac{\frac{1}{\Zapabs^2}\pap\Jone} \in \Ltwo$ and $\Rthree \in \Ltwo$, this implies that the right hand side of the above equation is in $\Ltwo$ and the proof of \thmref{thm:aprioriEhigh} is complete.

\medskip

\subsection{Equivalence of $\Ehigh$ and $\Ecalhigh$}\label{sec:equivEhighEcalhigh}

We now give a simpler description of the energy $\Ehigh$. Recall the definition of $\Ecalhigh$ from \secref{sec:results}
\begin{align*}
\Ecalhigh  = \norm[2]{\pap\frac{1}{\Zap}}^2  + \norm[2]{\frac{1}{\Zap^2}\pap^2\frac{1}{\Zap}}^2 + \norm[2]{\Ztapbar}^2 + \norm*[\Bigg][2]{\frac{1}{\Zap^2}\pap\Ztapbar}^2 + \norm*[\Bigg][\Hhalf]{\frac{1}{\Zap^3}\pap\Ztapbar}^2
\end{align*}
\begin{prop}\label{prop:equivEhighEcalhigh} 
There exists universal polynomials $P_1, P_2$ with non-negative coefficients so that if $(\Z,\Zt)(t)$ is a smooth solution to the water wave equation \eqref{eq:systemone} with $\sigma=0$ in the time interval $[0,T]$ satisfying $(\Zap-1,\frac{1}{\Zap} - 1, \Zt) \in \Linfty([0,T], H^{s}(\Rsp)\times H^{s}(\Rsp)\times H^{s+\half}(\Rsp))$ for all $s\geq 2$,  then for all $t \in [0,T]$ we have
\begin{align*}
\Ehigh(t) \leq P_1(\Ecalhigh(t)) \quad \tx{ and }\quad  \Ecalhigh(t) \leq P_2(\Ehigh(t))
\end{align*}
\end{prop}
\begin{proof}
Let $\Ehigh(t) < \infty$ and recall the definition of $\Ehigh$, $\Ecalhigh$ from \eqref{def:Ehigh}, \eqref{def:Ecalhigh} respectively. As $\Ehigh$ controls $\Esigma\vert_{\sigma = 0}$, we see from \propref{prop:equivEsigmaEcalsigma} that we already have control over $\Ecalsigma\vert_{\sigma=0}$. Now   from \secref{sec:quantEhigh} we have that $\dis \norm*[\bigg][2]{\frac{1}{\Zap^2}\pap^2 \frac{1}{\Zap}} \lesssim P_2(\Ehigh(t)) $.  The last term to control is $\dis \frac{1}{\Zap^3}\pap\Ztapbar$, which can be easily controlled in $\Hhalf$ by using \lemref{lem:CW}
\begin{align*}
\norm[\Ccal]{\frac{1}{\Zap^3}\pap\Ztapbar} \lesssim \norm[\Wcal]{\wbar}^3\norm[\Ccal]{\frac{1}{\Zapabs^3}\pap\Ztapbar} \lesssim P_2(\Ehigh(t)) 
\end{align*}

Now we assume that $\Ecalhigh(t) < \infty$.  We use  \propref{prop:LinftyHhalf} with $\dis f = \Dap\frac{1}{\Zap}$ and $\dis w = \frac{1}{\Zap}$ to obtain
\begin{align*}
& \norm[\Linfty\cap\Hhalf]{\Dap\frac{1}{\Zap}}^2 \\
& \lesssim \norm[2]{\pap\frac{1}{\Zap}}\norm[2]{\pap\brac{\frac{1}{\Zap^2}\pap\frac{1}{\Zap}}} + \norm[2]{\pap\frac{1}{\Zap}}^4 \\
& \lesssim \norm[2]{\pap\frac{1}{\Zap}}^2\norm[\Linfty\cap\Hhalf]{\Dap\frac{1}{\Zap}} + \norm[2]{\pap\frac{1}{\Zap}}\norm[2]{\frac{1}{\Zap^2}\pap^2\frac{1}{\Zap}} + \norm[2]{\pap\frac{1}{\Zap}}^4 
\end{align*}
Now use the inequality $ab \leq \frac{a^2}{2\ep} + \frac{\ep b^2}{2}$ on the first term to obtain $\dis \norm[\Linfty\cap\Hhalf]{\Dap\frac{1}{\Zap}} \lesssim P_1(\Ecalhigh(t))$. Hence we see that $\Ecalsigma\vert_{\sigma=0}$ is controlled by $\Ecalhigh$ and hence by  \propref{prop:equivEsigmaEcalsigma} we know that $\Esigma\vert_{\sigma=0}$ is controlled. Hence we only need to control the last two terms of $\Ehigh$ from \eqref{def:Ehigh}.  Now following the proof of $\dis \frac{1}{\Zapabs^2}\pap^2\frac{1}{\Zap} \in \Ltwo$ in \secref{sec:quantEhigh} we see that $\norm[2]{\Dapabs^2\Zttbar} \lesssim P_1(\Ecalhigh(t))$. Following the proof of $\Dapabs^2\Zttbar \in \Ltwo$ in \secref{sec:quantEhigh}, we see that  $\norm[2]{\Dt\Dap^2\Ztbar}  \lesssim P_1(\Ecalhigh(t))$. 

We now observe from \lemref{lem:CW} that
\begin{align*}
\norm[\Ccal]{\frac{1}{\Zapabs^3}\pap\Ztapbar} \lesssim \norm[\Wcal]{\w}^3\norm[\Ccal]{\frac{1}{\Zap^3}\pap\Ztapbar} \lesssim P_1(\Ecalhigh(t))
\end{align*}
Now by following the proof of $\dis \frac{1}{\Zapabs^3}\pap\Ztapbar \in \Ccal$ in \secref{sec:quantEhigh}, we see that $\dis \norm[\Ccal]{\frac{1}{\Zapabs}\Dap^2\Ztbar} \lesssim P_1(\Ecalhigh(t))$. Hence we have from \lemref{lem:CW}
\begin{align*}
\norm[\Ccal]{\frac{\sqrt{\Aone}}{\Zapabs}\Dap^2\Ztbar} \lesssim \norm[\Wcal]{\sqrt{\Aone} }\norm[\Ccal]{\frac{1}{\Zapabs}\Dap^2\Ztbar} \lesssim P_1(\Ecalhigh(t))
\end{align*}
This proves the proposition. 
\end{proof}

\medskip

\section{Auxiliary energy  $\Ecalaux$}\label{sec:aprioriEcalaux} 
\medskip

In this section we again consider a solution to the water wave equation with zero surface tension and prove an a priori energy estimate for the energy $\Ecalaux$ defined in  \eqref{def:Ecalaux}. More precisely we show that as long as $\Ecalhigh$ is controlled, the energy $\Ecalaux$ is also controlled. To prove this, we first define an energy $\Eaux$, prove an a priori estimate for it in \thmref{thm:aprioriEaux} and then in \propref{prop:equivEauxEcalaux} show that the energies $\Ecalaux$ and $\Eaux$ are equivalent. 

The energy $\Ecalaux$ is used in the definition of the energy $\EcalDelta$ in \secref{sec:results} and is used crucially in the proof of \thmref{thm:convergence}. We refer to \secref{sec:discussion} to the discussion regarding the necessity of having the energy $\Ecalaux$ as part of the energy $\EcalDelta$ and how it is used. 

Let us now define the energy $\Eaux$ and state out main result for this section.  Define
\begin{align}\label{def:Eaux}
\begin{split}
\Eaux &= \norm[\infty]{\Zapabs^\half\pap\frac{1}{\Zap}}^2 +  \norm[2]{\frac{1}{\Zaphalf}\pap\Ztbarap}^2  +  \norm[2]{\frac{1}{\Zaphalf}\pap^2\frac{1}{\Zap}}^2 + \norm[2]{\frac{1}{\Zaphalf}\pap\Dap^2\Ztbar}^2 \\
& \quad  +   \norm[2]{\Dt\brac{\frac{1}{\Zaphalf}\pap\Dap^2\Ztbar}}^2  + \norm[\Hhalf]{\nobrac{\frac{\sqrt{\Aone}}{\Zapabs}\brac{\frac{1}{\Zaphalf}\pap\Dap^2\Ztbar} } }^2
\end{split}
\end{align}

\begin{thm}\label{thm:aprioriEaux}
Let $T,\lambda>0$ and let $(\Z, \Zt)(t)$ be a smooth solution to \eqref{eq:systemone} with $\sigma=0$ in the time interval $[0,T]$, such that for all $s\geq 2$ we have $(\Zap-1,\frac{1}{\Zap} - 1, \Zt) \in \Linfty([0,T], H^{s}(\Rsp)\times H^{s}(\Rsp)\times H^{s+\half}(\Rsp))$. Then $\sup_{t\in[0,T]} \Ecalhigh(t)<\infty$, $\sup_{t\in[0,T]} \Eaux(t) < \infty$ and there exists a universal polynomial P with non-negative coefficients such that for all $t \in [0,T)$ we have
\begin{align}\label{eq:aprioriEaux}
\frac{d}{dt} (\lamb\Eaux(t)) \leq P(\Ecalhigh(t))(\lamb\Eaux(t))
\end{align}
\end{thm}
\begin{rem} \label{rem:aprioriElambaux}
Similar to the case of energy $\Esigma$ as mentioned in \remref{rem:aprioriEsigma}, the energy $\Eaux$ contains a term which is the $L^\infty$ norm of a function and hence for this term we replace the time derivative by the upper Dini derivative.
\end{rem}
Note that in the above result, the constant $\lambda>0$ does not actually play any role and one can simply rewrite \eqref{eq:aprioriEaux} as
\begin{align*}
\frac{d}{dt} \Eaux(t) \leq P(\Ecalhigh(t))\Eaux(t)
\end{align*}
The reason we add a $\lambda$ to the energy estimate \eqref{eq:aprioriEaux}, is that later on we will need to replace this $\lambda$ by $\sigma$ which denotes the surface tension of solution $A$ in \thmref{thm:convergence}. Another important reason is that having a $\lambda$ makes the proof of the \thmref{thm:aprioriEaux} more readable, as we will have lots of terms of both $\Ecalhigh$ and $\Eaux$ showing up in the proof of \thmref{thm:aprioriEaux}, and the constant $\lambda$ will help us distinguish whether the term corresponds to $\Ecalhigh$ or $\Eaux$. 

We employ a similar strategy to prove this theorem as we used to prove \thmref{thm:aprioriEhigh}. To prove this theorem, we first obtain the quasilinear equation in \secref{sec:quasiEaux}, then control the quantities controlled by $\lamb\Eaux$ in \secref{sec:quantEaux} and then finish the proof in \secref{sec:closeEaux}. Finally we show in \secref{sec:equivEauxEcalaux} that $\Eaux$ and $\Ecalaux$ are equivalent energies.

\subsection{Quasilinear equation}\label{sec:quasiEaux}

Let us recall the equation of $\Dap^2\Ztbar$ from \eqref{eq:Dap^2Ztbar}
\begin{align*}
\brac{\Dt^2 +i\frac{\Aone}{\Zapabs^2}\pap}\Dap^2\Ztbar = -i\wbar^3\Dapabs\brac{\frac{1}{\Zapabs^2}\pap\Jone} + \Rthree
\end{align*}
with $\Rthree$ as given in \eqref{form:Rthree} along with the identities $\Jone =  \Dt\Aone +  \Aone\brac{\bvarap - \Dap\Zt - \Dapbar\Ztbar}$ and $\dis \Zttbar - i = -i\frac{\Aone}{\Zap}$ from \eqref{form:Jone} and \eqref{form:Zttbar} respectively. Applying $\frac{\lamb^\half}{\Zaphalf}\pap$ to the above equation we obtain
\begin{align*}
\brac{\Dt^2 +i\frac{\Aone}{\Zapabs^2}\pap}\frac{\lamb^\half}{\Zaphalf}\pap\Dap^2\Ztbar =  -i\wbar^3\frac{\lamb^\half}{\Zaphalf}\pap\Dapabs\brac{\frac{1}{\Zapabs^2}\pap\Jone} + \Rfour
\end{align*}
where
\begin{align*}
\Rfour & = -3i\wbar^2\brac{\frac{\lamb^\half}{\Zaphalf}\pap\wbar}\Dapabs\brac{\frac{1}{\Zapabs^2}\pap\Jone} + \frac{\lamb^\half}{\Zaphalf}\pap\Rthree \\
& \quad + \sqbrac{\Dt^2 +i\frac{\Aone}{\Zapabs^2}\pap, \frac{\lamb^\half}{\Zaphalf}\pap}\Dap^2\Ztbar
\end{align*}
Let us try to simplify the terms above

\begin{enumerate}[label = \alph*)]

\item  $\dis \Dt\frac{1}{\Zaphalf} = \frac{-1}{2\Zap^{3/2}}\Dt\Zap = \frac{-1}{2\Zaphalf}(\Dap\Zt - \bvarap)$

\item $\dis \pap \frac{1}{\Zaphalf} = \half\Zaphalf\pap\frac{1}{\Zap} = \brac{\half\Zap\pap\frac{1}{\Zap}}\frac{1}{\Zaphalf}$

\item 
$
\begin{aligned}[t]
\sqbrac{\frac{\lamb^\half}{\Zaphalf}\pap, \Dt} = \frac{\lamb^\half}{\Zaphalf}\sqbrac{\pap,\Dt} + \sqbrac{\frac{\lamb^\half}{\Zaphalf}, \Dt}\pap = \brac{\frac{\bvarap}{2} + \frac{\Dap\Zt}{2}}\frac{\lamb^\half}{\Zaphalf}\pap
\end{aligned} 
$

\item 
$
\begin{aligned}[t]
& \sqbrac{\frac{\lamb^\half}{\Zaphalf}\pap, \Dt^2} \\
& = \Dt\sqbrac{\frac{\lamb^\half}{\Zaphalf}\pap, \Dt} + \sqbrac{\frac{\lamb^\half}{\Zaphalf}\pap, \Dt}\Dt \\
&=  \cbrac{\frac{\Dt\bvarap}{2} + \frac{\Dt\Dap\Zt}{2} - \brac{\frac{\bvarap}{2} + \frac{\Dap\Zt}{2}}^2 }\frac{\lamb^\half}{\Zaphalf}\pap +  \brac{\bvarap + \Dap\Zt}\frac{\lamb^\half}{\Zaphalf}\pap\Dt
\end{aligned}
$

\item 
$
\begin{aligned}[t]
& \sqbrac{\frac{\lamb^\half}{\Zaphalf}\pap, i\frac{\Aone}{\Zapabs^2}\pap} \\
&= \cbrac{2i\Aone\brac{\Dapabs\frac{1}{\Zapabs}} + \frac{i}{\Zapabs^2}\pap\Aone - \frac{i\Aone}{2}\brac{\Dapbar \frac{1}{\Zap}} }\frac{\lamb^\half}{\Zaphalf}\pap
\end{aligned}
$

\end{enumerate}
\smallskip
Combining the above formulae we have
\begin{align}\label{eq:lambpapDap^2Ztbar}
\brac{\Dt^2 +i\frac{\Aone}{\Zapabs^2}\pap}\frac{\lamb^\half}{\Zaphalf}\pap\Dap^2\Ztbar =  -i\wbar^3\frac{\lamb^\half}{\Zaphalf}\pap\Dapabs\brac{\frac{1}{\Zapabs^2}\pap\Jone} + \Rfour
\end{align}
where
\begin{align}\label{form:Rfour}
\begin{split}
\Rfour &=  -\cbrac{\frac{\Dt\bvarap}{2} + \frac{\Dt\Dap\Zt}{2} - \brac{\frac{\bvarap}{2} + \frac{\Dap\Zt}{2}}^2 }\frac{\lamb^\half}{\Zaphalf}\pap\Dap^2\Ztbar  + \frac{\lamb^\half}{\Zaphalf}\pap\Rthree \\
& \quad - \cbrac{2i\Aone\brac{\Dapabs\frac{1}{\Zapabs}} + \frac{i}{\Zapabs^2}\pap\Aone - \frac{i\Aone}{2}\brac{\Dapbar \frac{1}{\Zap}} }\frac{\lamb^\half}{\Zaphalf}\pap\Dap^2\Ztbar\\
& \quad -3i\wbar^2\brac{\frac{\lamb^\half}{\Zaphalf}\pap\wbar}\Dapabs\brac{\frac{1}{\Zapabs^2}\pap\Jone} - \brac{\bvarap + \Dap\Zt}\frac{\lamb^\half}{\Zaphalf}\pap\Dt\Dap^2\Ztbar
\end{split}
\end{align}
and $\Rthree$ is as defined in \eqref{form:Rthree}.
\bigskip

\subsection{Quantities controlled by the energy  $\lamb\Eaux$}\label{sec:quantEaux} 

In this section whenever we write $\dis f \in L^2_{\lamb^\alpha}$, what we mean is that there exists a universal polynomial $P$ with nonnegative coefficients such that $\dis \norm*[2]{f} \leq (\lamb\Eaux)^\alpha P(\Ecalhigh)$. Similar definitions for $\dis f\in \Hhalf_{\lamb^\alpha}$ and $\dis  f \in \Linfty_{\lamb^\alpha}$. We define the spaces $ \Ccal_{\lamb^\alpha}$ and $ \Wcal_{\lamb^\alpha}$ as follows
\begin{enumerate}
\item  If $\dis w \in \Linfty_{\lamb^\alpha}$ and $\dis \Dapabs w \in \Ltwo_{\lamb^\alpha}$, then we say $f\in \Wcal_{\lamb^\alpha}$. Define 
\[
 \norm[\Wcal_{\lamb^\alpha}]{w} = \norm[\Wcal]{w} = \norm[\infty]{w}+ \norm[2]{\Dapabs w}
 \]

\item If $\dis f \in \Hhalf_{\lamb^\alpha}$ and $\dis f\Zapabs \in \Ltwo_{\lamb^\alpha}$, then we say $f\in \mathcal{C}_{\lamb^\alpha}$. Define
\[
\norm[\Ccal_{\lamb^\alpha}]{f}  = \norm[\Ccal]{f} = \norm[\Hhalf]{f} + \brac*[\bigg]{1+ \norm*[\bigg][2]{\pap\frac{1}{\Zapabs}}}\norm[2]{f\Zapabs}
 \]
\end{enumerate}
Analogous to \lemref{lem:CW} we have the following lemma
\begin{lem} \label{lem:CWlamb} 
Let $\alpha_1, \alpha_2,\alpha_3 \geq0$ with $\alpha_1 + \alpha_2 = \alpha_3$. Then the following properties hold for the spaces $\Wcal_{\lamb^\alpha}$ and $\Ccal_{\lamb^\alpha}$
\begin{enumerate}[leftmargin =*, align=left]
\item If $w_1 \in \Wcal_{\lamb^{\alpha_1}}$, $w_2 \in \Wcal_{\lamb^{\alpha_2}}$, then $w_1w_2 \in \Wcal_{\lamb^{\alpha_3}}$. Moreover we have the estimate $\norm[\Wcal_{\lamb^{\alpha_{3}}}]{w_1w_2} \lesssim \norm[\Wcal_{\lamb^{\alpha_{1}}}]{w_1}\norm[\Wcal_{\lamb^{\alpha_{2}}}]{w_2}$
\item If $f\in \Ccal_{\lamb^{\alpha_1}}$ and $w\in \Wcal_{\lamb^{\alpha_2}}$, then $fw \in \Ccal_{\lamb^{\alpha_3}}$.  Moreover $\norm[\Ccal_{\lamb^{\alpha_3}}]{fw} \lesssim \norm[\Ccal_{\lamb^{\alpha_1}}]{f}\norm[\Wcal_{\lamb^{\alpha_2}}]{w}$
\item If $f \in \Ccal_{\lamb^{\alpha_1}}$, $g \in \Ccal_{\lamb^{\alpha_2}}$, then $fg\Zapabs \in \Ltwo_{\lamb^{\alpha_3}}$. Moreover $\norm[2]{fg\Zapabs} \lesssim \norm[\Ccal_{\lamb^{\alpha_1}}]{f}\norm[\Ccal_{\lamb^{\alpha_2}}]{g}$
\end{enumerate}
\end{lem}
When we write $\dis f \in \Ltwo$ we mean $f \in \Ltwo_{\lamb^{\alpha}}$ with $\alpha =0$. Similar notation for $\Hhalf, \Linfty, \Ccal$ and $\Wcal$. This notation is now consistent with the notation used in \secref{sec:quantEhigh}. Let us now control the important terms  controlled by the energy $\lamb\Eaux$.

\medskip
\begin{enumerate}[widest = 99, leftmargin =*, align=left, label=\arabic*)]

\item  $\dis  \lamb^\half\Zapabs^\half\pap\frac{1}{\Zap}  \in \Linftylambhalf, \lamb^\half\Zapabs^\half\pap\frac{1}{\Zapabs} \in \Linftylambhalf$ and $\dis \frac{\lamb^\half}{\Zapabs^\half}\pap\w \in \Linftylambhalf$
\medskip\\
Proof: From the energy we already know that $\dis \lamb^\half\Zapabs^\half\pap\frac{1}{\Zap} \in \Linftylambhalf$. Hence we easily have $\dis \lamb^\half\pap\frac{1}{\Zaphalf} \in \Linftylambhalf$. Recall from \eqref{form:RealImagTh} that
\[
\Real\brac{\frac{\Zap}{\Zapabs}\pap\frac{1}{\Zap} } = \pap \frac{1}{\Zapabs} \quad \qquad \Imag\brac{\frac{\Zap}{\Zapabs}\pap\frac{1}{\Zap}} =  i\brac{ \wbar\Dapabs \w} 
\]
Hence we obtain $\dis \lamb^\half\Zapabs^\half\pap\frac{1}{\Zapabs} \in \Linftylambhalf$ and $\dis \frac{\lamb^\half}{\Zapabs^\half}\pap\w \in \Linftylambhalf$
\bigskip

\item $\dis \frac{\lamb^\half}{\Zapabs^\half}\pap^2\frac{1}{\Zap} \in \Ltwolambhalf$, $\dis \frac{\lamb^\half}{\Zapabs^\half}\pap^2\frac{1}{\Zapabs} \in \Ltwolambhalf$ and $\dis \frac{\lamb^\half}{\Zapabs^\threebytwo}\pap^2\w \in \Ltwolambhalf$ and hence we have that $\dis \lamb^\half\Zapabs^\half\pap\frac{1}{\Zap} \in \Wcallambhalf$, $\dis \lamb^\half\Zapabs^\half\pap\frac{1}{\Zapabs} \in \Wcallambhalf$ and $\dis \frac{\lamb^\half}{\Zapabs^\half}\pap\w \in \Wcallambhalf$
\medskip\\
Proof: Observe that $\dis \frac{\lamb^\half}{\Zapabs^\half}\pap^2\frac{1}{\Zap} \in \Ltwolambhalf$ as it is part of the energy. Recall from \eqref{form:RealImagTh} that
\[
\Real\brac{\frac{\Zap}{\Zapabs}\pap\frac{1}{\Zap} } = \pap \frac{1}{\Zapabs} \quad \qquad \Imag\brac{\frac{\Zap}{\Zapabs}\pap\frac{1}{\Zap}} =  i\brac{ \wbar\Dapabs \w} 
\]
Hence taking derivatives we obtain
\begin{align*}
\norm[2]{\frac{\lamb^\half}{\Zapabs^\half}\pap^2\frac{1}{\Zapabs}} & \lesssim \norm[\infty]{\frac{\lamb^\half}{\Zapabs^\half}\pap\w}\norm[2]{\pap\frac{1}{\Zap}} + \norm[2]{ \frac{\lamb^\half}{\Zapabs^\half}\pap^2\frac{1}{\Zap}} \\
\norm[2]{ \frac{\lamb^\half}{\Zapabs^\threebytwo}\pap^2\w} & \lesssim \norm[\infty]{\frac{\lamb^\half}{\Zapabs^\half}\pap\w}\cbrac{\norm[2]{\pap\frac{1}{\Zap}} + \norm[2]{\Dapabs\w}}+ \norm[2]{ \frac{\lamb^\half}{\Zapabs^\half}\pap^2\frac{1}{\Zap}}
\end{align*}
We also see that
\begin{align*}
\norm[2]{\Dapabs\brac{ \lamb^\half\Zapabs^\half\pap\frac{1}{\Zap} } } \lesssim \norm[\infty]{\lamb^\half\Zapabs^\half\pap\frac{1}{\Zapabs}}\norm[2]{\pap\frac{1}{\Zap}} + \norm[2]{\frac{\lamb^\half}{\Zapabs^\half}\pap^2\frac{1}{\Zap}}
\end{align*}
Hence $\dis \lamb^\half\Zapabs^\half\pap\frac{1}{\Zap} \in \Wcallambhalf$. The rest are proven similarly. 
\bigskip

\item $\dis \frac{\lamb^\half}{\Zapabs^\half}\pap\Dap^2\Ztbar \in \Ltwolambhalf$, $\dis \frac{\lamb^\half}{\Zapabs^\half}\pap\Dap^2\Zt \in \Ltwolambhalf$, $\dis \frac{\lamb^\half}{\Zapabs^\fivebytwo}\pap^2\Ztbarap \in \Ltwolambhalf$ and similarly we have $\dis \frac{\lamb^\half}{\Zapabs^\threebytwo}\pap^2\Dap\Ztbar \in \Ltwolambhalf$, $\dis \frac{\lamb^\half}{\Zapabs^\threebytwo}\pap^2\Dap\Zt \in \Ltwolambhalf$ and $\dis \frac{\lamb^\half}{\Zapabs^\half}\pap\Dapabs\Dap\Zt \in \Ltwolambhalf$
\medskip\\
Proof: We already have $\dis \frac{\lamb^\half}{\Zapabs^\half}\pap\Dap^2\Ztbar \in \Ltwolambhalf$ as it is part of the energy. We now have
\begin{align*}
\frac{\lamb^\half}{\Zapabs^\half}\pap\Dap^2\Ztbar = \brac{\lamb^\half\Zapabs^\half\pap\frac{1}{\Zap}}(\Dapabs\Dap\Ztbar )+ \frac{\lamb^\half}{\Zap\Zapabs^\half}\pap^2\Dap\Ztbar 
\end{align*}
and hence
\begin{align*}
\norm[2]{ \frac{\lamb^\half}{\Zapabs^\threebytwo}\pap^2\Dap\Ztbar } \lesssim \norm[\infty]{\lamb^\half\Zapabs^\half\pap\frac{1}{\Zap}}\norm[2]{\Dapabs\Dap\Ztbar } + \norm[2]{\frac{\lamb^\half}{\Zapabs^\half}\pap\Dap^2\Ztbar}
\end{align*}
Now we have
\begin{align*}
\frac{\lamb^\half}{\Zapabs^\threebytwo}\pap^2\Dap\Ztbar & = \brac{\frac{\lamb^\half}{\Zapabs^\half}\pap^2\frac{1}{\Zap} }\Dapabs\Ztbar + 2\brac{\lamb^\half\Zapabs^\half\pap\frac{1}{\Zap}}\brac{\frac{1}{\Zapabs^2}\pap\Ztbarap} \\
& \quad + \frac{\lamb^\half}{\Zap\Zapabs^\threebytwo}\pap^2\Ztbarap
\end{align*}
From this we obtain
\begin{align*}
 & \norm[2]{\frac{\lamb^\half}{\Zapabs^\fivebytwo}\pap^2\Ztbarap} \\
 & \lesssim \norm[2]{\frac{\lamb^\half}{\Zapabs^\half}\pap^2\frac{1}{\Zap}}\norm[\infty]{\Dapabs\Ztbar} + \norm[\infty]{\lamb^\half\Zapabs^\half\pap\frac{1}{\Zap}}\norm[2]{\frac{1}{\Zapabs^2}\pap\Ztbarap} \\
& \quad + \norm[2]{\frac{\lamb^\half}{\Zapabs^\threebytwo}\pap^2\Dap\Ztbar}
\end{align*}
Hence we have $\dis \frac{\lamb^\half}{\Zapabs^\fivebytwo}\pap^2\Ztbarap \in \Ltwolambhalf$. By taking conjugation and retracing the steps backwards we easily obtain the other estimates $\dis \frac{\lamb^\half}{\Zapabs^\half}\pap\Dap^2\Zt \in \Ltwolambhalf$, $\dis \frac{\lamb^\half}{\Zapabs^\threebytwo}\pap^2\Dap\Zt \in \Ltwolambhalf$ and $\dis \frac{\lamb^\half}{\Zapabs^\half}\pap\Dapabs\Dap\Zt \in \Ltwolambhalf$.
\bigskip

\item $\dis \frac{\lamb^\half}{\Zapabs^\threebytwo}\pap\Ztbarap \in \Wcallambhalf\cap\Ccallambhalf$, $\dis \frac{\lamb^\half}{\Zapabs^\half}\pap\Dap\Ztbar \in \Wcallambhalf\cap\Ccallambhalf$, $\dis \frac{\lamb^\half}{\Zapabs^\half}\pap\Dap\Zt \in \Wcallambhalf\cap\Ccallambhalf$ and similarly $\dis \frac{\lamb^\half}{\Zapabs^\half}\pap\Dapabs\Zt \in \Wcallambhalf\cap\Ccallambhalf$
\medskip\\
Proof: We first see that
\begin{align*}
 \norm[2]{\pap\brac{\frac{\lamb^\half}{\Zapabs^\fivebytwo}\pap\Ztbarap}} \lesssim \norm[\infty]{\lamb^\half\Zapabs^\half\pap\frac{1}{\Zapabs}}\norm[2]{\frac{1}{\Zapabs^2}\pap\Ztapbar} + \norm[2]{\frac{\lamb^\half}{\Zapabs^\fivebytwo}\pap^2\Ztbarap}
\end{align*}
Now using  \propref{prop:LinftyHhalf} with $\dis f =   \frac{\lamb^\half}{\Zapabs^\threebytwo}\pap\Ztbarap$ and $\dis w = \frac{1}{\Zapabs}$ we see that
\begin{align*}
\norm[\Linfty\cap\Hhalf]{ \frac{\lamb^\half}{\Zapabs^\threebytwo}\pap\Ztbarap}^2 & \lesssim \norm[2]{\frac{\lamb^\half}{\Zapabs^\half}\pap\Ztbarap} \norm[2]{\pap\brac{\frac{\lamb^\half}{\Zapabs^\fivebytwo}\pap\Ztbarap}} \\
& \quad +  \norm[2]{\frac{\lamb^\half}{\Zapabs^\half}\pap\Ztbarap}^2\norm[2]{\pap\frac{1}{\Zapabs}}^2
\end{align*}
As $\dis \Zapabs f \in \Ltwolambhalf$ we have  $\dis \frac{\lamb^\half}{\Zapabs^\threebytwo}\pap\Ztbarap \in \Wcallambhalf\cap\Ccallambhalf$. Now we have
\begin{align*}
\frac{\lamb^\half}{\Zapabs^\half}\pap\Dap\Ztbar = \brac{\lamb^\half\Zapabs^\half\pap\frac{1}{\Zap}}\Dapabs\Ztbar + \frac{\lamb^\half}{\Zap\Zapabs^\half}\pap\Ztbarap
\end{align*}
Hence from \lemref{lem:CWlamb} we have
\begin{align*}
\norm[\Wcallambhalf\cap\Ccallambhalf]{\frac{\lamb^\half}{\Zapabs^\half}\pap\Dap\Ztbar} & \lesssim \norm[\Wcallambhalf]{\lamb^\half\Zapabs^\half\pap\frac{1}{\Zap}}\norm[\Wcal\cap\Ccal]{\Dapabs\Ztbar}\\
& \quad + \norm[\Wcal]{\wbar}\norm[\Wcallambhalf\cap\Ccallambhalf]{\frac{\lamb^\half}{\Zapabs^\threebytwo}\pap\Ztbarap}
\end{align*}
The estimates $\dis \frac{\lamb^\half}{\Zapabs^\half}\pap\Dap\Zt \in \Wcallambhalf\cap\Ccallambhalf$, $\dis \frac{\lamb^\half}{\Zapabs^\half}\pap\Dapabs\Zt \in \Wcallambhalf\cap\Ccallambhalf$ are proven similarly.
\bigskip

\item $\dis \frac{\lamb^\half}{\Zapabs^\half}\pap\Aone \in \Wcallambhalf$ 
\medskip\\
Proof: The proofs for $\dis \frac{\lamb^\half}{\Zapabs^\half}\pap\Aone \in \Linftylambhalf$ and $\dis \frac{\lamb^\half}{\Zapabs^\threebytwo}\pap^2\Aone \in \Ltwolambhalf$ follow in exactly the same was as the proofs of $\dis \frac{\sigma^\half}{\Zapabs^\half}\pap\Aone \in \Linfty$ and $\dis \frac{\sigma^\half}{\Zapabs^\threebytwo}\pap^2\Aone \in \Ltwo$ as done in Section 5.1 in \cite{Ag19}. Hence we have the estimates
\begin{align*}
\norm[\infty]{\frac{\lamb^\half}{\Zapabs^\half}\pap\Aone} & \lesssim \norm[\infty]{\lamb^\half\Zapabs^\half\pap\frac{1}{\Zapabs}}\norm[2]{\Ztapbar}^2 + \norm[2]{\frac{\lamb^\half}{\Zapabs^\half}\pap\Ztapbar}\norm[2]{\Ztapbar} \\
\lpar \norm[2]{\frac{\lamb^\half}{\Zapabs^\threebytwo}\pap^2\Aone} & \lesssim \cbrac{\norm[2]{\frac{\lamb^\half}{\Zapabs^\half}\pap^2\frac{1}{\Zap}} + \norm[\infty]{\lamb^\half\Zapabs^\half\pap\frac{1}{\Zap}}\norm[2]{\pap\frac{1}{\Zap}}} \norm[\infty]{\Aone}  \\
& \quad  + \norm[2]{\Ztap}\cbrac{\norm[\infty]{\frac{\lamb^\half}{\Zapabs^\threebytwo}\pap\Ztapbar} + \norm[2]{\pap\frac{1}{\Zap}}\norm[2]{\frac{\lamb^\half}{\Zapabs^\half}\pap\Ztapbar } }
\end{align*}
From this we get
\begin{align*}
\norm[2]{\Dapabs\brac{\frac{\lamb^\half}{\Zapabs^\half}\pap\Aone}} \lesssim \norm[2]{\pap\frac{1}{\Zapabs}}\norm[\infty]{\frac{\lamb^\half}{\Zapabs^\half}\pap\Aone} + \norm[2]{\frac{\lamb^\half}{\Zapabs^\threebytwo}\pap^2\Aone}
\end{align*}
\bigskip

\item $\dis \frac{\lamb^\half}{\Zapabs^\sevenbytwo}\pap^3\Aone \in \Ltwolambhalf$, $\dis \frac{\lamb^\half}{\Zapabs^\threebytwo}\pap^2\brac{\frac{1}{\Zapabs^2}\pap\Aone} \in \Ltwolambhalf$ and hence $\dis \frac{\lamb^\half}{\Zapabs^\fivebytwo}\pap^2\Aone \in \Wcallambhalf\cap\Ccallambhalf$
\medskip\\
Proof: We will first show that $\dis (\Id -\Hil)\cbrac*[\Bigg]{\frac{\lamb^\half}{\Zap^\sevenbytwo}\pap^3\Aone} \in \Ltwolambhalf$. Now from \eqref{form:Aonenew} we see that $\Aone = 1 + i\Zt\Ztapbar -i(\Id + \Hil)\cbrac{\Real(\Zt\Ztapbar)}$. Hence we have
\begin{flalign*}
\lpar \quad  &(\Id -\Hil)\cbrac*[\Bigg]{\frac{\lamb^\half}{\Zap^\sevenbytwo}\pap^3\Aone} &&\\
 & =  i(\Id - \Hil)\cbrac*[\Bigg]{\brac{\frac{\Zt}{\Zap}}\frac{\lamb^\half}{\Zap^\fivebytwo}\pap^3\Ztapbar } \\
 & \quad + i(\Id - \Hil) 
 \begin{aligned}[t]
 & \Biggl\{\brac*[\Bigg]{\frac{\lamb^\half}{\Zap^\fivebytwo}\pap^2\Ztap}(\Dap\Ztbar) + 3\brac*[\bigg]{\frac{1}{\Zap^2}\pap\Ztap}\brac*[\Bigg]{\frac{\lamb^\half}{\Zap^\threebytwo}\pap\Ztapbar} \\
 &  \quad + 3(\Dap\Zt)\brac*[\bigg]{\frac{\lamb^\half}{\Zap^\fivebytwo}\pap^2\Ztapbar} \Biggr\}
 \end{aligned}
\end{flalign*}
Now we see that 
\begin{align*}
i(\Id - \Hil)\cbrac*[\Bigg]{\brac{\frac{\Zt}{\Zap}}\frac{\lamb^\half}{\Zap^\fivebytwo}\pap^3\Ztapbar } & = -\frac{5i}{2}\sqbrac{\Zt, \Hil}\cbrac*[\Bigg]{\brac{\pap\frac{1}{\Zap}}\brac*[\bigg]{\frac{\lamb^\half}{\Zap^\fivebytwo}\pap^2\Ztapbar }} \\
& \quad +  i\sqbrac{\Pa\brac{\frac{\Zt}{\Zap}}, \Hil }\pap\brac*[\Bigg]{\frac{\lamb^\half}{\Zap^\fivebytwo}\pap^2\Ztapbar }
\end{align*}
Hence using \propref{prop:commutator} we have the estimate
\begin{flalign*}
\lpar \quad & \norm[2]{(\Id -\Hil)\cbrac*[\Bigg]{\frac{\lamb^\half}{\Zap^\sevenbytwo}\pap^3\Aone}} &&\\
& \lesssim \norm[2]{\frac{\lamb^\half}{\Zapabs^\fivebytwo}\pap^2\Ztapbar}\norm[\infty]{\Dapabs\Ztbar} + \norm[2]{\frac{1}{\Zapabs^2}\pap\Ztapbar}\norm[\infty]{\frac{\lamb^\half}{\Zapabs^\threebytwo}\pap\Ztapbar } \\
& \quad + \norm[2]{\frac{\lamb^\half}{\Zapabs^\fivebytwo}\pap^2\Ztapbar}\cbrac{\norm[2]{\Ztap}\norm[2]{\pap\frac{1}{\Zap}}  + \norm[\infty]{\pap\Pa\brac{\frac{\Zt}{\Zap}}}}
\end{flalign*}
Now lets come back to proving the main estimate $\dis \frac{\lamb^\half}{\Zapabs^\threebytwo}\pap^2\brac{\frac{1}{\Zapabs^2}\pap\Aone} \in \Ltwolambhalf$. Now as $\Aone $ is real valued we see that
\begin{align*}
 \frac{\lamb^\half}{\Zapabs^\threebytwo}\pap^2\brac{\frac{1}{\Zapabs^2}\pap\Aone} = \lamb^\half \Real\cbrac{  \frac{\w^\sevenbytwo\wbar^\sevenbytwo}{\Zapabs^\threebytwo}(\Id - \Hil)\pap^2\brac{\frac{1}{\Zapabs^2}\pap\Aone}}
\end{align*}
Hence it is enough to show that we have $\dis \frac{ \lamb^\half\wbar^\sevenbytwo}{\Zapabs^\threebytwo}(\Id - \Hil)\pap^2\brac{\frac{1}{\Zapabs^2}\pap\Aone} \in \Ltwolambhalf$. Now observe that 
\begin{align*}
& \frac{ \lamb^\half\wbar^\sevenbytwo}{\Zapabs^\threebytwo}(\Id - \Hil)\pap^2\brac{\frac{1}{\Zapabs^2}\pap\Aone} \\
& = - \sqbrac{\frac{\lamb^\half\wbar^\sevenbytwo}{\Zapabs^\threebytwo}, \Hil}\pap^2\brac{\frac{1}{\Zapabs^2}\pap\Aone} + (\Id - \Hil)\cbrac{\frac{ \lamb^\half\wbar^\sevenbytwo}{\Zapabs^\threebytwo}\pap^2\brac{\frac{1}{\Zapabs^2}\pap\Aone}}
\end{align*}
Hence by expanding the second term we see that
\begin{align*}
& \frac{ \lamb^\half\wbar^\sevenbytwo}{\Zapabs^\threebytwo}(\Id - \Hil)\pap^2\brac{\frac{1}{\Zapabs^2}\pap\Aone} \\
& =  - \sqbrac{\frac{\lamb^\half\wbar^\sevenbytwo}{\Zapabs^\threebytwo}, \Hil}\pap^2\brac{\frac{1}{\Zapabs^2}\pap\Aone} + (\Id -\Hil)\cbrac*[\Bigg]{\frac{\lamb^\half}{\Zap^\sevenbytwo}\pap^3\Aone} \\
& \quad +  (\Id - \Hil) 
\begin{aligned}[t]
& \Biggl\{ \frac{2\lamb^\half\wbar^2}{\Zap^\threebytwo}\brac{\pap\frac{1}{\Zapabs}}^2\pap\Aone +  \frac{2\lamb^\half\wbar^2}{\Zap^\threebytwo}\brac{\frac{1}{\Zapabs}\pap^2\frac{1}{\Zapabs} }\pap\Aone \\
& \quad  + \frac{4\lamb^\half\wbar^2}{\Zap^\threebytwo}\brac{\Dapabs\frac{1}{\Zapabs} }\pap^2\Aone \Biggr\}
\end{aligned}
\end{align*}
Hence from \propref{prop:commutator} we have the estimate
\begin{align*}
& \norm[2]{ \frac{\lamb^\half}{\Zapabs^\threebytwo}\pap^2\brac{\frac{1}{\Zapabs^2}\pap\Aone}} \\
& \lesssim \norm[2]{ (\Id -\Hil)\cbrac*[\Bigg]{\frac{\lamb^\half}{\Zap^\sevenbytwo}\pap^3\Aone}} +  \norm[2]{\frac{\lamb^\half}{\Zapabs^\half}\pap^2\frac{1}{\Zapabs}}\norm[\infty]{\frac{1}{\Zapabs^2}\pap\Aone} \\
& \quad + \norm[\infty]{\frac{1}{\Zapabs^2}\pap\Aone}\cbrac{\norm[2]{\frac{\lamb^\half}{\Zapabs^\threebytwo}\pap^2\w} + \norm[\infty]{\frac{\lamb^\half}{\Zapabs^\half}\pap\w}\norm[2]{\Dapabs\w}  } \\
& \quad +  \norm[\infty]{\lamb^\half\Zapabs^\half\pap\frac{1}{\Zapabs}}  \Bigg\{
\begin{aligned}[t]
&  \brac{\norm[2]{\Dapabs\w} + \norm[2]{\pap\frac{1}{\Zapabs}}}\norm[\infty]{\frac{1}{\Zapabs^2}\pap\Aone} \\
& \quad + \norm[2]{\frac{1}{\Zapabs^3}\pap^2\Aone } \Bigg\}
\end{aligned}
\end{align*}
From this we easily have the estimate
\begin{align*}
 & \norm[2]{\frac{\lamb^\half}{\Zapabs^\sevenbytwo}\pap^3\Aone}   \\
&  \lesssim \norm[\infty]{\frac{1}{\Zapabs^2}\pap\Aone}\cbrac{\norm[\infty]{\lamb^\half\Zapabs^\half\pap\frac{1}{\Zapabs}}\norm[2]{\pap\frac{1}{\Zapabs}} + \norm[2]{\frac{\lamb^\half}{\Zapabs^\half}\pap^2\frac{1}{\Zapabs} } } \\
& \quad + \norm[\infty]{\lamb^\half\Zapabs^\half\pap\frac{1}{\Zapabs}}\norm[2]{\frac{1}{\Zapabs^3}\pap^2\Aone} + \norm[2]{ \frac{\lamb^\half}{\Zapabs^\threebytwo}\pap^2\brac{\frac{1}{\Zapabs^2}\pap\Aone}}
\end{align*}
and 
\begin{align*}
\norm[2]{\pap\brac{\frac{\lamb^\half}{\Zapabs^\sevenbytwo}\pap^2\Aone }} \lesssim \norm[\infty]{\lamb^\half\Zapabs^\half\pap\frac{1}{\Zapabs}}\norm[2]{\frac{1}{\Zapabs^3}\pap^2\Aone}  + \norm[2]{\frac{\lamb^\half}{\Zapabs^\sevenbytwo}\pap^3\Aone}
\end{align*}
We also get the estimate $\dis \Dapabs\brac{\frac{\lamb^\half}{\Zapabs^\fivebytwo}\pap^2\Aone} \in \Ltwolambhalf$ similarly. Hence now using  \propref{prop:LinftyHhalf} with $\dis f = \frac{\lamb^\half}{\Zapabs^\fivebytwo}\pap^2\Aone$ and $\dis w = \frac{1}{\Zapabs}$ we obtain
\begin{align*}
& \norm[\Linfty\cap\Hhalf]{\frac{\lamb^\half}{\Zapabs^\fivebytwo}\pap^2\Aone}^2 \\
& \lesssim \norm[2]{\frac{\lamb^\half}{\Zapabs^\threebytwo}\pap^2\Aone}\norm[2]{\pap\brac{\frac{\lamb^\half}{\Zapabs^\sevenbytwo}\pap^2\Aone }} + \norm[2]{\frac{\lamb^\half}{\Zapabs^\threebytwo}\pap^2\Aone}^2\norm[2]{\pap\frac{1}{\Zapabs}}^2
\end{align*}
\bigskip

\item $\dis \frac{\lamb^\half}{\Zapabs^\half}\pap\bvarap \in \Linftylambhalf$
\medskip\\
Proof: The proof of this estimate is the same as the proof of $\dis \frac{\sigma^\half}{\Zapabs^\half}\pap\bvarap \in \Linfty$ as done in Section 5.1 in \cite{Ag19}. Hence we have the estimate
\begin{align*}
\!\!\! \norm[\infty]{\frac{\lamb^\half}{\Zapabs^\half}\pap\bvarap} & \lesssim \norm[\infty]{\lamb^\half\Zapabs^\half\pap\frac{1}{\Zapabs}}\norm[2]{\pap\frac{1}{\Zap}}\norm[2]{\Ztapbar} + \norm[2]{\frac{\lamb^\half}{\Zapabs^\half}\pap^2\frac{1}{\Zap}}\norm[2]{\Ztapbar} \\
& \quad + \norm[2]{\pap\frac{1}{\Zap}}\norm[2]{\frac{\lamb^\half}{\Zapabs^\half}\pap\Ztapbar} + \norm[\infty]{\frac{\lamb^\half}{\Zapabs^\half}\pap\Dap\Zt} 
\end{align*}
\bigskip

\item $\dis \frac{\lamb^\half}{\Zapabs^\threebytwo}\pap^2\bvarap \in \Ltwolambhalf$, $\dis \frac{\lamb^\half}{\Zapabs^\half}\pap\Dapabs\bvarap \in \Ltwolambhalf$, $\dis \frac{\lamb^\half}{\Zapabs^\half}\pap\Dap\bvarap \in \Ltwolambhalf$ and hence we also have $\dis \frac{\lamb^\half}{\Zapabs^\half}\pap\bvarap \in \Wcallambhalf$
\medskip\\
Proof:  The proof of this estimate is the same as the proof of $\dis \frac{\sigma^\half}{\Zapabs^\threebytwo}\pap^2\bvarap \in \Ltwo$ as done in Section 5.1 in \cite{Ag19}. Hence we have the estimate
\begin{align*}
 \norm*[\Bigg][2]{(\Id - \Hil)\cbrac*[\Bigg]{\frac{\lamb^\half}{\Zap^\threebytwo}\pap^2\bvarap}} & \lesssim \norm[2]{\frac{\lamb^\half}{\Zapabs^\threebytwo}\pap^2\Dap\Zt} + \norm[2]{\frac{1}{\Zapabs^2}\pap\Ztap}\norm[\infty]{\lamb^\half\Zapabs^\half\pap\frac{1}{\Zap}} \\
& \quad + \norm[2]{\frac{\lamb^\half}{\Zapabs^\half}\pap^2\frac{1}{\Zap}} 
\begin{aligned}[t]
& \cbrac{\norm[\infty]{\Dap\Zt} + \norm[2]{\Ztap}\norm[2]{\pap\frac{1}{\Zap}}}
\end{aligned}
\end{align*}
and using this we have
\begin{align*}
\norm[2]{ \frac{\lamb^\half}{\Zapabs^\threebytwo}\pap^2\bvarap} & \lesssim \norm[\infty]{\bvarap}\cbrac{\norm[2]{\frac{\lamb^\half}{\Zapabs^\half}\pap^2\frac{1}{\Zap}} + \norm[\infty]{\lamb^\half\Zapabs^\half\pap\frac{1}{\Zap}}\norm[2]{\pap\frac{1}{\Zap}} } \\
& \quad + \norm*[\Bigg][2]{(\Id - \Hil)\cbrac*[\Bigg]{\frac{\lamb^\half}{\Zap^\threebytwo}\pap^2\bvarap}}
\end{align*}
The estimates for the other terms can now be shown easily.
\bigskip

\item $\dis \frac{\lamb^\half}{\Zapabs^\half}\pap\Zttapbar \in \Ltwolambhalf$
\medskip\\
Proof: From \eqref{form:Zttbar} we know that $\dis \Zttbar - i = -i\frac{\Aone}{\Zap}$. Hence we have
\begin{align*}
\norm[2]{\frac{\lamb^\half}{\Zapabs^\half}\pap\Zttapbar} & \lesssim \norm[2]{\frac{\lamb^\half}{\Zapabs^\threebytwo}\pap^2\Aone} + \norm[2]{\Dapabs\Aone}\norm[\infty]{\lamb^\half\Zapabs^\half\pap\frac{1}{\Zap}} \\
& \quad + \norm[\infty]{\Aone}\norm[2]{\frac{\lamb^\half}{\Zapabs^\half}\pap^2\frac{1}{\Zap}}
\end{align*}
\bigskip

\item $\dis \frac{\lamb^\half}{\Zapabs^\fivebytwo}\pap^2\Zttapbar \in \Ltwolambhalf$, $\dis  \frac{\lamb^\half}{\Zapabs^\half}\pap\Dap^2\Zttbar \in \Ltwolambhalf$, $\dis  \frac{\lamb^\half}{\Zapabs^\half}\pap\Dap^2\Ztt \in \Ltwolambhalf$ and similarly we also have $\dis  \frac{\lamb^\half}{\Zapabs^\half}\pap\Dt\Dap^2\Ztbar \in \Ltwolambhalf$, $\dis  \frac{\lamb^\half}{\Zapabs^\half}\pap\Dapabs\Dt\Dap\Zt \in \Ltwolambhalf$
\medskip\\
Proof: From the energy we have $\dis \Dt\brac{\frac{\lamb^\half}{\Zaphalf}\pap\Dap^2\Ztbar} \in \Ltwolambhalf$. Hence we have
\begin{align*}
 & \norm[2]{\frac{\lamb^\half}{\Zapabs^\half}\pap\Dt\Dap^2\Ztbar} \\
 & \lesssim \norm[2]{\frac{\lamb^\half}{\Zaphalf}\pap\Dap^2\Ztbar}\cbrac*[\big]{\norm[\infty]{\bvarap} + \norm[\infty]{\Dap\Zt}} + \norm[2]{\Dt\brac{\frac{\lamb^\half}{\Zaphalf}\pap\Dap^2\Ztbar}}
\end{align*}
Using the commutator $\sqbrac{\Dap,\Dt } = (\Dap\Zt)\Dap$ we see that
\begin{align*}
\Dt\Dap^2\Ztbar = -2(\Dap\Zt)\Dap^2\Ztbar -(\Dap\Ztbar)\Dap^2\Zt + \Dap^2\Zttbar
\end{align*}
Hence we have
\begin{align*}
& \norm[2]{ \frac{\lamb^\half}{\Zapabs^\half}\pap\Dap^2\Zttbar} \\
& \lesssim  \cbrac{ \norm[\infty]{\frac{\lamb^\half}{\Zapabs^\half}\pap\Dap\Zt} + \norm[\infty]{\frac{\lamb^\half}{\Zapabs^\half}\pap\Dap\Ztbar}}\brac{\norm[2]{\Dap^2\Zt} + \norm[2]{\Dap^2\Ztbar}} \\
& \quad + \norm[\infty]{\Dapabs\Ztbar}\cbrac{\norm[2]{\frac{\lamb^\half}{\Zapabs^\half}\pap\Dap^2\Ztbar} + \norm[2]{\frac{\lamb^\half}{\Zapabs^\half}\pap\Dap^2\Zt}} + \norm[2]{\frac{\lamb^\half}{\Zapabs^\half}\pap\Dt\Dap^2\Ztbar}  \\
\end{align*}
Now we have the identity
\begin{align*}
\Dap^2\Zttbar = \brac{\Dap\frac{1}{\Zap}}\Zttapbar + \frac{1}{\Zap^2}\pap\Zttapbar
\end{align*}
Hence we have 
\begin{align*}
\!\!\!\!\norm[2]{ \frac{\lamb^\half}{\Zapabs^\fivebytwo}\pap^2\Zttapbar} & \lesssim \norm[\infty]{\Dapabs\Zttbar}\cbrac{\norm[\infty]{\lamb^\half\Zapabs^\half\pap\frac{1}{\Zap}}\norm[2]{\pap\frac{1}{\Zap}} + \norm[2]{\frac{\lamb^\half}{\Zapabs^\half}\pap^2\frac{1}{\Zap}}} \\
& \quad + \norm[\infty]{\lamb^\half\Zapabs^\half\pap\frac{1}{\Zap}}\norm[2]{\frac{1}{\Zapabs^2}\pap\Zttbar} + \norm[2]{ \frac{\lamb^\half}{\Zapabs^\half}\pap\Dap^2\Zttbar}
\end{align*}
We can prove $\dis  \frac{\lamb^\half}{\Zapabs^\half}\pap\Dap^2\Ztt \in \Ltwolambhalf$ similarly. Now observe that
\begin{align*}
\Dapabs\Dt\Dap\Zt = \Dapabs\brac{\Dap\Ztt - (\Dap\Zt)^2} = \w\Dap^2\Ztt  - 2(\Dapabs\Zt)\Dap^2\Zt
\end{align*}
Hence we have
\begin{align*}
& \norm[2]{\frac{\lamb^\half}{\Zapabs^\half}\pap\Dapabs\Dt\Dap\Zt} \\
& \lesssim \norm[\infty]{\frac{\lamb^\half}{\Zapabs^\half}\pap\w}\norm[2]{\Dap^2\Ztt} + \norm[2]{\frac{\lamb^\half}{\Zapabs^\half}\pap\Dap^2\Ztt} \\
& \quad  + \norm[\infty]{\frac{\lamb^\half}{\Zapabs^\half}\pap\Dapabs\Zt}\norm[2]{\Dap^2\Zt} + \norm[\infty]{\Dapabs\Zt}\norm[2]{\frac{\lamb^\half}{\Zapabs^\half}\pap\Dap^2\Zt}
\end{align*}

\bigskip

\item $\dis \frac{\lamb^\half}{\Zapabs^\threebytwo}\pap\Zttapbar \in \Wcallambhalf\cap\Ccallambhalf$, $\dis \frac{\lamb^\half}{\Zapabs^\threebytwo}\pap\Dt\Ztapbar \in \Wcallambhalf\cap\Ccallambhalf$, $\dis \frac{\lamb^\half}{\Zapabs^\half}\pap\Dap\Zttbar \in \Wcallambhalf\cap\Ccallambhalf$ and $\dis \frac{\lamb^\half}{\Zapabs^\half}\pap\Dap\Ztt \in \Wcallambhalf\cap\Ccallambhalf$
\medskip\\
Proof: Let $\dis f = \frac{\lamb^\half}{\Zapabs^\threebytwo}\pap\Zttapbar$ and $\dis w = \frac{1}{\Zapabs}$. Then we see that $\dis \frac{f}{w} \in \Ltwolambhalf$. Now we have
\begin{align*}
\norm[2]{\pap\brac{\frac{\lamb^\half}{\Zapabs^\fivebytwo}\pap\Zttapbar}} \lesssim \norm[2]{\frac{\lamb^\half}{\Zapabs^\fivebytwo}\pap^2\Zttapbar} + \norm[\infty]{\Dapabs\frac{1}{\Zapabs}}\norm[2]{\frac{\lamb^\half}{\Zapabs^\half}\pap\Zttapbar}
\end{align*}
As $\dis w' \in \Ltwo$, using  \propref{prop:LinftyHhalf} we obtain
\begin{align*}
\norm[\Linfty\cap\Hhalf]{\frac{\lamb^\half}{\Zapabs^\threebytwo}\pap\Zttapbar}^2 & \lesssim \norm[2]{\frac{\lamb^\half}{\Zapabs^\half}\pap\Zttapbar}\norm[2]{\pap\brac{\frac{\lamb^\half}{\Zapabs^\fivebytwo}\pap\Zttapbar}} \\
& \quad + \norm[2]{\frac{\lamb^\half}{\Zapabs^\half}\pap\Zttapbar}^2\norm[2]{\pap\frac{1}{\Zapabs}}^2
\end{align*}
Hence we have $\dis \frac{\lamb^\half}{\Zapabs^\threebytwo}\pap\Zttapbar \in \Wcallambhalf\cap\Ccallambhalf$. Now 
\begin{align*}
\frac{\lamb^\half}{\Zapabs^\threebytwo}\pap\Dt\Ztapbar = -\frac{\lamb^\half}{\Zapabs^\threebytwo}\pap\brac{\bvarap\Ztapbar} + \frac{\lamb^\half}{\Zapabs^\threebytwo}\pap\Zttapbar
\end{align*}
Hence from \lemref{lem:CWlamb} we obtain
\begin{align*}
& \norm[\Wcallambhalf\cap\Ccallambhalf]{\frac{\lamb^\half}{\Zapabs^\threebytwo}\pap\Dt\Ztapbar} \\
& \lesssim \norm[\Wcallambhalf]{\frac{\lamb^\half}{\Zapabs^\half}\pap\bvarap }\norm[\Wcal\cap\Ccal]{\Dapabs\Ztbar} + \norm[\Wcallambhalf\cap\Ccallambhalf]{\frac{\lamb^\half}{\Zapabs^\threebytwo}\pap\Zttapbar} \\
& \quad + \norm[\Wcal]{\bvarap}\norm[\Wcallambhalf\cap\Ccallambhalf]{\frac{\lamb^\half}{\Zapabs^\threebytwo}\pap\Ztapbar}
\end{align*}
We also have from \lemref{lem:CWlamb}
\begin{align*}
& \norm[\Wcallambhalf\cap\Ccallambhalf]{\frac{\lamb^\half}{\Zapabs^\half}\pap\Dap\Zttbar} \\
& \lesssim \norm[\Wcallambhalf]{\lamb^\half\Zapabs^\half\pap\frac{1}{\Zap}}\norm[\Wcal\cap\Ccal]{\Dapabs\Zttbar} + \norm[\Wcallambhalf\cap\Ccallambhalf]{ \frac{\lamb^\half}{\Zapabs^\threebytwo}\pap\Zttapbar }
\end{align*}
We prove $\dis \frac{\lamb^\half}{\Zapabs^\half}\pap\Dap\Ztt \in \Wcallambhalf\cap\Ccallambhalf$ similarly. 
\bigskip

\item $\dis \frac{\lamb^\half}{\Zapabs^\fivebytwo}\pap^3\frac{1}{\Zap} \in \Ltwolambhalf$, $\dis \frac{\lamb^\half}{\Zapabs^\fivebytwo}\pap^3\frac{1}{\Zapabs} \in \Ltwolambhalf$, $\dis \frac{\lamb^\half}{\Zapabs^\sevenbytwo}\pap^3\w \in \Ltwolambhalf$ and similarly we also have $\dis \frac{\lamb^\half}{\Zapabs^\half}\pap\Dap^2\frac{1}{\Zap} \in \Ltwolambhalf$
\medskip\\
Proof: We know from \eqref{form:Zttbar} that $\dis \Zttbar - i = -i\frac{\Aone}{\Zap}$ and as $\Aone \geq 1$ we have
\begin{align*}
& \norm[2]{\frac{\lamb^\half}{\Zapabs^\fivebytwo}\pap^3\frac{1}{\Zap}} \\
& \lesssim \norm[2]{\frac{1}{\Zapabs^3}\pap^2\Aone}\norm[\infty]{\lamb^\half\Zapabs^\half\pap\frac{1}{\Zap}} + \norm[\infty]{\frac{1}{\Zapabs^2}\pap\Aone}\norm[2]{\frac{\lamb^\half}{\Zapabs^\half}\pap^2\frac{1}{\Zap}} \\
& \quad  + \norm[2]{\frac{\lamb^\half}{\Zapabs^\fivebytwo}\pap^3\Zttbar} + \norm[2]{\frac{\lamb^\half}{\Zapabs^\sevenbytwo}\pap^3\Aone}
\end{align*}
Recall from \eqref{form:RealImagTh} that
\[
\Real\brac{\frac{\Zap}{\Zapabs}\pap\frac{1}{\Zap} } = \pap \frac{1}{\Zapabs} \quad \qquad \Imag\brac{\frac{\Zap}{\Zapabs}\pap\frac{1}{\Zap}} =  i\brac{ \wbar\Dapabs \w} 
\]
Hence by taking derivatives we obtain
\begin{align*}
\norm[2]{\frac{\lamb^\half}{\Zapabs^\fivebytwo}\pap^3\frac{1}{\Zapabs}} & \lesssim \norm[2]{\frac{\lamb^\half}{\Zapabs^\threebytwo}\pap^2\w}\norm[\infty]{\Dapabs\frac{1}{\Zap}} + \norm[\infty]{\frac{1}{\Zapabs^2}\pap\w}\norm[2]{\frac{\lamb^\half}{\Zapabs^\half}\pap^2\frac{1}{\Zap}} \\
& \quad + \norm[2]{\frac{\lamb^\half}{\Zapabs^\fivebytwo}\pap^3\frac{1}{\Zap}}
\end{align*}
and also
\begin{align*}
\norm[2]{\frac{\lamb^\half}{\Zapabs^\sevenbytwo}\pap^3\w} & \lesssim \norm[2]{\frac{\lamb^\half}{\Zapabs^\threebytwo}\pap^2\w}\norm[\infty]{\Dapabs\frac{1}{\Zap}} + \norm[\infty]{\frac{1}{\Zapabs^2}\pap\w}\norm[2]{\frac{\lamb^\half}{\Zapabs^\half}\pap^2\frac{1}{\Zap}} \\
& \quad + \norm[2]{\frac{\lamb^\half}{\Zapabs^\fivebytwo}\pap^3\frac{1}{\Zap}}
\end{align*}
The estimate $\dis \frac{\lamb^\half}{\Zapabs^\half}\pap\Dap^2\frac{1}{\Zap} \in \Ltwolambhalf$ can be proved similarly. 
\bigskip

\item $\dis \frac{\lamb^\half}{\Zapabs^\threebytwo}\pap^2\frac{1}{\Zap} \in \Wcallambhalf\cap\Ccallambhalf$, $\dis \frac{\lamb^\half}{\Zapabs^\threebytwo}\pap^2\frac{1}{\Zapabs} \in \Wcallambhalf\cap\Ccallambhalf$ and $\dis \frac{\lamb^\half}{\Zapabs^\fivebytwo}\pap^2\w \in \Wcallambhalf\cap\Ccallambhalf$
\medskip\\
Proof: We first note that
\begin{align*}
\norm[2]{\pap\brac{\frac{\lamb^\half}{\Zapabs^\fivebytwo}\pap^2\frac{1}{\Zap}}} \lesssim \norm[\infty]{\Dapabs\frac{1}{\Zapabs}}\norm[2]{\frac{\lamb^\half}{\Zapabs^\half}\pap^2\frac{1}{\Zap}} + \norm[2]{\frac{\lamb^\half}{\Zapabs^\fivebytwo}\pap^3\frac{1}{\Zap}}
\end{align*}
Now taking $\dis f =  \frac{\lamb^\half}{\Zapabs^\threebytwo}\pap^2\frac{1}{\Zap}$ and $\dis w = \frac{1}{\Zapabs}$ in  \propref{prop:LinftyHhalf} we obtain
\begin{align*}
\norm[\Linfty\cap\Hhalf]{ \frac{\lamb^\half}{\Zapabs^\threebytwo}\pap^2\frac{1}{\Zap}}^2 &  \lesssim \norm[2]{\frac{\lamb^\half}{\Zapabs^\half}\pap^2\frac{1}{\Zap}}\norm[2]{\pap\brac{\frac{\lamb^\half}{\Zapabs^\fivebytwo}\pap^2\frac{1}{\Zap}}} \\
& \quad + \norm[2]{\frac{\lamb^\half}{\Zapabs^\half}\pap^2\frac{1}{\Zap}}^2\norm[2]{\pap\frac{1}{\Zapabs}}^2
\end{align*}
As $\dis \frac{\lamb^\half}{\Zapabs^\half}\pap^2\frac{1}{\Zap} \in \Ltwolambhalf$ this shows that $\dis \frac{\lamb^\half}{\Zapabs^\threebytwo}\pap^2\frac{1}{\Zap} \in \Wcallambhalf\cap\Ccallambhalf$. We can prove the other estimates similarly.

\bigskip

\item $\dis \frac{\lamb^\half}{\Zapabs\Zaphalf}\pap\Dap^2\Ztbar \in \Ccallambhalf$, $\dis \frac{\lamb^\half}{\Zapabs^\fivebytwo}\pap^2\Dap\Ztbar \in \Ccallambhalf$, $\dis \frac{\lamb^\half}{\Zapabs^\sevenbytwo}\pap^2\Ztapbar \in \Ccallambhalf$ and similarly $\dis \frac{\lamb^\half}{\Zapabs^\fivebytwo}\pap^2\Dap\Zt \in \Ccallambhalf$, $\dis \frac{\lamb^\half}{\Zapabs^\threebytwo}\pap\Dapabs\Dap\Zt \in \Ccallambhalf$
\medskip\\
Proof: We first see that $\dis \frac{\lamb^\half\sqrt{\Aone}}{\Zaphalf}\pap\Dap^2\Ztbar \in \Ltwolambhalf$ as $\Aone \in \Linfty$ and $\dis \frac{\lamb^\half}{\Zaphalf}\pap\Dap^2\Ztbar \in \Ltwolambhalf$ as it is part of the energy. Hence from the energy we see that $\dis \frac{\lamb^\half\sqrt{\Aone}}{\Zapabs\Zaphalf}\pap\Dap^2\Ztbar \in \Ccallambhalf$. Now using \lemref{lem:CWlamb} we get
\begin{align*}
\norm[\Ccallambhalf]{\frac{\lamb^\half}{\Zapabs\Zaphalf}\pap\Dap^2\Ztbar} \lesssim \norm[\Ccallambhalf]{\frac{\lamb^\half\sqrt{\Aone}}{\Zapabs\Zaphalf}\pap\Dap^2\Ztbar}\norm[\Wcal]{\frac{1}{\sqrt{\Aone}}}
\end{align*}
As $\w^\half \in \Wcal$ we see that $\dis \frac{\lamb^\half}{\Zapabs^\threebytwo}\pap\Dap^2\Ztbar \in \Ccallambhalf$. Now again using \lemref{lem:CWlamb} we have
\begin{align*}
\norm[\Ccallambhalf]{ \frac{\lamb^\half}{\Zapabs^\fivebytwo}\pap^2\Dap\Ztbar} & \lesssim \norm[\Wcal]{\w}\norm[\Wcallambhalf]{\lamb^\half\Zapabs^\half\pap\frac{1}{\Zap}}\norm[\Ccal]{\frac{1}{\Zapabs^2}\pap\Dap\Ztbar}  \\
& \quad + \norm[\Wcal]{\w}\norm[\Ccallambhalf]{\frac{\lamb^\half}{\Zapabs^\threebytwo}\pap\Dap^2\Ztbar}
\end{align*}
Now observe that
\begin{align*}
 \frac{\lamb^\half}{\Zapabs^\fivebytwo}\pap^2\Dap\Ztbar =  \frac{\lamb^\half}{\Zapabs^\fivebytwo}\pap\cbrac{\brac{\pap\frac{1}{\Zap}}\Ztapbar + \frac{1}{\Zap}\pap\Ztapbar}
\end{align*}
Hence from \lemref{lem:CWlamb} we have
\begin{align*}
 & \norm[\Ccallambhalf]{\frac{\lamb^\half}{\Zapabs^\sevenbytwo}\pap^2\Ztapbar} \\
& \lesssim \norm[\Wcal]{\w}\norm[\Wcallambhalf]{\lamb^\half\Zapabs^\half\pap\frac{1}{\Zap}}\norm[\Ccal]{\frac{1}{\Zapabs^3}\pap\Ztapbar} + \norm[\Wcal]{\w}\norm[\Ccallambhalf]{ \frac{\lamb^\half}{\Zapabs^\fivebytwo}\pap^2\Dap\Ztbar}      \\
& \quad + \norm[\Wcal]{\w}\norm[\Wcal]{\Dapabs\Ztbar}\norm[\Ccallambhalf]{\frac{\lamb^\half}{\Zapabs^\threebytwo}\pap^2\frac{1}{\Zap}} + \norm[\Wcal]{\w}\norm[\Wcal]{\Dapabs\frac{1}{\Zap}}\norm[\Ccallambhalf]{\frac{\lamb^\half}{\Zapabs^\threebytwo}\pap\Ztapbar}
\end{align*}
Now the estimates $\dis \frac{\lamb^\half}{\Zapabs^\fivebytwo}\pap^2\Dap\Zt \in \Ccallambhalf$, $\dis \frac{\lamb^\half}{\Zapabs^\threebytwo}\pap\Dapabs\Dap\Zt \in \Ccallambhalf$ are proven similarly.
\bigskip

\item $\dis \frac{\lamb^\half}{\Zapabs^\half}\pap\Th \in \Ltwolambhalf$
\medskip\\
Proof: Using formula \eqref{form:Real} we see that
\begin{align*}
\frac{\lamb^\half}{\Zapabs^\half}\pap\Th = (\Id + \Hil)\Real\cbrac{\frac{\lamb^\half}{\Zapabs^\half}\pap\Th} + i\Imag (\Id - \Hil)\cbrac{\frac{\lamb^\half}{\Zapabs^\half}\pap\Th}
\end{align*}
Now using the formula \eqref{form:RealImagTh} we see that
\begin{align*}
\Real\cbrac{\frac{\lamb^\half}{\Zapabs^\half}\pap\Th} = -i\frac{\lamb^\half}{\Zapabs^\half}\pap\brac{\Dap\w}
\end{align*}
Hence we have
\begin{align*}
\lpar \qquad  \norm[2]{\Real\cbrac{\frac{\lamb^\half}{\Zapabs^\half}\pap\Th}} \lesssim \norm[\infty]{\lamb^\half\Zapabs^\half\pap\frac{1}{\Zap}}\norm[2]{\Dapabs\w} + \norm[2]{\frac{\lamb^\half}{\Zapabs^\threebytwo}\pap^2\w}
\end{align*}
We also have
\begin{align*}
 (\Id - \Hil)\cbrac{\frac{\lamb^\half}{\Zapabs^\half}\pap\Th} = \sqbrac{\frac{\lamb^\half}{\Zapabs^\half},\Hil}\pap\Th
\end{align*}
and hence we have from \propref{prop:commutator}
\begin{align*}
\norm[2]{(\Id - \Hil)\cbrac{\frac{\lamb^\half}{\Zapabs^\half}\pap\Th}} \lesssim \norm[\infty]{\lamb^\half\Zapabs^\half\pap\frac{1}{\Zapabs}}\norm[2]{\Th}
\end{align*}

\bigskip

\item $\dis \frac{\lamb^\half}{\Zapabs^\threebytwo}\pap\Th \in \Ccallambhalf$, $\dis \frac{\lamb^\half}{\Zap^\threebytwo}\pap\Th \in \Ccallambhalf$
\medskip\\
Proof: As $ \frac{\lamb^\half}{\Zapabs^\half}\pap\Th \in \Ltwolambhalf$, we only need to prove  $\frac{\lamb^\half}{\Zapabs^\threebytwo}\pap\Th \in \Hhalflambhalf$. Now using formula \eqref{form:Real} we see that
\begin{align*}
\frac{\lamb^\half}{\Zapabs^\threebytwo}\pap\Th = (\Id + \Hil)\Real\cbrac{\frac{\lamb^\half}{\Zapabs^\threebytwo}\pap\Th} + i\Imag (\Id - \Hil)\cbrac{\frac{\lamb^\half}{\Zapabs^\threebytwo}\pap\Th}
\end{align*}
Using the formula \eqref{form:RealImagTh} we see that
\begin{align*}
\Real\cbrac{\frac{\lamb^\half}{\Zapabs^\threebytwo}\pap\Th} = -i\frac{\lamb^\half}{\Zapabs^\threebytwo}\pap\brac{\Dap\w}
\end{align*}
Hence we have from \lemref{lem:CWlamb}
\begin{align*}
& \norm[\Hhalf]{\Real\cbrac{\frac{\lamb^\half}{\Zapabs^\threebytwo}\pap\Th}} \\
& \lesssim \norm[\Wcallambhalf]{\lamb^\half\Zapabs^\half\pap\frac{1}{\Zap}}\norm[\Ccal]{\frac{1}{\Zapabs^2}\pap\w} + \norm[\Wcal]{\wbar}\norm[\Ccallambhalf]{\frac{\lamb^\half}{\Zapabs^\fivebytwo}\pap^2\w}
\end{align*}
Now observe that
\begin{align*}
(\Id - \Hil)\cbrac{\frac{\lamb^\half}{\Zapabs^\threebytwo}\pap\Th} = \lamb^\half\sqbrac{\frac{\w^\half}{\Zapabs},\Hil}\frac{1}{\Zaphalf}\pap\Th
\end{align*}
Hence we have from \propref{prop:commutator}
\begin{align*}
\norm[\Hhalf]{(\Id - \Hil)\cbrac{\frac{\lamb^\half}{\Zapabs^\threebytwo}\pap\Th} } \lesssim \brac{\norm[2]{\Dapabs\w} + \norm[2]{\pap\frac{1}{\Zapabs}}}\norm[2]{\frac{\lamb^\half}{\Zapabs^\half}\pap\Th}
\end{align*}
Hence $\dis \frac{\lamb^\half}{\Zapabs^\threebytwo}\pap\Th \in \Ccallambhalf$. Now as $\wbar^\threebytwo \in \Wcal$ we obtain the other estimate easily by multiplying and using \lemref{lem:CWlamb}. 
\bigskip

\item $\dis  \frac{\lamb^\half}{\Zapabs^\half}\pap\Dt\Th \in \Ltwolambhalf$
\medskip\\
Proof: Using formula \eqref{form:Real} we see that
\begin{align*}
\frac{\lamb^\half}{\Zapabs^\half}\pap\Dt\Th = (\Id + \Hil)\Real\cbrac{\frac{\lamb^\half}{\Zapabs^\half}\pap\Dt\Th} + i\Imag (\Id - \Hil)\cbrac{\frac{\lamb^\half}{\Zapabs^\half}\pap\Dt\Th}
\end{align*}
We control the terms individually. 
\begin{enumerate}

\item Using the formula \eqref{form:RealDtTh} we see that
\begin{align*}
\Real\cbrac{\frac{\lamb^\half}{\Zapabs^\half}\pap\Dt\Th} = -\frac{\lamb^\half}{\Zapabs^\half}\pap\Imag\cbrac{\Dapabs\Dapbar\Ztbar + i(\Real\Th)\Dapbar\Ztbar}
\end{align*}
Hence we have
\begin{align*}
\norm[2]{\Real\cbrac{\frac{\lamb^\half}{\Zapabs^\half}\pap\Dt\Th}} & \lesssim \norm[2]{\frac{\lamb^\half}{\Zapabs^\half}\pap\Dapabs\Dapbar\Ztbar} + \norm[2]{\frac{\lamb^\half}{\Zapabs^\half}\pap\Th}\norm[\infty]{\Dapbar\Ztbar} \\
& \quad + \norm[2]{\Th}\norm[\infty]{\frac{\lamb^\half}{\Zapabs^\half}\pap\Dapbar\Ztbar}
\end{align*}

\item We note that
\begin{align*}
 (\Id - \Hil)\cbrac{\frac{\lamb^\half}{\Zapabs^\half}\pap\Dt\Th} & = \sqbrac{\frac{\lamb^\half}{\Zapabs^\half},\Hil}\pap\Dt\Th - \w^\half\sqbrac{\frac{\lamb^\half\wbar^\half}{\Zapabs^\half},\Hil}\pap\Dt\Th \\
 & \quad + \w^\half(\Id - \Hil)\cbrac{\frac{\lamb^\half}{\Zaphalf}\pap\Dt\Th}
\end{align*}
Now observe that using the identities from \secref{sec:quasiEaux} we have
\begin{align*}
\Dt\brac{\frac{\lamb^\half}{\Zaphalf}\pap\Th} & = \sqbrac{\Dt, \frac{\lamb^\half}{\Zaphalf}\pap} \Th + \frac{\lamb^\half}{\Zaphalf}\pap\Dt\Th \\
& = -\brac{\frac{\bvarap}{2} + \frac{\Dap\Zt}{2}}\frac{\lamb^\half}{\Zaphalf}\pap\Th + \frac{\lamb^\half}{\Zaphalf}\pap\Dt\Th
\end{align*}
Now as $\dis \frac{\lamb^\half}{\Zaphalf}\pap\Th$ is holomorphic and $\Dt = \pt + \bvar\pap$ we have
\begin{align}\label{eq:lambpapDtTh}
\begin{split}
& (\Id - \Hil)\cbrac{\frac{\lamb^\half}{\Zaphalf}\pap\Dt\Th} \\
& = (\Id - \Hil)\cbrac{\brac{\frac{\bvarap}{2} + \frac{\Dap\Zt}{2}}\frac{\lamb^\half}{\Zaphalf}\pap\Th}  + \sqbrac{\bvar,\Hil}\pap\cbrac{\frac{\lamb^\half}{\Zaphalf}\pap\Th}
\end{split}
\end{align}
Hence using \propref{prop:commutator} we obtain
\begin{align*}
& \norm[2]{(\Id - \Hil)\cbrac{\frac{\lamb^\half}{\Zapabs^\half}\pap\Dt\Th}} \\
& \lesssim \brac{\norm[\infty]{\bvarap} + \norm[\infty]{\Dap\Zt} }\norm[2]{\frac{\lamb^\half}{\Zaphalf}\pap\Th} \\
& \quad + \brac{\norm[\infty]{\lamb^\half\Zapabs^\half\pap\frac{1}{\Zapabs}} + \norm[2]{\frac{\lamb^\half}{\Zapabs^\half}\pap\w}}\norm[2]{\Dt\Th}
\end{align*}

\end{enumerate}
\bigskip

\item $\dis  \frac{\lamb^\half}{\Zapabs^\threebytwo}\pap\Dt\Th \in \Ccallambhalf$
\medskip\\
Proof: Note that we only need to prove the $\Hhalf$ estimate. Using formula \eqref{form:Real} we see that
\begin{align*}
\frac{\lamb^\half}{\Zapabs^\threebytwo}\pap\Dt\Th = (\Id + \Hil)\Real\cbrac{\frac{\lamb^\half}{\Zapabs^\threebytwo}\pap\Dt\Th} + i\Imag (\Id - \Hil)\cbrac{\frac{\lamb^\half}{\Zapabs^\threebytwo}\pap\Dt\Th}
\end{align*}
We control the term individually. 
\begin{enumerate}

\item Using the formula \eqref{form:RealDtTh} we see that
\begin{align*}
\Real\cbrac{\frac{\lamb^\half}{\Zapabs^\threebytwo}\pap\Dt\Th} = -\frac{\lamb^\half}{\Zapabs^\threebytwo}\pap\Imag\cbrac{\Dapabs\Dapbar\Ztbar + i(\Real\Th)\Dapbar\Ztbar}
\end{align*}
Hence we have from \lemref{lem:CWlamb}
\begin{align*}
& \norm[\Ccallambhalf]{\Real\cbrac{\frac{\lamb^\half}{\Zapabs^\threebytwo}\pap\Dt\Th}} \\
& \lesssim \norm[\Ccallambhalf]{\frac{\lamb^\half}{\Zapabs^\threebytwo}\pap\Dapabs\Dapbar\Ztbar} + \norm[\Ccallambhalf]{\frac{\lamb^\half}{\Zapabs^\threebytwo}\pap\Th}\norm[\Wcal]{\Dapbar\Ztbar} \\
& \quad + \norm[\Wcal]{\frac{\Th}{\Zapabs}}\norm[\Ccallambhalf]{\frac{\lamb^\half}{\Zapabs^\half}\pap\Dapbar\Ztbar}
\end{align*}

\item We note that
\begin{align*}
 & (\Id - \Hil)\cbrac{\frac{\lamb^\half}{\Zapabs^\threebytwo}\pap\Dt\Th}  \\
 & = \sqbrac{\frac{1}{\Zapabs},\Hil}\frac{\lamb^\half}{\Zapabs^\half}\pap\Dt\Th + \w^\threebytwo\frac{\wbar^\threebytwo}{\Zapabs}(\Id - \Hil)\cbrac{\frac{\lamb^\half}{\Zapabs^\half}\pap\Dt\Th}
\end{align*}
as we have from \propref{prop:commutator}
\begin{align*}
\norm[\Hhalf]{\sqbrac{\frac{1}{\Zapabs},\Hil}\frac{\lamb^\half}{\Zapabs^\half}\pap\Dt\Th} \lesssim \norm[2]{\pap\frac{1}{\Zapabs}}\norm[2]{\frac{\lamb^\half}{\Zapabs^\half}\pap\Dt\Th}
\end{align*}
we only need to show the second term is in $\Hhalf$. Now as $\w^\threebytwo \in \Wcal$, it is enough to show that $\dis \frac{\wbar^\threebytwo}{\Zapabs}(\Id - \Hil)\cbrac{\frac{\lamb^\half}{\Zapabs^\half}\pap\Dt\Th} \in \Ccallambhalf$. As $\dis \frac{\lamb^\half}{\Zapabs^\half}\pap\Dt\Th \in \Ltwolambhalf$ we only need to show the $\Hhalf$ estimate. Now
\begin{align*}
& \frac{\wbar^\threebytwo}{\Zapabs}(\Id - \Hil)\cbrac{\frac{\lamb^\half}{\Zapabs^\half}\pap\Dt\Th}  \\
& =   -\sqbrac{\frac{\wbar^\threebytwo}{\Zapabs},\Hil}\frac{\lamb^\half}{\Zapabs^\half}\pap\Dt\Th  + (\Id - \Hil)\cbrac*[\Bigg]{\frac{\lamb^\half}{\Zap^\threebytwo}\pap\Dt\Th}
\end{align*}
Now by a similar computation as done in \eqref{eq:lambpapDtTh}, we see that
\begin{align*}
& (\Id - \Hil)\cbrac*[\Bigg]{\frac{\lamb^\half}{\Zap^\threebytwo}\pap\Dt\Th} \\
& = (\Id - \Hil)\cbrac*[\Bigg]{\brac{-\frac{\bvarap}{2} + \frac{3\Dap\Zt}{2}}\frac{\lamb^\half}{\Zap^\threebytwo}\pap\Th}  + \sqbrac{\bvar,\Hil}\pap\cbrac*[\Bigg]{\frac{\lamb^\half}{\Zap^\threebytwo}\pap\Th}
\end{align*}
From this we obtain by \propref{prop:commutator} and \lemref{lem:CWlamb}
\begin{align*}
\lpar & \norm[\Hhalf]{\frac{\wbar^\threebytwo}{\Zapabs}(\Id - \Hil)\cbrac{\frac{\lamb^\half}{\Zapabs^\half}\pap\Dt\Th}} \\
& \lesssim  \cbrac{\norm[2]{\pap\frac{1}{\Zapabs}} + \norm[2]{\Dapabs\w} }\norm*[\Bigg][2]{\frac{\lamb^\half}{\Zapabs^\half}\pap\Dt\Th}  \\
& \quad + \brac{\norm[\Wcal]{\bvarap} + \norm[\Wcal]{\Dap\Zt} }\norm*[\Bigg][\Ccallambhalf]{\frac{\lamb^\half}{\Zap^\threebytwo}\pap\Th} + \norm[\Hhalf]{\bvarap}\norm*[\Bigg][\Hhalf]{\frac{\lamb^\half}{\Zap^\threebytwo}\pap\Th}
\end{align*}

\end{enumerate}
\bigskip

\item $\dis \frac{\lamb^\half}{\Zapabs^\half}\pap\Rthree \in \Ltwolambhalf$
\medskip\\
Proof: Recall from \eqref{form:Rthree} the formula of $\Rthree$
\begin{align*}
\begin{split}
\Rthree & =  \cbrac{-2(\Dap^2\Ztt) + 6(\Dap\Zt)(\Dap^2\Zt)}(\Dap\Ztbar)  + \cbrac{-4(\Dap\Ztt) + 6(\Dap\Zt)^2}(\Dap^2\Ztbar) \\
& \quad   -2(\Dap^2\Zt)(\Dap\Zttbar) - 4(\Dap\Zt)(\Dap^2\Zttbar)  -2i\wbar(\Dap\wbar)\brac{\frac{1}{\Zapabs^2}\pap\Jone}  \\
& \quad - i(\Dap\Jone)\brac{\Dap\frac{1}{\Zap}} -i\Jone\brac{\Dap^2\frac{1}{\Zap}}
\end{split}
\end{align*}
Let us control the terms individually
\begin{enumerate}[leftmargin =*]
\item We see that
\begin{align*}
& \norm[2]{\frac{\lamb^\half}{\Zapabs^\half}\pap \brac*[\big]{\cbrac{-2(\Dap^2\Ztt) + 6(\Dap\Zt)(\Dap^2\Zt)}(\Dap\Ztbar)}} \\
& \lesssim \norm[2]{ \frac{\lamb^\half}{\Zapabs^\half}\pap\Dap^2\Ztt}\norm[\infty]{\Dap\Ztbar} + \norm[2]{\Dap^2\Ztt}\norm[\infty]{\frac{\lamb^\half}{\Zapabs^\half}\pap \Dap\Ztbar} \\
& \quad + \norm[\infty]{\frac{\lamb^\half}{\Zapabs^\half}\pap \Dap\Zt}\norm[2]{\Dap^2\Zt}\norm[\infty]{\Dap\Ztbar} + \norm[\infty]{\Dap\Zt}^2\norm[2]{\frac{\lamb^\half}{\Zapabs^\half}\pap\Dap^2\Zt} \\
& \quad + \norm[\infty]{\frac{\lamb^\half}{\Zapabs^\half}\pap \Dap\Ztbar}\norm[2]{\Dap^2\Zt}\norm[\infty]{\Dap\Zt} 
\end{align*}
\item We have
\begin{align*}
& \norm[2]{\frac{\lamb^\half}{\Zapabs^\half}\pap\brac*[\big]{\cbrac{-4(\Dap\Ztt) + 6(\Dap\Zt)^2}(\Dap^2\Ztbar)}} \\
& \lesssim \norm[\infty]{\frac{\lamb^\half}{\Zapabs^\half}\pap \Dap\Ztt}\norm[2]{\Dap^2\Ztbar} + \norm[\infty]{\Dap\Ztt}\norm[2]{\frac{\lamb^\half}{\Zapabs^\half}\pap \Dap^2\Ztbar} \\
& \quad + \norm[\infty]{\frac{\lamb^\half}{\Zapabs^\half}\pap\Dap\Zt}\norm[\infty]{\Dap\Zt}\norm[2]{\Dap^2\Ztbar} + \norm[\infty]{\Dap\Zt}^2\norm[2]{\frac{\lamb^\half}{\Zapabs^\half}\pap\Dap^2\Ztbar}
\end{align*}
\item We have
\begin{align*}
& \norm[2]{\frac{\lamb^\half}{\Zapabs^\half}\pap\brac{  -2(\Dap^2\Zt)(\Dap\Zttbar) - 4(\Dap\Zt)(\Dap^2\Zttbar) }} \\
& \lesssim \norm[2]{\frac{\lamb^\half}{\Zapabs^\half}\pap\Dap^2\Zt}\norm[\infty]{\Dap\Zttbar} + \norm[2]{\Dap^2\Zt}\norm[\infty]{\frac{\lamb^\half}{\Zapabs^\half}\pap\Dap\Zttbar} \\
& \quad + \norm[\infty]{\frac{\lamb^\half}{\Zapabs^\half}\pap\Dap\Zt}\norm[2]{\Dap^2\Zttbar} + \norm[\infty]{\Dap\Zt}\norm[2]{\frac{\lamb^\half}{\Zapabs^\half}\pap\Dap^2\Zttbar}
\end{align*}
\item We have
\begin{align*}
& \norm[2]{\frac{\lamb^\half}{\Zapabs^\half}\pap \brac{-2i\wbar(\Dap\wbar)\brac{\frac{1}{\Zapabs^2}\pap\Jone}}} \\
& \lesssim \norm[\infty]{\frac{\lamb^\half}{\Zapabs^\half}\pap\wbar}\norm[2]{\Dapabs\wbar}\norm[\infty]{\frac{1}{\Zapabs^2}\pap\Jone} + \norm[2]{\frac{\lamb^\half}{\Zapabs^\half}\pap\Dapabs\wbar}\norm[\infty]{\frac{1}{\Zapabs^2}\pap\Jone} \\
& \quad + \norm[\infty]{\frac{\lamb^\half}{\Zapabs^\half}\pap\wbar}\norm[2]{\Dapabs\brac{\frac{1}{\Zapabs^2}\pap\Jone}}
\end{align*}
\item We have
\begin{align*}
& \norm[2]{\frac{\lamb^\half}{\Zapabs^\half}\pap\brac{- i(\Dap\Jone)\brac{\Dap\frac{1}{\Zap}} -i\Jone\brac{\Dap^2\frac{1}{\Zap}}}} \\
& \lesssim \brac{\norm[\infty]{\frac{\lamb^\half}{\Zapabs^\half}\pap\wbar} + \norm[\infty]{\lamb^\half\Zapabs^\half\pap\frac{1}{\Zap}}}\norm[\infty]{\frac{1}{\Zapabs^2}\pap\Jone}\norm[2]{\pap\frac{1}{\Zap}} \\
& \quad + \norm[2]{\Dapabs\brac{\frac{1}{\Zapabs^2}\pap\Jone}}\norm[\infty]{\lamb^\half\Zapabs^\half\pap\frac{1}{\Zap}}  + \norm[\infty]{\frac{1}{\Zapabs^2}\pap\Jone}\norm[2]{\frac{\lamb^\half}{\Zapabs^\half}\pap^2\frac{1}{\Zap}} \\
& \quad + \norm[\infty]{\Jone}\norm[2]{\frac{\lamb^\half}{\Zapabs^\half}\pap\Dap^2\frac{1}{\Zap}}
\end{align*}
\end{enumerate}

\bigskip

\item $\dis \Rfour \in \Ltwolambhalf$
\medskip\\
Proof: Recall from \eqref{form:Rfour} the formula of $\Rfour$
\begin{align*}
\begin{split}
\Rfour &=  -\cbrac{\frac{\Dt\bvarap}{2} + \frac{\Dt\Dap\Zt}{2} - \brac{\frac{\bvarap}{2} + \frac{\Dap\Zt}{2}}^2 }\frac{\lamb^\half}{\Zaphalf}\pap\Dap^2\Ztbar  + \frac{\lamb^\half}{\Zaphalf}\pap\Rthree \\
& \quad - \cbrac{2i\Aone\brac{\Dapabs\frac{1}{\Zapabs}} + \frac{i}{\Zapabs^2}\pap\Aone - \frac{i\Aone}{2}\brac{\Dapbar \frac{1}{\Zap}} }\frac{\lamb^\half}{\Zaphalf}\pap\Dap^2\Ztbar\\
& \quad -3i\wbar^2\brac{\frac{\lamb^\half}{\Zaphalf}\pap\wbar}\Dapabs\brac{\frac{1}{\Zapabs^2}\pap\Jone} - \brac{\bvarap + \Dap\Zt}\frac{\lamb^\half}{\Zaphalf}\pap\Dt\Dap^2\Ztbar
\end{split}
\end{align*}
Hence we have the estimate
\begin{align*}
& \norm[2]{\Rfour} \\
& \lesssim \cbrac{\norm[\infty]{\Dt\bvarap} + \norm[\infty]{\Dt\Dap\Zt} + \brac{\norm[\infty]{\bvarap} + \norm[\infty]{\Dap\Zt}}^2  }\norm[2]{\frac{\lamb^\half}{\Zapabs^\half}\pap\Dap^2\Ztbar} \\
& \quad + \norm[2]{ \frac{\lamb^\half}{\Zapabs^\half}\pap\Rthree} \! +  \cbrac{\norm[\infty]{\Aone}\norm[\infty]{\Dapabs\frac{1}{\Zapabs}} + \norm[\infty]{\frac{1}{\Zapabs^2}\pap\Aone}}\norm[2]{\frac{\lamb^\half}{\Zapabs^\half}\pap\Dap^2\Ztbar} \\
& \quad + \norm[\infty]{\Aone}\norm[\infty]{\Dapabs\frac{1}{\Zap}} \norm[2]{\frac{\lamb^\half}{\Zapabs^\half}\pap\Dap^2\Ztbar} +  \norm[\infty]{\frac{\lamb^\half}{\Zapabs^\half}\pap\w}\norm[2]{\Dapabs\brac{\frac{1}{\Zapabs^2}\pap\Jone}} \\
& \quad  + \cbrac*[\big]{\norm[\infty]{\bvarap} + \norm[\infty]{\Dap\Zt} }\norm[2]{\frac{\lamb^\half}{\Zapabs^\half}\pap\Dt\Dap^2\Ztbar}
\end{align*}
\bigskip

\item $\dis (\Id - \Hil)\Dt^2\brac{\frac{\lamb^\half}{\Zaphalf}\pap\Dap^2\Ztbar} \in \Ltwolambhalf$
\medskip\\
Proof: For a function $f$ satisfying $\Pa f = 0$ we have from \propref{prop:tripleidentity}
\begin{align*}
(\Id - \Hil)\Dt^2f & = \sqbrac{\Dt,\Hil}\Dt f + \Dt\sqbrac{\Dt,\Hil}f \\
& = \sqbrac{\bvar,\Hil}\pap\Dt f + \Dt\sqbrac{\bvar,\Hil}\pap f \\
& = 2\sqbrac{\bvar,\Hil}\pap\Dt f + \sqbrac{\Dt\bvar,\Hil}\pap f - \sqbrac{\bvar, \bvar ; \pap f}
\end{align*}
As $\Pa \brac{\frac{\lamb^\half}{\Zaphalf}\pap\Dap^2\Ztbar} = 0$ we obtain from \propref{prop:commutator}
\begin{align*}
& \norm[2]{(\Id - \Hil)\Dt^2\brac{\frac{\lamb^\half}{\Zaphalf}\pap\Dap^2\Ztbar}} \\
& \lesssim \norm[\infty]{\bvarap}\norm[2]{\Dt\brac{\frac{\lamb^\half}{\Zaphalf}\pap\Dap^2\Ztbar}} + \norm[\Hhalf]{\pap\Dt\bvar}\norm[2]{\frac{\lamb^\half}{\Zaphalf}\pap\Dap^2\Ztbar}  \\
& \quad + \norm[\infty]{\bvarap}^2\norm[2]{\frac{\lamb^\half}{\Zaphalf}\pap\Dap^2\Ztbar}
\end{align*}
\bigskip

\item $\dis (\Id - \Hil)\cbrac{i\frac{\Aone}{\Zapabs^2}\pap\brac{\frac{\lamb^\half}{\Zaphalf}\pap\Dap^2\Ztbar}} \in \Ltwolambhalf$
\medskip\\
Proof: We see that
\[
(\Id - \Hil)\cbrac{i\frac{\Aone}{\Zapabs^2}\pap\brac{\frac{\lamb^\half}{\Zaphalf}\pap\Dap^2\Ztbar}} = i\sqbrac{\frac{\Aone}{\Zapabs^2},\Hil}\pap\brac{\frac{\lamb^\half}{\Zaphalf}\pap\Dap^2\Ztbar}
\]
and hence we have from \propref{prop:commutator}
\begin{align*}
& \norm[2]{(\Id - \Hil)\cbrac{i\frac{\Aone}{\Zapabs^2}\pap\brac{\frac{\lamb^\half}{\Zaphalf}\pap\Dap^2\Ztbar}}} \\
& \lesssim \norm[2]{\frac{\lamb^\half}{\Zaphalf}\pap\Dap^2\Ztbar}\cbrac{\norm*[\bigg][\infty]{\frac{1}{\Zapabs^2}\pap\Aone} + \norm[\infty]{\Aone}\norm[\infty]{\Dapabs\frac{1}{\Zapabs}} }
\end{align*}
\bigskip

\item $\dis \frac{\lamb^\half}{\Zapabs^\half}\pap\Dapabs\brac{\frac{1}{\Zapabs^2}\pap\Jone} \in \Ltwolambhalf$
\medskip\\
Proof: Recall the equation of $\dis \frac{\lamb^\half}{\Zaphalf}\pap\Dap^2\Ztbar$ from \eqref{eq:lambpapDap^2Ztbar} 
\begin{align*}
\brac{\Dt^2 +i\frac{\Aone}{\Zapabs^2}\pap}\frac{\lamb^\half}{\Zaphalf}\pap\Dap^2\Ztbar =  -i\wbar^3\frac{\lamb^\half}{\Zaphalf}\pap\Dapabs\brac{\frac{1}{\Zapabs^2}\pap\Jone} + \Rfour
\end{align*}
Applying $(\Id - \Hil)$ to the above equation we obtain the estimate
\begin{align*}
&\norm[2]{(\Id - \Hil)\cbrac{\frac{\lamb^\half \wbar^3}{\Zaphalf}\pap\Dapabs\brac{\frac{1}{\Zapabs^2}\pap\Jone}}} \\
& \lesssim  \norm[2]{\Rfour} + \norm[2]{(\Id - \Hil)\Dt^2\brac{\frac{\lamb^\half}{\Zaphalf}\pap\Dap^2\Ztbar}} + \norm[2]{(\Id - \Hil)\cbrac{i\frac{\Aone}{\Zapabs^2}\pap\brac{\frac{\lamb^\half}{\Zaphalf}\pap\Dap^2\Ztbar}}}
\end{align*}
Now as $\Jone$ is real valued we see that
\begin{align*}
\frac{\lamb^\half}{\Zapabs^\half}\pap\Dapabs\brac{\frac{1}{\Zapabs^2}\pap\Jone} = \Real\cbrac{\frac{\lamb^\half \w^\sevenbytwo\wbar^3}{\Zaphalf}\pap (\Id - \Hil)\Dapabs\brac{\frac{1}{\Zapabs^2}\pap\Jone}}
\end{align*}
and we see that
\begin{align*}
\frac{\lamb^\half \wbar^3}{\Zaphalf}\pap (\Id - \Hil)\Dapabs\brac{\frac{1}{\Zapabs^2}\pap\Jone} & = -\sqbrac{\frac{\lamb^\half \wbar^3}{\Zaphalf}, \Hil}\pap\Dapabs\brac{\frac{1}{\Zapabs^2}\pap\Jone} \\
& \quad + (\Id - \Hil)\cbrac{\frac{\lamb^\half \wbar^3}{\Zaphalf}\pap\Dapabs\brac{\frac{1}{\Zapabs^2}\pap\Jone}}
\end{align*}
Hence from \propref{prop:commutator} we obtain
\begin{align*}
& \norm[2]{\frac{\lamb^\half}{\Zapabs^\half}\pap\Dapabs\brac{\frac{1}{\Zapabs^2}\pap\Jone}} \\
& \lesssim \cbrac{\norm[\infty]{\lamb^\half\Zapabs^\half\pap\frac{1}{\Zap}} + \norm[\infty]{\frac{\lamb^\half}{\Zapabs^\half}\pap\w} }\norm[2]{\Dapabs\brac{\frac{1}{\Zapabs^2}\pap\Jone}} \\
& \quad + \norm[2]{(\Id - \Hil)\cbrac{\frac{\lamb^\half \wbar^3}{\Zaphalf}\pap\Dapabs\brac{\frac{1}{\Zapabs^2}\pap\Jone}}}
\end{align*}

\bigskip

\end{enumerate}

\subsection{Closing the energy estimate for  $\lamb\Eaux$}\label{sec:closeEaux} 

We now complete the proof of \thmref{thm:aprioriEaux}.  To simplify the calculations, we will continue to use the notation used in \secref{sec:quantEaux} and introduce another notation: If $a(t), b(t) $ are functions of time we write $a \approx b$ if there exists a universal polynomial $P$ with non-negative coefficients so that $\abs{a(t)-b(t)} \leq P(\Ehigh(t))(\lamb\Eaux(t))$. Observe that $\approx$ is an equivalence relation. With this notation, proving \thmref{thm:aprioriEaux} is equivalent to showing $ \frac{d}{dt} (\lamb\Eaux(t)) \approx 0$. We control the first four terms of the energy \eqref{def:Eaux} directly and for the last two terms we use the equation \eqref{eq:lambpapDap^2Ztbar}.

\begin{enumerate}[leftmargin =*]
\item Controlling the time derivative of $\dis \norm[\infty]{\lamb^\half\Zapabs^\half\pap\frac{1}{\Zap}}^2$ proceeds in exactly the same as that of controlling $\dis \norm[\infty]{\sigma^\half\Zapabs^\half\pap\frac{1}{\Zap}}^2$ for the energy $\Esigma$ which was done in Sec 5.2.1 in \cite{Ag19}. First observe that
\begin{align*}
\norm*[\bigg][\infty]{\Dt\brac*[\bigg]{\lamb^\half\Zapabs^\half \pap\frac{1}{\Zap}}} & \lesssim \brac{\norm[\infty]{\Dap\Zt} + \norm[\infty]{\bvarap}}\norm[\infty]{\lamb^\half\Zapabs^\half\pap\frac{1}{\Zap}}  \\
& \quad + \norm*[\bigg][\infty]{\frac{\lamb^\half}{\Zapabs^\half}\pap\bvarap} + \norm*[\bigg][\infty]{\frac{\lamb^\half}{\Zapabs^\half}\pap\Dap\Zt }
\end{align*}
Now using the computation from Sec 5.2.1 in \cite{Ag19} we obtain
\begin{flalign*}
\lpar \lpar \quad & \limsup_{s \to 0^+} \frac{\norm*[\Big][\infty]{ \lamb^{\half}  \Zapabs^\half \pap\frac{1}{\Zap}}^2(t+s) -  \norm*[\Big][\infty]{  \lamb^{\half} \Zapabs^\half \pap\frac{1}{\Zap}}^2(t)  }{s} && \\
& \lesssim \norm*[\Big][\infty]{  \lamb^{\half} \Zapabs^\half \pap\frac{1}{\Zap}}(t)\norm*[\bigg][\infty]{\Dt\brac*[\bigg]{\lamb^\half\Zapabs^\half \pap\frac{1}{\Zap}}}(t) \\
& \lesssim P(\Ecalhigh(t))(\lamb\Eaux(t))
\end{flalign*}

\item We first observe that
\begin{align*}
\lpar \Dt\brac{\frac{\lamb^\half}{\Zaphalf}\pap\Ztapbar} = -\brac{\frac{\bvarap}{2} + \frac{\Dap\Zt}{2} }\frac{\lamb^\half}{\Zaphalf}\pap\Ztapbar + \frac{\lamb^\half}{\Zaphalf}\pap\brac{ -\bvarap\Ztapbar + \Zttapbar}
\end{align*}
Hence
\begin{align*}
\norm[2]{\Dt\brac{\frac{\lamb^\half}{\Zaphalf}\pap\Ztapbar}} & \lesssim \cbrac{\norm[\infty]{\bvarap} + \norm[\infty]{\Dap\Zt} }\norm[2]{\frac{\lamb^\half}{\Zapabs^\half}\pap\Ztapbar}  \\
& \quad + \norm[\infty]{\frac{\lamb^\half}{\Zapabs^\half}\pap\bvarap}\norm[2]{\Ztapbar} + \norm[2]{\frac{\lamb^\half}{\Zapabs^\half}\pap\Zttapbar}
\end{align*}
Now by using \lemref{lem:timederiv} we obtain
\begin{align*}
& \frac{d}{dt} \int \abs*[\bigg]{\frac{\lamb^\half}{\Zaphalf}\pap\Ztbarap}^2 \diff\ap \\
& \lesssim \norm[\infty]{\bvarap}\norm[2]{\frac{\lamb^\half}{\Zapabs^\half}\pap\Ztapbar}^2 + \norm[2]{\frac{\lamb^\half}{\Zapabs^\half}\pap\Ztapbar}\norm[2]{\Dt\brac{\frac{\lamb^\half}{\Zaphalf}\pap\Ztapbar}} \\
& \lesssim P(\Ecalhigh)(\lamb\Eaux)
\end{align*}

\item We observe that from \eqref{form:DtoneoverZap}
\begin{align*}
& \Dt\brac{\frac{\lamb^\half}{\Zaphalf}\pap^2\frac{1}{\Zap}} \\
& = -\brac{\frac{\bvarap}{2} + \frac{\Dap\Zt}{2} }\frac{\lamb^\half}{\Zaphalf}\pap^2\frac{1}{\Zap} + \frac{\lamb^\half}{\Zaphalf}\pap\brac{\Dt\pap\frac{1}{\Zap}} \\
& =  -\brac{\frac{\bvarap}{2} + \frac{\Dap\Zt}{2} }\frac{\lamb^\half}{\Zaphalf}\pap^2\frac{1}{\Zap} + \frac{\lamb^\half}{\Zaphalf}\pap\brac{\Dap\bvarap -\Dap^2\Zt - \brac{\pap\frac{1}{\Zap}}\Dap\Zt} 
\end{align*}
Hence we have
\begin{align*}
\!\norm[2]{\Dt\brac{\frac{\lamb^\half}{\Zaphalf}\pap^2\frac{1}{\Zap}}} & \lesssim  \cbrac{\norm[\infty]{\bvarap} + \norm[\infty]{\Dap\Zt} }\norm[2]{\frac{\lamb^\half}{\Zapabs^\half}\pap^2\frac{1}{\Zap}} + \norm[2]{\frac{\lamb^\half}{\Zapabs^\half}\pap\Dap^2\Zt} \\
& \quad + \norm[2]{\pap\frac{1}{\Zap}}\norm[\infty]{\frac{\lamb^\half}{\Zapabs^\half}\pap\Dap\Zt} + \norm[2]{\frac{\lamb^\half}{\Zapabs^\half}\pap\Dap\bvarap}
\end{align*}
Now by using \lemref{lem:timederiv} we obtain
\begin{align*}
& \frac{d}{dt} \int \abs*[\bigg]{\frac{\lamb^\half}{\Zaphalf}\pap^2\frac{1}{\Zap}}^2 \diff\ap \\
& \lesssim \norm[\infty]{\bvarap}\norm[2]{\frac{\lamb^\half}{\Zapabs^\half}\pap^2\frac{1}{\Zap}}^2 + \norm[2]{\frac{\lamb^\half}{\Zapabs^\half}\pap^2\frac{1}{\Zap}}\norm[2]{\Dt\brac{\frac{\lamb^\half}{\Zaphalf}\pap^2\frac{1}{\Zap}} }\\
& \lesssim P(\Ecalhigh)(\lamb\Eaux)
\end{align*}

\item  Using \lemref{lem:timederiv} we see that
\begin{align*}
& \frac{d}{dt} \int \abs*[\bigg]{\frac{\lamb^\half}{\Zaphalf}\pap\Dap^2\Ztbar}^2 \diff\ap \\
& \lesssim \norm[\infty]{\bvarap}\norm[2]{\frac{\lamb^\half}{\Zapabs^\half}\pap\Dap^2\Ztbar}^2 + \norm[2]{\frac{\lamb^\half}{\Zapabs^\half}\pap\Dap^2\Ztbar}\norm[2]{\Dt\brac{\frac{\lamb^\half}{\Zaphalf}\pap\Dap^2\Ztbar} }\\
& \lesssim P(\Ecalhigh)(\lamb\Eaux)
\end{align*}

\item The quantity left to control is the time derivative of 
\begin{align*}
 \int \abs{\Dt f}^2\difff\ap +  \int \abs*[\bigg]{ \papabs^\half \brac*[\bigg]{\frac{\sqrt{\Aone}}{\Zapabs} f}}^2\difff\ap
\end{align*}
where $f = \frac{\lamb^\half}{\Zaphalf}\pap\Dap^2\Ztbar$ and we have $\Ph f = f$. Following the same computation as done in \secref{sec:closeEhigh} for the energy $\Ehigh$ we see that
\begin{align*}
\frac{d}{dt} (\lamb\Eaux(t)) \approx 2 \Real \int \brac{ \Dt^2 f +i\frac{\Aone}{\Zapabs^2}\pap f }(\Dt \bar{f}) \diff\ap
\end{align*}
As $\Dt\brac{\frac{\lamb^\half}{\Zaphalf}\pap\Dap^2\Ztbar} \in \Ltwolambhalf$ we only need to show that the other term in in $\Ltwolambhalf$. Now the equation for $\frac{\lamb^\half}{\Zaphalf}\pap\Dap^2\Ztbar$ from \eqref{eq:lambpapDap^2Ztbar} implies
\begin{align*}
\brac{\Dt^2 +i\frac{\Aone}{\Zapabs^2}\pap}\frac{\lamb^\half}{\Zaphalf}\pap\Dap^2\Ztbar =  -i\wbar^3\frac{\lamb^\half}{\Zaphalf}\pap\Dapabs\brac{\frac{1}{\Zapabs^2}\pap\Jone} + \Rfour
\end{align*}
As we have shown $\dis \frac{\lamb^\half}{\Zapabs^\half}\pap\Dapabs\brac{\frac{1}{\Zapabs^2}\pap\Jone} \in \Ltwolambhalf$ and $\Rfour \in \Ltwolambhalf$, this implies that  the right hand side is in $\Ltwolambhalf$ and so the proof of \thmref{thm:aprioriEaux} is complete.

\end{enumerate}

\subsection{Equivalence of $\lamb\Eaux$ and $\lamb\Ecalaux$}\label{sec:equivEauxEcalaux} 

We now give a simpler description of the energy $\Eaux$ defined by \eqref{def:Eaux}. Recall the definition of $\Ecalaux$ from \eqref{def:Ecalaux}
\begin{align*}
\Ecalaux & = \norm[\infty]{\Zap^\half\pap\frac{1}{\Zap}}^2 + \norm*[\Bigg][2]{\frac{1}{\Zap^\half}\pap^2\frac{1}{\Zap}}^2 + \norm*[\Bigg][2]{\frac{1}{\Zap^\fivebytwo}\pap^3\frac{1}{\Zap}}^2 \\
& \quad + \norm*[\Bigg][2]{\frac{1}{\Zap^\half}\pap\Ztapbar}^2 +  \norm*[\Bigg][2]{\frac{1}{\Zap^\fivebytwo}\pap^2\Ztapbar}^2 + \norm*[\Bigg][\Hhalf]{\frac{1}{\Zap^\sevenbytwo}\pap^2\Ztapbar}^2
\end{align*}
\begin{prop}\label{prop:equivEauxEcalaux}
There exists universal polynomials $P_1, P_2$ with non-negative coefficients so that if $(\Z,\Zt)(t)$ is a smooth solution to the water wave equation \eqref{eq:systemone} with $\sigma=0$ in the time interval $[0,T]$ satisfying $(\Zap-1,\frac{1}{\Zap} - 1, \Zt) \in \Linfty([0,T], H^{s}(\Rsp)\times H^{s}(\Rsp)\times H^{s+\half}(\Rsp))$ for all $s\geq 2$ and if $\lamb>0$ is some given constant,  then for all $t \in [0,T]$ we have
\begin{align*}
\lamb\Eaux \leq P_1(\Ecalhigh)(\lamb\Ecalaux) \quad \tx{ and }\quad  \lamb\Ecalaux \leq P_2(\Ecalhigh)(\lamb\Eaux)
\end{align*}
\end{prop}
\begin{proof}
Let $\lamb\Eaux < \infty$. We have already pretty much controlled all the terms of $\lamb\Ecalaux$ and the only term not directly controlled is $\dis \frac{\lamb^\half}{\Zap^\sevenbytwo}\pap^2\Ztapbar $.  This term can be easily controlled by using \lemref{lem:CWlamb}
\begin{align*}
\norm*[\Bigg][\Ccallambhalf]{\frac{\lamb^\half}{\Zap^\sevenbytwo}\pap^2\Ztapbar } \lesssim \norm*[\big][\Wcal]{\wbar^\sevenbytwo}\norm[\Ccallambhalf]{\frac{\lamb^\half}{\Zapabs^\sevenbytwo}\pap^2\Ztapbar }
\end{align*}

Now assume that $\lamb\Ecalaux < \infty$. We see that the first three terms of $\lamb\Eaux$ are controlled and following the proofs of $\dis \lamb^\half\Zapabs^\half\pap\frac{1}{\Zapabs} \in \Linftylambhalf$ and $\dis \frac{\lamb^\half}{\Zapabs^\threebytwo}\pap^2\w \in \Ltwolambhalf$ in  \secref{sec:quantEaux}, we immediately get
\begin{align*}
& \norm[\infty]{\lamb^\half\Zapabs^\half\pap\frac{1}{\Zapabs}}  + \norm[\infty]{\frac{\lamb^\half}{\Zapabs^\half}\pap\w} + \norm[2]{\frac{\lamb^\half}{\Zapabs^\half}\pap^2\frac{1}{\Zapabs}} + \norm[2]{\frac{\lamb^\half}{\Zapabs^\threebytwo}\pap^2\w} \\
& \lesssim P_1(\Ecalhigh)(\lamb\Ecalaux)^\half
\end{align*}
Now following the proof of $\dis \frac{\lamb^\half}{\Zapabs^\half}\pap\Dap\Zt \in \Wcallambhalf\cap\Ccallambhalf$ from  \secref{sec:quantEaux}, we see that
\begin{align*}
& \norm[\Wcallambhalf\cap\Ccallambhalf]{\frac{\lamb^\half}{\Zapabs^\threebytwo}\pap\Ztapbar } + \norm[\Wcallambhalf\cap\Ccallambhalf]{ \frac{\lamb^\half}{\Zapabs^\half}\pap\Dap\Zt} + \norm[\Wcallambhalf\cap\Ccallambhalf]{\frac{\lamb^\half}{\Zapabs^\half}\pap\Dap\Ztbar} \\
& \lesssim P_1(\Ecalhigh)(\lamb\Ecalaux)^\half
\end{align*}
Similarly following the proof of $\dis \frac{\lamb^\half}{\Zapabs^\threebytwo}\pap^2\frac{1}{\Zap} \in \Wcallambhalf\cap\Ccallambhalf$ from  \secref{sec:quantEaux}, we see that
\begin{align*}
\norm[\Wcallambhalf\cap\Ccallambhalf]{ \frac{\lamb^\half}{\Zapabs^\threebytwo}\pap^2\frac{1}{\Zap} } \lesssim P_1(\Ecalhigh)(\lamb\Ecalaux)^\half
\end{align*}
Now following the proofs of $\dis \frac{\lamb^\half}{\Zapabs^\sevenbytwo}\pap^3\Aone \in \Ltwolambhalf $ and $\dis \frac{\lamb^\half}{\Zapabs^\fivebytwo}\pap^2\Ztapbar \in \Ltwolambhalf$ in  \secref{sec:quantEaux}, we see that 
\begin{align*}
\norm[2]{ \frac{\lamb^\half}{\Zapabs^\sevenbytwo}\pap^3\Aone} + \norm[2]{\frac{\lamb^\half}{\Zaphalf}\pap\Dap^2\Zt} + \norm[2]{\frac{\lamb^\half}{\Zaphalf}\pap\Dap^2\Ztbar} \lesssim P_1(\Ecalhigh)(\lamb\Ecalaux)^\half
\end{align*}
Now following the proof of $\dis \frac{\lamb^\half}{\Zapabs^\fivebytwo}\pap^3\frac{1}{\Zap} \in \Ltwolambhalf$ from  \secref{sec:quantEaux}, we obtain 
\begin{align*}
\norm[2]{\frac{\lamb^\half}{\Zapabs^\fivebytwo}\pap^3\Zttbar} \lesssim P_1(\Ecalhigh)(\lamb\Ecalaux)^\half
\end{align*}
Now we follow the proof of $\dis \frac{\lamb^\half}{\Zapabs^\fivebytwo}\pap^3\Zttbar \in \Ltwolambhalf$  from  \secref{sec:quantEaux} to obtain 
\begin{align*}
\norm[2]{\Dt\brac{\frac{\lamb^\half}{\Zaphalf}\pap\Dap^2\Ztbar}}  \lesssim P_1(\Ecalhigh)(\lamb\Ecalaux)^\half
\end{align*}
Now we use \lemref{lem:CWlamb} to obtain the estimate
\begin{align*}
\norm[\Ccallambhalf]{\frac{\lamb^\half}{\Zapabs^\sevenbytwo}\pap^2\Ztapbar} \lesssim \norm*[\big][\Wcal]{\w^\sevenbytwo}\norm[\Ccallambhalf]{\frac{\lamb^\half}{\Zap^\sevenbytwo}\pap^2\Ztapbar}  \lesssim P_1(\Ecalhigh)(\lamb\Ecalaux)^\half
\end{align*}
Hence by following the proof of $\dis \frac{\lamb^\half}{\Zapabs^\sevenbytwo}\pap^2\Ztapbar \in \Ccallambhalf$ from  \secref{sec:quantEaux},  we see that 
\begin{align*}
\norm[\Ccallambhalf]{ \frac{\sqrt{\Aone}}{\Zapabs}\brac{\frac{\lamb^\half}{\Zaphalf}\pap\Dap^2\Ztbar}}  \lesssim P_1(\Ecalhigh)(\lamb\Ecalaux)^\half
\end{align*}
Hence proved. 
\end{proof}

\Bigskip

\section{The energy  $\EcalDelta$ }\label{sec:aprioriEcalDelta}
 \bigskip

In this section we consider two solutions of the water wave equation \eqref{eq:systemone}, one with and one without surface tension and prove an a priori estimate for the difference of the two solutions. Let $(\Z,\Zt)_a$ and $(\Z,\Zt)_b$ be two solutions of the water wave equation \eqref{eq:systemone} with surface tensions $\sigma_a = \sigma$ and $\sigma_b = 0$ respectively, i.e.  $(\Z,\Zt)_a$ is the solution with surface tension $\sigma$ and $(\Z,\Zt)_b$ is the solution with zero surface tension. We denote the two solutions as $A$ and $B$ respectively for simplicity.  In this section we consider an energy $\EDelta$ defined below, which is a weighted Sobolev norm for the difference of the two solutions $A$ and $B$, and prove our main energy estimate for this energy in \thmref{thm:aprioriEDelta}. To prove this theorem, we first control the quantities controlled by $\EDelta$ in \secref{sec:quantEDelta} and then close the energy estimate in \secref{sec:closeEDelta}. Finally in \secref{sec:EDeltaEcalDelta} we show that the energy $\EDelta$ is equivalent to the energy $\EcalDelta$ defined in \eqref{def:EcalDelta}. 

  \thmref{thm:aprioriEDelta} allows different initial data for the solutions $A$ and $B$ and hence the energies $\EDelta$ and $\EcalDelta$ takes this into account. The energy $\EcalDelta$ used in \secref{sec:results} is a simplified version of the actual energy defined in \eqref{def:EcalDelta}. In this section whenever we talk about the energy $\EcalDelta$, we will mean the one given by \eqref{def:EcalDelta}.

Let us recall some of the notation used in \secref{sec:results}. We will denote the terms and operators for each solution by their subscript $a$ or $b$.  Let $h_a, h_b$ be the homeomorphisms from \eqref{eq:h} for the respective solutions and let the material derivatives by given by $(\Dt)_a = U_{h_a}^{-1}\pt U_{h_a}$ and $(\Dt)_b = U_{h_b}^{-1}\pt U_{h_b}$. We define
\begin{align}\label{def:Util}
\htil = h_b \compose h_a^{-1} \quad \tx{ and } \quad \Util = U_{\htil} = U_{h_a}^{-1}U_{h_b}
\end{align}
We also define 
\begin{align}\label{def:Delta}
\Delta (f) = f_a - \Util(f_b)
\end{align}
See \secref{sec:results} for more details about the notation $\Delta (f)$. Define the operators
\begin{align}\label{eq:Hcal}
\begin{split}
(\Hcal f)(\ap) &= \frac{1}{i\pi} p.v. \int \frac{\htilbp(\bp)}{\htil(\ap) - \htil(\bp)}f(\bp) \diff \bp \\
 (\Hcaltil f)(\ap) &=\frac{1}{i\pi} p.v. \int \frac{1}{\htil(\ap) - \htil(\bp)}f(\bp) \diff \bp
 \end{split}
\end{align}
See \propref{prop:HcalHtilest} for some properties of these operators. In this section in addition to these above operators, we will also heavily use the notations $\sqbrac{f_1, f_2;  f_3}$ and $\sqbrac{f_1, f_2;  f_3}_{g}$ defined in \eqref{eq:foneftwofthree} and \eqref{eq:foneftwofthreeg} respectively. 

We are now ready to define the energy for the difference of the solutions. Define
\begingroup
\allowdisplaybreaks
\begin{align*}
& \EDeltazero =  
\begin{aligned}[t]
&  \norm[\infty]{\brac{\sigma^{\half} \Zapabs^\half \pap\frac{1}{\Zap}}\la}^2  +   \norm[2]{\brac{\sigma^\onebysix\Zapabs^\half\pap\frac{1}{\Zap}}\la}^6  + \norm[2]{\brac{\frac{\sigma^\half}{\Zapabs^\half}\pap^2\frac{1}{\Zap}}\la}^2 \\
& + \norm[\infty]{\Delta(\w)}^2  +  \norm[2]{\Delta\brac{ \pap\frac{1}{\Zap}}}^2 + \norm*[\big][\Linfty\cap\Hhalf]{\htilap - 1}^2 + \norm[2]{\Dapabs_a(\htilap -1)}^2 \\
&  + \norm[\infty]{\Zapabs_a\Util\brac{\frac{1}{\Zapabs_b}} - 1}^2
\end{aligned}\\
& \EDeltaone = \norm[\Hhalf]{\Delta\cbrac{(\Zttbar-i)\Zap}}^2 + \norm[2]{(\sqrt{\Aone})_a\Delta(\Ztbarap)}^2 + \norm[2]{\brac{\frac{\sigma^\half}{\Zapabs^\half}\pap\Ztapbar}\la}^2\\
& \EDeltatwo = \norm[2]{\Delta(\Dt\Ztapbar)}^2 + \norm[\Hhalf]{\brac{\frac{\sqrt{\Aone}}{\Zapabs}}\la \Delta( \Ztapbar)}^2 + \norm[\Hhalf]{\brac{\frac{\sigma^\half}{\Zapabs^\threebytwo}\pap\Ztapbar}\la}^2 \\
& \EDeltathree = \norm[2]{\Delta(\Dt\Th)}^2 + \norm[\Hhalf]{\brac{\frac{\sqrt{\Aone}}{\Zapabs}}\la \Delta( \Th)}^2 + \norm[\Hhalf]{\brac{\frac{\sigma^\half}{\Zapabs^\threebytwo}\pap\Th}\la}^2 \\
& \EDeltafour = \norm[\Hhalf]{\Delta(\Dt\Dapbar\Ztbar)}^2 + \norm[2]{(\sqrt{\Aone})_a\Dapabs_a\Delta(\Dapbar\Ztbar)}^2 + \norm[2]{\brac{\frac{\sigma^\half}{\Zapabs^\half}\pap\Dapabs\Dapbar\Ztbar}\la}^2\\
& \EDelta = \sigma(\Eaux)_b + \EDeltazero + \EDeltaone + \EDeltatwo + \EDeltathree + \EDeltafour
\end{align*}
\endgroup

Note that here $(\Eaux)_b$ is the energy $\Eaux$ defined in \eqref{def:Eaux} for the solution $B$. Hence the term $\sigma(\Eaux)_b$ couples the zero surface tension solution $B$ with the value of surface tension $\sigma$ of the solution $A$. This coupling term is crucial to closing the energy estimate and this is discussed in more detail in \secref{sec:discussion}. The other terms in the energy $\EDelta$ come from taking a difference in the energy $\Esigma$ defined in \secref{sec:aprioriEcalsigma}. We can now state our main a priori energy estimate about the difference of the two solutions. 

\begin{thm}\label{thm:aprioriEDelta}
Let $T>0$ and let $(\Z,\Zt)_a(t)$, $(\Z,\Zt)_b(t)$ be two smooth solutions in $[0,T]$ to  \eqref{eq:systemone} with surface tension $\sigma$ and zero surface tension respectively, such that for all $s\geq 2$ we have $(\Zap-1,\frac{1}{\Zap} - 1, \Zt)_i \in \Linfty([0,T], H^{s}(\Rsp)\times H^{s}(\Rsp)\times H^{s+\half}(\Rsp))$ for both $i=a,b$.  Let $L_1>0$ be such that 
\begin{align*}
\sup_{t \in [0,T]}(\Ecalhigh)_b(t), \sup_{t\in [0,T]}(\Ecalsigma)_a (t), \norm[\infty]{\Zapabs_a\Util\brac{\frac{1}{\Zapabs_b}}}(0),  \norm[\infty]{\frac{1}{\Zapabs_a}\Util\brac{\Zapabs_b}}(0) \leq L_1
\end{align*}
Then there exists a constant $C(L_1)$ depending only on $L_1$ so that for all $t\in [0,T)$ we have
\begin{align*}
\frac{d}{dt}\EDelta(t) \leq C(L_1) \EDelta(t) 
\end{align*}
\end{thm}
Note that the above theorem allows the initial data of the two solutions $A$ and $B$ to be different. The theorem simplifies a little bit if we work with the same initial data for the two solutions $A$ and $B$ which was how \thmref{thm:convergence} was stated. The energy $\EDelta$ is a very strong norm which compares the two solutions $A$ and $B$. In particular observe that the energy $\EDeltazero$ controls $\norm[\infty]{\Delta(\w)}$ and hence the above theorem implies that if the initial data of $A$ and $B$ are the same, then the angle of the interface $\thvar_a \to \thvar_b$ in $\Linfty$ as $\sigma \to 0$ (as stated in \secref{sec:results}).

The rest of this section is devoted to the proof of \thmref{thm:aprioriEDelta}. We will first need some basic properties which we use throughout the section. The following two lemmas capture some basic identities regarding the operators $\Util, \Hcal, \Hcaltil$ and $\Delta$.  
\begin{lem}\label{lem:basicUtil}
Let $\Util$ be defined by \eqref{def:Util} and let $\Hcal, \Hcaltil$ be defined by \eqref{eq:Hcal}. Then
\begin{enumerate}
\item $(\Dt)_a \Util = \Util (\Dt)_b$
\item $\pap\Util = \htilap\Util\pap \quad$ $\dis \pap\Util^{-1} = \frac{1}{\htilap\compose\htil^{-1}}\Util^{-1}\pap$ and $\dis \htilap = U_{h_a}^{-1}\brac{\frac{(\hal)_b}{(\hal)_a}}$
\item $\Hcal \Util = \Util \Hil $
\item $\Util\sqbrac{f, \Hil}\pap g = \sqbrac*{(\Util f), \Htil}\pap(\Util g) $
\item $\Util [f_1, f_2;\pap f_3 ] = [(\Util f_1), (\Util f_2) ; \pap(\Util f_3)]_{\htil}$
\end{enumerate}
\end{lem}
\begin{proof}
The proofs are quite straightforward.
\begin{enumerate}[leftmargin =*]
\item We see that $(\Dt)_a \Util = (U_{h_a}^{-1}\pt U_{h_a}) (U_{h_a}^{-1}U_{h_b}) = U_{h_a}^{-1}\pt U_{h_b}$ and similarly we have $\Util (\Dt)_b = (U_{h_a}^{-1}U_{h_b}) (U_{h_b}^{-1}\pt U_{h_b}) = U_{h_a}^{-1}\pt U_{h_b}$.
\item As $\Util(f)(\ap) = f(\htil(\ap))$ we see that $\pap\Util(f)(\ap) = \htilap(\ap)f_\ap(\htil(\ap))$ and hence $\pap\Util = \htilap\Util\pap$. Similarly we have $\dis \pap\Util^{-1} = \frac{1}{\htilap\compose\htil^{-1}}\Util^{-1}\pap$. Now as $\htil = h_b \compose h_a^{-1}$ we see that 
\begin{align*}
\htilap = \frac{\nobrac{(\hal)_b\compose h_a^{-1}}}{\nobrac{(\hal)_a\compose h_a^{-1}}}
\end{align*}
\item We see that
\begin{align*}
(\Hcal\Util f)(\ap) = \frac{1}{i\pi} \int \frac{\htilbp(\bp)}{\htil(\ap) - \htil(\bp)}f(\htil(\bp)) \diff \bp = \frac{1}{i\pi}\int \frac{1}{\htil(\ap) - \bp}f(\bp) \diff \bp = (\Util \Hil f)(\ap)
\end{align*}
\item We observe that
\begin{align*}
\Util\sqbrac{f,\Hil}\pap g & = \Util\cbrac{f \Hil(\pap g) - \Hil(f\pap g)} \\
& = \Util(f)\Hcal(\Util(\pap g)) - \Hcal(\Util(f\pap g)) \\
& = \Util(f)\Hcaltil(\pap\Util(g)) - \Hcaltil(\Util(f)\pap\Util(g)) \\
& =  \sqbrac*{(\Util f), \Htil}\pap(\Util g)
\end{align*}
\item We see that
\begin{align*}
& \Util [f_1, f_2;\pap f_3 ](\ap) \\
& = \frac{1}{i\pi} \int \brac{\frac{f_1(\htil(\ap)) - f_1(\bp)}{\htil(\ap) - \bp}}\brac{\frac{f_2(\htil(\ap)) - f_2(\bp)}{\htil(\ap)-\bp}} (f_3)_\bp(\bp) \diff \bp \\
& = \frac{1}{i\pi} \int \brac{\frac{f_1(\htil(\ap)) - f_1(\htil(\bp))}{\htil(\ap) - \htil(\bp)}}\brac{\frac{f_2(\htil(\ap)) - f_2(\htil(\bp))}{\htil(\ap)-\htil(\bp)}} (f_3\compose\htil)_\bp(\bp) \diff \bp \\
& =  [(\Util f_1), (\Util f_2) ; \pap(\Util f_3)]_{\htil}(\ap)
\end{align*}
\end{enumerate}
\end{proof}

\begin{lem}\label{lem:Delta}
Let $\Delta$ be defined by \eqref{def:Delta}. Then
\begin{enumerate}[leftmargin =*, align=left]
\item $\dis \Delta (f_1f_2\cdots f_n) = \sum_{i=1}^{n} \cbrac*[\big]{\Util(f_1)_b\cdots \Util(f_{i-1})_b} \Delta (f_{i}) \cbrac*[\big]{(f_{i+1})_a \cdots (f_n)_a} $
\item $\dis \Delta \sqbrac{f,\Hil}\pap g = \sqbrac{\Delta f, \Hil}\pap(g_a) + \sqbrac*[\big]{\Util(f)_b, \Hil - \Htil}\pap(g_a) + \Util\cbrac*[\big]{\sqbrac{f_b,\Hil}\pap\brac*[\big]{\Util^{-1}(\Delta g)} }$
\item $
\begin{aligned}[t]
\Delta \sqbrac{f_1,f_2; \pap f_3} & =  \sqbrac{\Delta f_1, (f_2)_a;\pap(f_3)_a} + \sqbrac*[\big]{\Util(f_1)_b, \Delta f_2; \pap(f_3)_a}  \\
& \quad + \sqbrac*[\big]{\Util(f_1)_b, \Util(f_2)_b; \pap(\Delta f_3)} \\
& \quad + \cbrac{\sqbrac*[\big]{\Util(f_1)_b,\Util(f_2)_b; \pap\Util(f_3)_b } - \sqbrac*[\big]{\Util(f_1)_b,\Util(f_2)_b; \pap\Util(f_3)_b }_{\htil} }
\end{aligned}
$
\end{enumerate}
\end{lem}
\begin{proof}
The first identity follows immediately from the definition. For the second, we see that
\begin{align*}
 \Delta \sqbrac{f,\Hil}\pap g & = \sqbrac{f_a,\Hil}\pap g_a - \Util\sqbrac{f_b,\Hil}\pap g_b \\
& = \sqbrac{f_a,\Hil}\pap g_a - \sqbrac*[\big]{(\Util f_b),\Hcaltil}\pap (\Util g_b) \\
& =  \sqbrac{\Delta(f),\Hil}\pap g_a +  \sqbrac*[\big]{\Util(f_b),\Hil}\pap g_a - \sqbrac*[\big]{(\Util f_b),\Hcaltil}\pap (\Util g_b) \\
& =  \sqbrac{\Delta f, \Hil}\pap(g_a) + \sqbrac*[\big]{\Util(f_b), \Hil - \Htil}\pap(g_a) + \sqbrac*[\big]{\Util(f_b), \Htil}\pap(\Delta(g)) \\
& = \sqbrac{\Delta f, \Hil}\pap(g_a) + \sqbrac*[\big]{\Util(f_b), \Hil - \Htil}\pap(g_a) + \Util\cbrac*[\big]{\sqbrac{f_b,\Hil}\pap\brac*[\big]{\Util^{-1}(\Delta g)} }
\end{align*}
For the third estimate we see that
\begin{align*}
\Delta \sqbrac{f_1,f_2; \pap f_3} & =  \sqbrac{(f_1)_a,(f_2)_a; \pap (f_3)_a} - \Util  \sqbrac{(f_1)_b,(f_2)_b; \pap (f_3)_b} \\
& =  \sqbrac{(f_1)_a,(f_2)_a; \pap (f_3)_a} -  \sqbrac*[\big]{\Util(f_1)_b,\Util(f_2)_b; \pap (\Util(f_3)_b)}_{\htil} \\
& = \sqbrac{\Delta f_1, (f_2)_a;\pap(f_3)_a} + \sqbrac*[\big]{\Util(f_1)_b, \Delta f_2; \pap(f_3)_a}  \\
& \quad + \sqbrac*[\big]{\Util(f_1)_b, \Util(f_2)_b; \pap(\Delta f_3)} \\
& \quad + \cbrac{\sqbrac*[\big]{\Util(f_1)_b,\Util(f_2)_b; \pap\Util(f_3)_b } - \sqbrac*[\big]{\Util(f_1)_b,\Util(f_2)_b; \pap\Util(f_3)_b }_{\htil} }
\end{align*}
\end{proof}

The following lemma gives us control of some basic quantities required for the proof of \thmref{thm:aprioriEDelta}.

\begin{lem}\label{lem:quantM}
Assume the hypothesis of \thmref{thm:aprioriEDelta}. We will suppress the dependence of $L_1$ i.e. we write $a\lesssim b$ instead of $a\leq C(L_1)b$. Let $f\in \Scalsp(\Rsp)$. With this notation we have the following estimates for all  $t \in [0,T)$
\begin{enumerate}[leftmargin =*, align=left]
\item $\dis \norm*[\Linfty]{\htilap}(t), \norm[\Linfty]{\frac{1}{\htilap}}(t) \lesssim 1$ 
\item $\dis \abs*[\bigg]{\frac{\htil(\ap,t) - \htil(\bp,t)}{\ap-\bp}}, \abs*[\bigg]{\frac{\ap-\bp}{\htil(\ap,t) - \htil(\bp,t)}} \lesssim 1$ for all $\ap \neq \bp$
\item $\norm*[2]{\Util f} \lesssim \norm[2]{f}$ and $\norm*[\Hhalf]{\Util f} \lesssim \norm[\Hhalf]{f}$. These estimates are also true for the operator $\Util^{-1}$ instead of $\Util$. 
\item $\norm[2]{\Hcal(f)} \lesssim  \norm[2]{f}$, $\norm[\Hhalf]{\Hcal(f)} \lesssim \norm[\Hhalf]{f}$ and $\norm*[2]{\Htil(f)} \lesssim  \norm[2]{f}$ 
\item $\dis \norm*[\Hhalf]{\htilap}(t), \norm[\Hhalf]{\frac{1}{\htilap}}(t) \lesssim 1$ 
\item $\dis \norm[\infty]{\Zapabs_a\Util \brac{\frac{1}{\Zapabs}}_b}(t), \norm[\infty]{\frac{1}{\Zapabs_a}\Util \brac{\Zapabs}_b}(t) \lesssim 1 $ 
\item $\dis \norm*[2]{(\Dapabs)_a\htilap}(t) \lesssim 1$ 
\item $\dis \norm[2]{(\Dapabs)_a\cbrac{\Zapabs_a\Util \brac{\frac{1}{\Zapabs}}_b}}(t), \norm[2]{(\Dapabs)_a\cbrac{\frac{1}{\Zapabs_a}\Util \brac{\Zapabs}_b}}(t) \lesssim 1 $  
\end{enumerate}
\end{lem}

\begin{proof}
We will prove each of the estimates individually.
\begin{enumerate}[leftmargin =*, align=left]

\item From \eqref{eq:h} we know that $h_a(\ap,0) = h_b(\ap,0) = \ap$ at time $t=0$. Now
observe that 
\begin{align*}
U_{h}^{-1}\brac{\frac{h_{t \al}}{\hal}} = \bvarap \quad \tx{ and } \pt \hal = \brac{\frac{h_{t \al}}{\hal}}\hal \quad \tx{ and } \pt \frac{1}{\hal} = -\brac{\frac{h_{t \al}}{\hal}}\frac{1}{\hal}
\end{align*}
and as $\norm[\infty]{(\bvarap)_b} $ is controlled by $(\Ecalhigh)_b(t)$ and $\norm[\infty]{(\bvarap)_a} $ is controlled by $(\Ecalsigma)_a(t)$, we see that $(\hal)_i $ and $\brac{\frac{1}{\hal}}_{ i}$ remain bounded for both $i=a,b$.  Now 
\begin{align*}
(\Dt)_a \htilap = U_{h_a}^{-1}\brac{\pt \frac{(\hal)_b}{(\hal)_a}} = \htilap U_{h_a}^{-1}\cbrac{\brac{\frac{h_{t \al}}{\hal}}_b - \brac{\frac{h_{t \al}}{\hal}}_a} = \htilap\brac*{\Util(\bvarap)_b - (\bvarap)_a }
\end{align*}
Hence as $\htilap = 1$ at time $t=0$, we see that $\norm*[\infty]{\htilap}(t) \lesssim 1$. Similarly for $\dis\frac{1}{\htilap}$.  

\item This is an easy consequence of $\dis \norm*[\Linfty]{\htilap}(t), \norm[\Linfty]{\frac{1}{\htilap}}(t) \lesssim 1$ and the fact that $\htil(\cdot,t)$ is a homeomorphism. 

\item We see that
\begin{align*}
\norm*[2]{\Util f}^2 = \int \abs*[\big]{f(\htil(\ap))}^2 \diff \ap = \int \frac{\abs*{f(s)}^2}{\abs*[\big]{(\htilap\compose\htil^{-1})(s)}} \diff s \lesssim \norm[2]{f}^2
\end{align*}
Similarly we have that
\begin{align*}
\norm*[\Hhalf]{\Util f}^2 & = \frac{1}{2\pi} \int  \int \frac{\abs*{f(\htil(\ap)) - f(\htil(\bp))}^2}{(\ap - \bp)^2} \diff \ap \diff\bp \\
& = \frac{1}{2\pi} \int \int \brac{\frac{\abs{f(x) - f(y)}^2}{(\htil^{-1}(x) - \htil^{-1}(y))^2}} \frac{1}{\abs*[\big]{(\htilap\compose\htil^{-1})(x)(\htilap\compose\htil^{-1})(y)}} \diff x \diff y \\
& \lesssim   \frac{1}{2\pi} \int  \int \frac{\abs*{f(x) - f(y)}^2}{(\ap - \bp)^2} \diff x \diff y \\
& \lesssim \norm[\Hhalf]{f}^2
\end{align*}
In the same way we can prove the estimates for $\Util^{-1}$.

\item This follows directly from \propref{prop:HcalHtilest}. 

\item As $\norm[\Linfty\cap\Hhalf]{(\bvarap)_i}$ and  is controlled by $C(L_1)$ for $i=a,b$, using \lemref{lem:timederiv} and \propref{prop:Leibniz} we see that
\begin{align*}
\frac{d}{dt} \norm*[\Hhalf]{\htilap}^2 & \lesssim \norm*[\Hhalf]{\htilap}^2 + \norm*[\Hhalf]{\htilap}\norm*[\Hhalf]{(\Dt)_a \htilap} \\
&  \lesssim \norm*[\Hhalf]{\htilap}^2 + \norm*[\Hhalf]{\htilap}\norm*[\Hhalf]{\htilap\brac*{\Util(\bvarap)_b - (\bvarap)_a }} \\
&  \lesssim \norm*[\Hhalf]{\htilap}^2 + \norm*[\Hhalf]{\htilap}\brac*{\norm*[\infty]{\htilap} + \norm*[\Hhalf]{\htilap} } \\
&  \lesssim \brac*{\norm*[\Hhalf]{\htilap}^2 + 1}
\end{align*}
Now as $\norm*[\Hhalf]{\htilap}(0) = 0$, we obtain that $\norm*[\Hhalf]{\htilap}(t) \lesssim 1$ for $t \in [0,T)$.

\item Using \eqref{form:DtZapabs} we observe that
\begin{align*}
& (\Dt)_a\cbrac{\Zapabs_a\Util\brac{\frac{1}{\Zapabs}}_b } \\
& = \cbrac{\Zapabs_a\Util\brac{\frac{1}{\Zapabs}}_b }\Real\cbrac{(\Dap\Zt)_a - \Util(\Dap\Zt)_b - (\bvarap)_a + \Util(\bvarap)_b  }
\end{align*}
Hence we have that $\dis \norm[\infty]{\Zapabs_a\Util \brac{\frac{1}{\Zapabs}}_b}(t) \lesssim 1$ for all $t\in[0,T)$. The other estimate is proven similarly. 

\item We first observe that
\begin{align*}
 (\Dt)_a \Dapabs_a\htilap & = -\Real(\Dap\Zt)_a \Dapabs_a\htilap + \Dapabs_a(\Dt)_a\htilap \\
& = -\Real(\Dap\Zt)_a \Dapabs_a\htilap + \Dapabs_a\cbrac{\htilap\brac*{\Util(\bvarap)_b - (\bvarap)_a } } \\
& = -\Real(\Dap\Zt)_a \Dapabs_a\htilap + (\Dapabs_a\htilap)(\Util(\bvarap)_b - (\bvarap)_a)  \\
& \quad + \htilap\cbrac{\htilap\brac{\frac{1}{\Zapabs_a}\Util(\Zapabs)_b }\Util(\Dapabs\bvarap)_b - (\Dapabs\bvarap)_a }
\end{align*}
Hence using \lemref{lem:timederiv} we have
\begin{align*}
\frac{d}{dt}\norm[2]{\Dapabs_a\htilap}^2 & \lesssim \norm[2]{\Dapabs_a\htilap}^2 + \norm[2]{\Dapabs_a\htilap}\norm[2]{(\Dt)_a\Dapabs_a\htilap} \\
& \lesssim \cbrac{ \norm[2]{\Dapabs_a\htilap}^2 + 1}
\end{align*}
As $\Dapabs_a\htilap = 0$ at time $t=0$, we are done. 

\item Using \eqref{form:DtZapabs} we observe that 
\begin{align*}
& (\Dapabs)_a\cbrac{\Zapabs_a\Util \brac{\frac{1}{\Zapabs}}_b} \\
& = \cbrac{\Zapabs_a\Util \brac{\frac{1}{\Zapabs}}_b}\cbrac{-\brac{\pap\frac{1}{\Zapabs}}_a + \htilap\brac{\frac{1}{\Zapabs_a}\Util(\Zapabs)_b }\Util\brac{\pap\frac{1}{\Zapabs}}_b }
\end{align*}
Hence we see that $ \dis \norm[2]{(\Dapabs)_a\cbrac{\Zapabs_a\Util \brac{\frac{1}{\Zapabs}}_b}}(t) \lesssim 1$ for all $t\in [0,T)$. The other estimate is proven similarly.

\end{enumerate}
\end{proof}

\bigskip
 
\subsection{Quantities controlled by $\EDelta$}\label{sec:quantEDelta} 

In this section whenever we write $\dis f \in L^2_{\Delta^\alpha}$, what we mean is that there exists a constant $C(L_1)$ depending only on $L_1$ such that $\dis \norm*[2]{f} \leq C(L_1)(\EDelta)^\alpha $.  Similar definitions for $\dis f\in \Lone_{\Delta^\alpha}$, $\dis f\in \Hhalf_{\Delta^\alpha}$ and $\dis  f \in \Linfty_{\Delta^\alpha}$. We define the spaces $ \Ccal_{\Delta^\alpha}$ and $ \Wcal_{\Delta^\alpha}$ as follows
\begin{enumerate}
\item  If $\dis w \in \Linfty_{\Delta^\alpha}$ and $\dis \Dapabs_a w \in \Ltwo_{\Delta^\alpha}$, then we say $f\in \Wcal_{\Delta^\alpha}$. Define 
\[
 \norm[\Wcal_{\Delta^\alpha}]{w} = \norm[\Wcal]{w} = \norm[\infty]{w}+ \norm[2]{\Dapabs_a w}
 \]

\item If $\dis f \in \Hhalf_{\Delta^\alpha}$ and $\dis f\Zapabs_a \in \Ltwo_{\Delta^\alpha}$, then we say $f\in \mathcal{C}_{\Delta^\alpha}$. Define
\[
\norm[\Ccal_{\Delta^\alpha}]{f} = \norm[\Ccal]{f} = \norm[\Hhalf]{f} + \brac*[\bigg]{1+ \norm*[\bigg][2]{\brac{\pap\frac{1}{\Zapabs}}\la}}\norm[2]{f\Zapabs_a}
 \]
\end{enumerate}
Now analogous to \lemref{lem:CW} we have the following lemma
\begin{lem} \label{lem:CWDelta} 
Let $\alpha_1, \alpha_2,\alpha_3 \geq0$ with $\alpha_1 + \alpha_2 = \alpha_3$. Then the following properties hold for the spaces $\Wcal_{\Delta^\alpha}$ and $\Ccal_{\Delta^\alpha}$
\begin{enumerate}[leftmargin =*, align=left]
\item If $w_1 \in \Wcal_{\Delta^{\alpha_1}}$, $w_2 \in \Wcal_{\Delta^{\alpha_2}}$, then $w_1w_2 \in \Wcal_{\Delta^{\alpha_3}}$. Moreover we have the estimate $\norm[\Wcal_{\Delta^{\alpha_{3}}}]{w_1w_2} \lesssim \norm[\Wcal_{\Delta^{\alpha_{1}}}]{w_1}\norm[\Wcal_{\Delta^{\alpha_{2}}}]{w_2}$
\item If $f\in \Ccal_{\Delta^{\alpha_1}}$ and $w\in \Wcal_{\Delta^{\alpha_2}}$, then $fw \in \Ccal_{\Delta^{\alpha_3}}$.  Moreover $\norm[\Ccal_{\Delta^{\alpha_3}}]{fw} \lesssim \norm[\Ccal_{\Delta^{\alpha_1}}]{f}\norm[\Wcal_{\Delta^{\alpha_2}}]{w}$
\item If $f \in \Ccal_{\Delta^{\alpha_1}}$, $g \in \Ccal_{\Delta^{\alpha_2}}$, then $fg\Zapabs_a \in \Ltwo_{\Delta^{\alpha_3}}$. Moreover $\norm[2]{fg\Zapabs_a} \lesssim \norm[\Ccal_{\Delta^{\alpha_1}}]{f}\norm[\Ccal_{\Delta^{\alpha_2}}]{g}$
\end{enumerate}
\end{lem}

When we write $\dis f \in \Ltwo$ we mean $f \in \Ltwo_{\Delta^{\alpha}}$ with $\alpha =0$. Similar notation for $\Hhalf, \Linfty, \Ccal$ and $\Wcal$.  It is important to note that if $f \in \Ltwo$ and $f\in \Ltwo_{\Delta^{\alpha}}$,  then we have $f\in \Ltwo_{\Delta^{\beta}}$ for all $0<\beta<\alpha$.  We say that $a \approx_{\Ltwo_{\Delta^{\alpha}}} b$ if $a-b \in \Ltwo_{\Delta^{\alpha}}$. It should be noted that $\approx_{\Ltwo_{\Delta^{\alpha}}}$ is an equivalence relation. Similar definitions for $\approx_{\Lone_{\Delta^{\alpha}}}, \approx_{\Linfty_{\Delta^{\alpha}}}$, $\dis\approx_{\Hhalf_{\Delta^{\alpha}}} $, $ \approx_{\Wcal_{\Delta^\alpha}} $ and $\approx_{\mathcal{C}_{\Delta^\alpha}}$.

In this section, we will need to commute weights and derivatives with the operators $\Util, \Delta$ quite frequently and hence the following lemma will be very frequently used.
\begin{lem}\label{lem:Deltaapprox}
Let $f,g \in \Scalsp(\Rsp)$ and let $\alpha \in \Rsp$. Then 
\begin{enumerate}[leftmargin =*, align=left]

\item If $g_a\Util\brac{\pap f}_b \in \Ltwo$ then 
\begin{enumerate}
\item $g_a\pap\Util ( f)_b \in \Ltwo$
\item $g_a\Util (\pap f)_b \approxLtwoDeltahalf g_a\pap\Util ( f)_b$
\item $g_a\Delta(\pap f) \approxLtwoDeltahalf g_a\pap\Delta(f)$
\end{enumerate}
These estimates are also true if we replace $(\Ltwo, \LtwoDeltahalf)$ with $(\Linfty, \LinftyDeltahalf)$, $(\Linfty\cap\Hhalf, \LinftyDeltahalf\cap\HhalfDeltahalf)$, $(\Wcal, \WcalDeltahalf)$ or $(\Ccal, \CcalDeltahalf)$.

\item If $g_a \Util(f)_b \in \Ltwo$ then  
\begin{enumerate}
\item $(g\Zapabs^{\alpha})_a\Util\brac{\Zapabs^{-\alpha} f}_b \in \Ltwo$
\item $g_a \Util(f)_b \approxLtwoDeltahalf (g\Zapabs^{\alpha})_a\Util\brac{\Zapabs^{-\alpha} f}_b  $
\item $g_a\Delta(f) \approxLtwoDeltahalf (g\Zapabs^{\alpha})_a\Delta(\Zapabs^{-\alpha} f)$
\end{enumerate}
These estimates are also true if we replace $(\Ltwo, \LtwoDeltahalf)$ with $(\Linfty, \LinftyDeltahalf)$, $(\Wcal, \WcalDeltahalf)$ or $(\Ccal, \CcalDeltahalf)$.

\item If $g_a \Util(f)_b \in \Ltwo$ then  
\begin{enumerate}
\item $(g\w^{\alpha})_a\Util\brac{\w^{-\alpha} f}_b \in \Ltwo$
\item $g_a \Util(f)_b \approxLtwoDeltahalf (g\w^{\alpha})_a\Util\brac{\w^{-\alpha} f}_b  $
\item $g_a\Delta(f) \approxLtwoDeltahalf (g\w^{\alpha})_a\Delta(\w^{-\alpha} f)$
\end{enumerate} 
These estimates are also true if we replace $(\Ltwo, \LtwoDeltahalf)$ with $(\Linfty, \LinftyDeltahalf)$, $(\Wcal, \WcalDeltahalf)$ or $(\Ccal, \CcalDeltahalf)$.

\end{enumerate}
\end{lem}
\begin{proof}
We prove each of the statements individually.
\begin{enumerate}[leftmargin =*, align=left]

\item We first observe that the energy $\EDelta$ controls $(\htilap -1) \in \LinftyDeltahalf\cap\HhalfDeltahalf\cap\WcalDeltahalf$. Now notice that $\dis g_a\pap\Util ( f)_b = \htilap g_a\Util\brac{\pap f}_b$. As $\htilap \in \Linfty$ we see that $g_a\pap\Util ( f)_b \in \Ltwo$. Now we have
\begin{align*}
g_a\pap\Util ( f)_b - g_a\Util\brac{\pap f}_b = (\htilap -1)g_a\Util\brac{\pap f}_b
\end{align*}
and hence $\norm*[2]{g_a\pap\Util ( f)_b - g_a\Util\brac{\pap f}_b} \leq \norm*[\infty]{\htilap -1}\norm*[2]{g_a\Util\brac{\pap f}_b} \leq C(L_1)(\EDelta)^\half $. The other estimates are shown similarly using the fact that $\htilap \in \Linfty\cap\Hhalf\cap\Wcal$ and $(\htilap -1) \in \LinftyDeltahalf\cap\HhalfDeltahalf\cap\WcalDeltahalf$.

\item Observe that the energy $\EDelta$ controls $\Delta w \in \LinftyDeltahalf$ and $\dis \Delta\brac{\pap\frac{1}{\Zap}} \in \LtwoDeltahalf$. Hence by using \eqref{form:RealImagTh} we see that $\dis \Delta\brac{\pap\frac{1}{\Zapabs}} \in \LtwoDeltahalf$ and $\Delta(\Dapabs\w) \in \LtwoDeltahalf$.

Now observe that for $\alpha \in \Rsp$, we have $\abs{z^\alpha  -1} \leq C(\alpha)\abs{z-1}\max(\abs{z}^\alpha, 1)$ for $z \in \Csp$. Hence using the fact that $\dis \Zapabs_a\Util\brac{\frac{1}{\Zapabs}}\lb - 1 \in \LinftyDeltahalf$ we see that $\dis \brac{\Zapabs_a\Util\brac{\frac{1}{\Zapabs}}\lb}^\alpha - 1 \in \LinftyDeltahalf$. In particular we have $\dis \frac{1}{\Zapabs_a}\Util(\Zapabs)_b - 1 \in \LinftyDeltahalf$. Now we have
\begin{align*}
& (\Dapabs)_a\cbrac{\Zapabs_a\Util \brac{\frac{1}{\Zapabs}}_b} \\
& = \cbrac{\Zapabs_a\Util \brac{\frac{1}{\Zapabs}}_b}\cbrac{-\brac{\pap\frac{1}{\Zapabs}}_a + \htilap\brac{\frac{1}{\Zapabs_a}\Util(\Zapabs)_b }\Util\brac{\pap\frac{1}{\Zapabs}}_b }
\end{align*}
Hence we see that $\dis \Zapabs_a\Util\brac{\frac{1}{\Zapabs}}\lb - 1 \in \WcalDeltahalf$ or more generally $\dis \brac{\Zapabs_a\Util\brac{\frac{1}{\Zapabs}}\lb}^\alpha - 1 \in \WcalDeltahalf$. Now coming back we see that
\begin{align*}
(g\Zapabs^{\alpha})_a\Util\brac{\Zapabs^{-\alpha} f}_b = \brac{\Zapabs_a\Util\brac{\frac{1}{\Zapabs}}\lb}^\alpha g_a\Util(f)_b
\end{align*}
Now as $\dis \Zapabs_a\Util\brac{\frac{1}{\Zapabs}}\lb \in \Linfty$, we see that $(g\Zapabs^{\alpha})_a\Util\brac{\Zapabs^{-\alpha} f}_b \in \Ltwo$. Now
\begin{align*}
(g\Zapabs^{\alpha})_a\Util\brac{\Zapabs^{-\alpha} f}_b - g_a\Util(f)_b =  \cbrac{ \brac{\Zapabs_a\Util\brac{\frac{1}{\Zapabs}}\lb}^\alpha -1 }g_a\Util(f)_b
\end{align*}
Hence we have $\norm[2]{(g\Zapabs^{\alpha})_a\Util\brac{\Zapabs^{-\alpha} f}_b - g_a\Util(f)_b} \leq C(L_1)(\EDelta)^\half$. The other estimates are proven similarly using $\dis \Zapabs_a\Util\brac{\frac{1}{\Zapabs}}\lb \in \Linfty\cap\Wcal$ and $\dis \cbrac{\brac{\Zapabs_a\Util\brac{\frac{1}{\Zapabs}}\lb}^\alpha - 1} \in \LinftyDeltahalf\cap\WcalDeltahalf$. 

\item This is proved exactly the same as above. Here we use the estimate $\Delta(w) \in \LinftyDeltahalf$ and $\abs{w} = 1$ to see that $\w_a\Util(\w^{-1})_b - 1 \in \LinftyDeltahalf$ or more generally $\w_a^\alpha\Util(\w^{-\alpha})_b - 1 \in \LinftyDeltahalf$. Now we observe that
\begin{align*}
&\Dapabs_a \brac{\w_a\Util(\w^{-1})_b} \\
& = \brac{\w_a\Util(\w^{-1})_b}\cbrac{\wbar_a\Dapabs_a\w_a - \frac{\htil_\al}{\Zapabs_a}\Util(\Zapabs)_b\Util(\wbar)_b\Util(\Dapabs\w)_b}
\end{align*}
Now as the energy $\EDelta$ controls $\dis \Delta(\Dapabs\w) \in \LtwoDeltahalf$, we see that $\w_a\Util(\w^{-1})_b - 1 \in \WcalDeltahalf$ or more generally $\w_a^\alpha\Util(\w^{-\alpha})_b - 1 \in \WcalDeltahalf$. Now using 
\begin{align*}
  (g\w^{\alpha})_a\Util\brac{\w^{-\alpha} f}_b - g_a \Util(f)_b = g_a \Util(f)_b\cbrac{ \w_a^\alpha\Util(\w^{-\alpha})_b - 1}
\end{align*}
we easily reach the desired conclusion. 
\end{enumerate}
\end{proof}

One important conclusion from the proof of the previous lemma is that for any $\alpha \in \Rsp$ we have  $\dis \cbrac{\Zapabs_a\Util\brac{\frac{1}{\Zapabs}}\lb}^\alpha - 1 \in \WcalDeltahalf$ and $\w_a^\alpha\Util(\w^{-\alpha})_b - 1 \in \WcalDeltahalf$. This estimate along with the fact that $(\htilap -1) \in \WcalDeltahalf\cap\HhalfDeltahalf$ will be heavily used in the proof of the energy estimate. 
 
Let us now control the main terms controlled by $\EDelta$. We will heavily use the lemmas proven earlier namely \lemref{lem:basicUtil}, \lemref{lem:Delta}, \lemref{lem:quantM}, \lemref{lem:CWDelta} and \lemref{lem:Deltaapprox} along with \propref{prop:commutator} and \propref{prop:HilHtilcaldiff} from the appendix. As we will be using \lemref{lem:basicUtil}, \lemref{lem:Delta} and \lemref{lem:quantM} in almost every step, we will skip referencing them. The proof of the estimates in this section follow exactly analogous to the proof of the control of the quantities controlled by $\Esigma$ from  Sec 5.1 of \cite{Ag19}, and in particular we use the identities established there and then subtract the terms of the two solutions.  As the proofs are straighforward modifications of the proofs in Sec 5.1 of \cite{Ag19}, we will just control a few terms and show how it is done and the rest are proved analogously.  A few terms require a bit more work and we give more details for those terms. The numbering system employed here for the quantities controlled is the same as used in Sec 5.1 of \cite{Ag19}. 
 
\bigskip 
\begin{enumerate}[leftmargin =*, align=left]

\item $\Delta(\Ztapbar) \in \LtwoDeltahalf, \pap\Delta(\Ztbar) \in \LtwoDeltahalf$ and $\Dapabs_a\Delta(\Dapbar\Ztbar) \in \LtwoDeltahalf, \Delta(\Dapabs\Dapbar\Ztbar) \in \LtwoDeltahalf$
\medskip\\
Proof: As $(\Aone)_a \geq 1$, we see that $\EDelta $ controls $\Delta(\Ztapbar) \in \LtwoDeltahalf$ and $\Dapabs_a\Delta(\Dapbar\Ztbar) \in \LtwoDeltahalf$. We obtain the other two estimates by using \lemref{lem:Deltaapprox}. 
\bigskip

\item $\dis \Delta(\Aone) \in \LinftyDeltahalf \cap \HhalfDeltahalf$ 
\medskip \\
Proof: Recall from \eqref{eq:systemone} that $\displaystyle \Aone = 1 - \Imag[\Zt,\Hil]\Ztapbar$ and hence we have
\begin{align*}
& \Delta(\Aone)  \\
& =  -\Imag \cbrac{\Delta \sqbrac{\Zt,\Hil}\pap \Ztbar} \\
& = -\Imag\cbrac{\sqbrac{\Delta \Zt, \Hil}\pap(\Ztbar)_a + \sqbrac*[\big]{\Util(\Zt)_b, \Hil - \Htil}\pap(\Ztbar)_a + \Util\cbrac*[\Big]{\sqbrac{(\Zt)_b,\Hil}\pap\brac*[\big]{\Util^{-1}(\Delta \Ztbar)} }}
\end{align*}
Hence using \propref{prop:commutator}, \propref{prop:HilHtilcaldiff} and \lemref{lem:Deltaapprox} we get 
\begin{align*}
\norm[\Linfty\cap\Hhalf]{\Delta(\Aone)} \leq C(L_1) (\EDelta)^\half
\end{align*}
\bigskip

\item $\dis \Delta\brac{\pap\frac{1}{\Zap}} \in \LtwoDeltahalf, \Delta\brac{\pap\frac{1}{\Zapabs}} \in \LtwoDeltahalf, \Delta(\Dapabs \w) \in \LtwoDeltahalf$ and hence $\Delta(\w) \in \WcalDeltahalf$
\medskip\\
Proof: Observe that $\dis \Delta\brac{\pap\frac{1}{\Zap}} \in \LtwoDeltahalf$ as it is part of the energy $\EDeltazero$. Recall from \eqref{form:RealImagTh} that
\[
\Real\brac{\frac{\Zap}{\Zapabs}\pap\frac{1}{\Zap} } = \pap \frac{1}{\Zapabs} \quad \qquad \Imag\brac{\frac{\Zap}{\Zapabs}\pap\frac{1}{\Zap}} =  i\brac{ \wbar\Dapabs \w} 
\]
Using $\Delta(w) \in \LinftyDeltahalf$ we obtain $\dis  \Delta\brac{\pap\frac{1}{\Zapabs}} \in \LtwoDeltahalf$ and $ \Delta(\Dapabs \w) \in \LtwoDeltahalf$. Now using \lemref{lem:Deltaapprox} we obtain $\dis \Dapabs_a \Delta(w) \in \LtwoDeltahalf$ and hence  $\Delta(\w) \in \WcalDeltahalf$. 
\bigskip

\item $\dis \Delta(\Dapbar\Ztbar) \in \LinftyDeltahalf, \Delta(\Dapabs\Ztbar) \in \LinftyDeltahalf $ and $\Delta(\Dap\Ztbar) \in \LinftyDeltahalf $ 
\medskip\\
Proof: First observe that from \lemref{lem:Deltaapprox} we have $\Zapabs_a\Delta(\Dapbar\Ztbar) \approxLtwoDeltahalf \Delta(\w\Ztapbar) \in \LtwoDeltahalf$. Hence we have
\begin{align*}
\pap(\Delta(\Dapbar\Ztbar))^2 = 2 \cbrac{\Zapabs_a\Delta(\Dapbar\Ztbar)}\Dapabs_a\Delta(\Dapbar\Ztbar) \in \LoneDelta
\end{align*}
Hence $\dis \Delta(\Dapbar\Ztbar) \in \LinftyDeltahalf$. The other ones are obtained easily by using   \lemref{lem:Deltaapprox}. 
\bigskip

\item $\dis \Delta(\Dapbar^2\Ztbar) \in \LtwoDeltahalf, \Delta(\Dapabs^2\Ztbar) \in \LtwoDeltahalf, \Delta(\Dap^2\Ztbar) \in \LtwoDeltahalf$ and $\dis \Delta\brac*[\bigg]{\frac{1}{\Zap^2}\pap\Ztapbar} \in  \LtwoDeltahalf$
\medskip\\
Proof: We already know that $\Delta (\Dapabs\Dapbar\Ztbar) \in \LtwoDeltahalf$ and hence using \lemref{lem:Deltaapprox} we have $\Delta (\Dapbar^2\Ztbar) \in \LtwoDeltahalf$. Now 
\begin{align*}
\Dapbar^2\Ztbar = \Dapbar\brac{\w\Dapabs\Ztbar} = (\Dapbar\w)\Dapabs\Ztbar + \w^2\Dapabs^2\Ztbar
\end{align*}
Applying $\Delta$ to the above equation we easily get $\dis \Delta(\Dapabs^2\Ztbar) \in \LtwoDeltahalf$. The estimate $\dis \Delta(\Dap^2\Ztbar) \in \LtwoDeltahalf$ and $\dis \Delta\brac*[\bigg]{\frac{1}{\Zap^2}\pap\Ztapbar} \in  \LtwoDeltahalf$ is proven similarly.
\bigskip

\item $\dis \Delta(\Dapbar\Ztbar) \in \WcalDeltahalf\cap\CcalDeltahalf, \Delta(\Dapabs\Ztbar) \in \WcalDeltahalf\cap\CcalDeltahalf, \Delta(\Dap\Ztbar) \in \WcalDeltahalf\cap\CcalDeltahalf $
\medskip\\
Proof: As $ \Delta(\Dapbar\Ztbar) \in \LinftyDeltahalf$ and $\Dapabs_a \Delta(\Dapbar\Ztbar) \in \LtwoDeltahalf$ we see that $ \Delta(\Dapbar\Ztbar) \in \WcalDeltahalf$. Now $ \Zapabs_a\Delta(\Dapbar\Ztbar) \in \LtwoDeltahalf$ and using  \propref{prop:LinftyHhalf} with $f =  \Delta(\Dapbar\Ztbar)$ and $\dis  w = \frac{1}{\Zapabs_a}$ we see that
\begin{align*}
\norm[\Hhalf]{ \Delta(\Dapbar\Ztbar)}^2 & \lesssim \norm[2]{ \Zapabs_a\Delta(\Dapbar\Ztbar)}\norm[2]{\pap\brac{\frac{1}{\Zapabs_a}\Delta(\Dapbar\Ztbar) }} \\
& \quad + \norm[2]{ \Zapabs_a\Delta(\Dapbar\Ztbar)}^2\norm[2]{\pap\frac{1}{\Zapabs_a}}^2
\end{align*}
From this we obtain $\Delta(\Dapbar\Ztbar) \in \WcalDeltahalf\cap\CcalDeltahalf$. As $\wbar \in \Wcal$, using \lemref{lem:Deltaapprox} and \lemref{lem:CWDelta} we see that $\Delta(\Dapabs\Ztbar) \in \WcalDeltahalf\cap\CcalDeltahalf$. The other one is proven similarly. 
\bigskip

\item $\dis \Delta \cbrac{\pap \Pa\brac{\frac{\Zt}{\Zap}}} \in \LinftyDeltahalf$
\medskip\\
Proof: In Sec 5.1 of \cite{Ag19} we have shown the formula 
\begin{align*}
2\pap \Pa \brac*[\Big]{\frac{\Zt}{\Zap}}  = 2\Dap\Zt + \sqbrac{\frac{1}{\Zap},\Hil}\Ztap + \sqbrac{\Zt, \Hil}\pap\frac{1}{\Zap} 
\end{align*}
Now applying $\Delta$ to the formula above, the estimate follows easily from \propref{prop:commutator}, \propref{prop:HilHtilcaldiff} and \lemref{lem:Deltaapprox}.
\bigskip

\item $\dis \Delta(\Dapabs \Aone) \in \LtwoDeltahalf$ and hence $\dis \Delta(\Aone) \in \WcalDeltahalf, \Delta(\sqrt{\Aone}) \in \WcalDeltahalf, \Delta\brac{\frac{1}{\Aone}} \in \WcalDeltahalf$ and also $\dis  \Delta\brac{\frac{1}{\sqrt{\Aone}}} \in \WcalDeltahalf$
\medskip\\
Proof: In Sec 5.1 of \cite{Ag19} we established the formula
\begin{align*}
(\Id - \Hil)\Dap\Aone = i(\Id - \Hil)\brac{(\Dap\Zt)\Ztapbar}+ i\sqbrac{\Pa\brac{\frac{\Zt}{\Zap}},\Hil}\pap\Ztapbar
\end{align*}
Now by applying $\Delta$ to the formula above and using \propref{prop:commutator}, \propref{prop:HilHtilcaldiff} and \lemref{lem:Deltaapprox}, we see that $(\Id - \Hil)\Dap\Aone \in \LtwoDeltahalf$. Now in Sec 5.1 of \cite{Ag19} we also established the formula
\begin{align*}
\Dapabs\Aone = \Real \cbrac{\w(\Id - \Hil)\Dap\Aone} -\Real\cbrac{\w\sqbrac{\frac{1}{\Zap},\Hil}\pap\Aone} 
\end{align*}
Hence by applying $\Delta$ to the formula above and using \propref{prop:commutator}, \propref{prop:HilHtilcaldiff} and \lemref{lem:Deltaapprox}, we easily get $\dis \Delta(\Dapabs \Aone) \in \LtwoDeltahalf$. Hence using \lemref{lem:Deltaapprox} we see that $\dis \Delta(\Aone) \in \WcalDeltahalf$. Now as $(\Aone)_a \geq 1$ and $(\Aone)_a \in \Wcal$, we get from \lemref{lem:CWDelta} that $\dis \frac{1}{(\Aone)_a}\Util(\Aone)_b - 1 \in \WcalDeltahalf$.  The inequality $\abs{z^\alpha  -1} \leq C(\alpha)\abs{z-1}\max(\abs{z}^\alpha, 1)$ for $z \in \Csp$, $\alpha \in \Rsp$ implies that we have $\dis \brac{\frac{1}{(\Aone)_a}\Util(\Aone)_b}^\alpha - 1 \in \WcalDeltahalf$. Choosing suitable values of $\alpha$ imply all the other estimates. 
\bigskip

\item $\Delta(\Th) \in \LtwoDeltahalf$ and $\Delta(\Dt\Th) \in \LtwoDeltahalf$
\medskip\\
Proof: The proof of $\Delta(\Th) \in \LtwoDeltahalf$ follows easily by applying $\Delta$ to the formula \eqref{form:Th} and using \propref{prop:HilHtilcaldiff}. Also $\Delta(\Dt\Th) \in \LtwoDeltahalf$ as it part of $\EDelta$. 
\bigskip

\item $\dis \frac{1}{\Zapabs_a}\Delta(\Th) \in \CcalDeltahalf, \Delta\brac{\frac{\Th}{\Zapabs}} \in \CcalDeltahalf$
\medskip\\
Proof: The energy $\EDelta$ controls $\dis \frac{\sqrt{(\Aone)_a}}{\Zapabs_a}\Delta(\Th) \in \HhalfDeltahalf$. Now as  $\dis \sqrt{(\Aone)_a} \Delta(\Th) \in \LtwoDeltahalf$ this implies that $\dis \frac{\sqrt{(\Aone)_a}}{\Zapabs_a}\Delta(\Th) \in \CcalDeltahalf$. Hence by using \lemref{lem:CWDelta} and observing that $\dis \frac{1}{\sqrt{(\Aone)_a}} \in \Wcal$ we obtain $\dis \frac{1}{\Zapabs_a}\Delta(\Th) \in \CcalDeltahalf$. The other estimate is now obtained by using \lemref{lem:Deltaapprox}.
\Bigskip

From now on we will just state the estimates and the proof follows exactly as in Sec 5.1 of \cite{Ag19} and can be easily obtained by applying $\Delta$ to the formulae there and using \lemref{lem:CWDelta}, \lemref{lem:Deltaapprox}, \propref{prop:commutator} and \propref{prop:HilHtilcaldiff} as shown in the above examples. For estimates which do not follow this pattern we give more details. 
\Bigskip

\item $\dis \Delta\brac{\Dap\frac{1}{\Zap}} \in \CcalDeltahalf$, $\dis \Delta\brac{\Dapabs\frac{1}{\Zap}} \in \CcalDeltahalf$, $\dis \Delta\brac{\Dapabs\frac{1}{\Zapabs}} \in \CcalDeltahalf$ and similarly we have $\dis   \Delta\brac{\frac{1}{\Zapabs^2}\pap\w} \in \CcalDeltahalf$
\smallskip

\item $\dis \Delta\brac{\frac{1}{\Zapabs^2}\pap \Aone} \in \LinftyDeltahalf\cap\HhalfDeltahalf  $ and hence $\dis  \Delta\brac{\frac{1}{\Zapabs^2}\pap \Aone} \in \CcalDeltahalf$
\smallskip

\item $\dis \Delta\brac{\frac{1}{\Zapabs^3}\pap^2\Aone} \in \LtwoDeltahalf $, $\dis \Dapabs\Delta\brac*[\bigg]{\frac{1}{\Zapabs^2}\pap\Aone} \in \LtwoDeltahalf $ and $\dis \Delta\brac{\frac{1}{\Zapabs^2}\pap \Aone} \in \WcalDeltahalf$
\smallskip

\item $\dis \Delta(\bvarap) \in \LinftyDeltahalf\cap\HhalfDeltahalf $ and $ \Delta(\Hil(\bvarap)) \in \LinftyDeltahalf\cap\HhalfDeltahalf $
\smallskip

\item $\dis \Delta(\Dapabs\bvarap) \in \LtwoDeltahalf$ and hence $\dis \Delta(\bvarap) \in \WcalDeltahalf$
\smallskip

\item $\dis \Delta\cbrac{\pap\Dt\frac{1}{\Zap}} \in \LtwoDeltahalf$, $\dis \Delta\cbrac{\Dt\pap\frac{1}{\Zap}} \in \LtwoDeltahalf$
\smallskip

\item $\dis \Delta(\Zttapbar) \in \LtwoDeltahalf$
\smallskip

\item $\Delta(\Dapbar\Zttbar) \in \CcalDeltahalf$, $\Delta(\Dapabs\Zttbar) \in \CcalDeltahalf$, $\Delta(\Dt\Dapbar\Ztbar) \in \CcalDeltahalf$ and $\Delta(\Dt\Dapabs\Ztbar) \in \CcalDeltahalf$
\smallskip

\item $\dis \Delta(\Dt\Aone) \in \LinftyDeltahalf\cap\HhalfDeltahalf$
\smallskip

\item $\Delta(\Dt(\bvarap - \Dap\Zt - \Dapbar\Ztbar)) \in \LinftyDeltahalf\cap\HhalfDeltahalf$ and hence $\dis \Delta(\Dt\bvarap) \in \HhalfDeltahalf, \Delta(\pap\Dt \bvar) \in \HhalfDeltahalf $
\smallskip

Now we start controlling terms with surface tension. Note that these estimates are only for the solution A and hence the estimates have already been shown in Sec 5.1 in \cite{Ag19} and no new work has to be done at all. For most of the estimates we will have that the power of $\sigma$ will be the same as that of the power of $\Delta$. For e.g. we have $\dis \brac{\sigma^{\onebysix}  \Zapabs^\half \pap\frac{1}{\Zap}}\la \in \Ltwo_{\Delta^\onebysix}$ and both $\sigma$ and $\Delta$ are raised to the same power $1/6$. However the estimates  derived from $\sigma\pap\Th, \sigma\pap\Dap\Th$ and $\sigma\Dap^2\Th$ will not follow this pattern. For e.g. we will have $\dis (\sigma^\onebythree\Th) \in \Linfty_{\Delta^\onebysix}$ and not $\dis (\sigma^\onebythree\Th) \in \Linfty_{\Delta^\onebythree}$. The reason is that $\EDelta$ controls $\Delta\brac{(\Zttbar -i)\Zap} \in \HhalfDeltahalf$ not $\Hhalf_{\Delta}$ and hence we have $(\sigma\pap\Th) \in \HhalfDeltahalf$. Similarly for $\sigma\pap\Dap\Th$ and $\sigma\Dap^2\Th$. 

\medskip
\item $\dis \brac{\sigma^{\half}  \Zapabs^\half \pap\frac{1}{\Zap}}\la \in \LinftyDeltahalf$,  $\dis \brac{\sigma^{\half}  \Zapabs^\half \pap\frac{1}{\Zapabs}}\la \in \LinftyDeltahalf$, $\dis \brac{\frac{ \sigma^\half}{\Zapabs^\half}\pap\w}\la \in \LinftyDeltahalf$ and  $\dis \brac{\sigma^\half\Zapabs^\half\Real\Th}\la \in \LinftyDeltahalf$
\smallskip

\item $\dis \brac{\sigma^{\onebysix}  \Zapabs^\half \pap\frac{1}{\Zap}}\la \in \Ltwo_{\Delta^\onebysix}$,  $\dis \brac{\sigma^{\onebysix}  \Zapabs^\half \pap\frac{1}{\Zapabs}}\la \in \Ltwo_{\Delta^\onebysix}$, $\dis \brac{\frac{ \sigma^\onebysix}{\Zapabs^\half}\pap\w}\la \in \Ltwo_{\Delta^\onebysix}$
\smallskip

\item $\dis \brac{\sigma \pap\Th}_a \in \HhalfDeltahalf$
\medskip\\
Proof: Note carefully that we claim the estimate $\dis \brac{\sigma \pap\Th}_a \in \HhalfDeltahalf$ and not $\dis \brac{\sigma \pap\Th}_a \in \Hhalf_{\Delta}$. From the fundamental equation \eqref{form:Zttbar} we have
\[
\Delta\brac{(\Zttbar -i)\Zap} = -i \Delta(\Aone) + \brac{\sigma \pap \Th}_a
\]
and we know that $\EDelta$ controls $\dis \Delta\brac{(\Zttbar -i)\Zap} \in \HhalfDeltahalf$. As $\Delta(\Aone) \in \HhalfDeltahalf$, we obtain the above estimate. 
\bigskip

\item $\dis  \brac{\sigma^\twobythree\pap\Th}\la \in \Ltwo_{\Delta^\onebythree}$
\smallskip

\item $\dis \brac{\sigma^\twobythree \nobrac{\pap^2\frac{1}{\Zap}}}\la \in \Ltwo_{\Delta^\onebythree}$, $\dis \brac{\sigma^\twobythree \nobrac{\pap^2\frac{1}{\Zapabs}}}\la \in \Ltwo_{\Delta^\onebythree}$, $\dis \brac{\frac{\sigma^\twobythree}{\Zapabs}\pap^2\w}\la \in \Ltwo_{\Delta^\onebythree}$ and similarly we have $\dis \brac{\sigma^\twobythree\pap\Dapabs\w}\la \in \Ltwo_{\Delta^\onebythree}$
\medskip\\
Proof: Note carefully that we have $\dis \Ltwo_{\Delta^\onebythree}$ and not $\dis \Ltwo_{\Delta^\twobythree}$ in the above estimate. First see that $\dis \brac{\sigma^\half\Zapabs^\half\pap\frac{1}{\Zap}}\la \in \Ltwo$ as  $(\Esigma)_a (t) \leq C(L_1) $ and $(\Esigma)_a$ controls it. But we have already shown above that $\dis \brac{\sigma^\half\Zapabs^\half\pap\frac{1}{\Zap}}\la \in \LtwoDeltahalf$. Hence we have $\dis \brac{\sigma^\half\Zapabs^\half\pap\frac{1}{\Zap}}\la \in \Ltwo_{\Delta^{\beta}}$ for all $0\leq\beta\leq 1/2$. In a similar way we can show that $\dis \brac{\sigma^\half\Zapabs^\half\pap\frac{1}{\Zapabs}}\la \in \Ltwo_{\Delta^{\beta}}$ etc.  for all $0\leq\beta\leq 1/2$. Due to this argument, the proof of the above estimates follow in the same way as is shown in Sec 5.1 in \cite{Ag19}.
\bigskip

\item $\dis \brac*{\sigma^\onebythree \Th}_a \in \Linfty_{\Delta^\onebysix}\cap \Hhalf_{\Delta^\onebysix}$
\smallskip

\item $\dis \brac{\sigma^\onebythree \pap \frac{1}{\Zap}}\la \in \Linfty_{\Delta^\onebysix}\cap \Hhalf_{\Delta^\onebysix}$,  $\dis \brac{\sigma^\onebythree \pap \frac{1}{\Zapabs}}\la \in \Linfty_{\Delta^\onebysix}\cap \Hhalf_{\Delta^\onebysix}$, $\dis \brac{\sigma^\onebythree \Dapabs\w }\la\in \Linfty_{\Delta^\onebysix}\cap \Hhalf_{\Delta^\onebysix}$
\smallskip

\item $ \brac{\sigma \pap\Dap\Th}_a \in \LtwoDeltahalf$, $ \brac{\sigma \Dapabs\pap\Th}_a \in \LtwoDeltahalf$, $ \brac{\sigma \pap\Dapabs\Th}_a \in \LtwoDeltahalf$
\smallskip

\item $\dis \brac{\frac{\sigma}{\Zapabs}\pap^3\frac{1}{\Zap}}\la \in \LtwoDeltahalf$, $\dis \brac{\frac{\sigma}{\Zapabs}\pap^3\frac{1}{\Zapabs}}\la \in \LtwoDeltahalf$, $\dis \brac{\frac{\sigma}{\Zapabs^2}\pap^3\w}\la \in \LtwoDeltahalf$
\smallskip

\item $\dis \brac{\frac{\sigma^\half}{\Zapabs^\half}\pap^2\frac{1}{\Zap}}\la \in \LtwoDeltahalf$, $\dis \brac{\frac{\sigma^\half}{\Zapabs^\half}\pap^2\frac{1}{\Zapabs}}\la \in \LtwoDeltahalf$, $\dis \brac{\frac{\sigma^\half}{\Zapabs^\threebytwo}\pap^2\w}\la \in \LtwoDeltahalf$ and also $\dis \brac{\frac{\sigma^\half}{\Zapabs^\half}\pap\Th}\la \in \LtwoDeltahalf $
\smallskip

\item $\dis \brac{\sigma^{\half}  \Zapabs^\half \pap\frac{1}{\Zap}}\la \in \WcalDeltahalf$,  $\dis \brac{\sigma^{\half}  \Zapabs^\half \pap\frac{1}{\Zapabs}}\la \in \WcalDeltahalf$, $\dis \brac{\frac{ \sigma^\half}{\Zapabs^\half}\pap\w}\la \in \WcalDeltahalf$ and   $\dis \brac{\sigma^\half\Zapabs^\half\Real\Th}_a \in \WcalDeltahalf$
\smallskip

\item $\dis \brac{\frac{\sigma^\fivebysix}{\Zapabs^\half}\pap\Th}\la \in \Linfty_{\Delta^{\frac{5}{12}}}\cap\Hhalf_{\Delta^{\frac{5}{12}}}$
\smallskip

\item $\dis \brac{\frac{\sigma^\fivebysix}{\Zapabs^\half}\pap^2\frac{1}{\Zap}}\la \in \Linfty_{\Delta^{\frac{5}{12}}}\cap\Hhalf_{\Delta^{\frac{5}{12}}}$, $\dis \brac{\frac{\sigma^\fivebysix}{\Zapabs^\half}\pap^2\frac{1}{\Zapabs}}\la \in \Linfty_{\Delta^{\frac{5}{12}}}\cap\Hhalf_{\Delta^{\frac{5}{12}}}$ and similarly  $\dis \brac{\frac{\sigma^\fivebysix}{\Zapabs^\threebytwo}\pap^2\w}\la \in \Linfty_{\Delta^{\frac{5}{12}}}\cap\Hhalf_{\Delta^{\frac{5}{12}}}$
\smallskip

\item $\dis \brac{\frac{\sigma^\half}{\Zapabs^\threebytwo}\pap\Th}\la \in \CcalDeltahalf$
\smallskip

\item $\dis \brac{\frac{\sigma^\half}{\Zapabs^\threebytwo}\pap^2\frac{1}{\Zap}}\la \in \CcalDeltahalf$, $\dis \brac{\frac{\sigma^\half}{\Zapabs^\threebytwo}\pap^2\frac{1}{\Zapabs}}\la \in \CcalDeltahalf$, $\dis \brac{\frac{\sigma^\half}{\Zapabs^\fivebytwo}\pap^2\w}\la \in \CcalDeltahalf$
\smallskip

\item $\dis \brac{\sigma \Dapbar\Dap\Th}_a \in \CcalDeltahalf$, $\dis \brac{\sigma \Dap^2\Th}_a \in \CcalDeltahalf$, $\dis \brac{\sigma \Dapabs^2\Th}_a \in \CcalDeltahalf $, $\dis \brac{\frac{\sigma}{\Zapabs^2}\pap^2\Th}\la \in \CcalDeltahalf$
\smallskip

\item $\dis \brac{\frac{\sigma}{\Zapabs^2}\pap^3\frac{1}{\Zap}}\la \in \CcalDeltahalf$, $\dis \brac{\frac{\sigma}{\Zapabs^2}\pap^3\frac{1}{\Zapabs}}\la \in \CcalDeltahalf$, $\dis \brac{\sigma\pap^3\frac{1}{\Zapabs^3}}\la \in \CcalDeltahalf$ and similarly $\dis \brac{\frac{\sigma}{\Zapabs^3}\pap^3\w}\la \in \CcalDeltahalf$
\smallskip

\item $\dis \brac{\frac{\sigma^\half}{\Zapabs^\half}\pap\Ztapbar}\la \in \LtwoDeltahalf$, $\dis \brac{\sigma^\half\Zapabs^\half\pap\Dapabs\Ztbar}\la \in \LtwoDeltahalf$ and $\dis \brac{\sigma^\half\Zapabs^\half\pap\Dapbar\Ztbar}\la \in \LtwoDeltahalf$
\smallskip

\item $\dis \brac{\frac{\sigma^\half}{\Zapabs^\fivebytwo}\pap^2\Ztapbar}\la  \in \LtwoDeltahalf$, $\dis \brac{\frac{\sigma^\half}{\Zapabs^\threebytwo}\pap^2\Dapbar\Ztbar}\la \in \LtwoDeltahalf$, $\dis \brac{\frac{\sigma^\half}{\Zapabs^\half}\pap\Dapabs\Dapbar\Ztbar}\la \in \LtwoDeltahalf$ and also $\dis \brac{\frac{\sigma^\half}{\Zapabs^\half}\Dapabs^2\Ztapbar}\la \in \LtwoDeltahalf$, $\dis \brac{\frac{\sigma^\half}{\Zapabs^\half}\pap\Dapbar^2\Ztbar}\la \in \LtwoDeltahalf$
\smallskip

\item $\dis \brac{\frac{\sigma^\half}{\Zapabs^\threebytwo}\pap\Ztapbar}\la \in \WcalDeltahalf\cap\CcalDeltahalf$, $\dis \brac{\frac{\sigma^\half}{\Zapabs^\half}\pap\Dapabs\Ztbar}\la \in \WcalDeltahalf\cap\CcalDeltahalf$ and similarly we have  $\dis \brac{\frac{\sigma^\half}{\Zapabs^\half}\pap\Dapbar\Ztbar}\la \in \WcalDeltahalf\cap\CcalDeltahalf$, $\dis \brac{\frac{\sigma^\half}{\Zapabs^\half}\pap\Dap\Ztbar}\la \in \WcalDeltahalf\cap\CcalDeltahalf$
\smallskip

\item $\dis \brac{\frac{\sigma^\onebysix}{\Zapabs^\threebytwo}\pap\Ztapbar}\la \in \Ltwo_{\Delta^\onebysix}$, $\dis \brac{\frac{\sigma^\onebysix}{\Zapabs^\half}\pap\Dapabs\Ztbar}\la \in\Ltwo_{\Delta^\onebysix}$, $\dis \brac{\frac{\sigma^\onebysix}{\Zapabs^\half}\pap\Dapbar\Ztbar}\la \in \Ltwo_{\Delta^\onebysix}$
\smallskip

\item $\dis \brac{\frac{\sigma^\onebythree}{\Zapabs}\pap\Ztapbar}\la \in  \Ltwo_{\Delta^\onebythree}$, $\dis \brac{\sigma^\onebythree\pap\Dapabs\Ztbar}\la \in  \Ltwo_{\Delta^\onebythree}$, $\dis \brac{\sigma^\onebythree\pap\Dapbar\Ztbar}\la \in \Ltwo_{\Delta^\onebythree}$
\smallskip

\item $\dis \brac{\sigma^\onebysix \frac{\Ztapbar}{\Zapabs^\half}}\la \in \Wcal_{\Delta^\onebysix}$
\smallskip

\item $\dis \cbrac{\sigma^\onebysix\pap\Pa\brac{\frac{\Zt}{\Zaphalf}}}\la \in \Linfty_{\Delta^\onebysix} $
\smallskip

\item $\dis \brac{\frac{\sigma^\onebythree}{\Zapabs^2}\pap\Ztapbar}\la \in \Linfty_{\Delta^\onebythree}\cap\Hhalf_{\Delta^\onebythree}$
\smallskip

\item $\dis \brac*[\big]{\sigma^\onebythree\pap\bvarap}_a \in  \Ltwo_{\Delta^\onebythree}$
\smallskip

\item $\dis \brac{\frac{\sigma^\onebysix}{\Zapabs^\half}\pap\bvarap}\la \in \Ltwo_{\Delta^\onebysix}$ 
\smallskip

\item $\dis \brac{\frac{\sigma^\half}{\Zapabs^\half}\pap\bvarap}\la \in \LinftyDeltahalf$
\smallskip

\item $\dis \brac{\frac{\sigma^\half}{\Zapabs^\threebytwo}\pap^2\bvarap}\la \in \LtwoDeltahalf$, $\dis \brac{\frac{\sigma^\half}{\Zapabs^\half}\pap\Dapabs\bvarap}\la \in \LtwoDeltahalf$ and $\dis \brac{\frac{\sigma^\half}{\Zapabs^\half}\pap\Dap\bvarap}\la \in \LtwoDeltahalf$
\smallskip

\item $\dis \brac{\frac{\sigma^\half}{\Zapabs^\half}\pap\Aone}\la \in \LinftyDeltahalf$
\smallskip

\item $\dis \brac{\frac{\sigma^\half}{\Zapabs^\threebytwo}\pap^2\Aone}\la \in \LtwoDeltahalf$
\smallskip

\item $\dis \Delta\brac{(\Id - \Hil)\Dt^2\Th} \in \LtwoDeltahalf$, $\dis \Delta\brac{(\Id - \Hil)\Dt^2\Ztapbar} \in \LtwoDeltahalf$, $\dis \Delta\brac{(\Id - \Hil)\Dt^2 \Dap\Ztbar} \in \HhalfDeltahalf$
\smallskip\\
Proof: Observe that the terms are of the form $\Delta\brac{(\Id - \Hil)\Dt^2 f}$ with $f$ satisfying $\Pa f = 0$. Hence we can use \propref{prop:Dt^2hol} to get
\begin{align*}
\begin{aligned}[t]
 & (\Id - \Hil)\Dt^2 f \\
 & = 2\sqbrac{\Pa\brac{\frac{\Zt}{\Zap}},\Hil}\pap(\Dt f) - \sqbrac{\Pa\brac{\frac{\Zt}{\Zap}}, \Pa\brac{\frac{\Zt}{\Zap}}; \pap f} \\
& \quad + \frac{1}{4}(\Id - \Hil)\cbrac{\brac{\sqbrac{\frac{1}{\Zap},\Hil}\Ztap }\sqbrac{\Zt,\Hil}\Dap f } - \frac{1}{4}(\Id - \Hil)\cbrac{\brac{\sqbrac{\Zt,\Hil}\pap\frac{1}{\Zap}}^2 f} \\
& \quad + \half \sqbrac{\sqbrac{\Zt, \sqbrac{\Zt,\Hil} }\pap\frac{1}{\Zap},\Hil }\pap\brac{\frac{f}{\Zap}} + \sqbrac{\Ztt,\Hil}\Dap f
\end{aligned}
\end{align*}
Now we apply $\Delta$ to the above equation and handle each term individually. 
\begin{enumerate}
\item Using \propref{prop:commutator} and \propref{prop:HilHtilcaldiff} the term $\dis \Delta\cbrac{\sqbrac{\Pa\brac{\frac{\Zt}{\Zap}},\Hil}\pap(\Dt f)} $ is easily shown to be in $\LtwoDeltahalf$ for $f = \Th, \Ztapbar$ and in $\HhalfDeltahalf$ for $f = \Dap\Ztbar$. 

\item Using \propref{prop:triple} and \propref{prop:HilHtilcaldiff} we see that $\dis \Delta\sqbrac{\Pa\brac{\frac{\Zt}{\Zap}}, \Pa\brac{\frac{\Zt}{\Zap}}; \pap f}$ is in $\LtwoDeltahalf$ for $f = \Th, \Ztapbar$ and in $\HhalfDeltahalf$ for $f = \Dap\Ztbar$. 

\item Using \propref{prop:commutator}  and \propref{prop:HilHtilcaldiff} we see that $\dis \Delta\brac{\sqbrac{\frac{1}{\Zap},\Hil}\Ztap } \in \LinftyDeltahalf\cap\HhalfDeltahalf $

\item For $f = \Dap\Ztbar$ we see by using \propref{prop:commutator} and \propref{prop:HilHtilcaldiff} that $\dis \Delta \brac{\sqbrac{\Zt,\Hil}\Dap f} \in \LinftyDeltahalf\cap\HhalfDeltahalf $. For $f = \Th, \Ztapbar$ we observe that
\begin{align*}
\sqbrac{\Zt,\Hil}\Dap f = \sqbrac{\Pa\brac{\frac{\Zt}{\Zap}}, \Hil}\pap f
\end{align*}
Hence by using  \propref{prop:commutator} and \propref{prop:HilHtilcaldiff} we see that $\dis \Delta \brac{\sqbrac{\Zt,\Hil}\Dap f} \in \LtwoDeltahalf$ for $f = \Th, \Ztapbar$. Hence combining these estimates with the previous estimate we see that $\dis \Delta (\Id - \Hil)\cbrac{\brac{\sqbrac{\frac{1}{\Zap},\Hil}\Ztap }\sqbrac{\Zt,\Hil}\Dap f } $ is in $\LtwoDeltahalf$ for $f = \Th, \Ztapbar$ and in $\HhalfDeltahalf$ for $f = \Dap\Ztbar$. 

\item Using \propref{prop:commutator} and \propref{prop:HilHtilcaldiff} we see that $\dis \Delta\brac{\sqbrac{\Zt,\Hil}\pap\frac{1}{\Zap}} \in \LinftyDeltahalf\cap\HhalfDeltahalf$. Hence using this we see that $\dis \Delta (\Id - \Hil)\cbrac*[\bigg]{\brac{\sqbrac{\Zt,\Hil}\pap\frac{1}{\Zap}}^2 f}$ is in $\LtwoDeltahalf$ for $f = \Th, \Ztapbar$ and in $\HhalfDeltahalf$ for $f = \Dap\Ztbar$. 

\item Using \propref{prop:triple} and \propref{prop:HilHtilcaldiff} we see that $\dis \pap\Delta\brac{\sqbrac{\Zt, \sqbrac{\Zt,\Hil} }\pap\frac{1}{\Zap} } \in \LtwoDeltahalf$

\item For $f = \Dap\Ztbar$ we observe that $\dis \pap\Delta\brac{\frac{f}{\Zap}}  \in \LtwoDeltahalf$. For $f = \Th,\Ztapbar$ we see that $\dis \Delta\brac{\frac{f}{\Zap}} \in \CcalDeltahalf$.  Hence combining this with the previous estimate and using \propref{prop:commutator} and \propref{prop:HilHtilcaldiff} we see that $\dis \sqbrac{\sqbrac{\Zt, \sqbrac{\Zt,\Hil} }\pap\frac{1}{\Zap},\Hil }\pap\brac{\frac{f}{\Zap}} $  is in $\LtwoDeltahalf$ for $f = \Th, \Ztapbar$ and in $\HhalfDeltahalf$ for $f = \Dap\Ztbar$. 

\item For $f = \Dap\Ztbar$ we observe that by using \propref{prop:commutator} and \propref{prop:HilHtilcaldiff} we obtain $\dis \Delta\brac{\sqbrac{\Ztt,\Hil}\Dap f} \in \HhalfDeltahalf$. For $f = \Th,\Ztapbar$ we see that
\begin{align*}
\sqbrac{\Ztt,\Hil}\Dap f =  - \sqbrac{\Ztt,\Hil}\cbrac*[\Big]{\brac*[\Big]{\pap\frac{1}{\Zap}} f} +  \sqbrac{\Ztt,\Hil}\pap\brac{\frac{f}{\Zap}}
\end{align*}
Now observe that for $f = \Th,\Ztapbar$ we have $\dis \Delta \cbrac*[\Big]{\brac*[\Big]{\pap\frac{1}{\Zap}} f} \in \Lone_{\sqrt{\Delta}}$ and $\dis \Delta \brac{\frac{f}{\Zap}} \in \CcalDeltahalf$. Hence using \propref{prop:commutator} and \propref{prop:HilHtilcaldiff} we see that $\dis \Delta\brac{\sqbrac{\Ztt,\Hil}\Dap f} \in \LtwoDeltahalf$ for $f = \Th,\Ztapbar$. 

\end{enumerate}

Hence we have the required estimate.

\medskip

\item $\dis \brac{\sigma(\Id - \Hil)\Dapabs^3\Th}_a \in \LtwoDeltahalf, \brac{\sigma(\Id - \Hil)\Dapabs^3\Ztapbar}_a \in \LtwoDeltahalf, \dis \brac{\sigma(\Id - \Hil)\Dapabs^3\Dap\Ztbar}_a \in \HhalfDeltahalf $
\smallskip

\item $\dis \Delta\cbrac{\sqbrac{\Dt^2, \frac{1}{\Zap}}\Ztapbar} \in \CcalDeltahalf$, $\dis \Delta\cbrac{\sqbrac{\Dt^2, \frac{1}{\Zapbar}}\Ztapbar} \in \CcalDeltahalf$
\smallskip

\item $\dis \Delta\cbrac{\sqbrac{i\frac{\Aone}{\Zapabs^2}\pap, \frac{1}{\Zap}}\Ztapbar} \in \CcalDeltahalf$, $\dis \Delta\cbrac{\sqbrac{i\frac{\Aone}{\Zapabs^2}\pap, \frac{1}{\Zapbar}}\Ztapbar} \in \CcalDeltahalf$
\smallskip

\item  $\dis \cbrac{(\Id - \Hil)\sqbrac{i\sigma\Dapabs^3, \frac{1}{\Zap}}\Ztapbar}\la \in \HhalfDeltahalf$, $\dis \cbrac{(\Id - \Hil)\sqbrac{i\sigma\Dapabs^3, \frac{1}{\Zapbar}}\Ztapbar}\la \in \HhalfDeltahalf$ and we also have $\dis \cbrac{\Zapabs\sqbrac{i\sigma\Dapabs^3, \frac{1}{\Zap}}\Ztapbar}\la \in \LtwoDeltahalf$, $\dis \cbrac{\Zapabs\sqbrac{i\sigma\Dapabs^3, \frac{1}{\Zapbar}}\Ztapbar}\la \in \LtwoDeltahalf$
\smallskip

\item $\dis \Delta(\Rone) \in \CcalDeltahalf$
\smallskip

\item $\dis \Delta(\Jone) \in \LinftyDeltahalf\cap\HhalfDeltahalf$
\smallskip

\item $\dis \Delta(\Dapabs\Jone) \in \LtwoDeltahalf$ and hence $\dis \Delta(\Jone) \in \WcalDeltahalf$
\smallskip

\item $\dis \Delta(\Rtwo) \in \LtwoDeltahalf$
\smallskip

\item $\dis \Delta(\Jtwo) \in \LtwoDeltahalf$
\smallskip

\item $\dis \brac{\sigma\sqbrac{\frac{1}{\Zap},\Hil}\Dapabs^3\Ztapbar}\la \in \HhalfDeltahalf$, $\dis \brac{\sigma\sqbrac{\frac{1}{\Zapbar},\Hil}\Dapabs^3\Ztapbar}\la \in \HhalfDeltahalf$
\smallskip

\item $\dis \Delta\cbrac{(\Id - \Hil)\Dt^2 \Dapbar\Ztbar} \in \HhalfDeltahalf$
\smallskip

\item $\dis \brac{\sigma(\Id - \Hil)\Dapabs^3\Dapbar\Ztbar}_a \in \HhalfDeltahalf $
\smallskip

\item $\dis \Delta\brac{\frac{1}{\Zapabs^2}\pap\Jone} \in \HhalfDeltahalf$ and hence $\dis \Delta\brac{\frac{1}{\Zapabs^2}\pap\Jone} \in \CcalDeltahalf$
\smallskip

\end{enumerate} 
\bigskip

\subsection{Closing the energy estimate for  $\EDelta$}\label{sec:closeEDelta}

We now complete the proof of \thmref{thm:aprioriEDelta}. To simplify the calculations, we will continue to use the notation used in \secref{sec:quantEDelta} and introduce another notation: If $a(t), b(t) $ are functions of time we write $a \approx b$ if there exists a constant $C(L_1)$ depending only on $L_1$ (where $L_1$ was defined in \thmref{thm:aprioriEDelta}) with $\abs{a(t)-b(t)} \leq C(L_1) \EDelta (t)$. Observe that $\approx$ is an equivalence relation. With this notation, proving \thmref{thm:aprioriEDelta} is equivalent to showing $ \frac{d\EDelta(t)}{dt}  \approx 0$. 

First observe that by replacing $\lambda$ by $\sigma$ we get from \thmref{thm:aprioriEaux} 
\begin{align*}
\frac{d}{dt} (\sigma(\Eaux)_b(t)) \leq P((\Ecalhigh)_b(t))(\sigma(\Eaux)_b(t))
\end{align*}
Now from the assumptions of  \thmref{thm:aprioriEaux}, we have that $\sup_{t \in [0,T]}(\Ecalhigh)_b(t) \leq L_1$ and so we get
\begin{align*}
\frac{d}{dt} (\sigma(\Eaux)_b(t)) \leq C(L_1)\EDelta(t)
\end{align*}
So we now need to control the other components of $\EDelta$.
\bigskip

\subsubsection{Controlling $\EDeltazero$ \nopunct}  \hspace*{\fill} \medskip

We recall that 
\begin{align*}
\EDeltazero =  
\begin{aligned}[t]
&  \norm[\infty]{\brac{\sigma^{\half} \Zapabs^\half \pap\frac{1}{\Zap}}\la}^2  +   \norm[2]{\brac{\sigma^\onebysix\Zapabs^\half\pap\frac{1}{\Zap}}\la}^6  + \norm[2]{\brac{\frac{\sigma^\half}{\Zapabs^\half}\pap^2\frac{1}{\Zap}}\la}^2 \\
& + \norm[\infty]{\Delta(\w)}^2  +  \norm[2]{\Delta\brac{ \pap\frac{1}{\Zap}}}^2 + \norm*[\big][\Linfty\cap\Hhalf]{\htilap - 1}^2 + \norm[2]{\Dapabs_a(\htilap -1)}^2 \\
&  + \norm[\infty]{\Zapabs_a\Util\brac{\frac{1}{\Zapabs_b}} - 1}^2
\end{aligned}
\end{align*}

The control of the time derivatives of the first three quantities is the same as the one done in Sec 5.2.1 in \cite{Ag19} while controlling the time derivative of $\Esigmazero$ there. Now we control the other quantities.
\begin{enumerate}[leftmargin =*, align=left]
\item We observe from \eqref{form:Dtg} that
\begin{align*}
(\Dt)_a \Delta (\w) = \Delta(\Dt\w) = -\Delta\brac{i\w\Imag(\Dapbar\Ztbar) }  \in \LinftyDeltahalf
\end{align*}
Now using the computation from Sec 5.2.1 in \cite{Ag19} we obtain
\begin{align*}
\frac{d}{dt}\norm[\infty]{\Delta(\w)}^2 \lesssim \norm[\infty]{\Delta{\w}}\norm[\infty]{(\Dt)_a\Delta(\w)} \leq C(L_1)\EDelta
\end{align*}

\item By using \lemref{lem:timederiv} we obtain
\begin{align*}
\frac{d}{dt} \norm[2]{\Delta\brac{\pap\frac{1}{\Zap}}}^2 & \lesssim \norm[\infty]{(\bvarap)_a}\norm[2]{\Delta\brac{\pap\frac{1}{\Zap}}}^2 + \norm[2]{\Delta\brac{\pap\frac{1}{\Zap}}}\norm[2]{(\Dt)_a\Delta\brac{\pap\frac{1}{\Zap}}} \\
& \leq C(L_1)\EDelta
\end{align*}

\item By the calculation of \lemref{lem:quantM} we have
\begin{align*}
(\Dt)_a \htilap = \htilap\brac*{\Util(\bvarap)_b - (\bvarap)_a } = -\htilap\Delta(\bvarap)
\end{align*}
As $\Delta(\bvarap) \in \LinftyDeltahalf\cap\HhalfDeltahalf$ we have that $(\Dt)_a \htilap \in \LinftyDeltahalf\cap\HhalfDeltahalf$. Hence by \lemref{lem:timederiv} and by using the computation from Sec 5.2.1 in \cite{Ag19} we have
\begin{align*}
\frac{d}{dt}\norm*[\big][\Linfty\cap\Hhalf]{\htilap - 1}^2 \leq C(L_1)\norm*[\big][\Linfty\cap\Hhalf]{\htilap - 1}\norm*[\big][\Linfty\cap\Hhalf]{(\Dt)_a\htilap } \leq C(L_1)\EDelta
\end{align*}

\item By the calculation of \lemref{lem:quantM} we have
\begin{align*}
 (\Dt)_a \Dapabs_a\htilap & =   -\Real(\Dap\Zt)_a \Dapabs_a\htilap + (\Dapabs_a\htilap)(\Util(\bvarap)_b - (\bvarap)_a)  \\
& \quad + \htilap\cbrac{\htilap\brac{\frac{1}{\Zapabs_a}\Util(\Zapabs)_b }\Util(\Dapabs\bvarap)_b - (\Dapabs\bvarap)_a }
\end{align*}
Now as $\Dapabs_a\htilap \in \LtwoDeltahalf$, $\htilap - 1 \in \LinftyDeltahalf$, $\dis \frac{1}{\Zapabs_a}\Util(\Zapabs)_b -1 \in \LinftyDeltahalf$ and $\dis \Delta(\Dapabs\bvarap) \in \LtwoDeltahalf$ we see that $\dis (\Dt)_a \Dapabs_a\htilap  \in \LtwoDeltahalf$. Hence by \lemref{lem:timederiv}
\begin{align*}
\frac{d}{dt}\norm[2]{\Dapabs_a(\htilap -1)}^2 \leq C(L_1)\norm[2]{\Dapabs_a(\htilap -1)}\norm[2]{(\Dt)_a \Dapabs_a\htilap} \leq C(L_1)\EDelta
\end{align*}
\smallskip

\item By the calculation of \lemref{lem:quantM} we have
\begin{align*}
& (\Dt)_a\cbrac{\Zapabs_a\Util\brac{\frac{1}{\Zapabs}}_b } \\
&= \cbrac{\Zapabs_a\Util\brac{\frac{1}{\Zapabs}}_b }\Real\cbrac{(\Dap\Zt)_a - \Util(\Dap\Zt)_b - (\bvarap)_a + \Util(\bvarap)_b  }
\end{align*}
Now as $\Delta(\Dap\Zt) \in \LinftyDeltahalf$ and $\Delta(\bvarap) \in \LinftyDeltahalf$ we see that $\dis (\Dt)_a\cbrac{\Zapabs_a\Util\brac{\frac{1}{\Zapabs}}_b } \in \LinftyDeltahalf$. Hence by using the computation from Sec 5.2.1 in \cite{Ag19} we have
\begin{align*}
\frac{d}{dt} \norm[\infty]{\Zapabs_a\Util\brac{\frac{1}{\Zapabs_b}} - 1}^2 & \leq C(L_1) \norm[\infty]{\Zapabs_a\Util\brac{\frac{1}{\Zapabs_b}} - 1} \norm[\infty]{(\Dt)_a\cbrac{\Zapabs_a\Util\brac{\frac{1}{\Zapabs}}_b }} \\
& \leq C(L_1)\EDelta
\end{align*}

\end{enumerate}
\medskip

\subsubsection{Controlling $\EDeltaone$\nopunct}  \hspace*{\fill} \medskip

Recall that 
\begin{align*}
\EDeltaone = \norm[\Hhalf]{\Delta\cbrac{(\Zttbar-i)\Zap}}^2 + \norm[2]{(\sqrt{\Aone})_a\Delta(\Ztbarap)}^2 + \norm[2]{\brac{\frac{\sigma^\half}{\Zapabs^\half}\pap\Ztapbar}\la}^2
\end{align*}
We will first simplify the time derivative of each of the individual terms before combining them. 
\begin{enumerate}[leftmargin =*, align=left]
\item By using \lemref{lem:timederiv} we get
\begin{align*}
& \frac{d}{dt} \int \abs{\papabs^\half \Delta \cbrac{ (\Zttbar -i)\Zap }}^2\difff\ap \\
& \approx 2\Real \int \cbrac{\papabs \Delta\brac{ (\Ztt +i)\Zapbar }}(\Dt)_a\Delta\brac{(\Zttbar -i)\Zap}  \difff\ap
\end{align*}
Now $(\Dt)_a\Delta\brac{(\Zttbar -i)\Zap} = \Delta\brac{\Dt(\Zttbar -i)\Zap}$ and we have 
\begin{align*}
 \Dt\brac{(\Zttbar -i)\Zap} &= \Ztttbar\Zap + (\Dap\Zt -\bvarap)(\Zttbar-i)\Zap \\
 & = \Ztttbar\Zap + (\Dap\Zt -\bvarap)(-i\Aone + \sigma\pap\Th)
\end{align*}
Now by following the argument in Sec 5.2.2 of \cite{Ag19} we see that $\Delta \cbrac{(\Dap\Zt -\bvarap)(-i\Aone + \sigma\pap\Th)} \in \HhalfDeltahalf$. Hence we have 
\begin{align*}
\frac{d}{dt} \int \abs{\papabs^\half \Delta \cbrac{ (\Zttbar -i)\Zap }}^2\difff\ap \approx 2\Real \int \cbrac{\papabs \Delta\brac{ (\Ztt +i)\Zapbar }}\Delta\brac{\Ztttbar\Zap}  \difff\ap
\end{align*}

\item We see from \lemref{lem:timederiv} that
\begin{align*}
& \frac{d}{dt} \int (\Aone)_a \abs{\Delta( \Ztapbar) }^2\difff\ap \\
& =  \int (\bvarap\Aone + \Dt\Aone)_a\abs{\Delta(\Ztapbar)}^2 \diff\ap + 2\Real \int (\Aone)_a \Delta(\Ztapbar) \Delta\brac{-\bvarap\Ztap + \Zttap} \diff \ap \\
& \leq C(L_1)\EDelta
\end{align*}

\item By following the proof of time derivative of $\Esigmaone$ in Sec 5.2.2 in \cite{Ag19} we get
\begin{align*}
 & \sigma\frac{d}{dt}  \int \abs*[\bigg]{\brac{\frac{1}{\Zapabs^\half}\pap \Ztapbar }\la}^2\difff\ap  \\
 & \approx 2 \Real \int \cbrac{-i\sigma\pap\brac{\frac{1}{\Zapbar}\Dapabs\Ztapbar}}\la\papabs\brac{(\Ztt+i)\Zapbar  }_a \difff\ap \\
& = 2 \Real \int \cbrac{-i\sigma\pap\brac{\frac{1}{\Zapbar}\Dapabs\Ztapbar}}\la\papabs\Delta\brac{(\Ztt+i)\Zapbar  } \difff\ap  \\
& \quad + 2 \Real \int \cbrac{-i\sigma\pap\brac{\frac{1}{\Zapbar}\Dapabs\Ztapbar}}\la\papabs\Util\brac{(\Ztt+i)\Zapbar  }_b \difff\ap \\
\end{align*}
We now show that the second term is controlled. Observe that $((\Ztt+i)\Zapbar)_b = i(\Aone)_b$ and that $\dis \brac{\frac{\sigma^\half}{\Zapabs^\half}\pap\Ztapbar}\la \in \LtwoDeltahalf$. Hence we only need to show that $\dis \frac{\sigma^\half}{\Zapabs_a^\threebytwo}\pap\papabs\Util(\Aone)_b \in \LtwoDeltahalf$. Now
\begin{align*}
\frac{\sigma^\half}{\Zapabs_a^\threebytwo}\pap\papabs\Util(\Aone)_b = i\sqbrac{\frac{\sigma^\half}{\Zapabs_a^\threebytwo},\Hil}\pap^2\Util(\Aone)_b + i\Hil\cbrac{\frac{\sigma^\half}{\Zapabs_a^\threebytwo}\pap^2\Util(\Aone)_b}
\end{align*}
The first term is easily shown to be in $\LtwoDeltahalf$ as from \propref{prop:commutator} we have
\begin{align*}
\norm[2]{\sqbrac{\frac{\sigma^\half}{\Zapabs_a^\threebytwo},\Hil}\pap^2\Util(\Aone)_b} \lesssim \norm[2]{\sigma^\half\pap^2\frac{1}{\Zapabs_a^\threebytwo}}\norm[\infty]{\Util(\Aone)_b}
\end{align*}
Hence it is enough to show that $\dis \frac{\sigma^\half}{\Zapabs_a^\threebytwo}\pap^2\Util(\Aone)_b \in \LtwoDeltahalf$. We see that
\begin{align*}
\frac{\sigma^\half}{\Zapabs_a^\threebytwo}\pap^2\Util(\Aone)_b = \frac{\sigma^\half}{\Zapabs_a^\threebytwo}\pap\brac{\htilap\Util(\pap\Aone)_b}
\end{align*}
and hence we have the estimate
\begin{align*}
& \norm[2]{\frac{\sigma^\half}{\Zapabs_a^\threebytwo}\pap^2\Util(\Aone)_b} \\
& \leq C(L_1)\norm*[\big][2]{\Dapabs_a\htilap}\norm[\infty]{\brac{\frac{\sigma^\half}{\Zapabs^\half}\pap\Aone}_b} + C(L_1)\norm[2]{\brac{\frac{\sigma^\half}{\Zapabs^\threebytwo}\pap^2\Aone}_b}
\end{align*}
Now $\dis \brac{\frac{\sigma^\half}{\Zapabs^\half}\pap\Aone}_b \in \LtwoDeltahalf$ and $\dis \brac{\frac{\sigma^\half}{\Zapabs^\threebytwo}\pap^2\Aone}_b \in \LtwoDeltahalf $ as they are controlled by  $\sigma(\Eaux)_b$. Hence we have shown 
\begin{align*}
 & \sigma\frac{d}{dt}  \int \abs*[\bigg]{\brac{\frac{1}{\Zapabs^\half}\pap \Ztapbar }\la}^2\difff\ap  \\
 & \approx  2 \Real \int \cbrac{-i\sigma\pap\brac{\frac{1}{\Zapbar}\Dapabs\Ztapbar}\la}\papabs\Delta\brac{(\Ztt+i)\Zapbar  } \difff\ap 
\end{align*}

\medskip
\item Now combining the terms we have
\begin{align*}
\frac{d}{dt} \EDeltaone & \approx  2 \Real \int \cbrac{\Delta\brac{\Ztttbar\Zap} -i\sigma\pap\brac{\frac{1}{\Zapbar}\Dapabs\Ztapbar}\la}\papabs\Delta\brac{(\Ztt+i)\Zapbar  } \difff\ap 
\end{align*} 
Recall from \eqref{eq:ZtZap} that
\begin{align*} 
\begin{split}
& \Ztttbar\Zap + i\Aone\Dapbar\Ztbar -i\sigma\pap\brac{\frac{1}{\Zapbar}\Dapabs\Ztapbar  } \\
& = i\sigma\pap\cbrac{\brac{\Dapabs \frac{1}{\Zapbar}}\Ztapbar } -\sigma(\Dap\Zt)\pap\Th -\sigma\pap\cbrac{(\Real\Th)\Dapbar\Ztbar} -i\Jone
\end{split}
\end{align*}
Now we just apply $\Delta$ to the above equation and control the quantities. We see that $\Delta(\Jone) \in \HhalfDeltahalf$, $\Delta(\Aone\Dapbar\Ztbar) \in \LinftyDeltahalf\cap\HhalfDeltahalf$ and the other terms with $\sigma$ are controlled as in the proof of the time derivative of $\Esigmaone$ in Sec 5.2.2 in \cite{Ag19}. Hence
\begin{align*}
\Delta\brac{\Ztttbar\Zap} -i\sigma\pap\brac{\frac{1}{\Zapbar}\Dapabs\Ztapbar}\la \in \HhalfDeltahalf
\end{align*}
and hence we have shown that $\dis \frac{d}{dt}\EDeltaone \leq C(L_1)\EDelta$.  
\end{enumerate} 
\medskip

\subsubsection{Controlling $\EDeltatwo$ and $\EDeltathree$\nopunct}  \hspace*{\fill} \medskip 
 
Note that both $\EDeltatwo$ and $\EDeltathree$ are of the form
\begin{align*}
\E_{\Delta,i} = \norm[2]{\Delta(\Dt f)}^2 + \norm[\Hhalf]{\brac{\frac{\sqrt{\Aone}}{\Zapabs}}\la \Delta( f)}^2 + \norm[\Hhalf]{\brac{\frac{\sigma^\half}{\Zapabs^\threebytwo}\pap f}\la}^2
\end{align*}
Where $f = \Ztapbar$ for $i=2$ and $f=\Th$ for $i=3$. Also note that $\Ph f = f$ for these choices of $f$. We will simplify the time derivative of each of the terms individually before combining them.

\begin{enumerate}[leftmargin =*, align=left]

\item From \lemref{lem:timederiv} we have
\begin{align*}
\frac{d}{dt} \int \abs{\Delta(\Dt f)}^2 \diff\ap \approx 2\Real \int (\Delta(\Dt^2 f))\Delta(\Dt \bar{f}) \diff\ap
\end{align*}

\item By following the proof of time derivative of $\Esigmatwo,\Esigmathree$ in Sec 5.2.3 in \cite{Ag19} we have
\begin{align*}
& \frac{d}{dt} \int \abs*[\bigg]{ \papabs^\half \cbrac*[\bigg]{\brac{\frac{\sqrt{\Aone}}{\Zapabs}}\la \Delta( f)}}^2\difff\ap \\
& \approx 2\Real \int  \cbrac{\brac{\frac{\sqrt{\Aone}}{\Zapabs}}\la\papabs\brac*[\bigg]{\brac{\frac{\sqrt{\Aone}}{\Zapabs}}\la \Delta(f)}} \Delta(\Dt\bar{f}) \diff\ap
\end{align*}
Now using \propref{prop:commutator} and \lemref{lem:CWDelta} we see that 
\begin{align*}
 \brac{\frac{\sqrt{\Aone}}{\Zapabs}}\la\papabs\brac*[\bigg]{\brac{\frac{\sqrt{\Aone}}{\Zapabs}}\la \Delta(f)} & \approxLtwoDeltahalf i\Hil\cbrac{\brac{\frac{\sqrt{\Aone}}{\Zapabs}}\la\pap\brac{\brac{\frac{\sqrt{\Aone}}{\Zapabs}}\la\Delta(f)} } \\
& \approxLtwoDeltahalf i\Hil\cbrac{\brac{\frac{\Aone}{\Zapabs^2}}\la\pap\Delta(f) }
\end{align*}
Now using \lemref{lem:Deltaapprox} we have
\begin{align*}
i\Hil\cbrac{\brac{\frac{\Aone}{\Zapabs^2}}\la\pap\Delta(f) } & \approxLtwoDeltahalf i\Hil\cbrac{\Delta\brac{\frac{\Aone}{\Zapabs^2}\pap f}} \\
&  \approxLtwoDeltahalf i\Hil\cbrac{\frac{\Aone}{\Zapabs^2}\pap f}\la - i\Hil\Util\cbrac{\frac{\Aone}{\Zapabs^2}\pap f}\lb
\end{align*}
Now as $\dis \cbrac{\frac{\Aone}{\Zapabs^2}\pap f}\lb \in \Ltwo$ we can replace $\Hil$ in the second term with $\Hcal$. Hence we have
\begin{align*}
 \brac{\frac{\sqrt{\Aone}}{\Zapabs}}\la\papabs\brac*[\bigg]{\brac{\frac{\sqrt{\Aone}}{\Zapabs}}\la \Delta(f)}  \approxLtwoDeltahalf \Delta\cbrac{i\Hil\brac{\frac{\Aone}{\Zapabs^2}\pap f}}
\end{align*}
We can simplify the above term by using $\Hil f = f$. We see that 
\begin{align*}
i\Hil\brac{\frac{\Aone}{\Zapabs^2}\pap f} = - i\sqbrac{\frac{\Aone}{\Zapabs^2},\Hil}\pap f + i\frac{\Aone}{\Zapabs^2}\pap f
\end{align*}
Now apply $\Delta$ to the above equation. We can easily control the first term in $\LtwoDeltahalf$ by using \propref{prop:commutator}, \propref{prop:HilHtilcaldiff} and \lemref{lem:CWDelta} and hence we have
\begin{align*}
 \brac{\frac{\sqrt{\Aone}}{\Zapabs}}\la\papabs\brac*[\bigg]{\brac{\frac{\sqrt{\Aone}}{\Zapabs}}\la \Delta(f)}  \approxLtwoDeltahalf \Delta\brac{i\frac{\Aone}{\Zapabs^2}\pap f}
\end{align*}
Finally using this we obtain
\begin{align*}
\frac{d}{dt} \int \abs*[\bigg]{ \papabs^\half \cbrac*[\bigg]{\brac{\frac{\sqrt{\Aone}}{\Zapabs}}\la \Delta( f)}}^2\difff\ap \approx 2\Real \int  \Delta\brac{i\frac{\Aone}{\Zapabs^2}\pap f} \Delta(\Dt\bar{f}) \diff\ap 
\end{align*}

\item By using the argument in controlling the time derivative of $\Esigmatwo,\Esigmathree$ in Sec 5.2.3 in \cite{Ag19} we have
\begin{align*}
& \frac{d}{dt} \sigma \int  \abs*[\bigg]{\papabs^\half \brac*[\Bigg]{\frac{1}{\Zapabs^\threebytwo} \pap f}\la }^2 \difff\ap \\
 & \approx -2\sigma \Real \int  \pap\cbrac*[\Bigg]{\frac{1}{\Zapabs^\threebytwo}\papabs \brac*[\Bigg]{\frac{1}{\Zapabs^\threebytwo} \pap f}}\la(\Dt\bar{f})_a \diff\ap \\
& \approx -2\sigma \Real \int  \pap\cbrac*[\Bigg]{\frac{1}{\Zapabs^\threebytwo}\papabs \brac*[\Bigg]{\frac{1}{\Zapabs^\threebytwo} \pap f}}\la\Delta(\Dt\bar{f}) \diff\ap  \\
& \quad -2\sigma \Real \int  \pap\cbrac*[\Bigg]{\frac{1}{\Zapabs^\threebytwo}\papabs \brac*[\Bigg]{\frac{1}{\Zapabs^\threebytwo} \pap f}}\la\Util(\Dt\bar{f})_b \diff\ap 
\end{align*}
We now show that the second term is controlled. We see that
\begin{align*}
& -2\sigma \Real \int  \pap\cbrac*[\Bigg]{\frac{1}{\Zapabs^\threebytwo}\papabs \brac*[\Bigg]{\frac{1}{\Zapabs^\threebytwo} \pap f}}\la\Util(\Dt\bar{f})_b \diff\ap \\
& = 2\sigma \Real \int  \papabs \brac*[\Bigg]{\frac{1}{\Zapabs^\threebytwo} \pap f}\la \brac{\frac{1}{\Zapabs_a^\threebytwo}\pap \Util(\Dt\bar{f})_b} \diff\ap 
\end{align*}
Now we know that $\dis \brac{\frac{\sigma^\half}{\Zapabs^\threebytwo}\pap\Dt\Th}\lb \in \CcalDeltahalf$ and $\dis \brac{\frac{\sigma^\half}{\Zapabs^\threebytwo}\pap\Dt\Ztapbar}\lb \in \CcalDeltahalf$ as they are both controlled by $\sigma(\Eaux)_b$. Hence we also have that $\dis \frac{\sigma^\half}{\Zapabs_a^\threebytwo}\pap \Util(\Dt\Ztapbar)_b  \in \CcalDeltahalf$ and $\dis \frac{\sigma^\half}{\Zapabs_a^\threebytwo}\pap \Util(\Dt\Th)_b  \in \CcalDeltahalf$ by using \lemref{lem:Deltaapprox}. Therefore we now have
\begin{align*}
& \frac{d}{dt} \sigma \int  \abs*[\bigg]{\papabs^\half \brac*[\Bigg]{\frac{1}{\Zapabs^\threebytwo} \pap f}\la }^2 \difff\ap  \\
& \approx -2\sigma \Real \int  \pap\cbrac*[\Bigg]{\frac{1}{\Zapabs^\threebytwo}\papabs \brac*[\Bigg]{\frac{1}{\Zapabs^\threebytwo} \pap f}}\la\Delta(\Dt\bar{f}) \diff\ap 
\end{align*}
Now from the proof from Sec 5.2.3 in \cite{Ag19} we obtain 
\begin{align*}
\sigma \pap\cbrac*[\Bigg]{\frac{1}{\Zapabs^\threebytwo}\papabs \brac*[\Bigg]{\frac{1}{\Zapabs^\threebytwo} \pap f}}\la \approxLtwoDeltahalf  \brac{i\sigma\Dapabs^3 f}_a
\end{align*}
So we finally have
\begin{align*}
\frac{d}{dt} \sigma \int  \abs*[\bigg]{\papabs^\half \brac*[\Bigg]{\frac{1}{\Zapabs^\threebytwo} \pap f}\la }^2 \difff\ap  \approx -2 \Real \int  \brac{i\sigma\Dapabs^3 f}_a\Delta(\Dt \bar{f}) \diff\ap
\end{align*}

\medskip

\item Now combining all three terms we have for $i=2,3$
\begin{align*}
\frac{d}{dt} E_{\Delta,i} \approx 2\Real \int \cbrac{\Delta(\Dt^2 f) +  \Delta\brac{i\frac{\Aone}{\Zapabs^2}\pap f} -i\sigma(\Dapabs^3 f)_a  } \Delta(\Dt \bar{f}) \diff\ap
\end{align*}
For $f = \Ztapbar$ we obtain from  \eqref{eq:Ztbarap}
\begin{align*}
\begin{split}
& \brac{\Dt^2 +i\frac{\Aone}{\Zapabs^2}\pap  -i\sigma\Dapabs^3} \Ztbarap \\
 & =  \Rone\Zapbar -i\brac{\pap\frac{1}{\Zap}} \Jone -i\Dap \Jone - \Zapbar\sqbrac{\Dt^2 +i\frac{\Aone}{\Zapabs^2}\pap -i\sigma\Dapabs^3, \frac{1}{\Zapbar}} \Ztapbar
\end{split}
\end{align*}
Hence applying $\Delta$ on both sides, we easily see that the terms on the right hand side are in $\LtwoDeltahalf$. Similarly for  $f = \Th$ we have from \eqref{eq:Th}
\[
 \brac{\Dt^2 +i\frac{\Aone}{\Zapabs^2}\pap  -i\sigma\Dapabs^3} \Th = \Rtwo +i\Jtwo
\]
In this case also we apply $\Delta$ on both sides and see that the terms on the right are controlled. Hence we have shown that for $i=2,3$ we have
\begin{align*}
\frac{d}{dt} E_{\Delta,i}  \leq C(L_1)\EDelta
\end{align*}
\end{enumerate}

\medskip

\subsubsection{Controlling $\EDeltafour$\nopunct}\label{sec:controlEDeltafour}  \hspace*{\fill} \medskip 

Recall that 
\begin{align*}
 \EDeltafour = \norm[\Hhalf]{\Delta(\Dt\Dapbar\Ztbar)}^2 + \norm[2]{(\sqrt{\Aone})_a\Dapabs_a\Delta(\Dapbar\Ztbar)}^2 + \norm[2]{\brac{\frac{\sigma^\half}{\Zapabs^\half}\pap\Dapabs\Dapbar\Ztbar}\la}^2
\end{align*}
We again simplify the terms individually before combining them.
\begin{enumerate}[leftmargin =*, align=left]
\item By \lemref{lem:timederiv} we have
\begin{align*}
\frac{d}{dt} \int \abs*[\Big]{\papabs^\half \Delta \brac{\Dt \Dapbar \Ztbar}}^2 \difff\ap \approx 2\Real \int \Delta(\Dt^2\Dapbar\Ztbar)\papabs\Delta(\Dt\Dap\Zt) \diff \ap
\end{align*}
Now as $\Delta((\Id - \Hil)\Dt^2\Dapbar\Ztbar) \in \HhalfDeltahalf$ we see that 
\begin{align*}
\Delta(\Dt^2\Dapbar\Ztbar) \approxHhalfDeltahalf \Delta(\Hil\Dt^2\Dapbar\Ztbar)  \approxHhalfDeltahalf \Hil (\Dt^2\Dapbar\Ztbar)_a - \Hcal \Util(\Dt^2\Dapbar\Ztbar)_b
\end{align*}
But we know that $(\Dt^2\Dapbar\Ztbar)_b \in \Hhalf$ as it is controlled by $(\Ehigh)_b$. Hence we now have $(\Hil -\Hcal) \Util(\Dt^2\Dapbar\Ztbar)_b \in \HhalfDeltahalf$. From this we get 
\begin{align*}
\Delta(\Dt^2\Dapbar\Ztbar) \approxHhalfDeltahalf \Hil \Delta(\Dt^2\Dapbar\Ztbar)
\end{align*}
Now we use the fact that $\papabs = i\Hil\pap$ to obtain
\begin{align*}
\frac{d}{dt} \int \abs*[\Big]{\papabs^\half \Delta \brac{\Dt \Dapbar \Ztbar}}^2 \difff\ap \approx 2\Real \int \Delta(\Dt^2\Dapbar\Ztbar)\cbrac{-i\pap\Delta(\Dt\Dap\Zt)} \diff \ap
\end{align*} 

\item By following the proof of control of the time derivative of $\Esigmafour$ from Sec 5.2.4 in \cite{Ag19} we see that
\begin{align*}
& \frac{d}{dt} \int (\Aone)_a  \abs{\Dapabs_a \Delta (\Dapbar \Ztbar)}^2 \difff\ap \\
& \approx 2\Real \int \brac*[\bigg]{i\brac{\frac{\Aone}{\Zapabs^2}}\la\pap\Delta(\Dapbar\Ztbar)}\cbrac{-i\pap\Delta(\Dt\Dap\Zt)} \diff\ap
\end{align*}
Now we know that $\dis \brac{\frac{1}{\Zapabs^2}\pap\Dapbar\Ztbar}\lb \in \Ccal$ as it is controlled by $(\Ehigh)_b$. Hence as $(\Aone)_b \in \Wcal$, we have $\dis \brac{\frac{\Aone}{\Zapabs^2}\pap\Dapbar\Ztbar}\lb \in \Ccal$. Hence we see that 
\begin{align*}
i\brac{\frac{\Aone}{\Zapabs^2}}\la\pap\Delta(\Dapbar\Ztbar) \approxHhalfDeltahalf \Delta\brac{i\frac{\Aone}{\Zapabs^2}\pap\Dapbar\Ztbar}
\end{align*}
From this we get
\begin{align*}
\frac{d}{dt} \int (\Aone)_a  \abs{\Dapabs_a \Delta (\Dapbar \Ztbar)}^2 \difff\ap \approx 2\Real \int \Delta\brac*[\bigg]{i\frac{\Aone}{\Zapabs^2}\pap\Dapbar\Ztbar}\cbrac{-i\pap\Delta(\Dt\Dap\Zt)} \diff\ap
\end{align*}

\item By following the proof of control of the time derivative of $\Esigmafour$ from Sec 5.2.4 in \cite{Ag19} we see that
\begin{flalign*}
\lpar \quad &  \sigma\frac{d}{dt} \int \abs*[\bigg]{\brac{\frac{1}{\Zapabs^\half}\pap \Dapabs \Dapbar \Ztbar}\la }^2\difff\ap &&\\
& \approx 2\sigma \Real \int \cbrac{\frac{1}{\Zapabs^\half}\pap\Dapabs\Dapbar\Ztbar }\la\cbrac{\frac{1}{\Zapabs^\half}\pap\Dapabs\Dt\Dap\Zt }\la \\
 & \approx  2 \sigma\Real  \int \cbrac{\frac{1}{\Zapabs^\half}\pap\Dapabs\Dapbar\Ztbar }\la\cbrac{\frac{1}{\Zapabs_a^\half}\pap\Dapabs_a\Delta(\Dt\Dap\Zt) }  \\
 & \quad +  2 \sigma \Real  \int \cbrac{\frac{1}{\Zapabs^\half}\pap\Dapabs\Dapbar\Ztbar }\la\cbrac{\frac{1}{\Zapabs_a^\half}\pap\Dapabs_a\Util(\Dt\Dap\Zt)_b } 
\end{flalign*}
We now show that the second term is controlled. We first observe that 
\begin{align*}
& \frac{\sigma^\half}{\Zapabs_a^\half}\pap\Dapabs_a\Util(\Dt\Dap\Zt)_b \\
& = \frac{\sigma^\half}{\Zapabs_a^\half}\pap\brac{\frac{\htilap}{\Zapabs_a}\Util(\Zapabs)_b\Util(\Dapabs\Dt\Dap\Zt)_b } \\
& = \Zapabs_a^\half\Util\brac{\frac{1}{\Zapabs^\half}}\lb \Dapabs_a\brac{\frac{\htilap}{\Zapabs_a}\Util(\Zapabs)_b }\Util\brac{\frac{\sigma^\half}{\Zapabs^\half}\pap\Dt\Dap\Zt}\lb \\
& \quad + \frac{\htilap}{\Zapabs_a}\Util(\Zapabs)_b \brac{\frac{\sigma^\half}{\Zapabs_a^\half}\pap\Util(\Dapabs\Dt\Dap\Zt)_b}
\end{align*}
Now we know that $\dis \brac{\frac{\sigma^\half}{\Zapabs^\half}\pap\Dt\Dap\Zt}\lb \in \LinftyDeltahalf$ and $\dis \brac{\frac{\sigma^\half}{\Zapabs^\half}\pap\Dapabs\Dt\Dap\Zt }\lb \in \LtwoDeltahalf$ as they are controlled by $\sigma(\Eaux)_b$. We also know that $\htilap \in \Wcal$, $\dis \frac{1}{\Zapabs_a}\Util(\Zapabs)_b \in \Wcal$ and hence the above terms are controlled. Hence we have
\begin{flalign*}
\lpar \quad &  \sigma\frac{d}{dt} \int \abs*[\bigg]{\brac{\frac{1}{\Zapabs^\half}\pap \Dapabs \Dapbar \Ztbar}\la }^2\difff\ap &&\\
 & \approx  2 \sigma\Real  \int \cbrac{\frac{1}{\Zapabs^\half}\pap\Dapabs\Dapbar\Ztbar }\la\cbrac{\frac{1}{\Zapabs_a^\half}\pap\Dapabs_a\Delta(\Dt\Dap\Zt) } \\
 & \approx 2\Real \int \cbrac{-i\sigma\Dapabs^3\Dapbar\Ztbar }_a\cbrac{-i\pap\Delta(\Dt\Dap\Zt) } \diff\ap
\end{flalign*}

\medskip

\item Combining the three terms we obtain
\begin{align*}
& \frac{d}{dt}\EDeltafour  \approx 2\Real \int \cbrac{-i\pap\Delta(\Dt\Dap\Zt) }
\begin{aligned}[t]
& \bigg\{\Delta(\Dt^2\Dapbar\Ztbar) + \Delta\brac{i\frac{\Aone}{\Zapabs^2}\pap\Dapbar\Ztbar} \\
& \quad  -i\sigma(\Dapabs^3\Dapbar\Ztbar) _a\bigg\} \diff\ap
\end{aligned}
\end{align*}
From equation  \eqref{eq:DapbarZtbar} we see that
\[
 \brac{\Dt^2 +i\frac{\Aone}{\Zapabs^2}\pap  -i\sigma\Dapabs^3} \Dapbar\Ztbar  =  \Rone -i\brac{\Dapbar\frac{1}{\Zap}} \Jone -i\frac{1}{\Zapabs^2}\pap \Jone 
\]
Now we apply $\Delta$ to the above equation and see that the terms on the right hand side terms are controlled in $\HhalfDeltahalf$. Hence we have
\begin{align*}
 \frac{d}{dt}\EDeltafour \leq C(L_1) \EDelta
\end{align*}
This concludes the proof of \thmref{thm:aprioriEDelta}.

\end{enumerate}

\subsection{Equivalence of  $\EDelta$ and $\EcalDelta$}\label{sec:EDeltaEcalDelta}

We now give a simpler description of the energy $\EDelta$. Define
\begingroup
\allowdisplaybreaks
\begin{align}\label{def:EcalDelta}
\begin{split}
\EcalDeltazero & =   \norm[\infty]{\Delta(\w)}^2 + \norm*[\big][\Linfty\cap\Hhalf]{\htilap - 1}^2 + \norm[2]{\Dapabs_a(\htilap -1)}^2  +  \norm[\infty]{\Zapabs_a\Util\brac{\frac{1}{\Zapabs_b}} - 1}^2\\
\EcalDeltaone & = \norm[2]{\Delta\brac{\pap\frac{1}{\Zap}}}^2 + \norm*[\bigg][\Hhalf]{\Delta\brac{\frac{1}{\Zap}\pap\frac{1}{\Zap}}}^2 + \norm*[\Bigg][2]{\brac{\sigma^\onebysix\Zap^\half\pap\frac{1}{\Zap}}\la}^6  \\*
& \quad + \norm*[\Bigg][\infty]{\brac{\sigma^\half\Zap^\half\pap\frac{1}{\Zap}}\la}^2  +  \norm*[\Bigg][2]{\brac*[\Bigg]{\frac{\sigma^\half}{\Zap^\half}\pap^2\frac{1}{\Zap}}\la}^2 + \norm*[\Bigg][\Hhalf]{\brac*[\Bigg]{\frac{\sigma^\half}{\Zap^\threebytwo}\pap^2\frac{1}{\Zap}}\la}^2  \\* 
& \quad +  \norm[\Hhalf]{\brac{\sigma\pap\Th}_a}^2   + \norm[2]{\brac{\frac{\sigma}{\Zap}\pap^3\frac{1}{\Zap}}\la}^2  + \norm[\Hhalf]{\brac{\frac{\sigma}{\Zap^2}\pap^3\frac{1}{\Zap}}\la}^2\\
\EcalDeltatwo &= \norm[2]{\Delta\brac{\Ztapbar}}^2 + \norm*[\Bigg][2]{\Delta\brac{\frac{1}{\Zap^2}\pap\Ztapbar}}^2 + \norm*[\Bigg][2]{\brac*[\Bigg]{\frac{\sigma^\half}{\Zap^\half}\pap\Ztapbar}\la}^2 + \norm*[\Bigg][2]{\brac*[\Bigg]{\frac{\sigma^\half}{\Zap^\fivebytwo}\pap^2\Ztapbar}\la}^2 \\
\EcalDelta & = \sigma\brac{\Ecalaux}_b +  \EcalDeltazero + \EcalDeltaone + \EcalDeltatwo
\end{split}
\end{align}
\endgroup

Note that if the two solutions have the same initial data, then $\EcalDeltazero(0) = 0$ and hence we obtain the representation of the energy as stated in \secref{sec:results}.

\begin{prop}\label{prop:equivEDeltaEcalDelta}
Let $T>0$ and let $(\Z,\Zt)_a(t)$, $(\Z,\Zt)_b(t)$ be two smooth solutions in $[0,T]$ to  \eqref{eq:systemone} with surface tension $\sigma$ and zero surface tension respectively, such that for all $s\geq 2$ we have $(\Zap-1,\frac{1}{\Zap} - 1, \Zt)_i \in \Linfty([0,T], H^{s}(\Rsp)\times H^{s}(\Rsp)\times H^{s+\half}(\Rsp))$ for both $i=a,b$.  Let $L_1>0$ be such that 
\begin{align*}
\sup_{t \in [0,T]}(\Ecalhigh)_b(t), \sup_{t\in [0,T]}(\Ecalsigma)_a (t), \norm[\infty]{\Zapabs_a\Util\brac{\frac{1}{\Zapabs_b}}}(0),  \norm[\infty]{\frac{1}{\Zapabs_a}\Util\brac{\Zapabs_b}}(0) \leq L_1
\end{align*}
Then there exists constants $C_1(L_1), C_2(L_1) > 0$ depending only on $L_1$ so that for all $t \in [0,T]$ we have
\begin{align*}
\EDelta \leq C_1(L_1)\EcalDelta \quad \tx{ and }\quad  \EcalDelta \leq C_2(L_1)\EDelta
\end{align*}
\end{prop}
\begin{proof}
Let us first show $ \EcalDelta \leq C_2(L_1)\EDelta$. From \propref{prop:equivEauxEcalaux} we see that $\sigma(\Ecalaux)_b$ is controlled by $\sigma(\Eaux)_b$. Also observe that the energy $\EDelta$ directly controls $\EcalDeltazero$ and most of the terms of $\EcalDeltaone$ and $\EcalDeltatwo$. The terms which are not directly controlled are $\dis \brac*[\Bigg]{\frac{\sigma^\half}{\Zap^\threebytwo}\pap^2\frac{1}{\Zap}}\la$ and $\dis \brac{\frac{\sigma}{\Zap^2}\pap^3\frac{1}{\Zap}}\la$ both in $\HhalfDeltahalf$. To control these, we use \lemref{lem:CWDelta} to see that
\begin{align*}
\norm*[\Bigg][\Hhalf]{\brac*[\Bigg]{\frac{\sigma^\half}{\Zap^\threebytwo}\pap^2\frac{1}{\Zap}}\la} \lesssim \norm[\Wcal]{\wbar}^\threebytwo\norm[\CcalDeltahalf]{\brac*[\Bigg]{\frac{\sigma^\half}{\Zapabs^\threebytwo}\pap^2\frac{1}{\Zap}}\la} \leq C_2(L_1)\EDelta^\half
\end{align*}
Similarly we have
\begin{align*}
\norm*[\Bigg][\Hhalf]{\brac{\frac{\sigma}{\Zap^2}\pap^3\frac{1}{\Zap}}\la} \lesssim \norm[\Wcal]{\wbar}^2\norm[\CcalDeltahalf]{\brac{\frac{\sigma}{\Zapabs^2}\pap^3\frac{1}{\Zap}}\la} \leq C_2(L_1)\EDelta^\half
\end{align*}
This proves that $ \EcalDelta \leq C_2(L_1)\EDelta$.

Let us now show that $\EDelta \leq C_1(L_1)\EcalDelta$. We again observe from \propref{prop:equivEauxEcalaux} that $\sigma(\Eaux)_b$ is controlled by $\sigma(\Ecalaux)_b$.  It is also clear that $\EcalDelta$ controls $\EDeltazero$. Hence we now need to control $\EDeltaone, \EDeltatwo, \EDeltathree$ and $\EDeltafour$. As the proof of this is a little more involved, to simplify the presentation we will continue to use the same notation as in \secref{sec:quantEDelta} except for a few minor modifications. In the definitions, instead of using the energy $\EDelta$ we will use the energy $\EcalDelta$. So now whenever we write $\dis f \in \Ltwo_{\Delta^\alpha}$, what we mean is that there exists a constant $C_1(L_1)$ depending only on $L_1$ such that $\dis \norm*[2]{f} \leq C_1(L_1)(\EcalDelta)^\alpha $. Similar modifications for $\dis f\in \Lone_{\Delta^\alpha}$, $\dis f\in \Hhalf_{\Delta^\alpha}$ and $\dis  f \in \Linfty_{\Delta^\alpha}$. The definitions of the spaces $ \Ccal_{\Delta^\alpha}$ and $ \Wcal_{\Delta^\alpha}$ remain the same except for the fact that we have now changed the underlying definition of the spaces $\dis\Ltwo_{\Delta^\alpha}, \Hhalf_{\Delta^\alpha} $ and $\dis\Linfty_{\Delta^\alpha}$. We say that $f \in \Ltwo$ if $\dis f \in \Ltwo_{\Delta^\alpha}$ for $\alpha = 0$. Similar definitions for  $f \in \Hhalf, f \in \Linfty, f \in \Ccal$ and $f \in \Wcal$.  The definitions of $\approx_{\Ltwo_{\Delta^{\alpha}}}, \approx_{\Lone_{\Delta^{\alpha}}}, \approx_{\Linfty_{\Delta^{\alpha}}}, \dis\approx_{\Hhalf_{\Delta^{\alpha}}}, \approx_{\Wcal_{\Delta^\alpha}}  $ and $\approx_{\mathcal{C}_{\Delta^\alpha}}$ remain the same except the changes to the underlying spaces. Observe that there is no change to \lemref{lem:CWDelta}. 

We now make the important observation that \lemref{lem:Deltaapprox} still remains true with the new definitons. This is because in the proof of \lemref{lem:Deltaapprox}, the only properties of $\EDelta$ used were the control of $(\htilap -1) \in \LinftyDeltahalf\cap\HhalfDeltahalf\cap\WcalDeltahalf$, $\Delta w \in \LinftyDeltahalf$, $\dis \Delta\brac{\pap\frac{1}{\Zap}} \in \LtwoDeltahalf$ and the term $\dis \Zapabs_a\Util\brac{\frac{1}{\Zapabs}}\lb - 1 \in \LinftyDeltahalf$. All of the these quantities are also controlled by $\EcalDelta$ and hence the lemma still holds. Let us now control $\E_{\Delta,i}$ for $1\leq i\leq 4$. 

\begin{enumerate}[leftmargin =*, align=left]

\item Controlling $\EDeltaone$: From $\EcalDeltatwo$ we have $\Delta(\Ztapbar) \in \LtwoDeltahalf$. Hence we have $(\sqrt{\Aone})_a \Delta(\Ztapbar) \in \LtwoDeltahalf$. Hence now via \secref{sec:quantEDelta} we have $\Delta(\Aone) \in \LinftyDeltahalf\cap\HhalfDeltahalf$. Now we know from \eqref{form:Zttbar} that $(\Zttbar-i)\Zap = -i\Aone + \sigma\pap\Th$ and hence $\Delta\cbrac{(\Zttbar-i)\Zap} = -i\Delta(\Aone) + (\sigma\pap\Th)_a$. 
Now as $(\sigma\pap\Th)_a \in \HhalfDeltahalf$ we see that $\Delta\cbrac{(\Zttbar-i)\Zap} \in \HhalfDeltahalf$. From $\EcalDelta$ we also clearly see that $\dis \brac{\frac{\sigma^\half}{\Zapabs^\half}\pap\Ztapbar}\la \in \LtwoDeltahalf $ and hence $\EDeltaone$ is controlled.

\item Controlling $\EDeltatwo$: We prove this step by step.
\begin{enumerate}
\item As $\EcalDelta$ controls $\dis \Delta\brac{\pap\frac{1}{\Zap}} \in \LtwoDh$, from \secref{sec:quantEDelta} we easily obtain control of $\dis \Delta\brac{\pap\frac{1}{\Zapabs}} \in \LtwoDh, \Delta(\Dapabs\w) \in \LtwoDh$ and $\Delta(\w) \in \WcalDh$. 

\item As $\EcalDelta$ controls $\dis \Delta\brac{\frac{1}{\Zap^2}\pap\Ztapbar} \in \LtwoDh$, by using \lemref{lem:Deltaapprox} repeatedly we also have $\dis \frac{1}{(\Zap)_a^2}\pap\Delta(\Ztapbar) \in \LtwoDh$. Hence we  see that 

\begin{align*}
\pap\brac{\frac{1}{(\Zap)_a}\Delta(\Ztapbar)}^2 & = 2\brac{\frac{1}{(\Zap)_a}\Delta(\Ztapbar)}\pap\brac{\frac{1}{(\Zap)_a}\Delta(\Ztapbar)} \\
& = 2\brac{\frac{1}{(\Zap)_a}\Delta(\Ztapbar)}\nobrac{\brac{\pap\frac{1}{\Zap}}\la \Delta(\Ztapbar)} \\
& \quad + 2\Delta(\Ztapbar)\brac{\frac{1}{(\Zap)_a^2}\pap\Delta(\Ztapbar)}
\end{align*}
From this we obtain
\begin{align*}
\norm[\infty]{\frac{1}{(\Zap)_a}\Delta(\Ztapbar)}^2 & \leq  C_1(L_1)\norm[\infty]{\frac{1}{(\Zap)_a}\Delta(\Ztapbar)}\norm[2]{\Delta(\Ztapbar)} \\
& \quad + C_1(L_1)\norm[2]{\Delta(\Ztapbar)}\norm[2]{\frac{1}{(\Zap)_a^2}\pap\Delta(\Ztapbar)}
\end{align*}
Now using the inequality $\dis ab \leq \frac{\ep a^2}{2} + \frac{b^2}{2\ep}$, we see that $\dis \frac{1}{(\Zap)_a}\Delta(\Ztapbar) \in \LinftyDh$.  Now by using \lemref{lem:Deltaapprox} we see that $\Delta(\Dap\Ztbar) \in \LinftyDh $, $\Delta(\Dapabs\Ztbar) \in \LinftyDh $ and $\Delta(\Dapbar\Ztbar) \in \LinftyDh $

\item Observe that 
\begin{align*}
\Delta\brac{\Dap^2\Ztbar} = \Delta\cbrac{\brac{\pap\frac{1}{\Zap}}\Dap\Ztbar} + \Delta\brac{\frac{1}{\Zap^2}\pap\Ztapbar}
\end{align*}
Hence we have $\Delta\brac{\Dap^2\Ztbar} \in \LtwoDh$. Similarly we can also show $\Delta\brac{\Dapabs^2\Ztbar} \in \LtwoDh$ and $\Delta\brac{\Dapbar^2\Ztbar} \in \LtwoDh$. Now using \lemref{lem:Deltaapprox} we see that $\Delta\brac{\Dapabs\Dapbar\Ztbar} \in \LtwoDh$ and $\Dapabs_a \Delta(\Dapbar\Ztbar) \in \LtwoDh$. This in particular implies $(\sqrt{\Aone})_a \Dapabs_a \Delta(\Dapbar\Ztbar) \in \LtwoDh$ which is part of $\EDeltafour$.

\item Following the proof in \secref{sec:quantEDelta} we see that $\Delta(\Dap\Ztbar) \in \WcalDh\cap\CcalDh$, $\Delta(\Dapabs\Ztbar) \in \WcalDh\cap\CcalDh$ and $\Delta(\Dapbar\Ztbar) \in \WcalDh\cap\CcalDh$. Hence using \lemref{lem:Deltaapprox} we have $\dis \frac{1}{\Zapabs_a}\Delta(\Ztapbar) \in \CcalDh$. As $(\sqrt{\Aone})_a \in \Wcal$  we now obtain $\dis \brac{\frac{\sqrt{\Aone}}{\Zapabs}}\la \Delta(\Ztapbar) \in \CcalDh$ and hence we have controlled the second term of $\EDeltatwo$. 

\item Following the proof in \secref{sec:quantEDelta} we see that $\Delta(\Dapabs\Aone) \in \LtwoDh$ and hence we have $\Delta(\Aone) \in \WcalDh$ and $\Delta(\sqrt{\Aone}) \in \WcalDh$ 

\item From \secref{sec:quantEDelta} we see that $\Delta(\bvarap) \in \LinftyDh\cap\HhalfDh$, $\Delta(\Dapabs\bvarap) \in \LtwoDh$. Hence we also obtain $\Delta(\bvarap) \in \WcalDh$.

\item As $(\sigma\pap\Th)_a \in \HhalfDh$ and $\Th_a \in \Ltwo$, by interpolation we see that $(\sigma^\twobythree\pap\Th)_a \in \LtwoDeltaonebythree$ and $(\sigma^\onebythree\Th)_a \in \LinftyDeltaonebysix\cap\HhalfDeltaonebysix$

\item Following the proof in Sec 5.1 of \cite{Ag19} we see that $\dis \brac{\sigma^\twobythree\pap^2\frac{1}{\Zap}}\la \in \LtwoDeltaonebythree$, $\dis \brac{\sigma^\twobythree\pap^2\frac{1}{\Zapabs}}\la \in \LtwoDeltaonebythree$ and similarly  $\dis \brac{\frac{\sigma^\twobythree}{\Zapabs}\pap^2\w}\la \in \LtwoDeltaonebythree$, $\dis \brac{\sigma^\twobythree\pap\Dapabs\w}\la \in \LtwoDeltaonebythree$. In the same way we have $\dis \brac{\sigma^\onebythree\pap\frac{1}{\Zap}}\la \in \LinftyDeltaonebysix\cap\HhalfDeltaonebysix$, $\dis \brac{\sigma^\onebythree\pap\frac{1}{\Zapabs}}\la \in \LinftyDeltaonebysix\cap\HhalfDeltaonebysix$ and also $\dis \brac{\sigma^\onebythree\Dapabs\w}\la \in \LinftyDeltaonebysix\cap\HhalfDeltaonebysix$.

\item Following the proof in Sec 5.1 of \cite{Ag19} we see that $\dis \brac{\frac{\sigma}{\Zapabs}\pap^2\Th}\la \in \LtwoDh$ and from this we easily get $(\sigma\pap\Dap\Th)_a \in \LtwoDh$. Similarly we also get $\dis \brac{\frac{\sigma}{\Zapabs}\pap^3\frac{1}{\Zapabs}}\la \in \LtwoDeltahalf$ and $\dis \brac{\frac{\sigma}{\Zapabs^2}\pap^3\w}\la \in \LtwoDeltahalf$. 

\item Following the proof in Sec 5.1 of \cite{Ag19} we see that we have $\dis \brac{\frac{\sigma^\half}{\Zapabs^\half}\pap^2\frac{1}{\Zapabs}}\la \in \LtwoDeltahalf$, $\dis \brac{\frac{\sigma^\half}{\Zapabs^\threebytwo}\pap^2\w}\la \in \LtwoDeltahalf$ and $\dis \brac{\frac{\sigma^\half}{\Zapabs^\half}\pap\Th}\la \in \LtwoDeltahalf $. Similarly we also obtain control of $\dis \brac{\sigma^{\half}  \Zapabs^\half \pap\frac{1}{\Zap}}\la \in \WcalDeltahalf$,  $\dis \brac{\sigma^{\half}  \Zapabs^\half \pap\frac{1}{\Zapabs}}\la \in \WcalDeltahalf$ and  $\dis \brac{\frac{\sigma^\half}{\Zapabs^\half}\pap\w}\la \in \WcalDeltahalf$

\item We now recall from \eqref{form:Zttbar} 
\begin{align*}
\Zttbar -i = -i\frac{\Aone}{\Zap} + \sigma\Dap\Th
\end{align*}
Taking derivatives on both sides and applying $\Delta$ we get
\begin{align*}
\Delta(\Zttapbar) = -i\Delta\brac{\Aone\pap\frac{1}{\Zap}} -i\Delta(\Dap\Aone) + (\sigma\pap\Dap\Th)_a
\end{align*}
Hence we see that $\Delta(\Zttapbar) \in \LtwoDh$. As $\Dt\Ztapbar = -\bvarap\Ztapbar + \Zttapbar$ and $\Delta(\bvarap) \in \LinftyDh$ we obtain $\Delta(\Dt\Ztapbar) \in \LtwoDeltahalf$ which is the first term of $\EDeltatwo$.

\item By exactly following the proof of $\dis \frac{\sigma^\half}{\Zapabs^\fivebytwo}\pap^2\Ztapbar \in \Ltwo$ in Sec 5.1 of \cite{Ag19}, we easily get that $\dis \brac{\frac{\sigma^\half}{\Zapabs^\half}\pap\Dapabs\Dapbar\Ztbar}\la \in \LtwoDh$. Note that this is the last term of $\EDeltafour$.

\item Now we use can \propref{prop:LinftyHhalf} with $\dis \w = \frac{1}{\Zapabs_a}$ and $\dis f = \brac{\frac{\sigma^\half}{\Zapabs^\threebytwo}\pap\Ztapbar}\la $ and we obtain $\dis \brac{\frac{\sigma^\half}{\Zapabs^\threebytwo}\pap\Ztapbar}\la \in \LinftyDeltahalf\cap\HhalfDh$. Hence $\EDeltatwo $ is controlled.

\end{enumerate}

\item Controlling $\EDeltathree$: We prove this step by step.
\begin{enumerate}
\item By \eqref{form:Th} we see that $\Delta(\Th) \in \LtwoDh$. Similarly from \eqref{form:DtTh} and \eqref{eq:IdHilTh} we obtain $\Delta(\Dt\Th) \in \LtwoDh$. Hence the first term of $\EDeltathree$ is controlled. 

\item As we have $\dis \Delta\brac{\Dap\frac{1}{\Zap}} \in \CcalDh$, using \lemref{lem:Deltaapprox} and \lemref{lem:CWDelta} we see that we have $\dis \Delta\brac{\Dapbar\frac{1}{\Zap}} \in \CcalDh$. Now following the proof of $\dis \frac{1}{\Zapabs^2}\pap\Aone \in \Ccal$ in Sec 5.1 of \cite{Ag19}, we easily get $\dis \Delta\brac{\frac{1}{\Zapabs^2}\pap\Aone} \in \CcalDh$.

\item Following the proof of $\dis \Dap\frac{1}{\Zap} \in \Ccal$ in Sec 5.1 of \cite{Ag19}, we see that $\dis \Delta\brac{\frac{\Th}{\Zapabs}} \in \CcalDh$. Hence by \lemref{lem:Deltaapprox} and \lemref{lem:CWDelta} we see that $\dis \frac{1}{\Zapabs_a}\Delta(\Th) \in \CcalDh$ and $\dis \brac{\frac{\sqrt{\Aone}}{\Zapabs}}\la \Delta(\Th) \in \CcalDh$. Hence the second term of $\EDeltathree$ is controlled. 

\item As $\dis \brac*[\Bigg]{\frac{\sigma^\half}{\Zap^\threebytwo}\pap^2\frac{1}{\Zap}}\la \in \CcalDh$, by using \lemref{lem:Deltaapprox} we see that $\dis \brac{\frac{\sigma^\half}{\Zapabs^\threebytwo}\pap^2\frac{1}{\Zap}}\la \in \CcalDh$. Now by following the proof of $\dis \nobrac{\frac{\sigma^\half}{\Zapabs^\threebytwo}\pap^2\frac{1}{\Zap}} \in \Ccal$ in Sec 5.1 of \cite{Ag19}, we easily obtain $\dis \brac{\frac{\sigma^\half}{\Zapabs^\threebytwo}\pap\Th}\la \in \CcalDh$, $\dis \brac{\frac{\sigma^\half}{\Zapabs^\threebytwo}\pap^2\frac{1}{\Zapabs}}\la \in \CcalDeltahalf$ and $\dis \brac{\frac{\sigma^\half}{\Zapabs^\fivebytwo}\pap^2\w}\la \in \CcalDeltahalf$. Hence $\EDeltathree$ is controlled.

\end{enumerate}

\item Controlling $\EDeltafour$: Observe that we have already controlled the second and third term of $\EDeltafour$. Now by following the proof of $\dis \frac{\sigma}{\Zapabs^2}\pap^3\frac{1}{\Zap} \in \Ccal$ and $\dis \frac{\sigma}{\Zapabs^2}\pap^2\Th \in \Ccal$ in Sec 5.1 of \cite{Ag19}, we see that $(\sigma\Dapbar\Dap\Th)_a \in \CcalDh$. Hence by applying $\Dapbar$ in the formula \eqref{form:Zttbar} we obtain
\begin{align*}
\Delta(\Dap\Zttbar) = -i\Delta\brac{\Aone\Dapbar\frac{1}{\Zap}} -i\Delta\brac{\frac{1}{\Zapabs^2}\Aone} + (\sigma\Dapbar\Dap\Th)_a
\end{align*}
Hence $\Delta(\Dapbar\Zttbar) \in \CcalDh$. Now by applying $\Delta$ to the equation $\Dt\Dapbar\Ztbar = -(\Dapbar\Ztbar)^2 + \Dapbar\Zttbar$ we see that $\Delta(\Dt\Dapbar\Ztbar) \in \HhalfDeltahalf$ which shows that $\EDeltafour$ is controlled. Hence proved. 

\end{enumerate}
\end{proof}

\medskip
\section{Proof of \thmref{thm:convergence} and \corref{cor:examplenew}}\label{sec:proof}
 
We now prove our main results stated in \secref{sec:results}. We first prove a basic lemma.

\begin{lemma}\label{lem:equivlowerorder}
Let $(\Z,\Zt)(t)$ be a smooth solution to the water wave equation \eqref{eq:systemone} for $\sigma \geq 0$ in the time interval $[0,T]$ for $T>0$, satisfying $(\Zap-1,\frac{1}{\Zap} - 1, \Zt) \in \Linfty([0,T], H^{s+\half}(\Rsp)\times H^{s+\half }(\Rsp)\times H^{s }(\Rsp))$ for all $s\geq 3$. Let 
\begin{align*}
R_0 =  \norm[2]{\frac{1}{\Zap} - 1}(0) + \norm[2]{\Zt}(0) \qq \tx{ and } \q M_0 = \norm[\infty]{\Zap}(0) + \norm[2]{\frac{1}{\Zap} - 1}(0) + \norm[2]{\Zt}(0)
\end{align*}
Then there exists a universal increasing function $F:[0,\infty) \to [0,\infty)$ so that
\begin{enumerate}
\item If $\sigma > 0$, then
\begin{align*}
& \sup_{t\in [0,T]} \cbrac{ \norm[H^{3.5}]{\Zap - 1}(t) + \norm[H^{3.5}]{\frac{1}{\Zap} - 1}(t) + \norm[H^{3}]{\Zt}(t)} \\
& \leq F\brac{M_0 + \sup_{t\in[0,T]}\Ecalsigma(t) + T + \sigma + \frac{1}{\sigma}}
\end{align*}
\item If $\sigma = 0$, then 
\begin{align*}
& \sup_{t\in [0,T]} \cbrac{ \norm[H^3]{\Zap - 1}(t) + \norm[H^3]{\frac{1}{\Zap} - 1}(t) + \norm[H^{3.5}]{\Zt}(t)} \\
& \leq F\brac*[\bigg]{M_0 + \sup_{t\in[0,T]}\Ecalhigh(t) +  \sup_{t\in[0,T]}\Ecalaux(t) + T + 1}
\end{align*}
\item For $\sigma \geq 0$ we define
\begin{align*}
S(t) & =  \norm[2]{\frac{1}{\Zap} - 1}(t) + \norm[\Linfty\cap\Hhalf]{\frac{1}{\Zap}}(t) + \norm[2]{\Zt}(t) + \norm[\Linfty\cap\Hhalf]{\Zt}(t) + \norm[2]{\Aone - 1}(t)  \\
& \quad  + \norm[2]{\bvar}(t)  + \norm[\Linfty\cap\Hhalf]{\bvar}(t) + \norm[2]{\bvarap}(t) + \norm[2]{\Ztt}(t) + \norm[\Linfty\cap\Hhalf]{\Ztt}(t)
\end{align*}
Then we have the estimate
\begin{align*}
\sup_{t\in[0,T]} S(t) \leq F\brac*[\bigg]{R_0 + \sup_{t\in[0,T]}\Ecalsigma(t) + \sigma + T + 1}
\end{align*}
\end{enumerate}
\end{lemma}
\begin{proof}
The first estimate was already proved in Lemma 6.2 of \cite{Ag19}. Let us now prove the second estimate. Observe from the definition of $\Ehigh$ in \eqref{def:Ehigh}, \propref{prop:equivEhighEcalhigh} and \propref{prop:equivEsigmaEcalsigma} that $\Ecalhigh$ controls $\Ecalsigma\vert_{\sigma = 0}$. 
Hence from Lemma 6.2 of \cite{Ag19} we get
\begin{align*}
& \sup_{t\in [0,T]} \cbrac{  \norm[H^{1.5}]{\Zap - 1}(t) + \norm[H^{1.5}]{\frac{1}{\Zap} - 1}(t) + \norm[H^{2}]{\Zt}(t) } \\
& \leq F\brac{M_0 + \sup_{t\in[0,T]}\Ecalhigh(t) + T + 1}
\end{align*}
Hence we see that for each $t\in [0,T]$, we have $\Zap \in \Linfty$, $\Dapabs \Zap \in \Ltwo$ and hence $\Zap \in \Wcal$. Now using the energy $\Ecalaux(t)$ we see that 
\begin{align*}
\norm[2]{\pap^3 \frac{1}{\Zap}} \lesssim \norm[\infty]{\Zap}^\fivebytwo\norm[2]{\frac{1}{\Zap^\fivebytwo}\pap^3\frac{1}{\Zap}}
\end{align*}
and by \lemref{lem:CW}
\begin{align*}
\norm[\Ccal]{\pap^2\Ztapbar} \lesssim \norm[\Wcal]{\Zap}^\sevenbytwo \norm[\Ccal]{\frac{1}{\Zap^\sevenbytwo}\pap^2\Ztapbar}
\end{align*}
We control $\Zap - 1 \in H^{3}$ similarly. This proves the second estimate.

Let us now prove the third estimate. The proof follows in a similar fashion as the proof of Lemma 6.2 of \cite{Ag19}. To simplify the calculation we define
\begin{align*}
M = R_0 + \sup_{t\in[0,T]}\Ecalsigma(t) + \sigma + T + 1
\end{align*}
and we write $a\lesssim b$ if $a \leq C(M)b$ where $C(M)$ is a constant which depends only on $M$. Hence we need to prove that $\sup_{t\in[0,T]} S(t) \lesssim 1$. First observe that
\begin{align*}
\norm[\infty]{\frac{1}{\Zap}}(0) \lesssim 1 + \norm[2]{\frac{1}{\Zap} - 1}^\half(0) \norm[2]{\pap\frac{1}{\Zap}}^\half(0) \lesssim 1
\end{align*}
Now the evolution equation \eqref{eq:systemone} gives us
\begin{align*}
(\pt + \bvar\pap)\Zap = \Ztap -\bvarap \Zap = \brac{\Dap\Zt - \bvarap}\Zap
\end{align*}
Hence for all $0\leq t \leq T$ we have the estimate
\begin{align*}
\norm[\infty]{\frac{1}{\Zap}}(t) & \leq \norm[\infty]{\frac{1}{\Zap}}(0) \exp\cbrac{{\int_0^t \brac{\norm[\infty]{\Dap\Zt}(s) + \norm[\infty]{\bvarap}(s)} \diff s}} \lesssim 1
\end{align*}
Define
\begin{align*}
f(t) = \norm[2]{\frac{1}{\Zap} - 1}^2(t) + \norm[2]{\Zt}^2(t) + 1
\end{align*}
Observe that $f(0) \lesssim 1$. Now it was proved as part of the proof of Lemma 6.2 in \cite{Ag19} that
\begin{align*}
\norm[2]{b} + \norm[\infty]{b} + \norm[2]{\bvarap} + \norm[2]{\Aone - 1} + \norm[2]{\Ztt} \lesssim f^\half
\end{align*}
and that $\pt f \lesssim f$. Hence we have 
\begin{align*}
& \sup_{t \in [0,T]} \cbrac{\norm[2]{\frac{1}{\Zap} - 1}(t) + \norm[2]{\Zt}(t) +  \norm[2]{b}(t) + \norm[\infty]{b}(t) + \norm[2]{\bvarap}(t) + \norm[2]{\Aone - 1} (t)+ \norm[2]{\Ztt}(t) } \\
& \lesssim 1
\end{align*}
The other terms are easily controlled. Observe that
\begin{align*}
\norm[\Hhalf]{\frac{1}{\Zap}} \lesssim  \norm[2]{\frac{1}{\Zap} - 1}^\half \norm[2]{\pap\frac{1}{\Zap}}^\half \lesssim 1
\end{align*}
We also see that $\norm[\Linfty\cap\Hhalf]{\Zt} \lesssim \norm[2]{\Zt}^\half \norm[2]{\Ztap}^\half \lesssim 1$. Similarly we get $\norm[\Linfty\cap\Hhalf]{b} \lesssim 1$ and $\norm[\Linfty\cap\Hhalf]{\Ztt} \lesssim 1$. This finishes the proof of the lemma. 

\end{proof}
 
We can now prove our main theorem.  
\begin{proof}[Proof of \thmref{thm:convergence}]
In the proof we will frequently denote a constant which depends only on $L$ by $C(L)$. Let $M_0, N_0$ be defined as
\begin{align*}
M_0 = \norm[\infty]{\Zap}(0) + \norm[2]{\frac{1}{\Zap} - 1}(0) + \norm[2]{\Zt}(0) \qq \tx{ and } \q N_0 = \Ecalaux(\Z,\Zt)(0)
\end{align*}
Let $\ep\geq 0$ and consider the mollified initial data given by $(\Z^\ep,\Zt^\ep)(0) = (P_\ep\conv\Z, P_\ep\conv\Zt)(0)$ where $P_\ep$ is the Poisson kernel \eqref{eq:Poissonkernel}. Observe that there exists an $0<\ep_0 \leq 1$ small enough so that for all $0 < \ep\leq \ep_0$ we have
\begin{align*}
 \Ecalhigh(\Z^\ep,\Zt^\ep)(0), \Ecalsigma(\Z^\ep,\Zt^\ep)(0)  \leq 2L \qq \tx{ and } \quad \EcalDelta(\Z^{\ep, \sigma},\Z^{\ep})(0) \leq 2\EcalDelta(\Z^{\sigma},\Z)(0)
\end{align*}
and we also have
\begin{align*}
\norm[\infty]{\Zap^\ep}(0) + \norm[2]{\frac{1}{\Zap^\ep} - 1}(0) + \norm[2]{\Zt^\ep}(0) \leq M_0 \qq \tx{ and } \quad \Ecalaux(\Z^\ep,\Zt^\ep)(0) \leq  N_0
\end{align*}

Now let $0\leq \ep \leq \ep_0$. As $(\Zap-1,\frac{1}{\Zap} - 1, \Zt)(0) \in H^{3.5}(\Rsp)\times H^{3.5}(\Rsp)\times H^{3}(\Rsp)$, by \thmref{thm:existence} there exists a $T_1>0$ depending only on $L$ so that we have a unique solution $(\Z^{\ep,\sigma},\Zt^{\ep,\sigma})(t)$ in $[0,T_1]$ to \eqref{eq:systemone} with surface tension $\sigma$ and initial data $(\Z^\ep,\Zt^\ep)(0)$, satisfying $(\Zap^{\ep,\sigma} -1,\frac{1}{\Zap^{\ep,\sigma}} - 1, \Zt^{\ep,\sigma}) \in C^l([0,T_1], H^{3.5 - \threebytwo l}(\Rsp)\times H^{3.5 - \threebytwo l}(\Rsp)\times H^{3 - \threebytwo l}(\Rsp))$ for $l=0,1$ and we also have 
\begin{align*}
 \sup_{t\in [0,T_1]} \Ecalsigma (\Z^{\ep,\sigma},\Zt^{\ep,\sigma})(t) \leq C(L)
\end{align*}
We denote the solution for $\ep = 0$ in this case simply by $(\Z^{\sigma},\Zt^{\sigma})(t)$. Using \lemref{lem:equivlowerorder} we see that 
\begin{align*}
\sup_{t\in [0,T_1]} \cbrac{ \norm[H^{3.5}]{\Zap^{\ep,\sigma} - 1}(t) + \norm[H^{3.5}]{\frac{1}{\Zap^{\ep,\sigma}} - 1}(t) + \norm[H^{3}]{\Zt^{\ep,\sigma}}(t)} \leq C(L,M_0, T_1, \sigma)
\end{align*}
Also from Corollary 7.9 of \cite{Ag19} we obtain
\begin{align*}
 \sup_{t\in [0,T_1]}\cbrac{ \norm[H^{2}]{\Zap^{\ep,\sigma} - \Zap^{\sigma}}(t) + \norm[H^{2}]{\frac{1}{\Zap^{\ep,\sigma}} - \frac{1}{\Zap^{\sigma}}}(t) + \norm[H^{1.5}]{\Zt^{\ep,\sigma} - \Zt^{\sigma}}(t)} \to 0 \q \tx{ as } \ep \to 0
\end{align*}

Now let $0<\ep\leq \ep_0$. Using Theorem 2.3 of \cite{Wu19} we see that there exists a $T_\ep>0$ so that we have a unique smooth solution $(\Z^{\ep},\Zt^{\ep})(t)$ in $[0,T_{\ep}]$ to \eqref{eq:systemone} with zero surface tension and initial data $(\Z^\ep,\Zt^\ep)(0)$, satisfying $(\Zap^\ep-1,\frac{1}{\Zap^\ep} - 1, \Zt^\ep) \in \Linfty([0,T], H^{s}(\Rsp)\times H^{s}(\Rsp)\times H^{s + \half}(\Rsp))$ for all $s\geq 4$. Therefore by \thmref{thm:aprioriEhigh} we see that for all $t\in[0,T_\ep)$ we have
\begin{align*}
\frac{\diff \Ehigh(\Z^\ep,\Zt^\ep)(t)}{\diff t} \leq P(\Ehigh(\Z^\ep,\Zt^\ep)(t)) 
\end{align*}
Also from \thmref{thm:aprioriEaux} we have
\begin{align*}
\frac{\diff \Eaux(\Z^\ep,\Zt^\ep)(t)}{\diff t} \leq P(\Ecalhigh(\Z^\ep,\Zt^\ep)(t))\Eaux(\Z^\ep,\Zt^\ep)(t) \qq 
\end{align*}
Hence using \propref{prop:equivEhighEcalhigh}, \propref{prop:equivEauxEcalaux}, \lemref{lem:equivlowerorder} and the blow up criterion of Theorem 2.3 of \cite{Wu19}, we see that there exists a $T_2>0$ depending only on $L$ so that the solution  $(\Z^{\ep},\Zt^{\ep})(t)$ in fact exists in the time interval $[0,T_2]$ and we have the estimates 
\begin{align*}
 \sup_{t\in [0,T_2]} \Ecalhigh (\Z^{\ep},\Zt^{\ep})(t) \leq C(L) \qq \tx{ and } \q  \sup_{t\in [0,T_2]} \Ecalaux (\Z^{\ep},\Zt^{\ep})(t) \leq C(L,N_0)
\end{align*}
along with
\begin{align*}
\sup_{t\in [0,T_2]} \cbrac{ \norm[H^{3}]{\Zap^{\ep} - 1}(t) + \norm[H^{3}]{\frac{1}{\Zap^{\ep}} - 1}(t) + \norm[H^{3.5}]{\Zt^{\ep}}(t)} \leq C(L,M_0, N_0, T_2)
\end{align*}

Let $T = \min\cbrac{T_1,T_2}>0$. As the solutions $(\Z^{\ep},\Zt^{\ep})(t)$ and $(\Z^{\ep,\sigma},\Zt^{\ep,\sigma})(t)$ are smooth for $0<\ep\leq \ep_0$, we can now use \thmref{thm:aprioriEDelta} to see that for all $t \in [0,T)$ we have
\begin{align*}
\frac{d}{dt}\EDelta(Z^{\ep,\sigma}, Z^{\ep})(t) \leq C(L) \EDelta(Z^{\ep,\sigma}, Z^{\ep})(t) 
\end{align*}
Hence using \propref{prop:equivEDeltaEcalDelta} and the fact that $\EcalDelta(\Z^{\ep, \sigma},\Z^{\ep})(0) \leq 2\EcalDelta(\Z^{\sigma},\Z)(0)$ we see that there are constants $C_0,C_1$ depending only on $L$ so that
\begin{align*}
\sup_{t\in[0,T]}\EcalDelta(\Z^{\ep,\sigma},\Z^{\ep})(t) \leq C_1e^{C_0T}\EcalDelta(\Z^\sigma,\Z)(0)
\end{align*}
We now let $\ep \to 0$ and by using Theorem 3.9 of \cite{Wu19} and Lemma 5.9 of \cite{Wu97} we see that there exists a unique solution $(\Z,\Zt)(t)$ in $[0,T]$ to \eqref{eq:systemone} with zero surface tension and initial data $(\Z,\Zt)(0)$, satisfying $(\Zap -1,\frac{1}{\Zap} - 1, \Zt) \in C^l([0,T], H^{3 - \half l}(\Rsp)\times H^{3 - \half l}(\Rsp)\times H^{3.5 - \half l}(\Rsp))$ for $l=0,1$, along with the estimate
\begin{align*}
\sup_{t\in[0,T]}\EcalDelta(\Z^{\sigma},\Z)(t) \leq C_1e^{C_0T}\EcalDelta(\Z^\sigma,\Z)(0)
\end{align*}
Hence proved. 
\end{proof}

\begin{proof}[Proof of \corref{cor:examplenew}]

Without loss of generality we assume that $c=1$ and define $\dis \tau = \frac{\sigma}{\ep^{3/2}}$ which implies that $\tau\leq 1$.  To simplify the proof we will suppress the dependence of $\Mconst$ in the inequalities i.e. when we write $a \lesssim b$, what we mean is that there exists a constant $C(M)$ depending only on $\Mconst$ such that $a\leq C(M)b$.  

If $\sigma \geq 0, 0<\ep\leq 1$ and $\tau\leq 1$, then it was already shown in the proof of \corref{cor:example} that we have $\Ecalsigma(\Z^{\epsilon,\sigma}, \Zt^{\epsilon,\sigma})(0) \lesssim 1$. It is also clear from the definition of $M$ in \eqref{eq:M} that we have $\Ecalhigh(\Z^{\epsilon}, \Zt^{\epsilon})(0) \lesssim 1$. Hence by \thmref{thm:convergence} there exists $T_2, C_0>0$ depending only on $M$ so that the solutions  $(\Z^{\epsilon,\sigma}, \Zt^{\epsilon,\sigma})(t)$ exist in the time interval $[0,T_2]$, we have $\sup_{t\in[0,T_2]} \Ecalhigh(\Z^{\ep},\Zt^{\ep})(t) \lesssim 1$ and $\sup_{t\in[0,T_2]} \Ecalsigma(\Z^{\ep,\sigma},\Zt^{\ep,\sigma})(t) \lesssim 1$, and we have 
\begin{align*}
\sup_{t\in[0,T_2]}\EcalDelta(\Z^{\ep,\sigma},\Z^\ep)(t) \lesssim e^{C_0T_2}\EcalDelta(\Z^{\ep,\sigma},\Z^\ep)(0)
\end{align*}
Let us now prove each of the statements in the corollary.

\medskip

\noindent \textbf{Part 1:} From the above equation it is clear that we only need to prove $\EcalDelta(\Z^{\ep,\sigma},\Z^\ep)(0) \lesssim \tau$. As we only need to prove the estimates for $t=0$, we will suppress the time dependence of the solutions e.g. we will write $(\Z\conv P_\epsilon, \Zt\conv P_\epsilon)\vert_{t=0}$ by $(\Z,\Zt)_\ep$ for simplicity.

Recall that $\EcalDelta(\Z^{\ep,\sigma}, \Z^{\ep})(0) =  \EcalDeltaone(\Z^{\ep,\sigma}, \Z^{\ep})(0) + \EcalDeltatwo(\Z^{\ep,\sigma}, \Z^{\ep})(0)  + \sigma(\Ecalaux)_\ep(0)$ where the term $\sigma(\Ecalaux)_\ep(0)$ is given by
\begin{align*}
\sigma(\Ecalaux)_\ep(0) & = \norm[\infty]{\brac{\sigma^\half\Zap^\half\pap\frac{1}{\Zap}}_{\n \ep}}^2 + \norm*[\Bigg][2]{\brac*[\Bigg]{\frac{\sigma^\half}{\Zap^\half}\pap^2\frac{1}{\Zap}}_{\n\ep}}^2 + \norm*[\Bigg][2]{\brac*[\Bigg]{\frac{\sigma^\half}{\Zap^\fivebytwo}\pap^3\frac{1}{\Zap}}_{\n\ep}}^2 \\
& \quad + \norm*[\Bigg][2]{\brac*[\Bigg]{\frac{\sigma^\half}{\Zap^\half}\pap\Ztapbar}_{\n\ep}}^2 +  \norm*[\Bigg][2]{\brac*[\Bigg]{\frac{\sigma^\half}{\Zap^\fivebytwo}\pap^2\Ztapbar}_{\n\ep}}^2 + \norm*[\Bigg][\Hhalf]{\brac*[\Bigg]{\frac{\sigma^\half}{\Zap^\sevenbytwo}\pap^2\Ztapbar}_{\n\ep}}^2
\end{align*}
Now looking at the proof of \corref{cor:example} in \cite{Ag19}, it was shown there all the terms involving $\sigma$ in $\Ecalsigma(\Z^{\ep,\sigma}, \Zt^{\ep,\sigma})(0)$ are bounded above by $\tau$. Hence this directly implies that $ \EcalDeltaone(\Z^{\ep,\sigma}, \Z^{\ep})(0) + \EcalDeltatwo(\Z^{\ep,\sigma}, \Z^{\ep})(0) \lesssim \tau$. Hence we only need to prove that $\sigma(\Ecalaux)_\ep(0) \lesssim \tau$. Now observe that the only terms of $\sigma(\Ecalaux)_\ep(0)$  we really need to control are $\dis \brac*[\bigg]{\frac{\sigma^\half}{\Zap^\fivebytwo}\pap^3\frac{1}{\Zap}}_{\n\ep} \in \Ltwo$ and $\dis \brac*[\bigg]{\frac{\sigma^\half}{\Zap^\sevenbytwo}\pap^2\Ztapbar}_{\n\ep} \in \Hhalf$, as all the other terms are controlled as part of the terms involving $\sigma$ in $\Ecalsigma(\Z^{\ep,\sigma}, \Zt^{\ep,\sigma})(0)$. Let us now control these two terms.

\begin{enumerate}[leftmargin =*, align=left]

\item We use $\dis \sup_{y<0}\norm[\Lone(\Rsp, \diff x)]{\frac{1}{\Psi_z^3}\partial_z^3 \brac{\frac{1}{\Psi_z}}}  \lesssim 1 $ and \lemref{lem:conv} to get  
\begin{align*}
\norm*[\Bigg][2]{\sigma\pap\brac*[\Bigg]{\frac{1}{\Zap^3}\pap^3\frac{1}{\Zap}}_{\n\ep}} \lesssim \tau
\end{align*}
From this we also obtain
\begin{align*}
\norm[2]{\brac{\frac{\sigma}{\Zap^3}\pap^4\frac{1}{\Zap}}_{\n\ep}} & \lesssim \norm[\infty]{\brac{\frac{1}{\Zap}\pap\frac{1}{\Zap}}_{\n\ep}}\norm[2]{\brac{\frac{\sigma}{\Zap}\pap^3\frac{1}{\Zap}}_{\n\ep}} +  \norm*[\Bigg][2]{\sigma\pap\brac*[\Bigg]{\frac{1}{\Zap^3}\pap^3\frac{1}{\Zap}}_{\n\ep}} \\
& \lesssim \tau
\end{align*}
Now we have
\begin{align*}
& \norm[2]{\brac{\sigma\Zapabs^2\pap\brac{\frac{1}{\Zapabs^5}\pap^3\frac{1}{\Zap}} }_{\n\ep}} \\
& \lesssim \norm[2]{\brac{\frac{\sigma}{\Zap^3}\pap^4\frac{1}{\Zap}}_{\n\ep}} + \norm[\infty]{\brac{\sigma^\onebythree\pap\frac{1}{\Zapabs}}_{\n\ep}}\norm[2]{\brac{\frac{\sigma^\twobythree}{\Zapabs^2}\pap^3\frac{1}{\Zap}}_{\n\ep}} \\
&  \lesssim \tau
\end{align*}
where the last two terms in the product above were controlled in the proof of \corref{cor:example} in \cite{Ag19}. Hence by using integration by parts we see that
\begin{align*}
\norm*[\Bigg][2]{\brac*[\Bigg]{\frac{\sigma^\half}{\Zap^\fivebytwo}\pap^3\frac{1}{\Zap}}_{\n\ep}}^2 \lesssim \norm[2]{\brac{\frac{1}{\Zapabs^2}\pap^2\frac{1}{\Zap}}_{\n\ep}} \norm[2]{\brac{\sigma\Zapabs^2\pap\brac{\frac{1}{\Zapabs^5}\pap^3\frac{1}{\Zap}} }_{\n\ep}}  \lesssim \tau
\end{align*}

\item Using \propref{prop:Leibniz} we see that
\begin{align*}
& \norm*[\Bigg][\Hhalf]{\brac*[\Bigg]{\frac{\sigma^\half}{\Zap^\sevenbytwo}\pap^2\Ztapbar}_{\n\ep}} \\
& \lesssim \sigma^\half \norm[\infty]{\frac{1}{(\Zap)_\ep}}^\sevenbytwo\norm[\Hhalf]{(\pap^2\Ztapbar)_\ep}  + \sigma^\half\norm[\infty]{\frac{1}{(\Zap)_{\ep}}}^\fivebytwo\norm[2]{\brac{\pap\frac{1}{\Zap}}_{\n\ep}}\norm[2]{\brac{\pap^2\Ztapbar}_{\ep}} \\
& \lesssim \tau^\half
\end{align*}
This finishes the proof of $\EcalDelta(\Z^{\ep,\sigma},\Z^{\ep})(0) \lesssim \tau$ and hence we have proved the first part. 

\end{enumerate}

\medskip

\noindent \textbf{Part 2:} From the definition of $\Fcal_{\Delta}$ in \eqref{eq:Fcal} we observe that
\begin{align*}
\Fcal_{\Delta}(\Z^{\epsilon,\sigma}, \Z)(t) \leq  \Fcal_{\Delta}(\Z^{\epsilon,\sigma}, \Z^\ep)(t) +  \Fcal_{\Delta}(\Z^{\epsilon}, \Z)(t)
\end{align*}
Now from Theorem 3.7 of \cite{Wu19} we see that $\sup_{t\in[0,T]} \Fcal_{\Delta}(\Z^{\epsilon}, \Z)(t) \to 0 $ as $\ep \to 0$. Hence it is enough to show that $\sup_{t\in[0,T]} \Fcal_{\Delta}(\Z^{\epsilon,\sigma}, \Z^\ep)(t) \to 0 $ as $\tau \to 0$. Let $(\Z^\ep,\Zt^\ep)(t)$ be the solution A and let $(\Zt^{\ep,\sigma},\Zt^{\ep,\sigma})$ be solution B. Hence from \lemref{lem:quantM} we see that $\Util = U_{\htil} = U_{h_a}^{-1}U_{h_b}$ is bounded on $\Ltwo$ and $\Hhalf$, and the same is true for $\Util^{-1}$. Hence we see that 
\begin{align*}
\Fcal_{\Delta}(\Z^{\epsilon,\sigma}, \Z^\ep)(t) \lesssim \widetilde{\Fcal}_{\Delta}(\Z^{\epsilon,\sigma}, \Z^\ep)(t)
\end{align*}
where 
\begin{align}\label{eq:Ftilcal}
\begin{split}
\widetilde{\Fcal}_{\Delta}(\Z^{\epsilon,\sigma}, \Z^\ep)(t) & = \norm[\Hhalf]{\Delta(\Zt)} + \norm[\Hhalf]{\Delta(\Ztt)} + \norm[\Hhalf]{\Delta\brac{\frac{1}{\Zap}}} + \norm[2]{\Delta(\hal\compose\hinv)} \\
& \quad + \norm[2]{\Delta(\Dap\Zt)} + \norm[2]{\Delta(\Aone)} + \norm[2]{\Delta(\bvarap)}
\end{split}
\end{align}
We now control each of these terms. Hence using \lemref{lem:equivlowerorder}  we see that
\begin{align*}
& \norm[\Linfty\cap\Hhalf]{\Delta(\Zt)}^2  \lesssim \norm[2]{\Delta(\Zt)}\norm[2]{\pap\Delta(\Zt)} \lesssim (\norm[2]{(\Zt)_a} + \norm[2]{(\Zt)_b})\norm[2]{\pap\Delta(\Zt)} \lesssim \EcalDelta(\Z^{\ep,\sigma},\Z^{\ep})(t)^\half
\end{align*}
This implies that $ \norm[\Linfty\cap\Hhalf]{\Delta(\Zt)} \to 0$ as $\tau \to 0$. By the same argument we also see that $\norm[\Linfty\cap\Hhalf]{\Delta(\Ztt)} + \norm[\Linfty\cap\Hhalf]{\Delta\brac{\frac{1}{\Zap}}} \to 0$ as $\tau \to 0$. Now using \lemref{lem:equivlowerorder} we observe that
\begin{align*}
\norm[2]{\Delta(\Dap\Zt)} \lesssim \norm[\infty]{\Delta\brac{\frac{1}{\Zap}}} + \norm[2]{\Delta(\Ztap)}
\end{align*}
Therefore $\norm[2]{\Delta(\Dap\Zt)} \to 0$ as $\tau \to 0$. Now we recall the formula of $\Aone$ from \eqref{eq:systemone} which is $\Aone = 1 - \Imag \sqbrac{\Zt,\Hil}\Ztapbar$. Hence by using \lemref{lem:Delta}, \lemref{lem:equivlowerorder}, \propref{prop:commutator} and \propref{prop:HilHtilcaldiff} we see that
\begin{align*}
\norm[2]{\Delta(\Aone)} \lesssim \norm[\Hhalf]{\Delta(\Zt)} + \norm*[\big][\infty]{\htilap - 1} + \norm[2]{\pap\Delta(\Ztbar)}
\end{align*}
This shows that $\norm[2]{\Delta(\Aone)} \to 0$ as $\tau \to 0$. Now applying $\Real(\Id - \Hil)$ to the formula of $\bvarap$ from \eqref{form:bvarapnew} we see that
\begin{align*}
\bvarap = \Real\cbrac{\sqbrac{\frac{1}{\Zap},\Hil}\Ztap + 2\Dap\Zt + \sqbrac{\Zt,\Hil}\brac*[\Big]{\pap \frac{1}{\Zap}} }
\end{align*}
Hence by using \lemref{lem:Delta},\lemref{lem:equivlowerorder}, \propref{prop:commutator} and \propref{prop:HilHtilcaldiff} we see that
\begin{align*}
\norm[2]{\Delta(\bvarap)} &  \lesssim \norm[\Hhalf]{\Delta\brac{\frac{1}{\Zap}}} +  \norm*[\big][\infty]{\htilap - 1} +  \norm[2]{\pap\Delta(\Zt)} + \norm[2]{\Delta(\Dap\Zt)} + \norm[\Hhalf]{\Delta(\Zt)} \\
& \quad + \norm[2]{\pap\Delta\brac{\frac{1}{\Zap}}}
\end{align*}
This shows that $\norm[2]{\Delta(\bvarap)} \to 0$ as $\tau \to 0$. Finally using \lemref{lem:timederiv} we see that
\begin{align*}
\frac{\diff}{\diff t}\norm[2]{\Delta(\hal\compose\hinv)}^2 & \lesssim \norm[2]{\Delta(\hal\compose\hinv)}^2 + \norm[2]{\Delta(\hal\compose\hinv)}\norm[2]{(\Dt)_a\Delta(\hal\compose\hinv)} \\
&  \lesssim \norm[2]{\Delta(\hal\compose\hinv)}^2 + \norm[2]{\Delta(\hal\compose\hinv)}\norm[2]{\Delta(h_{t\al}\compose\hinv)} \\
&  \lesssim \norm[2]{\Delta(\hal\compose\hinv)}^2 + \norm[2]{\Delta(\hal\compose\hinv)}\norm[2]{\Delta(\bvarap)} \\
& \lesssim \norm[2]{\Delta(\hal\compose\hinv)}^2 + \norm[2]{\Delta(\bvarap)}^2
\end{align*} 
Now by integrating the above inequality and using the fact that $\norm[2]{\Delta(\bvarap)} \to 0$ as $\tau \to 0$, we see that $\norm[2]{\Delta(\hal\compose\hinv)} \to 0$ as $\tau \to 0$. Hence proved. 
\end{proof}

\medskip
\section{Appendix A}\label{sec:appendixA}

Here we collect the main identities and estimates proved in the appendix of \cite{Ag19} that we need in this paper. For the proof of the estimates we refer to the appendix of \cite{Ag19}.  Let $\Dt = \pt + \bvar\pap$ where $\bvar$ is given by \eqref{eq:systemone} and recall that $\sqbrac{f,g ; h}$ is defined as 
\begin{align*}
\sqbrac{f_1, f_2;  f_3}(\ap) = \frac{1}{i\pi} \int \brac{\frac{f_1(\ap) - f_1(\bp)}{\ap - \bp}}\brac{\frac{f_2(\ap) - f_2(\bp)}{\ap-\bp}} f_3(\bp) \diff \bp
\end{align*}

\begin{prop}\label{prop:tripleidentity}
Let $f,g,h \in \mathcal{S}(\Rsp)$. Then we have the following identities
\begin{enumerate}
\item $h\pap[f,\Hil]\pap g = \sqbrac{h\pap f,\Hil}\pap g + \sqbrac{f, \Hil}\pap\brac{h\pap g} - \sqbrac{h, f ; \pap g} $ 

\item $\Dt [f,\Hil]\pap g = \sqbrac{\Dt f, \Hil}\pap g + \sqbrac{f, \Hil}\pap(\Dt g) - \sqbrac{\bvar, f; \pap g} $  
\end{enumerate}
\end{prop}
\begin{proof}
See \cite{Ag19} for a proof. 
\end{proof}

\begin{prop}\label{prop:Coifman} 
Let $H \in C^1(\Rsp),A_i \in C^1(\Rsp) $ for $i=1,\cdots m$ and $F\in C^\infty(\Rsp)$. Define 
\begin{align*}
C_1(H,A,f)(x) & = p.v. \int F\brac{\frac{H(x)-H(y)}{x-y}}\frac{\Pi_{i=1}^{m}(A_i(x) - A_i(y))}{(x-y)^{m+1}}f(y)\diff y \\
C_2(H,A,f)(x) & = p.v. \int F\brac{\frac{H(x)-H(y)}{x-y}}\frac{\Pi_{i=1}^{m}(A_i(x) - A_i(y))}{(x-y)^{m}} \partial_y f(y)\diff y
\end{align*}
then there exists constants $c_1,c_2,c_3,c_4$ depending only on $F$ and $\norm[\infty]{H'}$ so that
\begin{enumerate}
\item $\norm[2]{C_1(H,A,f)} \leq c_1\norm[\infty]{A_1'}\cdots\norm[\infty]{A_m'}\norm[2]{f}$

\item $\norm[2]{C_1(H,A,f)} \leq c_2\norm[2]{A_1'}\norm[\infty]{A_2'}\cdots\norm[\infty]{A_m'}\norm[\infty]{f}$ 

\item $\norm[2]{C_2(H,A,f)} \leq c_3\norm[\infty]{A_1'}\cdots\norm[\infty]{A_m'}\norm[2]{f}$ 

\item $\norm[2]{C_2(H,A,f)} \leq c_4\norm[2]{A_1'}\norm[\infty]{A_2'}\cdots\norm[\infty]{A_m'}\norm[\infty]{f}$ 
\end{enumerate}
\end{prop}
\begin{proof}
See \cite{Ag19} for a proof.
\end{proof}

\begin{prop}\label{prop:Lemarie}
Let $T:\Dcalsp(\Rsp) \to \Dcalsp'(\Rsp)$ be a linear operator with kernel $K(x,y)$ such that on the open set $\{(x,y):x\neq y\} \subset \Rsp\times\Rsp$, $K(x,y)$ is a function satisfying
\begin{align*}
\abs{K(x,y)} \leq \frac{C_0}{\abs{x-y}} \quad \tx{ and } \quad \abs{\grad_x K(x,y)} \leq \frac{C_0}{\abs{x-y}^2}
\end{align*}
where $C_0$ is a constant. If $T$ is continuous on $\Ltwo(\Rsp)$ with $\norm[\Ltwo\to\Ltwo]{T} \leq C_0$  and if $T(1) =0$, then $T$ is bounded on $\dot{H}^s$ for $0<s<1$ with $\norm[\dot{H}^s\to\dot{H}^s]{T} \lesssim C_0$ 
\end{prop}
\begin{proof}
See \cite{Ag19} for a proof.
\end{proof}

\begin{prop}\label{prop:Hardy}
Let $f \in \Scalsp(\Rsp)$. Then we have
\begin{enumerate}

\item $\norm[\infty]{f} \lesssim \norm[H^s]{f}$ if $s>\frac{1}{2}$ and for $s=\half$ we have $\norm[BMO]{f} \lesssim \norm[\Hhalf]{f}$

\item $
\begin{aligned}[t]
\int \abs{\frac{f(\ap) - f(\bp)}{\ap - \bp}}^2 \diff \bp \lesssim \norm[2]{f'}^2 
\end{aligned}
$

\item $
\begin{aligned}[t]
\norm[\Ltwo(\Rsp, \diff \ap)]{\sup_{\bp}\abs{\frac{f(\ap) - f(\bp)}{\ap - \bp}}} \lesssim \norm[2]{f'}
\end{aligned}
$

\item $
\begin{aligned}[t]
\norm[\Hhalf]{f}^2 = \frac{1}{2\pi}\int\!\! \!\int \abs{\frac{f(\ap) - f(\bp)}{\ap - \bp}}^2 \diff \bp \diff\ap
\end{aligned}
$

\item $
\begin{aligned}[t]
\norm[\Ltwo(\Rsp^2, \diff\ap\diff\bp)]{\pbp\brac{\frac{f(\ap) - f(\bp)}{\ap - \bp}} } \lesssim \norm[\Hhalf]{f'}
\end{aligned}
$
\end{enumerate}
\end{prop}
\begin{proof}
See \cite{Ag19} for a proof. 
\end{proof}

\begin{prop}\label{prop:commutator}
Let $f,g \in \mathcal{S}(\Rsp)$ with $s,a\in \Rsp$ and $m,n \in \Zsp$. Then we have the following estimates
\begin{enumerate}
\item $\norm*[\big][2]{\papabs^s\sqbrac{f,\Hil}(\papabs^{a} g )} \lesssim  \norm*[\big][BMO]{\papabs^{s+a}f}\norm[2]{g}$ \quad  for $s,a \geq 0$

\item $\norm*[\big][2]{\papabs^s\sqbrac{f,\Hil}(\papabs^{a} g )} \lesssim  \norm*[\big][2]{\papabs^{s+a}f}\norm[BMO]{g}$ \quad  for $s\geq 0$ and $a>0$

\item $\norm*[\big][2]{\sqbrac*[\big]{f,\papabs^\half}g } \lesssim \norm*[\big][BMO]{\papabs^\half f}\norm[2]{g}$

\item $\norm*[\big][2]{\sqbrac*[\big]{f,\papabs^\half}(\papabs^\half g) } \lesssim \norm*[\big][BMO]{\papabs f}\norm[2]{g}$

\item $\norm[\Linfty\cap\Hhalf]{\pap^m\sqbrac{f,\Hil}\pap^n g} \lesssim \norm*[\big][2]{\pap^{(m+n+1)}f}\norm[2]{g}$ \quad  for $m,n \geq 0$

\item $\norm[2]{\partial_{\ap}^m\sqbrac{f,\Hil}\partial_{\ap}^n g} \lesssim \norm*[\infty]{\partial_\ap^{(m+n)} f}\norm[2]{g}$ \quad  for $m,n \geq 0$

\item $\norm[2]{\partial_{\ap}^m\sqbrac{f,\Hil}\partial_{\ap}^n g} \lesssim \norm*[2]{\partial_\ap^{(m+n)} f}\norm[\infty]{g}$ \quad for $m\geq 0$ and $n\geq 1$

\item $\norm[2]{\sqbrac{f,\Hil}g} \lesssim \norm[2]{f'}\norm[1]{g}$
\end{enumerate}
\end{prop}
\begin{proof}
See \cite{Ag19} for a proof.
\end{proof}

\begin{prop}\label{prop:Leibniz}
Let $f,g,h \in \mathcal{S}(\Rsp)$ with $s,a\in \Rsp$ and $m,n \in \Zsp$. Then we have the following estimates
\begin{enumerate}
\item $\norm[2]{\papabs^s (fg)} \lesssim  \norm[2]{\papabs^s f}\norm[\infty]{g} + \norm[\infty]{f}\norm[2]{\papabs^s g}$ \quad for $s > 0$
\item $\norm[\Hhalf]{fg} \lesssim \norm[\Hhalf]{f}\norm[\infty]{g} + \norm[\infty]{f}\norm[\Hhalf]{g}$
\item $\norm[\Hhalf]{fg} \lesssim \norm[2]{f'}\norm[2]{g} + \norm[\infty]{f}\norm[\Hhalf]{g}$
\end{enumerate}
\end{prop}
\begin{proof}
See \cite{Ag19} for a proof.
\end{proof}

\begin{prop}\label{prop:triple}
Let $f,g,h \in \mathcal{S}(\Rsp)$ . Then we have the following estimates
\begin{enumerate}

\item $\norm[2]{\sqbrac{f,g;h}} \lesssim \norm[2]{f'}\norm[2]{g'}\norm[2]{h}$

\item $\norm[2]{\pap\sqbrac{f,\sqbrac{g,\Hil}}h} \lesssim \norm[2]{f'}\norm[2]{g'}\norm[2]{h}$

\item $\norm[2]{\sqbrac{f,g; h'}} \lesssim \norm[\infty]{f'}\norm[\infty]{g'}\norm[2]{h}$

\item $\norm[\Hhalf]{\sqbrac{f,g; h'}} \lesssim \norm[\infty]{f'}\norm[\infty]{g'}\norm[\Hhalf]{h}$

\item $\norm[\Linfty\cap\Hhalf]{\sqbrac{f,g;h}} \lesssim \norm[\infty]{f'}\norm[2]{g'}\norm[2]{h}$

\end{enumerate}
\end{prop}
\begin{proof}
See \cite{Ag19} for a proof.
\end{proof}

\begin{prop} \label{prop:LinftyHhalf}
Let $f \in \mathcal{S}(\Rsp)$ and let $w$ be a smooth non-zero weight with $w,\frac{1}{w} \in \Linfty(\Rsp) $ and $w' \in \Ltwo(\Rsp)$. Then 
\begin{enumerate}
\item $\norm[\infty]{f}^2 \lesssim \norm[2]{\frac{f}{w}}\norm[2]{wf'}$
\item $\norm[\Linfty\cap\Hhalf]{f}^2 \lesssim \norm[2]{\frac{f}{w}}\norm[2]{(wf)'} +  \norm[2]{\frac{f}{w}}^2\norm[2]{w'}^2$
\end{enumerate}
\end{prop}
\begin{proof}
See \cite{Ag19} for a proof.
\end{proof}

\begin{prop}\label{prop:Hhalfweight}
Let $f,g \in \mathcal{S}(\Rsp)$ and let $w,h \in \Linfty(\Rsp)$ be smooth functions with $w',h' \in \Ltwo(\Rsp)$. Then 
\begin{align*}
\norm[\Hhalf]{fwh} \lesssim \norm[\Hhalf]{fw}\norm[\infty]{h} + \norm[2]{f}\norm[2]{(wh)'} + \norm[2]{f}\norm[2]{w'}\norm[\infty]{h}
\end{align*}
If in addition we assume that $w$ is real valued then
\begin{align*}
\norm[2]{fgw} \lesssim \norm[\Hhalf]{fw}\norm[2]{g} + \norm[\Hhalf]{gw}\norm[2]{f} + \norm[2]{f}\norm[2]{g}\norm[2]{w'} 
\end{align*}
\end{prop}
\begin{proof}
See \cite{Ag19} for a proof. 
\end{proof}

\begin{lemma}\label{lem:conv}
Let $K_\ep$ be the Poisson kernel from \eqref{eq:Poissonkernel}. If $\f \in L^q(\Rsp)$, then for $s\geq 0$ an integer we have
\begin{align*}
\norm[p]{(\pap^sf)\conv P_\ep} \lesssim \norm[q]{f}\ep^{-s-\brac{\frac{1}{q} - \frac{1}{p}}} \quad \tx{ for } 1\leq q\leq p \leq \infty
\end{align*}
Similarly for $s \in\Rsp, s\geq 0$ we have
\begin{align*}
\norm[p]{(\papabs^s f)\conv P_\ep} \lesssim \norm[q]{f}\ep^{-s-\brac{\frac{1}{q} - \frac{1}{p}}} \quad \tx{ for } 1\leq q\leq p \leq \infty
\end{align*}
\end{lemma}
\begin{proof}
See \cite{Ag19} for a proof.
\end{proof}

\medskip
\section{Appendix B}\label{sec:appendixB}

We now prove some estimates which are new. We first need some identities for holomorphic functions. 

\begin{prop}\label{prop:Dt^2hol}
Let $f \in \Scalsp(\Rsp)$ with $\Pa f = 0$. Then we have the following identities 
\begin{enumerate}[leftmargin =*, align=left]
\item $(\Id - \Hil)\Dt f = (\Id - \Hil)\brac{\Zt\Dap f}$ 

\item  
$\begin{aligned}[t]
 & (\Id - \Hil)\Dt^2 f \\
 & = 2\sqbrac{\Pa\brac{\frac{\Zt}{\Zap}},\Hil}\pap(\Dt f) - \sqbrac{\Pa\brac{\frac{\Zt}{\Zap}}, \Pa\brac{\frac{\Zt}{\Zap}}; \pap f} \\
& \quad + \frac{1}{4}(\Id - \Hil)\cbrac{\brac{\sqbrac{\frac{1}{\Zap},\Hil}\Ztap }\sqbrac{\Zt,\Hil}\Dap f } - \frac{1}{4}(\Id - \Hil)\cbrac{\brac{\sqbrac{\Zt,\Hil}\pap\frac{1}{\Zap}}^2 f} \\
& \quad + \half \sqbrac{\sqbrac{\Zt, \sqbrac{\Zt,\Hil} }\pap\frac{1}{\Zap},\Hil }\pap\brac{\frac{f}{\Zap}} + \sqbrac{\Ztt,\Hil}\Dap f
\end{aligned}
$
\end{enumerate}
\end{prop}
\begin{proof}
These identities essentially says that the material derivative of a holomorphic function remain essentially holomorphic. These identities are proved in \cite{Wu15}, see Appendix B for the first identity and section 4 for the second identity.
\end{proof}

\begin{cor}\label{cor:Coifman}
Let $H \in C^1(\Rsp),A_i \in C^1(\Rsp) $ for $i=1,\cdots m$ and let $\delta>0$ be such that 
\begin{align*}
\delta \leq \abs{\frac{H(x)-H(y)}{x-y} } \leq \frac{1}{\delta} \quad \tx{ for all } x\neq y
\end{align*}
Let $0\leq k \leq m+1$ and define
\begin{align*}
T(A,f)(x) & = p.v. \int \frac{\Pi_{i=1}^{m}(A_i(x) - A_i(y))}{(x-y)^{m+1-k}(H(x)-H(y))^k }f(y)\diff y
\end{align*}
then we have the estimates
\begin{enumerate}
\item $\norm[2]{T(A,f)} \leq C(\norm[\infty]{H'},\delta) \norm[\infty]{A_1'}\cdots\norm[\infty]{A_m'}\norm[2]{f}$ 
\item $\norm[2]{T(A,f)} \leq C(\norm[\infty]{H'},\delta) \norm[2]{A_1'}\norm[\infty]{A_2'}\cdots\norm[\infty]{A_m'}\norm[\infty]{f}$
\end{enumerate}

\end{cor}
\begin{proof}
If $k=0$, then the result follows directly from \propref{prop:Coifman}. If $k\geq 1$, we choose a smooth function $F$ with compact support such that $F(x) = 0$ if $\abs{x}\leq \frac{\delta}{2}$, $F(x) = 0$ if $\abs{x}\geq \frac{2}{\delta}$ and $F(x) = x^{-k} $ if $\delta\leq \abs{x} \leq \frac{1}{\delta}$. The result now follows from  \propref{prop:Coifman}.
\end{proof}

We now prove some basic estimates for the operators defined in \secref{sec:aprioriEcalDelta}. Recall from \eqref{eq:Hcal}
\begin{align*}
(\Hcal f)(\ap) &= \frac{1}{i\pi} p.v. \int \frac{\htilbp(\bp)}{\htil(\ap) - \htil(\bp)}f(\bp) \diff \bp \\
 (\Hcaltil f)(\ap) &=\frac{1}{i\pi} p.v. \int \frac{1}{\htil(\ap) - \htil(\bp)}f(\bp) \diff \bp
\end{align*}

\begin{prop}\label{prop:HcalHtilest}
Let $\Hcal, \Htil$ be defined as in \eqref{eq:Hcal} and let $f  \in \Scalsp(\Rsp)$. Assume that there is a constant $L>0$ so that
\begin{align*}
\frac{1}{L} \leq \abs{\frac{\htil(x) - \htil(y)}{x-y}} \leq L \quad \tx{ for all } x\neq y
\end{align*}
We will suppress the dependence of $L$ i.e. we write $a\lesssim b$ instead of $a\leq C(L)b$. With this notation we have the following estimates
\begin{enumerate}
\item $\norm[2]{\Hcal(f)} \lesssim \norm[2]{f}$ and $\norm*[2]{\Htil(f)} \lesssim \norm[2]{f}$. 
\item  $\norm[\Hhalf]{\Hcal(f)} \lesssim  \norm[\Hhalf]{f}$
\end{enumerate}
\end{prop}
\begin{proof}
The proofs are quite straighforward
\begin{enumerate}[leftmargin =*, align=left]
\item The estimate $\norm*[2]{\Htil(f)} \lesssim \norm[2]{f}$ follows directly from \corref{cor:Coifman}. Hence we see that $\norm*[2]{\Hcal(f)} \lesssim \norm*[2]{\htilap f}  \lesssim \norm[2]{f}$. 
\item We see that $\Hcal(1) = 0$ and $\norm[\Ltwo \to \Ltwo]{\Hcal} \lesssim 1$. The kernel of $\Hcal$ is
\begin{align*}
K(\ap,\bp) = \frac{\htilbp(\bp)}{\htil(\ap) - \htil(\bp)}
\end{align*}
From this we see that 
\begin{align*}
\abs{K(\ap,\bp)} \lesssim \frac{1}{\abs{\ap - \bp}} \qq \tx{ and } \quad \abs{\grad_{\ap}K(\ap,\bp)} \lesssim \frac{1}{\abs{\ap - \bp}^2}
\end{align*}
Hence from \propref{prop:Lemarie} we obtain $\norm[\Hhalf \to \Hhalf]{\Hcal} \lesssim 1$. 
\end{enumerate}
\end{proof}

We now prove the main estimate used to handle the difference of terms in the energy estimate of $\EcalDelta$. Recall that $\sqbrac{f_1,f_2 ; f_3}_{\htil}$ is defined as 
\begin{align*}
\sqbrac{f_1, f_2;  f_3}_{\htil}(\ap) = \frac{1}{i\pi} \int \brac{\frac{f_1(\ap) - f_1(\bp)}{\htil(\ap) - \htil(\bp)}}\brac{\frac{f_2(\ap) - f_2(\bp)}{\htil(\ap) - \htil(\bp)}} f_3(\bp) \diff \bp
\end{align*}

\begin{prop}\label{prop:HilHtilcaldiff}
Let $\Hil$ be the Hilbert transform and let $\Hcal, \Htil$ be defined as in \eqref{eq:Hcal} and let $f,f_1,f_2,f_3,g \in \Scalsp(\Rsp)$. Assume that there is a constant $L>0$ so that
\begin{align*}
\frac{1}{L} \leq \abs{\frac{\htil(x) - \htil(y)}{x-y}} \leq L \quad \tx{ for all } x\neq y
\end{align*}
We will suppress the dependence of $L$ i.e. we write $a\lesssim b$ instead of $a\leq C(L)b$. With this notation we have the following estimates
\begin{enumerate}
\item $\norm[2]{(\Hil - \Hcal) f} \lesssim \norm*[\infty]{\htilap -1}\norm[2]{f}$ 
\item $\norm[\Hhalf]{(\Hil - \Hcal) f} \lesssim \norm*[\infty]{\htilap -1}\norm[\Hhalf]{f}$ 
\item $\norm*[2]{\sqbrac*{f,\Hil - \Htil}g} \lesssim \norm*[\infty]{\htilap -1}\norm[2]{f'}\norm[1]{g}$
\item $\norm*[2]{\sqbrac*{f,\Hil - \Htil}g} \lesssim \norm*[\infty]{\htilap -1}\norm[\Hhalf]{f}\norm[2]{g}$
\item $\norm*[2]{\sqbrac*{f,\Hil - \Htil}\pap g} \lesssim \norm*[\infty]{\htilap -1}\norm[\infty]{f'}\norm[2]{g}$
\item $\norm*[2]{\sqbrac*{f,\Hil - \Htil}\pap g} \lesssim \norm*[\infty]{\htilap -1}\norm[2]{f'}\norm[\infty]{g}$
\item $\norm*[2]{\sqbrac*{f,\Hil - \Htil}\pap g} \lesssim \norm*[\infty]{\htilap -1}\norm[2]{f'}\norm[\Hhalf]{g}$
\item $\norm*[2]{\sqbrac*{f,\Hil - \Htil}\pap g} \lesssim \norm*[\infty]{\htilap -1}\cbrac*[\big]{\norm[\Hhalf]{f'}\norm[2]{g} + \norm[2]{f'g}}$
\item  $\norm*[2]{\pap\sqbrac*{f_1,\sqbrac*{f_2,\Hil - \Htil}}\pap f_3} \lesssim \norm*[\infty]{\htilap - 1}\norm[2]{f_1'}\norm[2]{f_2'}\norm[2]{f_3'}$ 
\item $\norm*[\Linfty\cap\Hhalf]{\sqbrac*{f,\Hil - \Htil}g} \lesssim \norm*[\infty]{\htilap -1}\norm[2]{f'}\norm[2]{g}$
\item $\norm*[\Hhalf]{\sqbrac*{f,\Hil - \Htil}\pap g} \lesssim \norm*[\infty]{\htilap -1}\norm[\infty]{f'}\norm[\Hhalf]{g}$
\item $\norm[\Linfty\cap\Hhalf]{\sqbrac*{f_1,f_2; f_3} - \sqbrac*{f_1,f_2; f_3}_{\htil} } \lesssim \norm*[\infty]{\htilap -1}\norm[\infty]{f_1'}\norm[2]{f_2'}\norm[2]{f_3}$
\item $\norm[2]{\sqbrac*{f_1,f_2; \pap f_3} - \sqbrac*{f_1,f_2; \pap f_3}_{\htil} } \lesssim \norm*[\infty]{\htilap -1}\norm[\infty]{f_1'}\norm[\infty]{f_2'}\norm[2]{f_3}$
\item $\norm[\Hhalf]{\sqbrac*{f_1,f_2; \pap f_3} - \sqbrac*{f_1,f_2; \pap f_3}_{\htil} } \lesssim \norm*[\infty]{\htilap -1}\norm[\infty]{f_1'}\norm[\infty]{f_2'}\norm[\Hhalf]{f_3}$
\end{enumerate}
\end{prop} 
\begin{proof}
We first observe that from \propref{prop:HcalHtilest} we get $\norm[2]{\Hcal(f)} \lesssim \norm[2]{f}$, $\norm[\Hhalf]{\Hcal(f)} \lesssim  \norm[\Hhalf]{f}$ and $\norm*[2]{\Htil(f)} \lesssim \norm[2]{f}$. To simplify the calculations we define
\[ 
\begin{array}{*2{>{\displaystyle}c}}
\lpar \qquad \dis F(a,b) = \frac{f(a)-f(b)}{a-b}  & \quad  F_{h}(a,b) = \frac{f(a)-f(b)}{\htil(a)-\htil(b)} \\
G(a,b) = \frac{g(a)-g(b)}{a-b} & \quad  G_{h}(a,b) = \frac{g(a)-g(b)}{\htil(a)-\htil(b)} \\
F_i(a,b) = \frac{f_i(a)-f_i(b)}{a-b}  & \quad  F_{ih}(a,b) = \frac{f_i(a)-f_i(b)}{\htil(a)-\htil(b)} \\
H(a,b) = \frac{(\htil(a)-a)-(\htil(b)-b)}{a-b} & \quad  H_h(a,b) = \frac{(\htil(a)-a)-(\htil(b)-b)}{\htil(a)-\htil(b)} 
\end{array}
\]
We have the identities
\begin{align*}
\frac{F(\ap,s)-F(\bp,s)}{\ap-\bp} & = \frac{F(\ap,\bp)-F(\bp,s)}{\ap-s} \\
\frac{H_h(\ap,s)-H_h(\bp,s)}{\ap-\bp} & = \frac{1}{\htil(\ap) - \htil(s)}\cbrac{H(\ap,\bp) - H_h(\bp,s)\brac{\frac{\htil(\ap)-\htil(\bp)}{\ap-\bp}} }
\end{align*}

\begin{enumerate}[leftmargin =*, align=left]

\item We write $(\Hil - \Hcal) = (\Hil - \Htil) + (\Htil - \Hcal)$. Observe that
\begin{align*}
(\Htil - \Hcal) f = \Htil ((1-\htilap)f)
\end{align*}
Hence as $\Htil $ is bounded on $\Ltwo$ we have $\norm*[2]{(\Htil-\Hcal)f} \lesssim \norm*[\infty]{\htilap-1}\norm[2]{f} $. Now we have
\begin{align*}
((\Hil - \Htil)f)(\ap) = \frac{1}{i\pi} \int \brac{\frac{1}{\ap-\bp} - \frac{1}{\htil(\ap)-\htil(\bp)} }f(\bp) \diff\bp = \frac{1}{i\pi} \int \frac{H_h(\ap,\bp)}{\ap-\bp} f(\bp) \diff \bp
\end{align*} 
Now using  \corref{cor:Coifman} we see that $\norm*[2]{(\Hil-\Htil)f} \lesssim \norm*[\infty]{\htil'-1}\norm[2]{f} $. Hence the required estimate follows.

\item Observe that $(\Hil - \Hcal)(1) = 0$ and that the kernel of this operator is
\begin{align*}
K(\ap,\bp) = \frac{1}{\ap-\bp} - \frac{\htilbp(\bp)}{\htil(\ap) - \htil(\bp)}
\end{align*}
Hence this kernel satisfies 
\begin{align*}
\abs{K(\ap.\bp)} \lesssim \frac{\norm*[\infty]{\htilap -1}}{\abs{\ap-\bp}} \qquad \abs{\nabla_\ap K(\ap.\bp)} \lesssim \frac{\norm*[\infty]{\htilap -1}}{\abs{\ap-\bp}^2}
\end{align*}
and by the first estimate of this proposition we also have  $\norm[\Ltwo \to \Ltwo]{\Hil - \Hcal}\lesssim \norm*[\infty]{\htilap -1}$. Hence by  \propref{prop:Lemarie} we have boundedness on $\Hhalf$ with $\norm[\Hhalf \to \Hhalf]{\Hil - \Hcal}\lesssim \norm*[\infty]{\htilap -1}$.

\item Note that
\begin{align*}
(\sqbrac*{f,\Hil - \Htil}g)(\ap) = \frac{1}{i\pi} \int F(\ap,\bp)H_h(\ap,\bp) g(\bp) \diff \bp
\end{align*}
and hence by Cauchy Schwarz we have
\begin{align*}
\abs*{\sqbrac*{f,\Hil - \Htil}g}(\ap) \lesssim \norm*[\infty]{\htilap -1}\norm[1]{g}^\half\brac{\int \abs{F(\ap,\bp)}^2\abs{g(\bp)} \diff \bp}^\half
\end{align*}
The estimate now follows from Hardy's inequality in \propref{prop:Hardy}.

\item From the earlier computation we see that
\begin{align*}
(\abs*{\sqbrac*{f,\Hil - \Htil}g})(\ap) \lesssim \norm*[\infty]{\htilap -1}\norm[2]{g}\brac{\int \abs{F(\ap,\bp)}^2 \diff \bp}^\half
\end{align*}
We now obtain the estimate easily as $\dis \int \!\!\! \int \abs{F(\ap,\bp)}^2 \diff\bp\diff\ap \lesssim \norm[\Hhalf]{f}^2$ from \propref{prop:Hardy}. 

\item We observe
\begin{align*}
& (\sqbrac*{f,\Hil - \Htil}\pap g)(\ap) \\
& = \frac{1}{i\pi} \int F(\ap,\bp)H_h(\ap,\bp) g_\bp(\bp) \diff \bp \\
& = \frac{1}{i\pi} \int \frac{H_h(\ap,\bp)}{\ap-\bp} f_\bp(\bp)g(\bp) \diff \bp + \frac{1}{i\pi} \int \frac{F(\ap,\bp)}{\htil(\ap)-\htil(\bp)} (\htil_\bp(\bp) - 1)g(\bp) \diff \bp \\
& \quad - \frac{1}{i\pi} \int \frac{F(\ap,\bp)}{\ap - \bp}H_h(\ap,\bp) g(\bp) \diff \bp - \frac{1}{i\pi} \int \frac{F(\ap,\bp)H_h(\ap,\bp)}{\htil(\ap)-\htil(\bp)} \htil_\bp(\bp)g(\bp) \diff \bp
\end{align*}
The estimate now follows from  \corref{cor:Coifman}.

\item This also follows from the computation above and  \corref{cor:Coifman}.

\item We see that 
\begin{align*}
(\sqbrac*{f,\Hil - \Htil}\pap g)(\ap) & = \frac{1}{i\pi} \int F(\ap,\bp)H_h(\ap,\bp) g_\bp(\bp) \diff \bp \\
& =  \frac{1}{i\pi} \int \pbp\brac{F(\ap,\bp)H_h(\ap,\bp)} (g(\ap)-g(\bp)) \diff \bp \\
\end{align*}
Now as $\dis F(\ap,\bp) =  \frac{f(\ap)-f(\bp)}{\ap-\bp} $, if the derivative falls on $f$ then we can use estimate 4 of this proposition. All other terms are bounded point-wise by
\begin{align*}
 \norm*[\infty]{\htilap -1} \int \abs{\frac{f(\ap)-f(\bp)}{\ap-\bp} }\abs{\frac{g(\ap)-g(\bp)}{\ap-\bp}} \diff \bp
\end{align*}
Now use Cauchy Schwarz and Hardy's inequality from \propref{prop:Hardy}.

\item We see that
\begin{align*}
&\sqbrac*{f,\Hil - \Htil}\pap g  \\
& = \frac{1}{i\pi} \int F(\ap,\bp)H_h(\ap,\bp) g_\bp(\bp) \diff \bp \\
& = -\frac{1}{i\pi} \int (\pbp F(\ap,\bp))H_h(\ap,\bp) g(\bp) \diff \bp -\frac{1}{i\pi} \int \cbrac{\pbp H_h(\ap,\bp)} f_\bp(\bp)g(\bp) \diff\bp \\
& \quad   - \frac{1}{i\pi} \int \brac{\frac{F(\ap,\bp) -f_\bp(\bp) }{\ap-\bp}} \cbrac{(\ap-\bp)\pbp H_h(\ap,\bp)}  g(\bp) \diff \bp
\end{align*}
For the first term we use Cauchy Schwarz inequality along with \propref{prop:Hardy}. The second term is easily handled by \corref{cor:Coifman}. For the last term we observe that $\pbp F(\ap,\bp) = \frac{F(\ap,\bp) -f_\bp(\bp) }{\ap-\bp}$ and then use Cauchy Schwarz inequality along with \propref{prop:Hardy}. 

\item We have
\begin{align*}
&\pap\sqbrac*{f_1,\sqbrac*{f_2,\Hil - \Htil}}\pap f_3 \\
& = \frac{1}{i\pi}\pap\int (f_1(\ap)-f_1(\bp))F_2(\ap,\bp)H_h(\ap,\bp) f_{3 \bp}(\bp) \diff\bp \\
& = f_1'(\ap) \brac{ \frac{1}{i\pi}\int F_2(\ap,\bp)H_h(\ap,\bp) f_{3 \bp}(\bp) \diff\bp} \\
& \quad + f_2'(\ap) \brac{ \frac{1}{i\pi}\int F_1(\ap,\bp)H_h(\ap,\bp) f_{3 \bp}(\bp) \diff\bp} \\
& \quad +  \frac{1}{i\pi}\int F_1(\ap,\bp)F_2(\ap,\bp) \cbrac*[\big]{(\ap-\bp)\pap H_h(\ap,\bp) - H_h(\ap,\bp)} f_{3 \bp}(\bp) \diff\bp
\end{align*}
Each of the terms are now easily controlled by using Cauchy Schwarz inequality and using \propref{prop:Hardy} and \propref{prop:triple}.

\item Note that
\begin{align*}
\sqbrac*{f,\Hil - \Htil}g = \frac{1}{i\pi} \int F(\ap,s)H_h(\ap,s) g(s) \diff s
\end{align*}
Hence the $\Linfty$ estimate follows immediately from Cauchy Schwarz inequality and \propref{prop:Hardy}. We now show the $\Hhalf$ estimate by using identity 4 in \propref{prop:Hardy}. We see that
\begin{align*}
& \frac{(\sqbrac*{f,\Hil - \Htil}g)(\ap) - (\sqbrac*{f,\Hil - \Htil}g)(\bp)}{\ap-\bp} & \\
& = \frac{1}{i\pi}\int \frac{F(\ap,s)H_h(\ap,s) - F(\bp,s)H_h(\bp,s) }{\ap-\bp}g(s) \diff s \\
& =  \frac{1}{i\pi}\int \frac{F(\ap,s) - F(\bp,s) }{\ap-\bp}H_h(\ap,s)g(s) \diff s +  \frac{1}{i\pi}\int \frac{H_h(\ap,s) - H_h(\bp,s) }{\ap-\bp}F(\bp,s)g(s) \diff s \\
\end{align*}
Now using the identities mentioned at the start of the proof of this proposition we get
\begin{align*}
& \frac{(\sqbrac*{f,\Hil - \Htil}g)(\ap) - (\sqbrac*{f,\Hil - \Htil}g)(\bp)}{\ap-\bp} & \\
& = \frac{F(\ap,\bp)}{i\pi} \int \frac{H_h(\ap,s)}{\ap-s}g(s) \diff s  - \frac{1}{i\pi}\int \frac{H_h(\ap,s)}{\ap-s}F(\bp,s)g(s)\diff s \\
& \quad + \frac{H(\ap,\bp)}{i\pi}\int \frac{F(\bp,s)}{\htil(\ap)-\htil(s)}g(s) \diff s \\
& \quad -\frac{1}{i\pi}\brac{\frac{\htil(\ap)-\htil(\bp)}{\ap-\bp}} \int \frac{F(\bp,s)}{\htil(\ap)-\htil(s)}H_h(\bp,s)g(s)\diff s \\
& = {\rm I + II + III + IV}
\end{align*}
We can control each of the terms. The first term is controlled as
\begin{align*}
\lpar \qq& \norm[\Ltwo(\Rsp\times\Rsp, \diff\ap \diff\bp)]{\frac{F(\ap,\bp)}{i\pi} \int \frac{H_h(\ap,s)}{\ap-s}g(s) \diff s} &&\\
& \lesssim \norm[\Ltwo(\diff \bp)]{\norm[\Linfty(\diff \ap)]{F(\ap,\bp)}\norm[\Ltwo(\diff \ap)]{\int \frac{H_h(\ap,s)}{\ap -s}g(s)\diff s}} \\
& \lesssim \norm*[\infty]{\htilap -1}\norm[2]{f'}\norm[2]{g}
\end{align*}
where we used \corref{cor:Coifman} and \propref{prop:Hardy}. For the second term we have
\begin{align*}
\lpar \qq & \norm[\Ltwo(\Rsp\times\Rsp, \diff\ap \diff\bp)]{\frac{1}{i\pi}\int \frac{H_h(\ap,s)}{\ap-s}F(\bp,s)g(s)\diff s} &&\\
& \lesssim \norm[\Ltwo(\diff \bp)]{\norm[\Ltwo(\diff \ap)]{\int \frac{H_h(\ap,s)}{\ap -s} F(\bp,s) g(s)\diff s}} \\
&  \lesssim \norm*[\infty]{\htilap -1}\norm[2]{g}\norm*[\Big][\Ltwo(\diff \bp)]{\norm[\Linfty]{F(\bp,\cdot)}} \\
& \lesssim \norm*[\infty]{\htilap -1}\norm[2]{f'}\norm[2]{g}
\end{align*}
Similarly the third term is controlled as
\begin{align*}
\lpar \qq & \norm[\Ltwo(\Rsp\times\Rsp, \diff\ap \diff\bp)]{\frac{H(\ap,\bp)}{i\pi}\int \frac{F(\bp,s)}{\htil(\ap)-\htil(s)}g(s) \diff s} && \\
& \lesssim \norm*[\infty]{\htilap -1}\norm[\Ltwo(\diff \bp)]{\norm[\Ltwo(\diff \ap)]{\int \frac{F(\bp,s)}{\htil(\ap) - \htil(s)}  g(s)\diff s}} \\
&  \lesssim \norm*[\infty]{\htilap -1}\norm[2]{g}\norm*[\Big][\Ltwo(\diff \bp)]{\norm[\Linfty]{F(\bp,\cdot)}} \\
& \lesssim \norm*[\infty]{\htilap -1}\norm[2]{f'}\norm[2]{g}
\end{align*}
and the last term is controlled similarly to the second term. Hence we have the required estimate.

\item From estimate 5 of this proposition we see that the operator $T: g \mapsto \sqbrac*{f,\Hil - \Htil}\pap(g)$ is bounded on $\Ltwo$ with $\norm[\Ltwo \to \Ltwo]{T} \lesssim  \norm*[\infty]{\htilap -1}\norm[\infty]{f'}$ and that $T(1)=0$. It is also easy to see that its kernel satisfies the conditions of  \propref{prop:Lemarie} and hence the estimate follows.   

\item This is proved in exactly the same was as we proved estimate 10 of this proposition. We leave the details to the reader. 

\item Consider the operator $T: f_3 \mapsto \sqbrac*{f_1,f_2; \pap f_3} - \sqbrac*{f_1,f_2; \pap f_3}_{\htil} $. We observe that
\begin{align*}
& T(f_3)(\ap) \\
& = - \frac{1}{i\pi}\int \pbp \cbrac{\frac{f_1(\ap) - f_1(\bp)}{\ap - \bp}\frac{f_2(\ap) - f_2(\bp)}{\ap - \bp} - \frac{f_1(\ap) - f_1(\bp)}{\htil(\ap) - \htil(\bp)}\frac{f_2(\ap) - f_2(\bp)}{\htil(\ap) - \htil(\bp)} }f_3(\bp) \diff \bp
\end{align*}
Now by repeated use of \corref{cor:Coifman} we obtain $\norm[\Ltwo \to \Ltwo]{T} \lesssim  \norm*[\infty]{\htilap -1}\norm[\infty]{f_1'}\norm[\infty]{f_2'}$. 

\item We again consider the operator $T: f_3 \mapsto \sqbrac*{f_1,f_2; \pap f_3} - \sqbrac*{f_1,f_2; \pap f_3}_{\htil} $ and from the previous estimate we have $\norm[\Ltwo \to \Ltwo]{T} \lesssim  \norm*[\infty]{\htilap -1}\norm[\infty]{f_1'}\norm[\infty]{f_2'}$. We observe that $T(1) = 0 $ and it is also easy to see that its kernel satisfies the conditions of \propref{prop:Lemarie}. Hence the required estimate follows.

\end{enumerate}
\end{proof}


\bibliographystyle{amsplain}
\bibliography{Main.bib}

\end{document}